\documentclass[fleqn,a4paper,12pt]{article}
\usepackage{amssymb,amsfonts,amsmath,amsthm,bbm,graphicx}

\usepackage[dvips]{epsfig}

\newtheorem{theorem}{Theorem}[section] 
\newtheorem{lemma}[theorem]{Lemma}     
\newtheorem{corollary}[theorem]{Corollary}
\newtheorem{proposition}[theorem]{Proposition}

\makeatletter
\def\old@comma{,}
\catcode`\,=13
\def,{%
  \ifmmode%
    \old@comma\discretionary{}{}{}%
  \else%
    \old@comma%
  \fi%
}
\makeatother

\newcommand{\bs}{\boldsymbol}

\renewcommand{\mod}{\mathrm{mod \,}}

\begin{document}


\begin{center}
{\Large\textbf{The infinite Fibonacci groups and relative asphericity }}

M.\ EDJVET\\
\textit{School of Mathematical Sciences, The University of Nottingham, University Park, Nottingham NG7 2RD}

A.\ JUH\'{A}SZ\\
\textit{Department of Mathematics, Technion--Israel Institute of Technology, Haifa 32000, Israel}\\

\bigskip

\textit{Dedicated to David L.\ Johnson}
\end{center}

\medskip

\begin{abstract}
We prove that the generalised Fibonacci group $F(r,n)$ is infinite for $(r,n) \in \{ (7+5k,5), (8+5k,5) \colon k \geq 0 \}$.  
This together with previously known results yields a complete classification of the finite $F(r,n)$, a problem that has its origins in a question by J H Conway in 1965.  
The method is to show that a related relative presentation is aspherical from which it can be deduced that the groups are infinite.

\medskip

\textbf{2000 Mathematics Subject Classification}: 20F05, 57M07.

\textbf{Keywords}: Fibonacci groups; relative asphericity.
\end{abstract}

\section{Introduction}
The generalised Fibonacci group $F(r,n)$ is the group defined by the \textit{cyclic presentation}
\[
\langle x_1, \ldots, x_n \mid x_1 x_2 \ldots x_r x_{r+1}^{-1}, x_2 x_3 \ldots x_{r+1} x_{r+2}^{-1}, \ldots , x_{n-1} x_n x_1 \ldots x_{r-2} x_{r-1}^{-1} ,
x_n x_1 x_2 \ldots x_{r-1} x_r^{-1} \rangle ,
\]
where $r > 1$, $n > 1$.
Thus there are $n$ generators and $n$ relators each of length $r+1$ and each relator is obtained from the first relator by cyclically permuting the subscripts and reducing modulo $n$ [10, Section 7.3].
There has been a great deal of interest in the study of these groups since the question in [5] by Conway about the order of $F(2,5)$.  Up to now the order of $F(r,n)$ was known except for the two infinite families $F\{ 7,5 \}$ and $F \{ 8,5 \}$ where $F\{r,n\} := \{ F(r+kn,n) \colon k \geq 0 \}$.  The reader is referred to [15] and the references therein together with [4] and [14] for further details.  In this paper we will show that each group in $F \{ 7,5 \}$ or $F \{ 8,5 \}$ is infinite.  This together with previous results yields the following theorem.

\begin{theorem}
\textit{The generalised Fibonacci group $F(r,n)$ is finite if and only if one of the following conditions is satisfied:}
\begin{enumerate}
\item[(i)]
$r=2$ \textit{and} $n \in \{ 2,3,4,5,7 \}$: \textit{indeed $F(2,2)$ is trivial}; $F(2,3) \cong Q_8$, the quaternion group of order 8; $F(2,4) \cong \mathbb{Z}_5$; $F(2,5) \cong \mathbb{Z}_{11}$;
and $F(2,7) \cong \mathbb{Z}_{29}$;
\item[(ii)]
$r=3$ \textit{and} $n \in \{ 2,3,5,6 \}$: \textit{indeed} $F(3,2) \cong Q_8$; $F(3,3) \cong \mathbb{Z}_2$; $F(3,5) \cong \mathbb{Z}_{22}$; \textit{and} $F(3,6)$ \textit{is non-metacyclic, soluble of order} 1512;
\item[(iii)]
$r \geq 4$ \textit{and} $r \equiv 0$ ($\textrm{mod}\, n$), \textit{in which case} $F(r,n) \cong \mathbb{Z}_{r-1}$;
\item[(iv)]
$r \geq 4$ \textit{and} $r \equiv 1$ ($\textrm{mod} \, n$), \textit{in which case} $F(r,n)$ \textit{is metacyclic of order} $r^n -1$;
\item[(v)]
$r \geq 4$, $n=4$ \textit{and} $r \equiv 2$ ($\textrm{mod}\, n$), \textit{in which case} $F(r,n)=F(4k+2,4)$ ($k \geq 1$) \textit{is metacyclic of order} $(4k+1)(2(4)^{2k}+2(-4)^k +1)$.
\end{enumerate}
\end{theorem}

A \textit{relative group presentation} is a presentation of the form $\mathcal{P}=\langle G, \bs{x} | \bs{r} \rangle$ where $G$ is a group, $\bs{x}$ a set disjoint from $G$, and $\bs{r}$ a set of cyclically reduced words in the free product $G \ast \langle \bs{x} \rangle$ where $\langle \bs{x} \rangle$ denotes the free group on $\bs{x}$ [2].  If $G(\mathcal{P})$ denotes the group defined by $\mathcal{P}$ then $G(\mathcal{P})$ is the quotient group
$G \ast \langle \bs{x} \rangle / N$ where $N$ denotes the normal closure in $G \ast \langle \bs{x} \rangle$ of $\bs{r}$.  A relative group presentation is defined in [2] to be \textit{aspherical} if every spherical picture over it contains a dipole, that is, fails to be reduced.  There is interest in when a relative presentation is aspherical, see, for example, [1], [2], [6], [7], [9] and [13].  In this paper we consider the situation when $G= \langle t | t^5 \rangle$, $\bs{x} = \{ u \}$ and $\bs{r} = \{ t^2 utu^{-n} \}$ and prove the following theorem.

\begin{theorem} 
\textit{The relative presentation $\mathcal{P}_n = \langle t,u | t^5, t^2 utu^{-n} \rangle$ is aspherical for $n \geq 7$.}
\end{theorem}

Applying, for example, statement (0.4) in the introduction of [2] and the fact that the group defined by $\mathcal{P}_n$ is neither trivial nor cyclic of order 5 we immediately obtain

\begin{corollary} 
\textit{If $G(\mathcal{P}_n )$ is the group defined by $\mathcal{P}_n$ then $G(\mathcal{P}_n)$ is infinite for $n \geq 7$, indeed $u$ has infinite order in $G(\mathcal{P}_n)$ for $n \geq 7$.}
\end{corollary}

We shall show in Section 2 that Corollary 1.3 implies that each group in $F\{7,5\},F\{8,5\}$ is infinite.
The remaining Sections 3--11 of the paper will be devoted entirely to proving Theorem 1.2.

\medskip

\textbf{Acknowledgement}. The authors wish to thank Chris Chalk for his many helpful comments concerning this paper.


\section{Fibonacci groups}

Consider the generalised Fibonacci group $F(r,n)$ of the introduction.  If $r=2$ or $2 \leq n \leq 4$
or $(r,n) \in \{(3,5),(3,6)\}$
or $n$ divides $r$ or $r \equiv 1$ ($\textrm{mod}\,n$) then Theorem 1.1 applies and these cases are discussed fully with relevant references in [15].  Assume then that none of these conditions holds.  In particular $r \geq 3$ and $n \geq 5$.
In [14] it is shown that if $n$ does not divide any of $r \pm 1$, $r+2$, $2r$, $2r+1$ or $3r$ then $F(r,n)$ is infinite.  If $n$ divides $r+1$ then $F(r,n)$ is infinite for $r \geq 3$ [11] so assume otherwise.  We are left therefore to consider the families $F\{ r,r+2 \}$; $F \{ r,2r \}$; $F\{ r,2r+1 \}$ and
$F \{ r,3r \}$.  In [4] it is shown that if $r \geq 4$ then each group in $F \{ r,r+2 \}$ and $F\{ r,2r \}$ is infinite; and if $r \geq 3$ then each group in $F \{ r,2r+1 \}$ is infinite.  This leaves $F \{ 8,5 \}$, $F \{ 9,6 \}$, $F\{7,5 \}$ and $F \{ r,3r \}$.
In [14] it is also shown that if $n$ does not divide any of $r \pm 1$, $r \pm 2$, $r + 3$, $2r$, $2r+1$ then $F(r,n)$ is infinite.  If $n$ divides $3r$ and $r+2$ we obtain the family $F \{ 4,6 \}$ which is $F \{ r,r+2 \}$ for $r=4$; if $n$ divides $3r$ and $r-2$ we obtain $F \{ 8,6 \}$ and each group in this family is infinite [3]; and if $n$ divides $3r$ and $r+3$ we obtain $F \{ 6,9 \}$.  By our assumptions $n$ does not divide $3r$ together with any of $r \pm 1$, $2r$ or $2r+1$.  It is also shown in [4] that each group in $F\{ 9,6 \}$ or $F \{ 6,9 \}$ is infinite, all of which leaves $F \{ 7,5 \}$ and $F \{ 8,5 \}$.  These families are
\[
\{ F(7+5k,5) \colon k \geq 0 \} \quad \text{and} \quad \{ F(8+5k,5) \colon k \geq 0 \} ,
\]
where $F(7+5k,5)$ and $F(8+5k,5)$ are defined respectively by the presentations
\[
\begin{array}{l}
\langle x_1,x_2,x_3,x_4,x_5 \mid (x_1 x_2 x_3 x_4 x_5)^{k+1} x_1 x_2 x_3^{-1}, \ldots , (x_5 x_1 x_2 x_3 x_4)^{k+1} x_5 x_1 x_2^{-1} \rangle ,\\
\langle x_1,x_2,x_3,x_4,x_5 \mid (x_1 x_2 x_3 x_4 x_5)^{k+1} x_1 x_2 x_3 x_4^{-1}, \ldots , (x_5 x_1 x_2 x_3 x_4)^{k+1} x_5 x_1 x_2 x_3^{-1} \rangle .
\end{array}
\]
We show how Corollary 1.3 can be used to prove Theorem 1.1.
Since cyclically permuting the generators induces an automorphism we can form a semi-direct product with the cyclic group of order 5 in the way described, for example, in [10, Section 10.2] and this yields the groups $E(7+5k,5)$ 
and 
$E(8+5k,5)$ defined respectively by the presentations
\[
\begin{array}{l}
\langle x,t \mid t^5, (xt^{-1} ) ^{7+5k} x^{-1} t^2 \rangle ,\\
\langle x,t \mid t^5, (xt^{-1})^{8+5k} x^{-1} t^3 \rangle .
\end{array}
\]

\noindent Now
\begin{align*}
\langle x,t \mid t^5, (xt^{-1})^{7+5k} x^{-1} t^2 \rangle &= \langle x,t,y \mid t^5, (xt^{-1})^{7+5k} x^{-1} t^2, y^{-1} xt^{-1} \rangle\\
&=\langle y,t \mid t^5, y^{7+5k} t^{-1} y^{-1} t^2 \rangle\\
&=\langle y,t \mid t^5, y^{7+5k} ty^{-1} t^3 \rangle \quad (\text{replacing } t \text{ by } t^{-1})\\
&=\langle y,t,s \mid t^5, y^{7+5k} ty^{-1} t^3, st^{-3} \rangle \quad (s=t^3)\\
&=\langle y,s \mid s^5, y^{7+5k} s^2 y^{-1} s \rangle \quad (s^2=t^6=t)\\
&=\langle u,t \mid t^5, t^2 utu^{-(7+5k)} \rangle \quad (s \leftrightarrow t,\,y=u^{-1}) \quad (\text{cyclic conjugate})
\end{align*}
and
\begin{align*}
\langle x,t \mid t^5, (xt^{-1})^{8+5k} x^{-1} t^3 \rangle &=
\langle x,t,y \mid t^5, (xt^{-1})^{8+5k} x^{-1} t^3, y^{-1} xt^{-1} \rangle\\
&= \langle t,y \mid t^5, y^{8+5k} t^{-1} y^{-1} t^3 \rangle \\
&=\langle u,t \mid t^5, t^2 utu^{-(8+5k)} \rangle \quad (\text{inverse},~t^{-3}=t^2,y=u) .
\end{align*}
Therefore Corollary 1.3 implies that each group in $\{ E(7+5k,5) \text{ and } E(8+5k,5) \colon k \geq 0 \}$ is infinite and, given this, Theorem 1.1 now follows.


\section{The amended picture and curvature}

The reader is referred to [2] and [12] for definitions of many of the basic terms used in this and subsequent sections.

Suppose by way of contradiction that the relative presentation
\[
\mathcal{P}_n = \langle t,u \mid t^5, t^2 utu^{-n} \rangle \quad (n \geq 7)
\]
is not aspherical, that is, there exists a \textit{reduced spherical picture} $\bs{P}$ over $\mathcal{P}_n$.  Then each \textit{arc} of $\bs{P}$ is equipped with a normal orientation and labelled by an element of $\{ u,u^{-1} \}$; each \textit{corner} of $\bs{P}$ is labelled by an element of
$\{ t^i \colon -2 \leq i \leq 2 \}$; reading the labels clockwise on the corners and arcs at a given vertex yields $t^2 utu^{-n}$ (up to cyclic permutation and inversion); and, since $t$ has order 5 in 
$G(\mathcal{P}_n )$, the product of the sequence of corner labels
encountered in an anti-clockwise traversal of any given \textit{region} of $\bs{P}$ yields the identity in $G = \langle t \mid t^5 \rangle$. 

Now let $\bs{D}$ be the dual of the picture $\bs{P}$ with the labelling of
$\bs{D}$ inherited from that of $\bs{P}$.  Then $\bs{D}$ is a (spherical) diagram such that: each corner label of $\bs{D}$ is $t^i$ where $-2 \leq i \leq 2$; and each edge is oriented and labelled $u$ or $u^{-1}$.
For convenience we adopt the notation
\[
buau^{-1} \lambda_{n-1}^{-1} u^{-1} \lambda_{n-2}^{-1} \ldots u^{-1} \lambda_1^{-1} u^{-1}
\]
for $t^2utu^{-n}$.  Thus $a=t^1$; $b=t^2$; and $\lambda_i=t^0$ ($1 \leq i \leq n-1$).
Each (oriented) region $\Delta$ of $\bs{D}$ is given (up to cyclic permutation and inversion) by Figure 1(i); and an example of how the regions are oriented is illustrated by Figure 1(iii).  (In subsequent figures we will not show the orientation of the regions or edges nor the edge labels $u,u^{-1}$.)
Note that the sum of the powers of $t$ read around any given vertex of $\bs{D}$ is congruent to 0 modulo 5.

For ease of presentation and to simplify matters further we will use $\lambda$ to denote $\lambda_i$ and $\mu$ to denote
$\lambda_j^{-1}$ ($1 \leq i,j \leq n-1$) throughout what follows.  This way the
\textit{star graph} $\Gamma$ for $\bs{D}$ is given by Figure 1(ii) with the the understanding that the edges labelled $\lambda$ and $\mu$ in $\Gamma$ are traversed only in the direction indicated.
Thus the edge labelled $\lambda$ represents the $n-1$ edges, labelled $\lambda_i$; and $\mu$ represents the $n-1$ inverse edges.  Recall that the vertex labels in $\bs{D}$ yield closed admissible paths in $\Gamma$.

We can make the following assumptions without any loss of generality:
\begin{enumerate}
\item[\textbf{A1}.]
$\bs{D}$ is minimal with respect to number of regions and so, in particular, is \textit{reduced}.
\item[\textbf{A2}.]
Subject to \textbf{A1}, $\bs{D}$ is maximal with respect to number of vertices of degree 2.
\end{enumerate}

\newpage
\begin{figure}
\begin{center}
\psfig{file=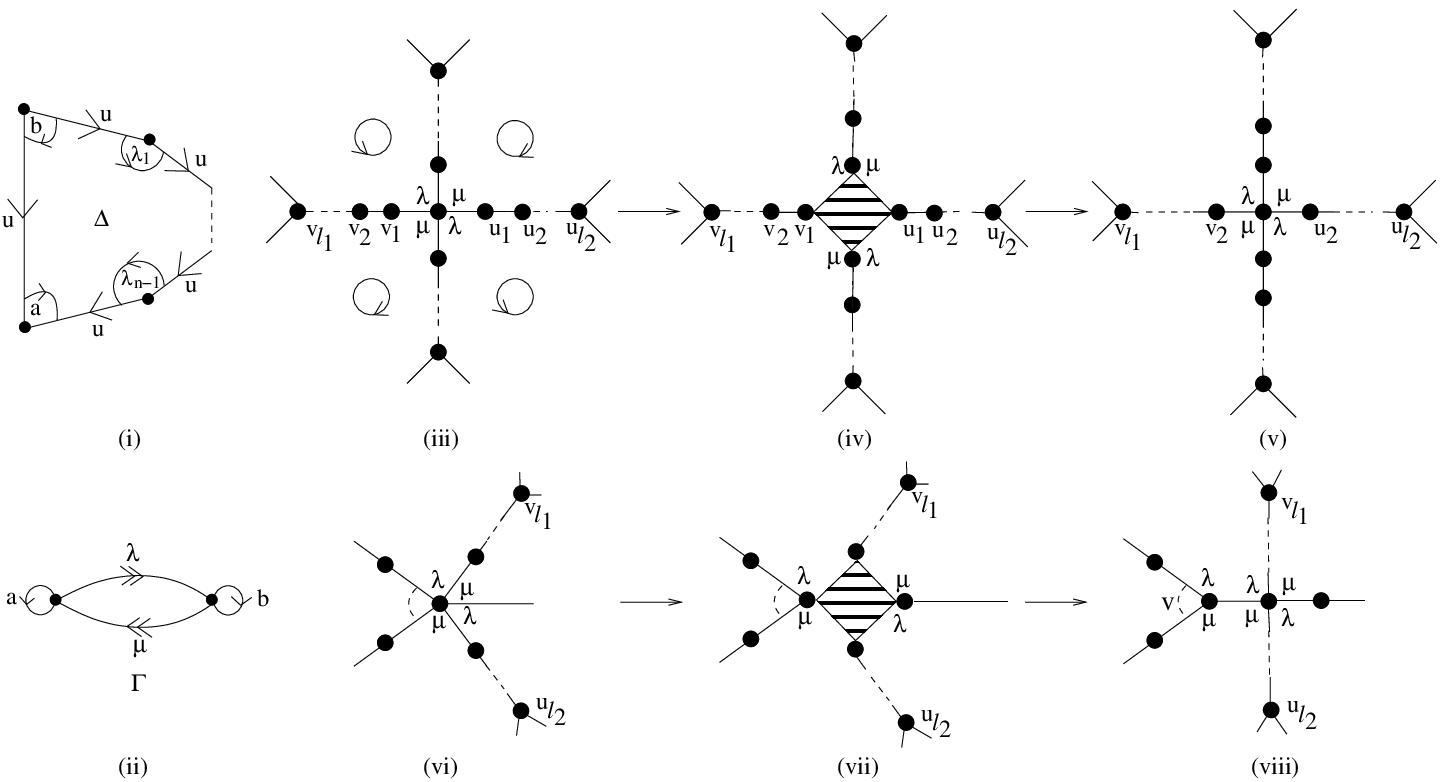}
\end{center}
\caption{ }
\end{figure}

We introduce some further notation.  If $v$ is a vertex of $\bs{D}$ then $l(v)$, the \textit{label} of $v$, is the cyclic word obtained from the corner labels of $v$ in a \textit{clockwise} direction; and $d(v)$ denotes the \textit{degree} of $v$. A $(v_1,v_2)$\textit{-edge} is an edge with
endpoints $v_1$ and $v_2$; and an edge is a $(\theta_1,\theta_2)$\textit{-edge relative to the region} $\Delta$ if its corner labels in
$\Delta$ are, in no particular order, $\theta_1$ and $\theta_2$. (Sometimes we will simply talk of a $(\theta_1,\theta_2)$-edge with the understanding that the corner labels are either $\theta_1$ and $\theta_2$ or
$\theta_1^{-1}$ and $\theta_2^{-1}$.)

\begin{lemma}
\textit{If $v$ is a vertex of $\bs{D}$ then $l(v) \neq ( \lambda \mu )^{\pm k}$ for $k \geq 2$.}
\end{lemma}

\textit{Proof}.  The proof is by induction on $k$. Consider the vertex of Figure 1(iii) having label $( \lambda \mu )^2$.  Apply $m=\min \{ l_1,l_2 \}$ bridge moves of the type shown in Figure 1(iv)-(v).  Then each of the first $m-1$ bridge moves will create and destroy two vertices of degree 2, leaving the total number unchanged.  The $m$th bridge move however will create two vertices of degree 2
but destroy at most one.  Since bridge moves leave the total number of regions unchanged we obtain a contradiction to assumption \textbf{A2}.  Now consider the vertex of Figure 1(vi) having label $( \lambda \mu )^k$ where $k \geq 3$.  Again apply $m=\min \{ l_1,l_2 \}$ bridge moves of the type shown in Figure 1(vii)-(viii).  The first such bridge move may decrease the total number of vertices of degree 2 by one, each subsequent bridge move creates two and destroys two until the $m$th bridge which increases the number by at least one.  This produces a new diagram with at most the same number of vertices of degree 2 as $\bs{D}$.  But applying an induction argument to the vertex $v$ of Figure 1(viii) where $l(v)=( \lambda \mu )^{k-1}$ will yield a contradiction to \textbf{A2} as before. $\Box$

\begin{lemma}
\textit{Let $v \in \bs{D}$.  (i) If $d(v)=2$ then $l(v)=(\lambda \mu )^{\pm 1}$ and
(ii) if $d(v) > 2$ then $l(v)$ contains at least three occurrences of} $a^{\pm 1}, b^{\pm 1}$.
\end{lemma}

\newpage
\begin{figure}
\begin{center}
\psfig{file=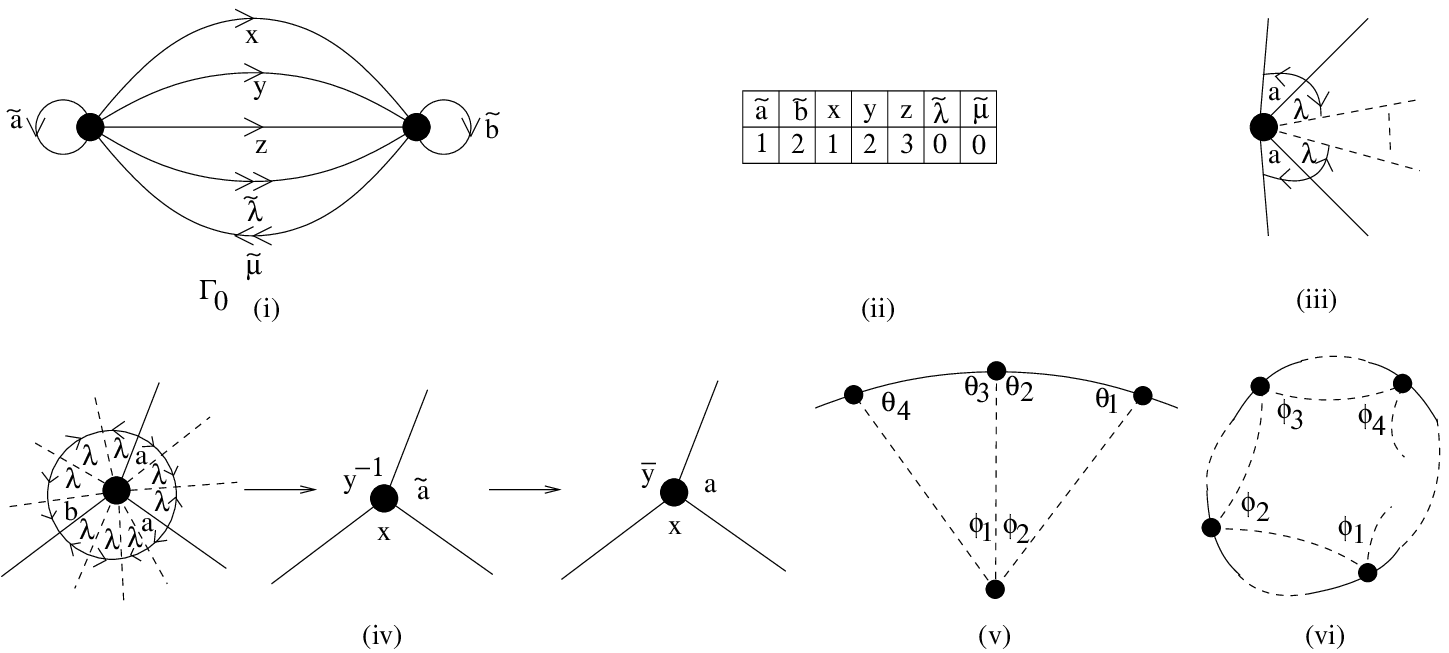}
\end{center}
\caption{}
\end{figure}

\textit{Proof}.  Both statements follow from the fact that the sum of the corner labels is congruent to $0 \, \mod 5$ together with Lemma 3.1 for (ii)
and the fact that, since $\bs{D}$ is reduced, no adjacent corner labels are inverse to each other (that is, no sublabels of the form $aa^{-1},bb^{-1},\lambda_i\lambda_i^{-1}$). $\Box$

\smallskip

We amend $\bs{D}$ as follows.  Delete all vertices of degree 2 and remove all edges that are not $(b,a)$-edges relative to any region, in so doing merging the adjacent regions.  This results in the diagram $\bs{K}$.  

\begin{lemma}
\textit{If $v \in \bs{K}$ then} $d(v) \geq 3$.
\end{lemma}

\textit{Proof}. We see from Lemma 3.2 that  $l(v)$ contains at least three occurrences of $a^{\pm 1}, b^{\pm 1}$ and each such occurrence contributes uniquely to $d(v)$. $\Box$

\smallskip

We claim that 
$\bs{K}$ contains a subdiagram $\bs{K}_0$ with the following properties: (1) $\bs{K}_0$ has connected interior and is simply connected; and (2) every connected component of $\bs{K}_0 \setminus \bs{K}_0^{(1)}$ is homeomorphic to an 
open disc.  If $\bs{K}$ satisfies these two properties then take $\bs{K}=\bs{K}_0$.  If not then $\bs{K} \backslash \bs{K}^{(1)}$ has a connected component $\bs{L}_1$ satisfying (1) which fails to satisfy (2).  (The merging of 
regions may produce open annuli.)  Consider $\bs{K} \backslash \bs{L}_1$.  It is the disjoint union of subdiagrams at least one of which $\bs{L}_2$, say, satisfies (1).  If $\bs{L}_2$ satisfies (2) then put $\bs{L}_2 = \bs{K}_0$; if not then repeat the argument with $\bs{L}_2 \backslash \bs{L}_2^{(1)}$ instead of $\bs{K} \backslash \bs{K}^{(1)}$ and so on.  This
procedure will terminate in a finite number of steps with a subdiagram $\bs{K}_0$ satisfying conditions (1) and (2).
Observe that if $\Delta$ is a region of $\bs{K}_0$ then it follows from Lemma 3.1 that any 2-segment in $\Delta$ when regarded as a region of $\bs{D}$ will have its endpoints on $\partial \Delta$.  (Recall that a \textit{2-segment}
is a segment where endpoints have degree $>2$ and whose intermediate vertices each have degree 2.)

\newpage
\begin{figure}
\begin{center}
\psfig{file=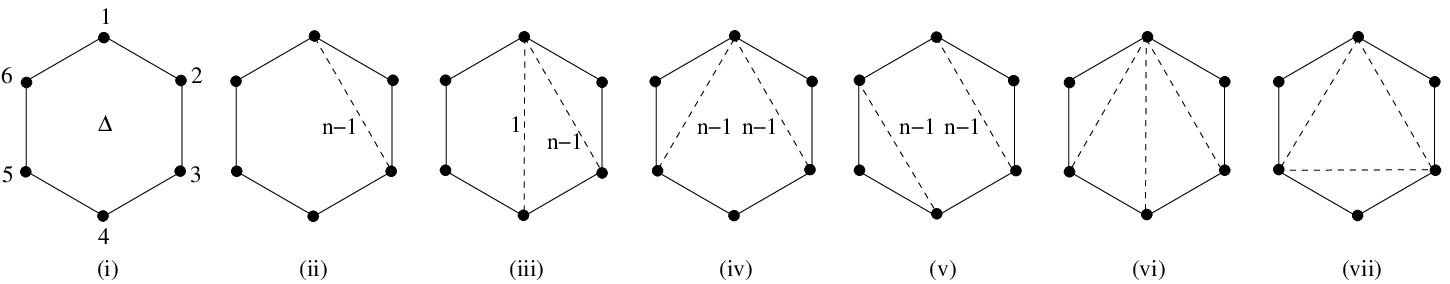}
\end{center}
\caption{}  
\end{figure}

The corner labels of $\bs{K}_0$ are obtained by taking the product of the corner labels of $\bs{D}$ used in forming each corner of $\bs{K}_0$.  (An example is shown in Figure 2(iv).)  
Since each corner of $\bs{K}_0$
is between two $(b,a)$-edges it follows that the corner labels of $\bs{K}_0$ are:
\begin{equation}
\renewcommand{\arraystretch}{1.25}
\begin{array}{ll}
\tilde{a} = a( \lambda \mu )^{k_1} & \text{(odd length)}\\
\tilde{b} = ( \mu \lambda )^{k_2} b & \text{(odd length)}\\
\tilde{\lambda} = ( \lambda \mu )^{k_3} \lambda & \text{(odd length)}\\
x = \tilde{a} \lambda & \text{(even length)}\\
y = \lambda \tilde{b} & \text{(even length)}\\
z = \tilde{a} \lambda \tilde{b} & \text{(odd length)}
\end{array}
\renewcommand{\arraystretch}{1}
\end{equation}
where $k_i \geq 0$ ($1 \leq i \leq 3$).  The star graph $\Gamma_0$ for $\bs{K}_0$ is given by Figure 2(i), and the table in Figure 2(ii) gives the power of $t$ each corner label represents.  Observe that $( \lambda \mu )^k$ for $k \geq 1$ cannot be a corner label in $\bs{K}_0$
for otherwise $\bs{D}$ would contain a subdiagram of the form shown
in Figure 2(iii) and this contradicts \textbf{A1} since after \textit{bridge moves} and cancellation it would be possible to reduce the number of regions of $\bs{D}$ by at least two.

\begin{lemma}
\textit{Let $v \in \bs{K}_0$.  If $d(v) < 6$ then $l(v)$ is one of the following:}
\[
\renewcommand{\arraystretch}{1.5}
\begin{array}{rlllll}
(\textrm{i})&d(v)=3:&\tilde{a} xy^{-1}&\tilde{b} \tilde{\mu} z\\
(\textrm{ii})&d(v)=4:&\tilde{a} \tilde{a} z \tilde{\mu}&\tilde{b} \tilde{b} x^{-1} y\\
(\textrm{iii})&d(v)=5:&\tilde{a} \tilde{a} \tilde{a} \tilde{a} \tilde{a}&\tilde{b} \tilde{b} \tilde{b} \tilde{b} \tilde{b}&\tilde{a} zx^{-1} y \tilde{\mu}&\tilde{b} x^{-1} \tilde{\lambda} z^{-1} y
\end{array}
\renewcommand{\arraystretch}{1}
\]
\end{lemma}

\textit{Proof}.  This follows from checking all reduced closed paths in $\Gamma _0$ whose exponent sum is 0 modulo 5 together with the fact that equations (3.1) can be used to show that the following paths of length 2 together with their inverses do not occur as sublabels:
$\tilde{a} \tilde{\lambda}$; $\tilde{a} y$; $\tilde{a}^{-1} x$; $\tilde{a}^{-1} z$; $\tilde{b} y^{-1}$; $\tilde{b} z^{-1}$; $\tilde{b}^{-1} x^{-1}$;
$\tilde{b}^{-1} \tilde{\mu}$; $\tilde{\lambda} x^{-1}$; $\tilde{\lambda} \mu$; $\tilde{\mu} y$; $\tilde{\mu} \tilde{\lambda}$; $x^{-1} z$; $yz^{-1}$.
For example $\tilde{a} \tilde{\lambda} = a( \lambda \mu )^{k_1} ( \lambda \mu )^{k_2} \lambda = a(\lambda \mu)^{k_1+k_2} \lambda = \tilde{a} \lambda = x$ after rewriting using equations (3.1).
Indeed if $\tilde{a}$ and $\tilde{\lambda}$ are adjacent corner labels then this would imply the existence of a non($b,a$)-edge in $\bs{K}_0$. $\Box$

\newpage
\begin{figure}
\begin{center}
\psfig{file=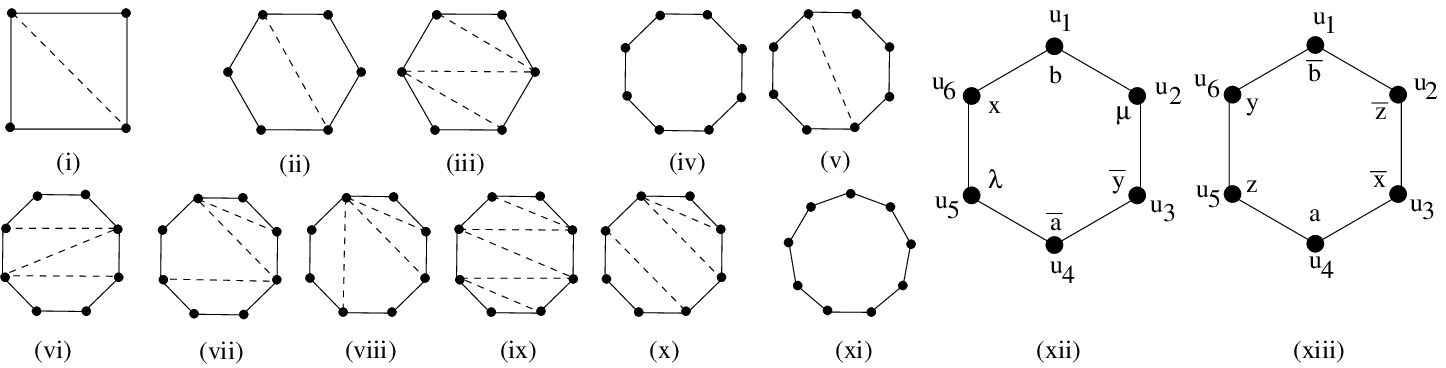}
\end{center}
\caption{}  
\end{figure}

\textbf{Convention}:  We will usually write $a,b,\lambda ,\mu$ for $\tilde{a}, \tilde{b}, \tilde{\lambda}, \tilde{\mu}$ simply for ease of presentation.
For example if $v \in \bs{D}$ has label $l(v)=a \lambda \mu a \lambda \mu \lambda b^{-1} \mu \lambda \mu$ then in $\bs{K}_0$ this transforms uniquely 
to $(a \lambda \mu )(a \lambda \mu \lambda )(b^{-1} \mu \lambda \mu ) = \tilde{a} xy^{-1}$ which we write as $axy^{-1}$ or as $a x \bar{y}$ in the figures. This is illustrated in Figure 2(iv).
Moreover, when drawing diagrams we use $\bar{\theta}$ for $\theta^{-1}$ where $\theta \in \{ a,b,x,y,z \}$.

\smallskip

We turn now to the regions of $\bs{K}_0$.  The edges or 2-segments deleted in forming $\bs{K}$ from $\bs{D}$ will be 
referred to as \textit{shadow edges} and will usually be denoted by dotted edges in our figures.  The number of edges in a 2-segment will be called its \textit{length}.  
Much use will be made of the fact that \textit{the number of edges in a region of $\bs{D}$ is} $n+1$.  
By \textit{length contradiction} we mean either a contradiction to this fact or to the fact that $n \geq 7$.

We will also use the fact that no region of $\bs{K}_0$ can contain the configuration of edges and shadow edges shown in Figure 2(v)-(vi).  To see this observe in Figure 2(v) that $\{ \phi_1, \phi_2 \} \subseteq \{ \lambda , \mu \}$ forcing each $\theta_i \in \{ a^{\pm 1}, b^{\pm 1} \}$ and any attempt at labelling forces $\theta_2 \theta_3 = aa^{-1}$ or $bb^{-1}$, a contradiction to $\bs{D}$ being reduced.  In Figure 2(vi) each
$\phi_i \in \{ \lambda , \mu \}$ and this produces a region in $\bs{D}$ without corner label $a^{\pm 1}$ or $b^{\pm 1}$.  We refer to the existence of each of these situations as a \textit{basic labelling contradiction}.

If $\Delta$ is a region of $\bs{K}_0$ then $d(\Delta)$ denotes the degree of $\Delta$, that is, the number of sides $\Delta$ has.
For example suppose that $\Delta \in \bs{K}_0$ and $d(\Delta)=6$.  If $\Delta$ contains no shadow edges as in Figure 3(i) then we obtain the length contradiction $n+1=6$.
Let $(pq)$ denote the shadow edge with endpoints $p$ and $q$.  If $\Delta$ contains exactly one shadow edge $e$ then, working modulo cyclic permutation and inversion, $e \in \{ (13),(14)\}$.
But $e=(13)$ yields the length contradiction $n+1=n+3$ as shown in Figure 3(ii) since the length of (13) must be $n-1$.
If $\Delta$ contains exactly two shadow edges $e_1$ and $e_2$ then $(e_1,e_2) \in \{( (13),(14) ),( (13),(15) ), ( (13),(46) )\}$.  But $(e_1,e_2)=( (13),(14) )$ yields the length
contradiction $n+1=4$; and $(e_1,e_2)=( (13),(15) )$ or
$( (13),(46) )$ implies $n+1=2n$ (see Figure 3(iii)-(v)).  Finally if $\Delta$ contains three shadow
edges $e_1$, $e_2$ and $e_3$ then $(e_1,e_2,e_3)=((13),(14),(15))$ or $((13),(15),(35))$ yielding a basic labelling contradiction (see Figure 3(vi)-(vii)); or $(e_1,e_2,e_3)=((13),(14),(46))$.

Similar elementary but somewhat lengthy arguments can be used to prove the following.
Full details are given in the Appendix.

\begin{lemma}
\textit{Let $\Delta$ be a region of $\bs{K}_0$.  If $d(\Delta) \leq 9$ then $d(\Delta) \in \{ 4,6,8,9 \}$ and $\Delta$ is given by Figure 4(i)-(xi).}
\end{lemma}

For example it follows from Lemma 3.5 that if $d(\Delta)=6$ then up to cyclic permutation and inversion $\Delta$ is given by Figure 4(xii)-(xiii).
In particular, if $\Delta$ contains an $(a,b)$-edge or $(x,y)$-edge then $d(\Delta) \geq 8$.

We will use similar curvature arguments to those used in [8].
Briefly, each corner at a vertex of degree $d$ is given the angle $\frac{2 \pi}{d}$ and so the curvature of each vertex is 0.
Thus, if $\Delta$ is an $m$-gon of $\bs{K}_0$ and the degrees of the vertices of $\Delta$ are $d_i$ ($1 \leq i \leq m$), then the \textit{curvature} of $\Delta$ is given by
\begin{equation}
c(\Delta)=c(d_1,\ldots,d_m)=(2-m) \pi + 2 \pi \sum_{i=1}^m \frac{1}{d_i}.
\end{equation}
(Observe that if $\rho$ is any permutation of $\{ 1, \ldots, m \}$ then $c(\Delta)=c(d_{\rho (1)}, \ldots, d_{\rho (m)} )$.  This fact will be used throughout without explicit reference.)
A list of some of the $c(d_1,\ldots,d_m)$ used in the paper is given in Tables 1 and 2 below for the reader's benefit.

\begin{table}[h]
\[
\renewcommand{\arraystretch}{1.5}
\begin{array}{|l|l|l|}
\hline
c(3,3,3,3)=\frac{2 \pi}{3}&c(3,3,5,5)=\frac{2 \pi}{15}&c(3,4,5,5)=-\frac{\pi}{30}\\
\hline
c(3,3,3,4)=\frac{\pi}{2}&c(3,3,5,6)=\frac{\pi}{15}&c(3,4,5,6)=-\frac{\pi}{10}\\
\hline
c(3,3,3,5)=\frac{2 \pi}{5}&c(3,3,5,7)=\frac{2 \pi}{105}&c(3,4,5,7)=-\frac{31 \pi}{210}\\
\hline
c(3,3,3,6)=\frac{\pi}{3}&c(3,3,6,6)=0&c(3,4,6,6)=-\frac{\pi}{6}\\
\hline
c(3,3,4,4)=\frac{\pi}{3}&c(3,4,4,5)=\frac{\pi}{15}&c(3,5,5,5)=-\frac{2 \pi}{15}\\
\hline
c(3,3,4,5)=\frac{7 \pi}{30}&c(3,4,4,6)=0&c(4,4,4,6)=-\frac{\pi}{6}\\
\hline
c(3,3,4,6)=\frac{\pi}{6}&c(3,4,4,7)=-\frac{\pi}{21}&c(4,4,6,6)=-\frac{\pi}{3}\\
\hline
\end{array}
\renewcommand{\arraystretch}{1}
\]
\caption{}
\end{table}

\begin{table}[h]
\[
\renewcommand{\arraystretch}{1.5}
\begin{array}{|l|l|}
\hline
c(3,3,3,3,3,4)=-\frac{\pi}{6}&c(3,3,3,4,4,4)=-\frac{\pi}{2}\\
\hline
c(3,3,3,3,4,4)=-\frac{\pi}{3}&c(3,3,3,4,4,5)=-\frac{3 \pi}{5}\\
\hline
c(3,3,3,3,4,5)=-\frac{13 \pi}{30}&c(3,3,3,4,4,6)=-\frac{2 \pi}{3}\\
\hline
c(3,3,3,3,4,6)=-\frac{\pi}{2}&c(3,3,4,4,4,4)=-\frac{2 \pi}{3}\\
\hline
\end{array}
\renewcommand{\arraystretch}{1}
\]
\caption{}
\end{table}


\section{Proof of Theorem 1.2}

It was assumed by way of contradiction that there is a reduced spherical picture $\bs{P}$ over $\mathcal{P}_n$.
As described in Section 3, the dual $\bs{D}$ of $\bs{P}$ was amended to produce the spherical diagram $\bs{K}$ and then the subdiagram $\bs{K}_0$.

Suppose first that $\bs{K}_0=\bs{K}$.  By Lemma 3.5, $\bs{K}_0$ has no regions of degree 5 or 7 and, since the curvature of the vertices are 0, the total curvature $c(\bs{K}_0)$ is given by
\[
c(\bs{K}_0)=\sum_{d(\Delta)=4} c(\Delta) +
\sum_{d(\Delta)=6} c(\Delta) +
\sum_{d(\Delta) \geq 8} c(\Delta).
\]
Now suppose that $\bs{K}_0 \neq \bs{K}$.  In this case delete all vertices and edges in $\bs{K} \backslash \bs{K}_0$ to produce a spherical diagram $\bs{K}_1$ consisting of the union of $\bs{K}_0$ and a single region $\Delta_0$ 
(which has essentially been obtained by merging all the regions of $\bs{K}$ not in $\bs{K}_0$).  Note that Lemma 3.5 holds for $\bs{K}_1$ and so
\[
c(\bs{K}_1)=\sum_{\substack{d(\Delta)=4\\\Delta \neq \Delta_0}} c(\Delta) +
\sum_{\substack{d(\Delta)=6\\\Delta \neq \Delta_0}} c(\Delta) +
\sum_{\substack{d(\Delta) \geq 8\\\Delta \neq \Delta_0}} c(\Delta) + c(\Delta_0).
\]
An elementary argument using Euler's formula for the sphere shows that $c(\bs{K}_0) = c(\bs{K}_1) = 4 \pi$ and it is this fact we seek to contradict.

The first step, given in detail in Section 5, is to define a \textit{positive curvature distribution scheme} for regions of degree 4.  That is, regions $\Delta \neq \Delta_0$ are located for which $c(\Delta) >0$, and so $d(\Delta)=4$, and $c(\Delta)$ is distributed to near regions $\hat{\Delta}$ of $\Delta$.  (\textbf{Remark}.  Throughout the paper $\Delta$ or $\Delta_i$ will usually be used to denote regions from which positive curvature is initially transferred, and
$\hat{\Delta}, \hat{\Delta}_j$ regions that receive, and possibly distribute further, positive curvature.)

For the region $\hat{\Delta}$ define $c^{\ast}(\hat{\Delta})$ to equal $c(\hat{\Delta})$ plus all the positive curvature $\hat{\Delta}$ receives minus all the positive curvature $\hat{\Delta}$ distributes as a result of the positive curvature distribution scheme that has been defined.

After completion of the first step, the following is proved in Section 6.

\begin{proposition}
\textit{If $\bs{K}_0=\bs{K}$ then $c(\bs{K}_0) \leq \sum_{d(\hat{\Delta}) \geq 6}
c^{\ast} (\hat{\Delta})$; or if $\bs{K}_0 \neq \bs{K}$ then $c(\bs{K}_1) \leq \sum_{\substack{d(\hat{\Delta}) \geq 6\\
\hat{\Delta} \neq \Delta_0}} c^{\ast} (\hat{\Delta}) + c^{\ast} (\Delta_0)$.}
\end{proposition}

The second step of the proof, given in detail in Section 7, is to define a positive curvature distribution scheme for regions $\hat{\Delta}$ of degree 6.  That is, regions $\hat{\Delta} \neq \Delta_0$ of degree 6 are located for which
$c^{\ast} (\hat{\Delta}) > 0$ and $c^{\ast} (\hat{\Delta})$ is distributed to near regions of $\hat{\Delta}$.

After completion of the second step, the following is proved in Section 8.

\begin{proposition}
\textit{If $\bs{K}_0 = \bs{K}$ then $c(\bs{K}_0) \leq \sum_{d(\hat{\Delta}) \geq 8}
c^{\ast} (\hat{\Delta})$; or if $\bs{K}_0 \neq \bs{K}$ then
$c(\bs{K}_1) \leq \sum_{\substack{d(\hat{\Delta}) \geq 8\\\hat{\Delta} \neq \Delta_0}} c^{\ast} (\hat{\Delta}) +
c^{\ast} (\Delta_0)$.}
\end{proposition}

After completion of the first two steps, for the third and final step the following is proved in Sections 10 and 11.

\begin{proposition}
\textit{If $d(\hat{\Delta}) \geq 8$ and $\hat{\Delta} \neq \Delta_0$ then
$c^{\ast}(\hat{\Delta}) \leq 0$.}
\end{proposition}

If $\bs{K}_0 = \bs{K}$ then Theorem 1.2 follows immediately since, combining the above results, we obtain the contradiction
$c(\bs{K}_0) \leq 0$.

If $\bs{K}_0 \neq \bs{K}$ then noting that the positive curvature distribution schemes are exactly the same with the proviso that if at any stage positive curvature is transferred to $\Delta_0$ then it remains with $\Delta_0$, it follows that $c(\bs{K}_1) \leq c^{\ast}(\Delta_0)$.
Let $d(\Delta_0)=k$.  It follows by inspection of steps one and two above (in Sections 5-7) that the maximum amount of curvature any region receives across an edge is $\frac{\pi}{3}$.  Therefore it can be seen from equation (3.2) 
that $c^{\ast}(\Delta_0) \leq (2-k) \pi + k \left( \frac{2
\pi}{3} \right) + k \left( \frac{\pi}{3} \right) = 2 \pi$.  This final contradiction completes the proof of Theorem 1.2.


\section{Distribution of positive curvature from 4-gons}

In this section we will describe the distribution of positive curvature from regions $\Delta \neq \Delta_0$ of the diagram $\bs{K}_0$ such that $c(\Delta) > 0$.  It follows from Lemma 3.5 that 
$d(\Delta)=4$ and $\Delta$ is given by 
Figure 5 with 
neighbouring regions $\hat{\Delta}_i$ ($1 \leq i \leq 4$) and vertices $v_i$ ($1 \leq i \leq 4$) which we fix for the remainder of this section.  There are 15 cases to consider according to which vertices of $\Delta$ have degree 3.
Our approach will be to consider neighbouring regions of $\Delta$, the valency and labels of their vertices and, if necessary, the neighbours of these also.

There will be exactly fourteen exceptions to the distribution of positive curvature rules given for the 15 cases.  These are contained within six exceptional Configurations A-F and will be fully described later in this section.

\medskip

For the benefit of the reader let us indicate briefly that a general rule for distribution is to try whenever possible to add curvature from $\Delta$ to neighbouring regions of degree greater than 4. Given 
this, we try to keep the number of times $\frac{\pi}{6}$ is exceeded to a minimum; and, given this, to keep the number of times $\frac{\pi}{5}$ is exceeded to a minimum, see, for example, Figure 7(iii). 
When avoidance of neighbouring regions of degree 4 is not possible we usually introduce distribution paths from $\Delta$ to nearby regions of degree greater than 4. For example in Figure 17(iv) there is a 
distribution path of length 2 fom $\Delta$ to $\hat{\Delta}_6$; or in Figure 19(iii) there is a distribution path of length 3 fom $\Delta$ to $\hat{\Delta}_{10}$. This approach turns out to be sufficient 
in almost all cases in terms of compensating positive curvature by negatively curved regions. However, a more complicated distribution scheme was required for some exceptional case and these are 
Configurations A-F mentioned immediately above and treated in detail in what follows.

\medskip

\textbf{Notes}. \textbf{1}.  In the figures the upper bound of the amount of curvature transferred will generally be indicated.

\textbf{2}.  It should be emphasised that whenever we identify regions, we do so modulo cyclic permutation and inversion.  For example in what follows we will identify $\Delta$ of Figure 20(vi) with $\Delta_1$ of Figure 31(i).

\textbf{3}.  Here and in what follows we use Lemma 3.4 to classify possible labelling of vertices.

\textbf{4}.  If $\bs{K}_0 \neq \bs{K}$ then it is assumed that $\Delta \neq \Delta_0$ and that if at any point positive curvature is transferred to $\Delta_0$ then it remains with $\Delta_0$.

\medskip

A complete description of the distribution of positive curvature from regions of degree 4 is given by Figures 6-32.
We give below an explanation for each case.

$\bs{d(v_i)=3}$ $\bs{(1 \leq i \leq 4)}$. \textbf{Figure 6}.
$c(\Delta)=\frac{2 \pi}{3}$ is distributed as shown.

$\bs{d(v_i)=3}$ $\bs{(1 \leq i \leq 3)}$. \textbf{Figure 7}. Either $d(v_4) > 5$ and $c(\Delta) \leq \frac{\pi}{3}$; or
$d(v_4)=5$ and $c(\Delta)=\frac{2\pi}{5}$; or $d(v_4)=4$ and $c(\Delta)=\frac{\pi}{2}$ (and distribute $c(\Delta)$ accordingly,as shown).

$\bs{d(v_1)=d(v_2)=d(v_4)=3}$. \textbf{Figure 8}. Either $d(v_3)>5$ or $d(v_3)=5$ or $d(v_3)=4$.

$\bs{d(v_1)=d(v_3)=d(v_4)=3}$. \textbf{Figure 9}. Either $d(v_2)>4$ or $d(v_2)=4$ and distribution of $c(\Delta)$ is given by Figure 9(i)-(iii).  There are two exceptions to these rules.  
There is an exception to Figure 9(ii) when $\Delta = \Delta_1$ of \textbf{Configuration F} in Figure 32(vi) and which, for convenience, we have reproduced in Figure 9(iv) with a rotation of
$\frac{\pi}{2}$ so that Figures 9(ii) and (iv) match up. Thus the exceptional rule (which is again described later for \textbf{Configuration F}) is that in Figure 9(iv) $\frac{\pi}{5},\frac{2\pi}{15}$
is added from $\Delta$ to $\hat{\Delta}_3,\hat{\Delta}_4$ instead of $\frac{\pi}{6},\frac{\pi}{6}$ (respectively) as in Figure 9(ii), and the dotted lines in Figure 9(iv) indicate the changes made. 
The second exception is to Figure 9(iii) and is when $\Delta = \Delta_1$ of \textbf{Configuration E} in Figure 32(iv). Again, for convenience, this has been reproduced (after inverting and rotating) in 
Figure 9(v). Thus the exceptional rule (which is again described later for \textbf{Configuration E}) is that in Figure 9(v) $\frac{2\pi}{15},\frac{\pi}{5}$ is added from $\Delta$ to 
$\hat{\Delta}_3,\hat{\Delta}_4$ instead of $\frac{\pi}{6},\frac{\pi}{6}$ (respectively) as in Figure 9(iii), and once more the dotted lines in Figure 9(v) indicate the changes made.

$\bs{d(v_i)=3}$ $\bs{(2 \leq i \leq 4)}$. \textbf{Figure 10}. Either $d(v_1) > 4$ or $d(v_1)=4$.

$\bs{d(v_1)=d(v_2)=3}$. \textbf{Figure 11}. If $d(v_3)=4$ and $d(v_4) \geq 6$ or $d(v_3) \geq 5$ and $d(v_4) \geq 5$ or $d(v_3) \geq 6$ and $d(v_4)=4$ then $c(\Delta) \leq c(3,3,4,6) = \frac{\pi}{6}$ is distributed as in Figure 11(i); otherwise the distribution is described by Figure 11(ii)-(viii).

$\bs{d(v_2)=d(v_3)=3}$. \textbf{Figure 12}. If $d(v_1)=4$ and $d(v_4) \geq 6$ or $d(v_1) \geq 5$ and $d(v_4) \geq 5$ or $d(v_1) \geq 6$ and $d(v_4)=4$ then $c(\Delta) \leq \frac{\pi}{6}$ is distributed as in Figure 12(i); otherwise the distribution is described by Figure 12(ii)-(viii).

$\bs{d(v_3)=d(v_4)=3}$. \textbf{Figure 13}. Distribution of $c(\Delta)$ is given by Figure 13(i).  There are five exceptions to this rule. When $\Delta = \Delta_3$ of \textbf{Configuration A} in Figure 
31(ii)-(iv) and, for convenience, $\Delta = \Delta_3$ has been reproduced (after inverting and rotating) in Figure 13(ii); when  $\Delta = \Delta$ of \textbf{Configuration C} in Figure 32(i) and $\Delta$
has been reproduced in Figure 13(iii); and when $\Delta=\Delta_1$ of \textbf{Configuration E} in Figure 32(iii) and $\Delta=\Delta_1$ has been reproduced (after inverting and rotating) in Figure 13(iv).  
Dotted lines in Figure 13(ii)-(iv) indicate the exceptional rules, so, for example, $\Delta$ adds $\frac{\pi}{30},\frac{3\pi}{10}$ to $\hat{\Delta}_1,\hat{\Delta}_3$ (respectively) in Figure 13(ii) 
instead of simply adding $\frac{\pi}{3}$ to $\hat{\Delta}_3$ as in Figure 13(i).

$\bs{d(v_4)=d(v_1)=3}$. \textbf{Figure 14}. Distribution of $c(\Delta)$ is given by Figure 14(i).  There are five exceptions to this rule. When $\Delta = \Delta_3$ of \textbf{Configuration B} in Figure 
31(vi)-(viii) and, for convenience, $\Delta = \Delta_3$ has been reproduced (after inverting and rotating) in Figure 14(ii); when $\Delta = \Delta$ of \textbf{Configuration D} and $\Delta$ has been 
reproduced (after rotating) in Figure 14(iii); and when $\Delta = \Delta_1$ of \textbf{Configuration F} in Figure 32(v) and $\Delta = \Delta_1$ has been reproduced (after rotating) in Figure 14(iv).
Dotted lines in Figure 14(ii)-(iv) indicate the exceptional rules.  so, for example, $\Delta$ adds $\frac{\pi}{30},\frac{3\pi}{10}$ to $\hat{\Delta}_2,\hat{\Delta}_4$ (respectively) in Figure 14(ii)  
instead of simply adding $\frac{\pi}{3}$ to $\hat{\Delta}_4$ as in Figure 13(i).

$\bs{d(v_2)=d(v_4)=4}$. \textbf{Figures 15-19}. There are four subcases.

(1)$\bs{(d(v_1),d(v_3)) \neq (4,4)}$. \textbf{Figure 15}. Either $(d(v_1),d(v_3)) \in \{ ( > 5,4),(5,4),(4,>5),(4,5)\}$ and distribution of $c(\Delta)$ is given by Figure 15(i)-(vi) or $d(v_1) \geq 5$ and $d(v_3) \geq 5$, 
$c(\Delta) \leq \frac{2 \pi}{15}$ and distribution is given by Figure 15(vii)-(x).

Suppose now that $d(v_1)=d(v_3)=4$.

(2)$\bs{(d(\hat{\Delta}_3),d(\hat{\Delta}_4)) \notin \{ (4,6),(4,4),(6,4)\}}$. \textbf{Figure 16}. $c(\Delta)$ is distributed as shown.

(3)$\bs{(d(\hat{\Delta}_3), d(\hat{\Delta}_4)) \in \{ (4,6),(6,4) \}}$. \textbf{Figure 17}. In each case add $\frac{\pi}{10}$ from $c(\Delta)$ to each of $c(\hat{\Delta}_1)$ and $c(\hat{\Delta}_2)$; and add $\frac{\pi}{15}$ from
$c(\Delta)$ to each of $c(\hat{\Delta}_3)$ and $c(\hat{\Delta}_4)$.
Let $d(\hat{\Delta}_3)=4$ and $d(\hat{\Delta}_4)=6$.  This is shown in Figure 17(i) in which $d(u_1) \geq 3$ and
$d(u_2) \geq 4$.  It remains to describe the further transfer (if any) of positive curvature from $c(\hat{\Delta}_3)$.

\begin{figure}
\begin{center}
\psfig{file=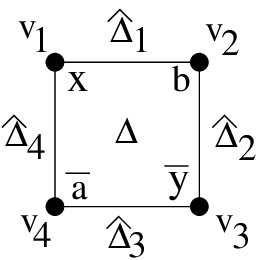}
\end{center}
\caption{}
\end{figure}

\begin{figure}
\begin{center}
\psfig{file=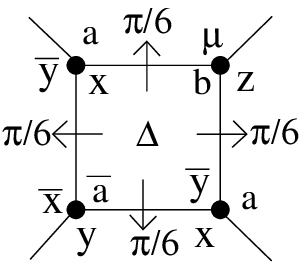}
\end{center}
\caption{$d(v_i)=3~(1 \leq i \leq 4)$}
\end{figure}

\begin{figure}
\begin{center}
\psfig{file=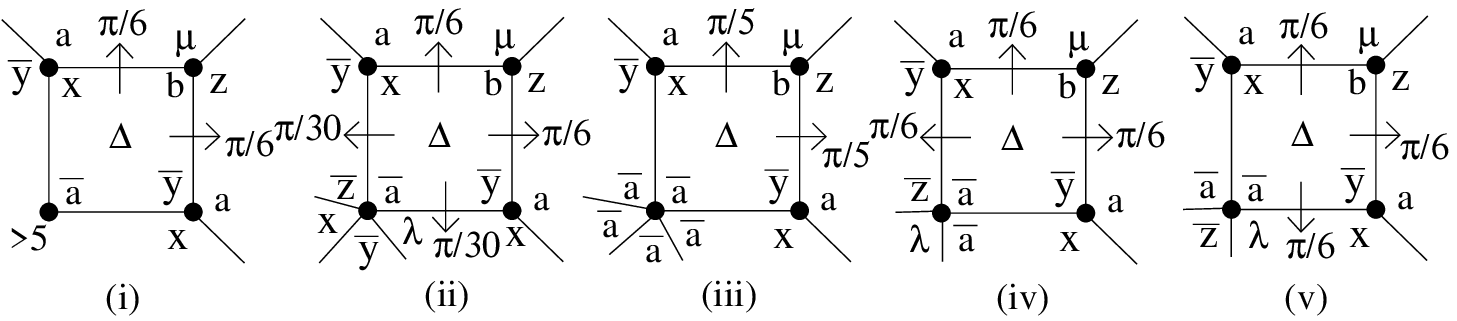}
\end{center}
\caption{$d(v_i)=3~(1 \leq i \leq 3)$}
\end{figure}

\begin{figure}
\begin{center}
\psfig{file=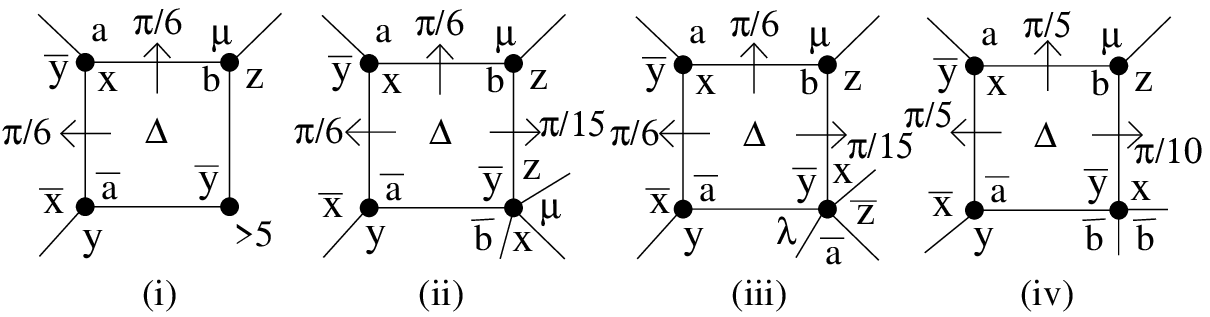}
\end{center}
\caption{$d(v_1)=d(v_2)=d(v_4)=3$}
\end{figure}

\begin{figure}
\begin{center}
\psfig{file=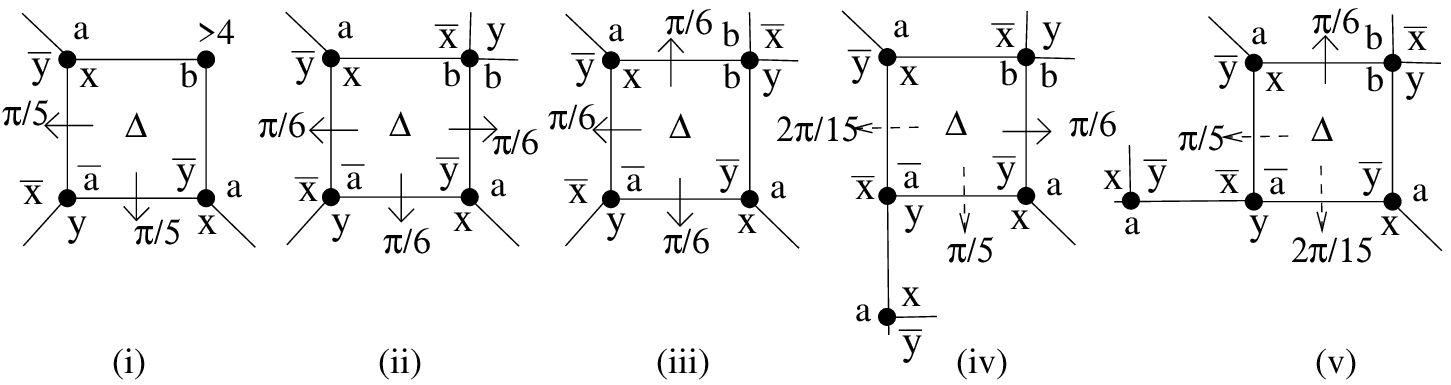}
\end{center}
\caption{$d(v_1)=d(v_3)=d(v_4)=3$}
\end{figure}

\begin{figure}
\begin{center}
\psfig{file=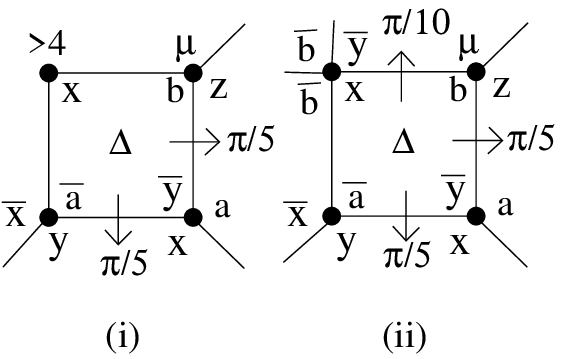}
\end{center}
\caption{$d(v_i)=3~(2 \leq i \leq 4)$}
\end{figure}

\begin{figure}
\begin{center}
\psfig{file=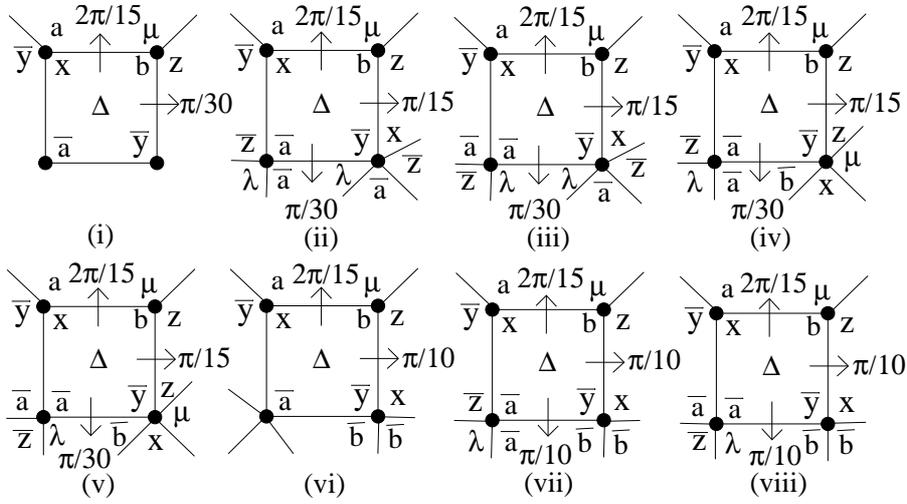}
\end{center}
\caption{$d(v_1)=d(v_2)=3$}
\end{figure}

\begin{figure}
\begin{center}
\psfig{file=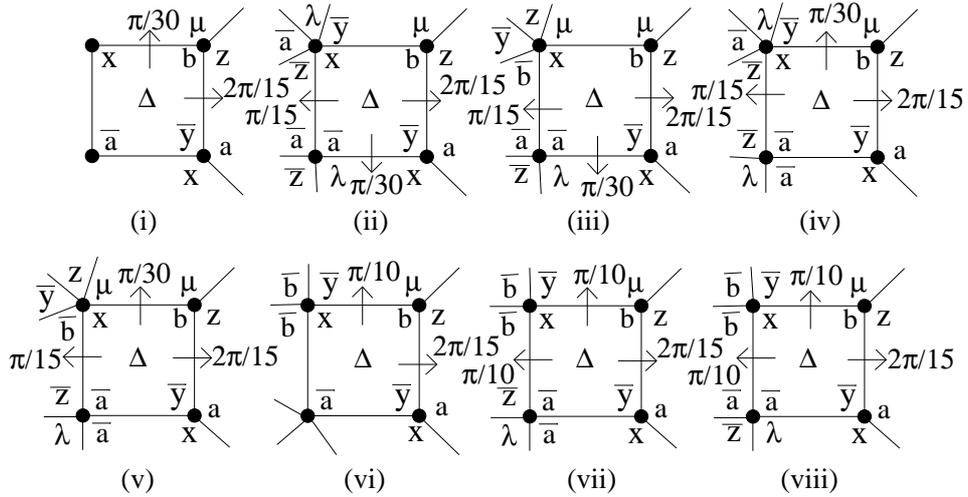}
\end{center}
\caption{$d(v_2)=d(v_3)=3$}
\end{figure}

\begin{figure}
\begin{center}
\psfig{file=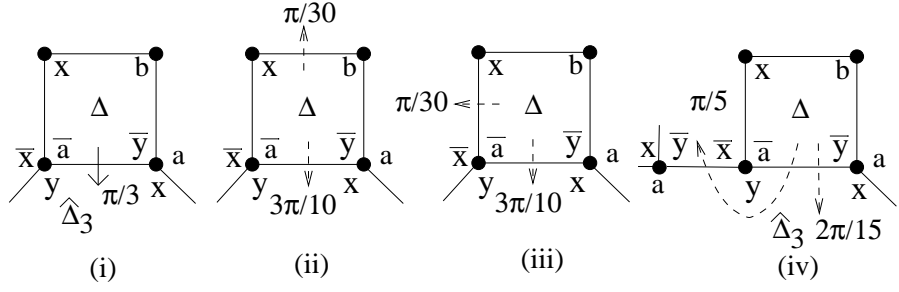}
\end{center}
\caption{$d(v_3)=d(v_4)=3$}
\end{figure}

\begin{figure}
\begin{center}
\psfig{file=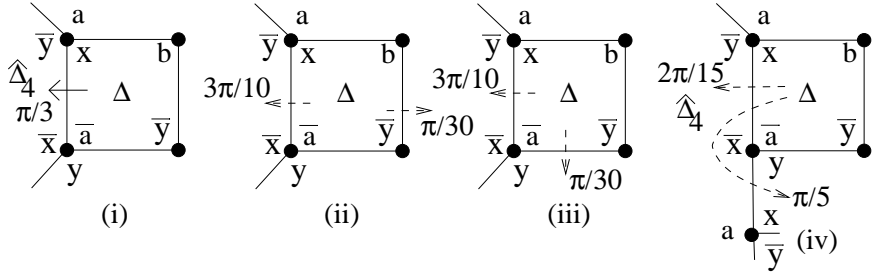}
\end{center}
\caption{$d(v_1)=d(v_4)=3$}
\end{figure}

\begin{figure}
\begin{center}
\psfig{file=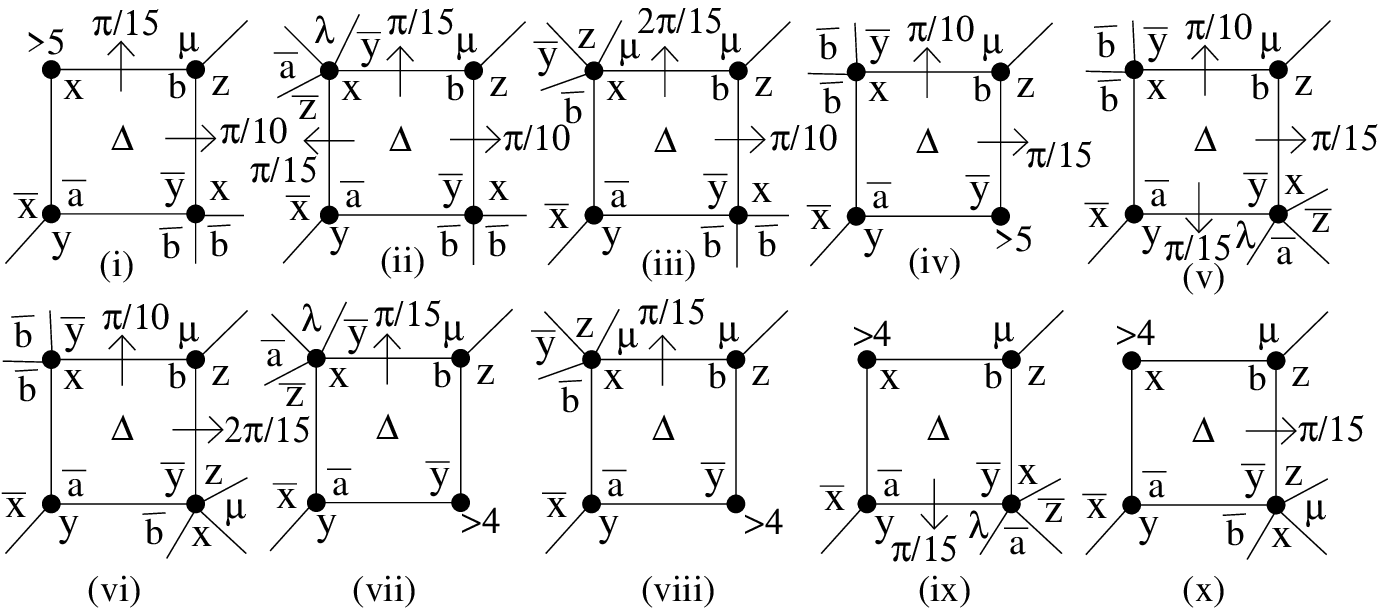}
\end{center}
\caption{$d(v_2)=d(v_4)=3  (subcase 1)$}
\end{figure}

\begin{figure}
\begin{center}
\psfig{file=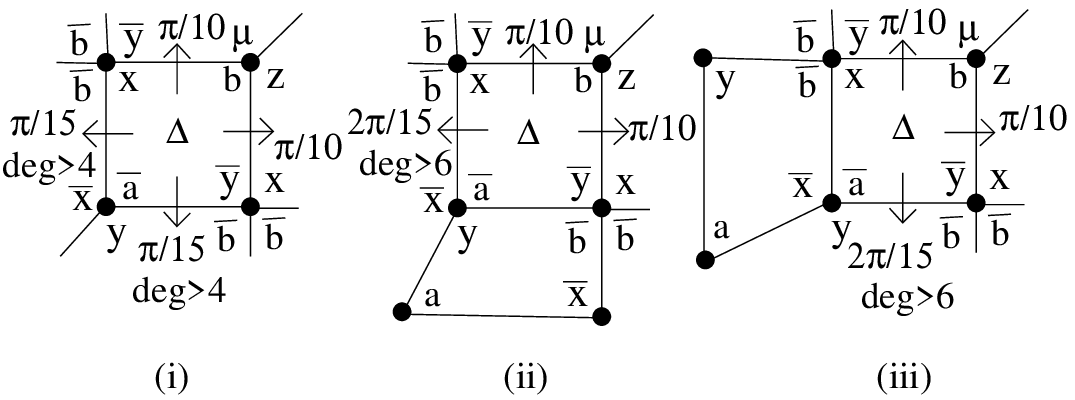}
\end{center}
\caption{$d(v_2)=d(v_4)=3 (subcase 2)$}
\end{figure}

\begin{figure}
\begin{center}
\psfig{file=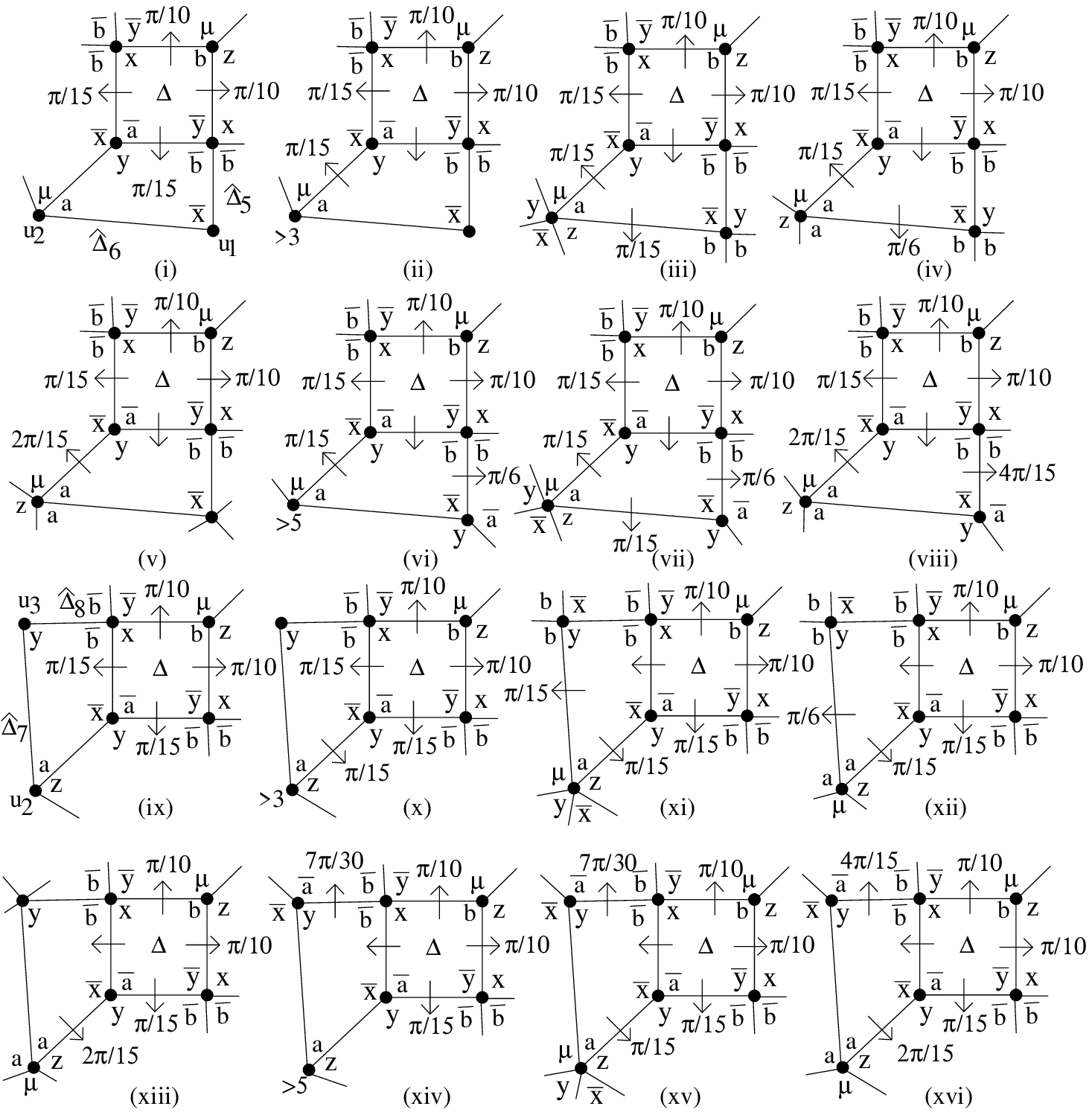}
\end{center}
\caption{$d(v_2)=d(v_4)=3 (subcase 3)$}
\end{figure}

\begin{figure}
\begin{center}
\psfig{file=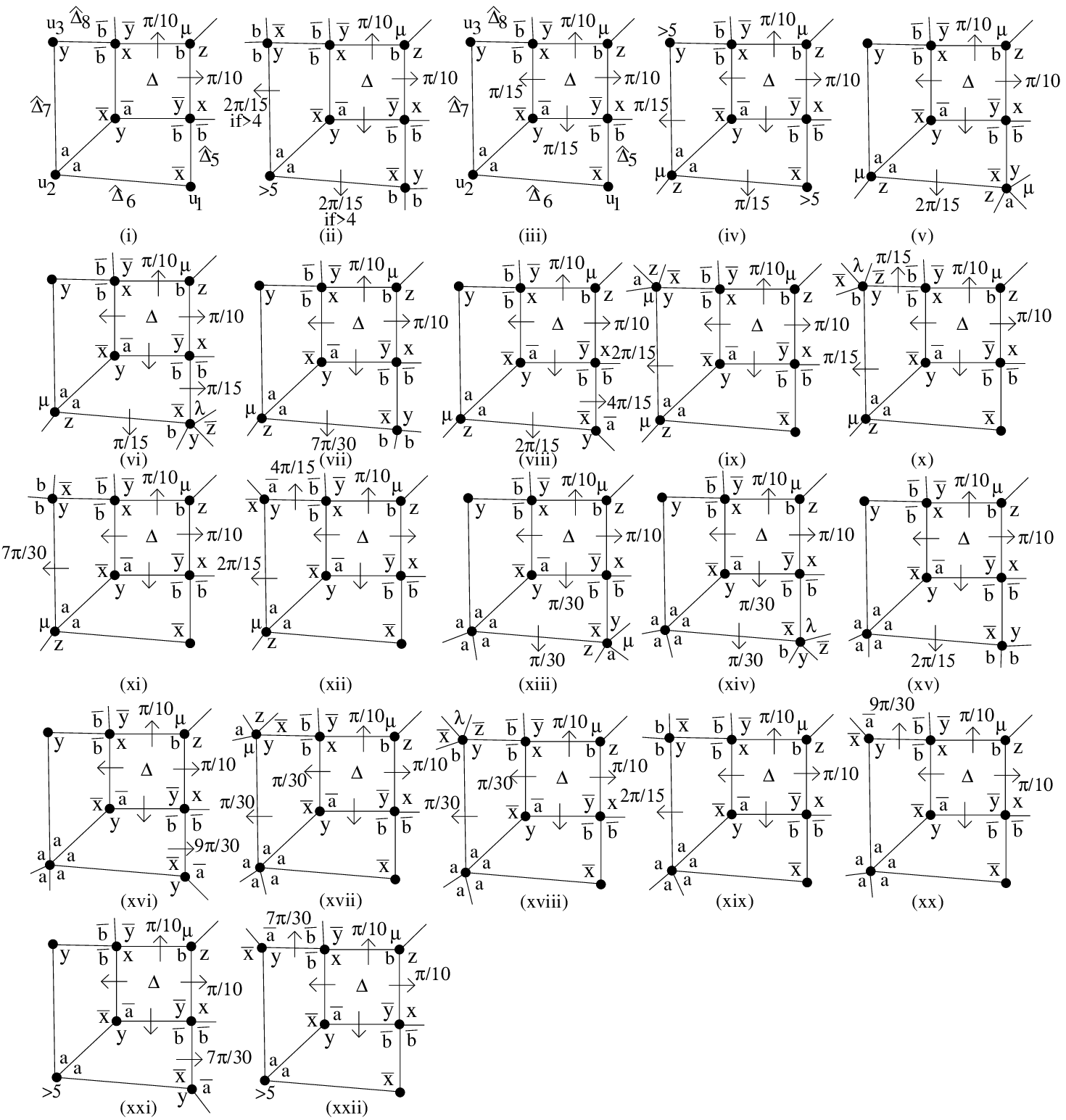}
\end{center}
\caption{$d(v_2)=d(v_4)=3 (subcase 4)$}
\end{figure}

\begin{figure}
\begin{center}
\psfig{file=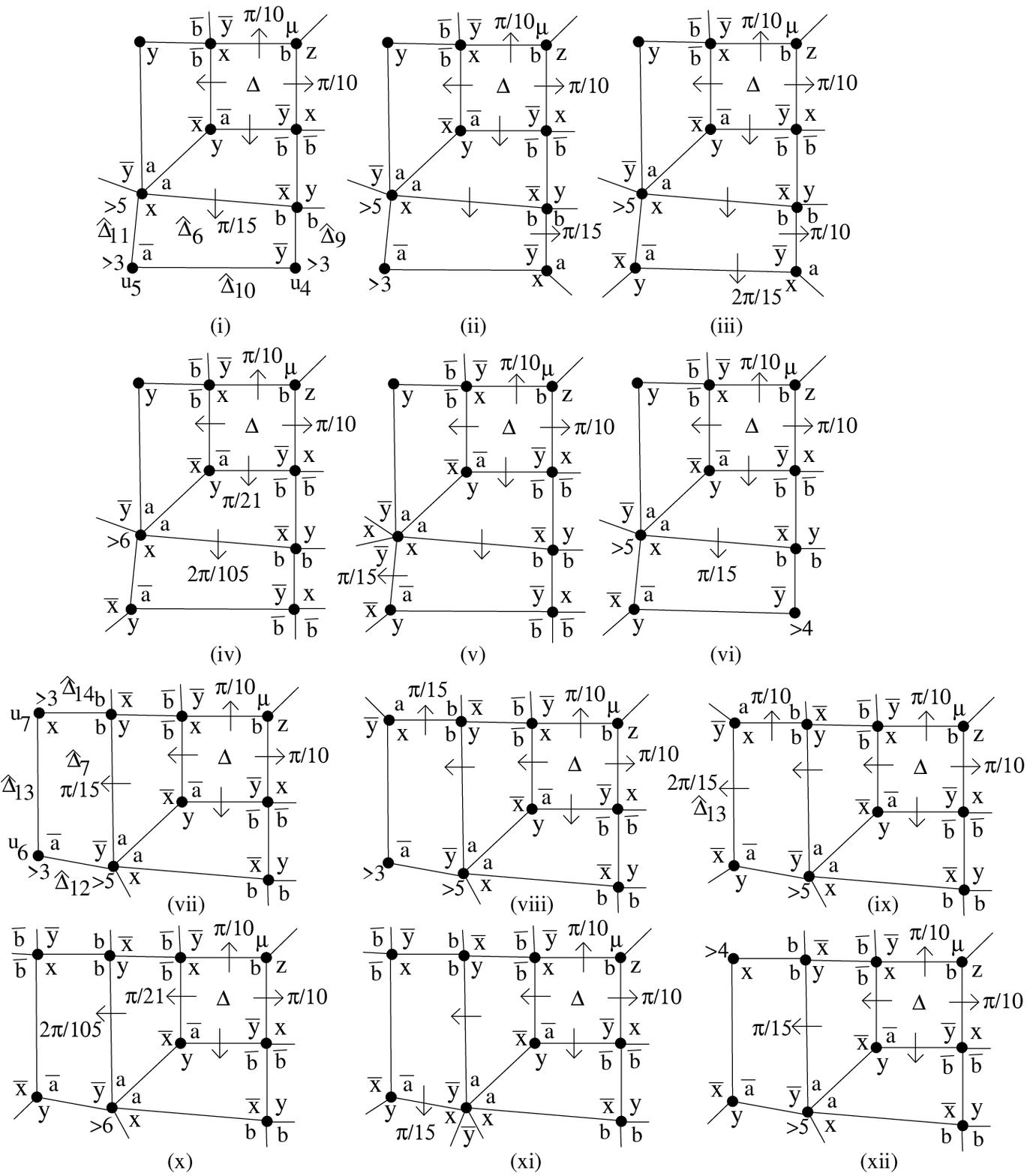}
\end{center}
\caption{$d(v_2)=d(v_4)=3 (subcase 4)$}
\end{figure}

\begin{figure}
\begin{center}
\psfig{file=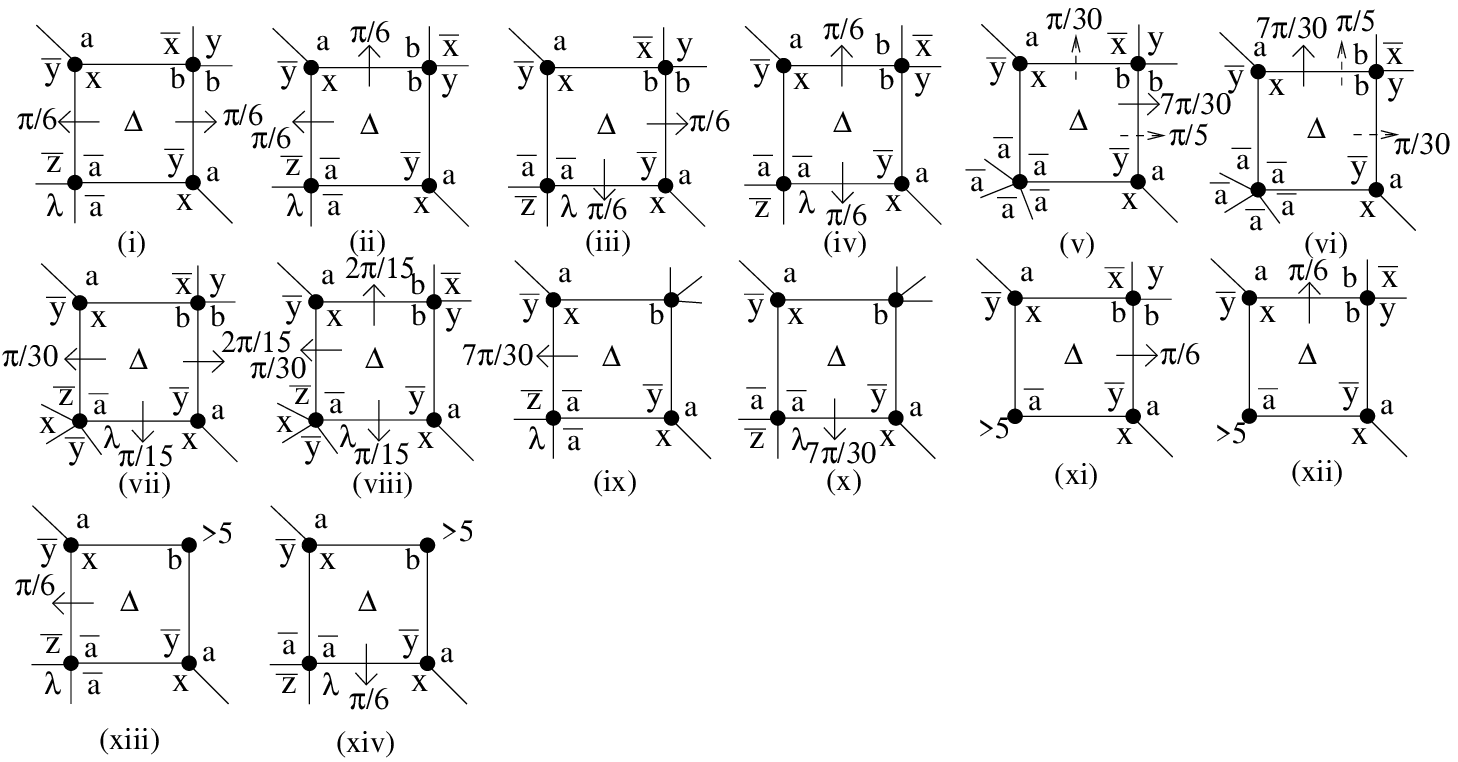}
\end{center}
\caption{$d(v_1)=d(v_3)=3 (subcase 1)$}
\end{figure}

\begin{figure}
\begin{center}
\psfig{file=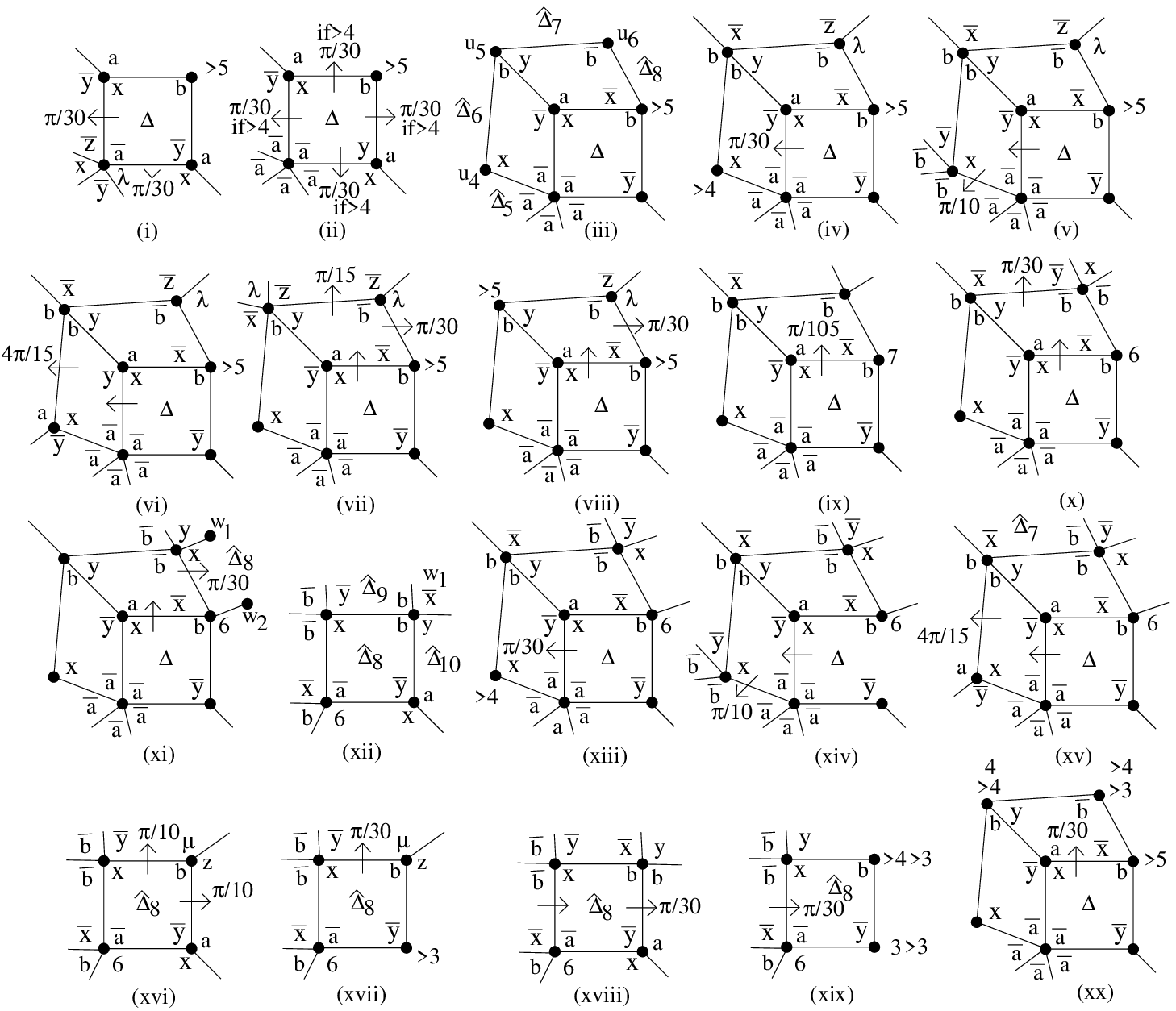}
\end{center}
\caption{$d(v_1)=d(v_3)=3 (subcase 2)$}
\end{figure}

\begin{figure}
\begin{center}
\psfig{file=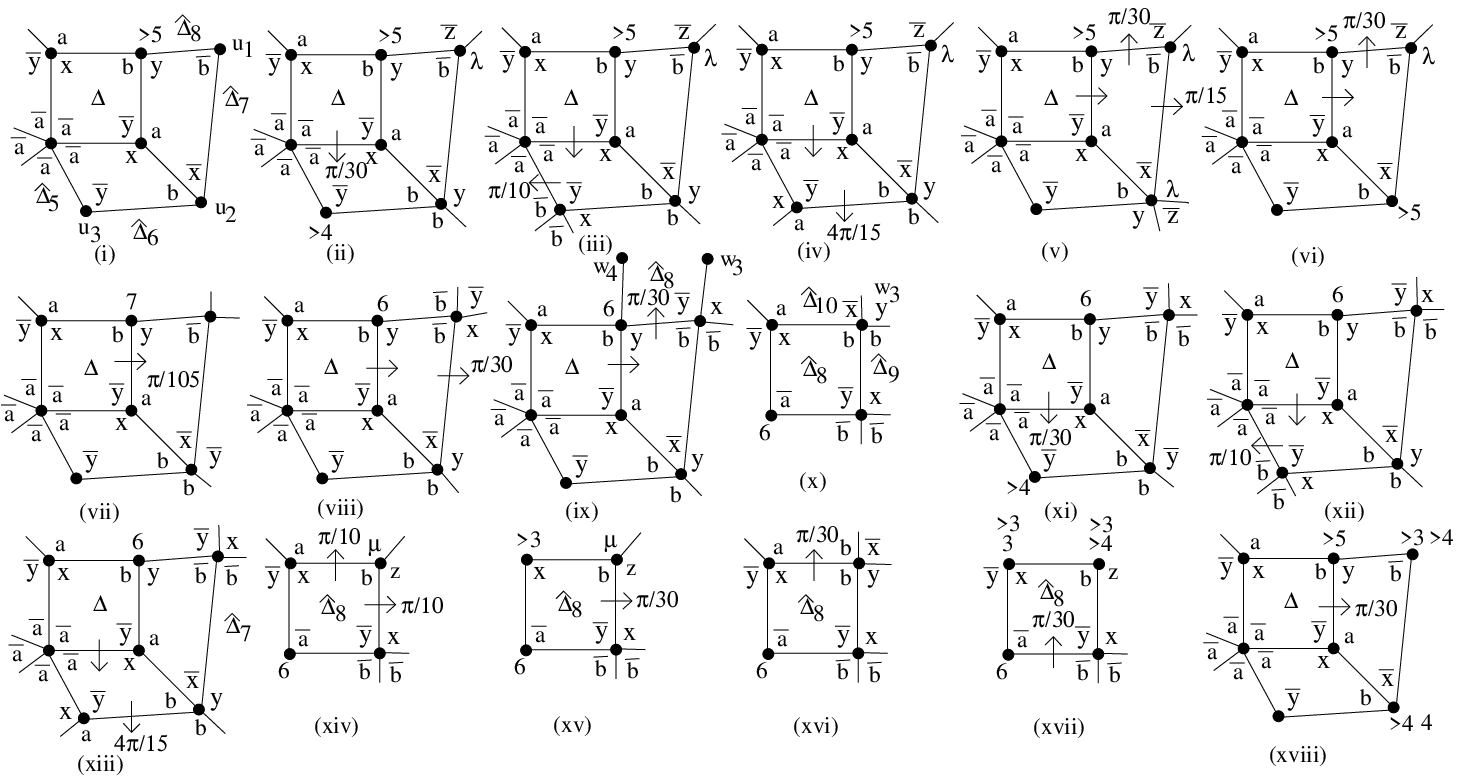}
\end{center}
\caption{$d(v_1)=d(v_3)=3 (subcase 2)$}
\end{figure}

\newpage

\begin{figure}
\begin{center}
\psfig{file=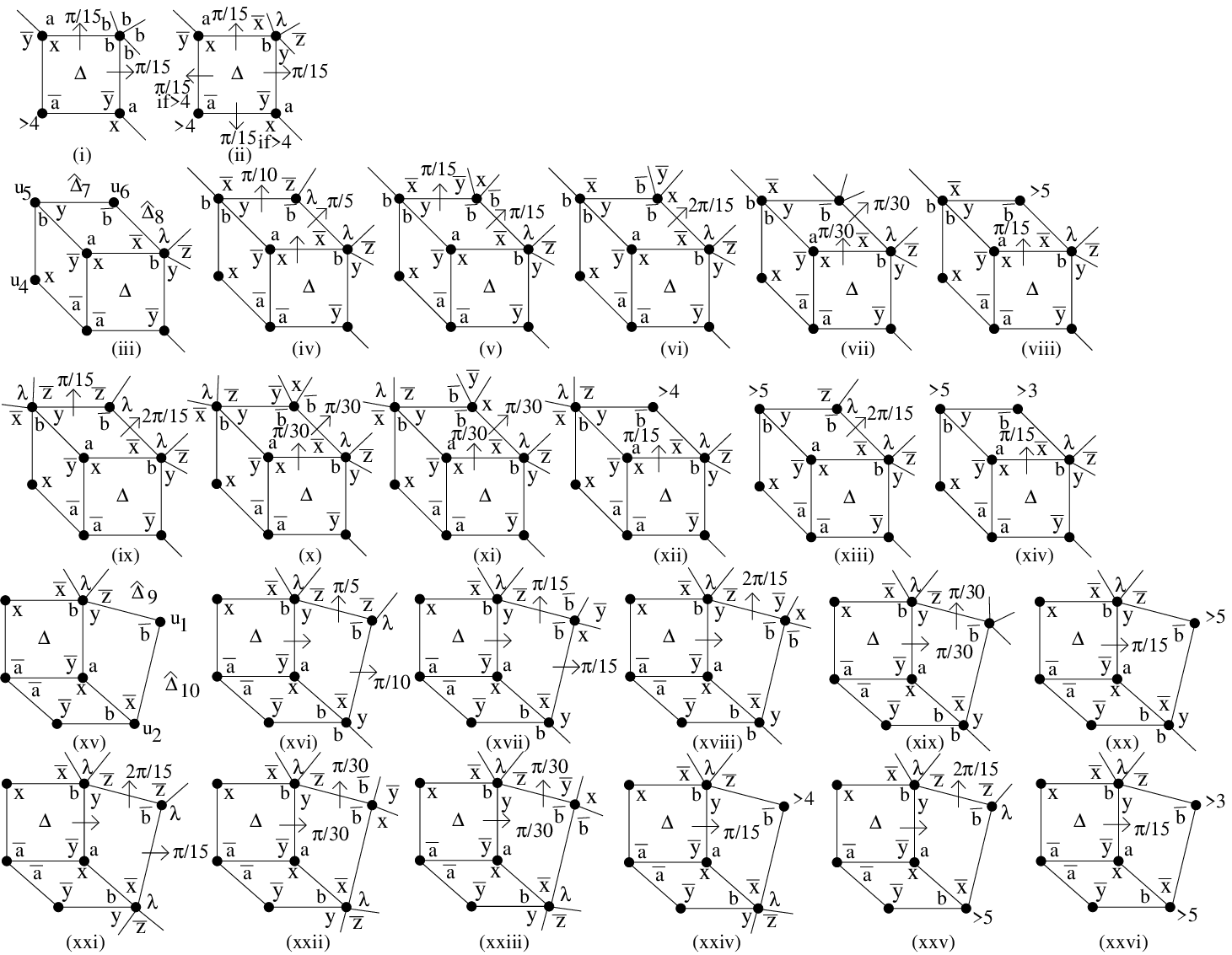}
\end{center}
\caption{$d(v_1)=d(v_3)=3 (subcase 3)$}
\end{figure}

\begin{figure}
\begin{center}
\psfig{file=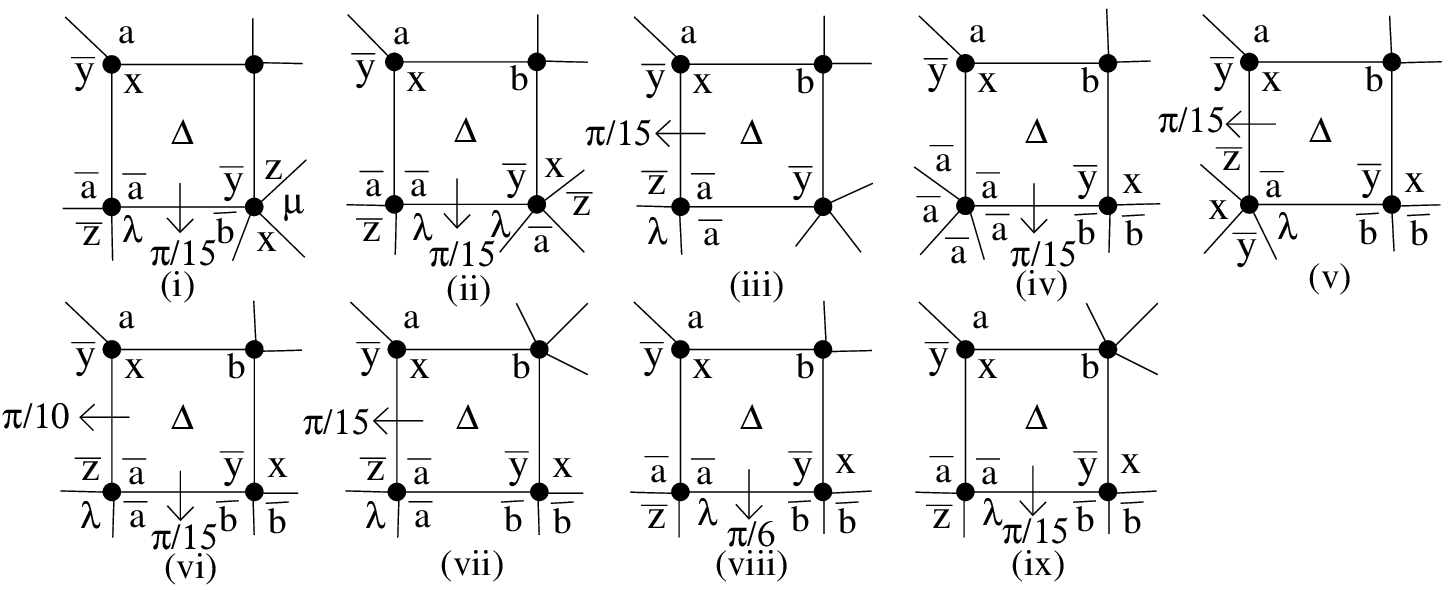}
\end{center}
\caption{$d(v_1)=3 $}
\end{figure}

\begin{figure}
\begin{center}
\psfig{file=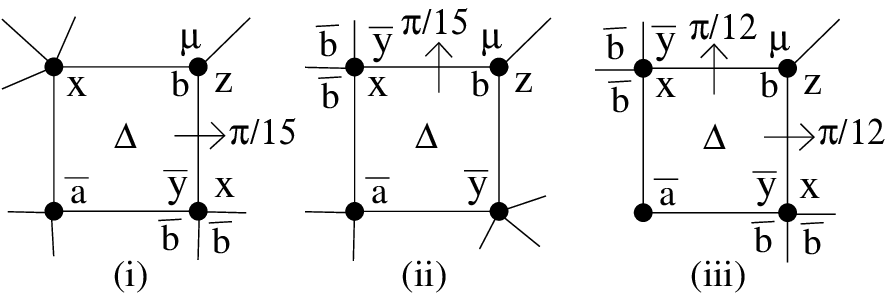}
\end{center}
\caption{$d(v_2)=3 $}
\end{figure}

\begin{figure}
\begin{center}
\psfig{file=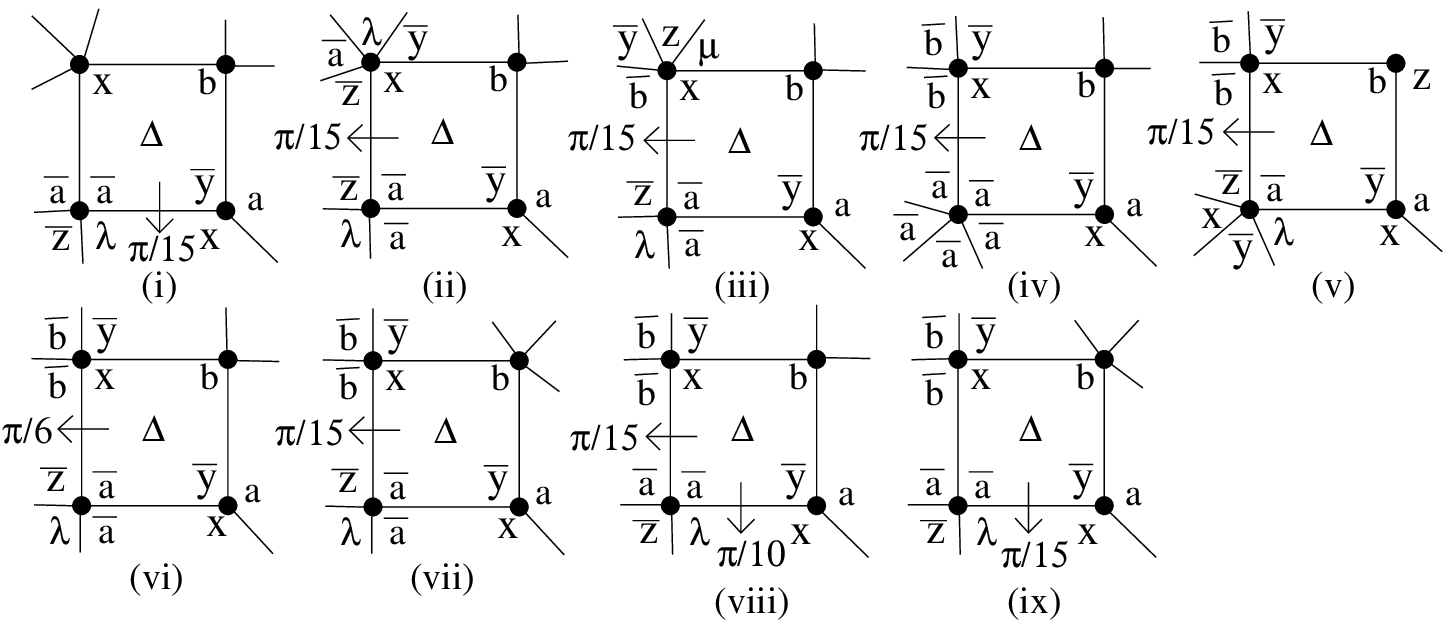}
\end{center}
\caption{$d(v_3)=3 $}
\end{figure}

\begin{figure}
\begin{center}
\psfig{file=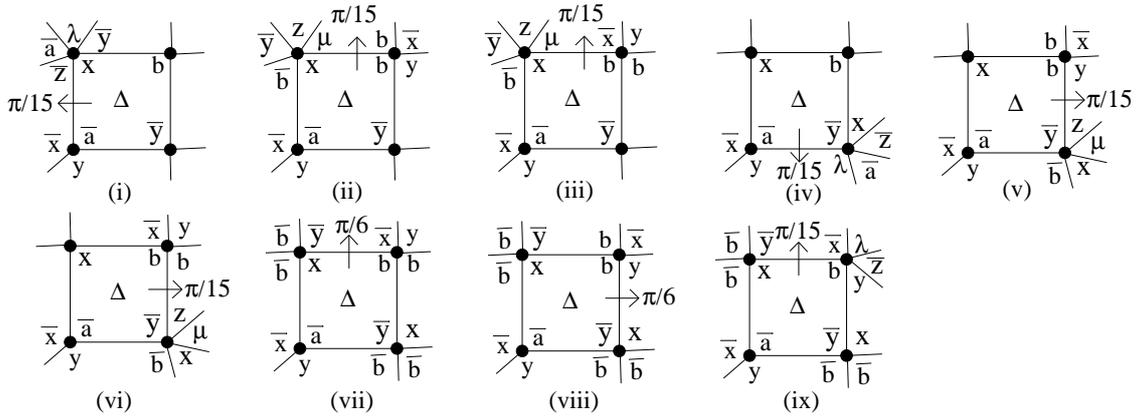}
\end{center}
\caption{$d(v_4)=3 (subcase1)$}
\end{figure}

\begin{figure}
\begin{center}
\psfig{file=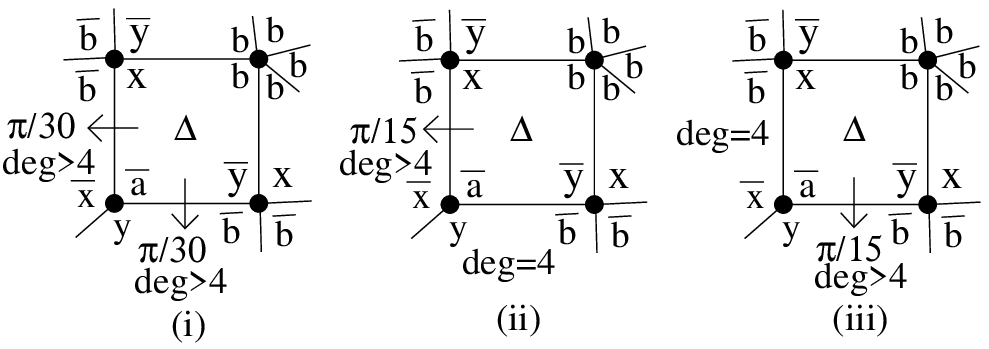}
\end{center}
\caption{$d(v_4)=3 (subcase2)$}
\end{figure}

\begin{figure}
\begin{center}
\psfig{file=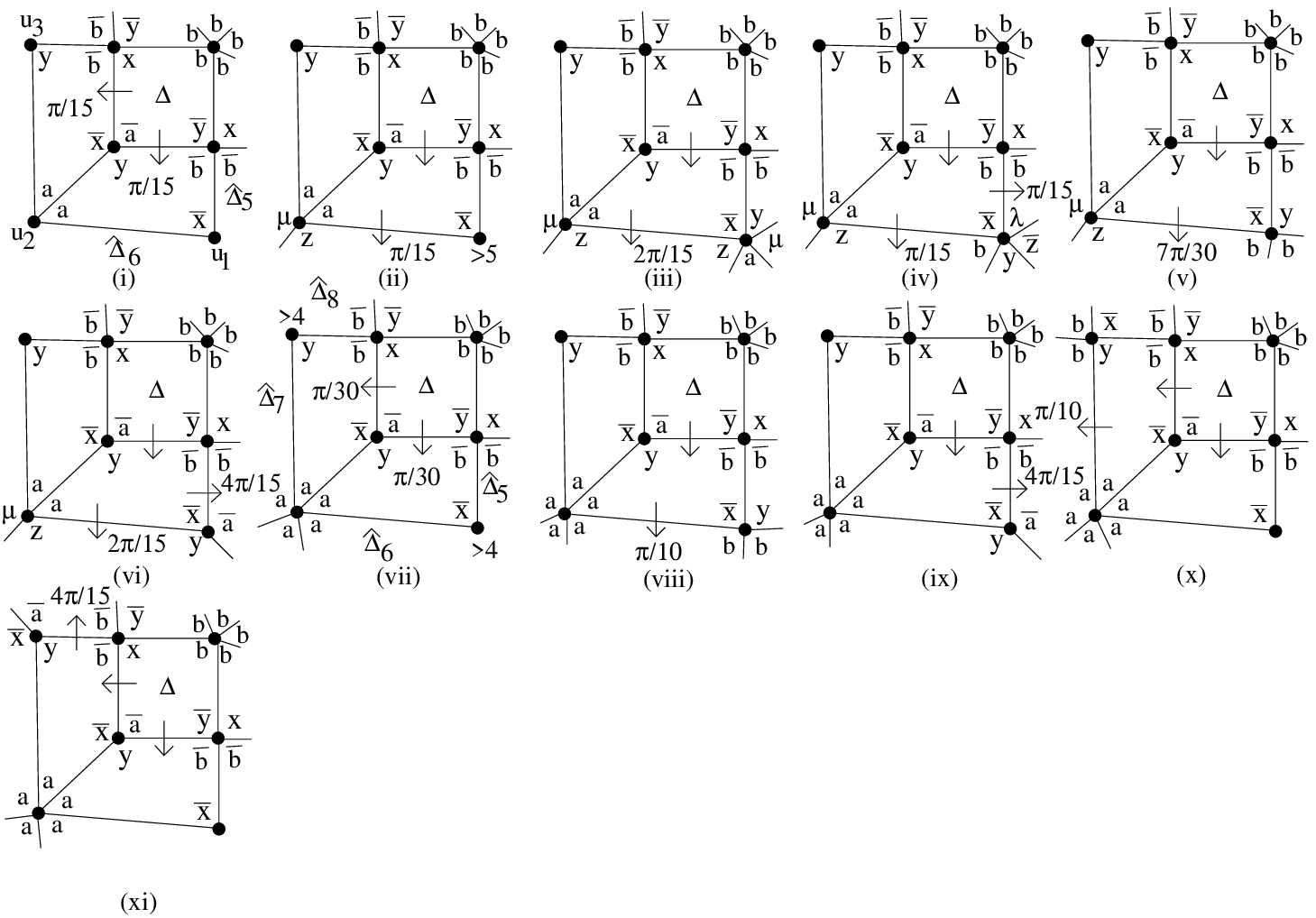}
\end{center}
\caption{$d(v_4)=3 (subcase3)$}
\end{figure}

\begin{figure}
\begin{center}
\psfig{file=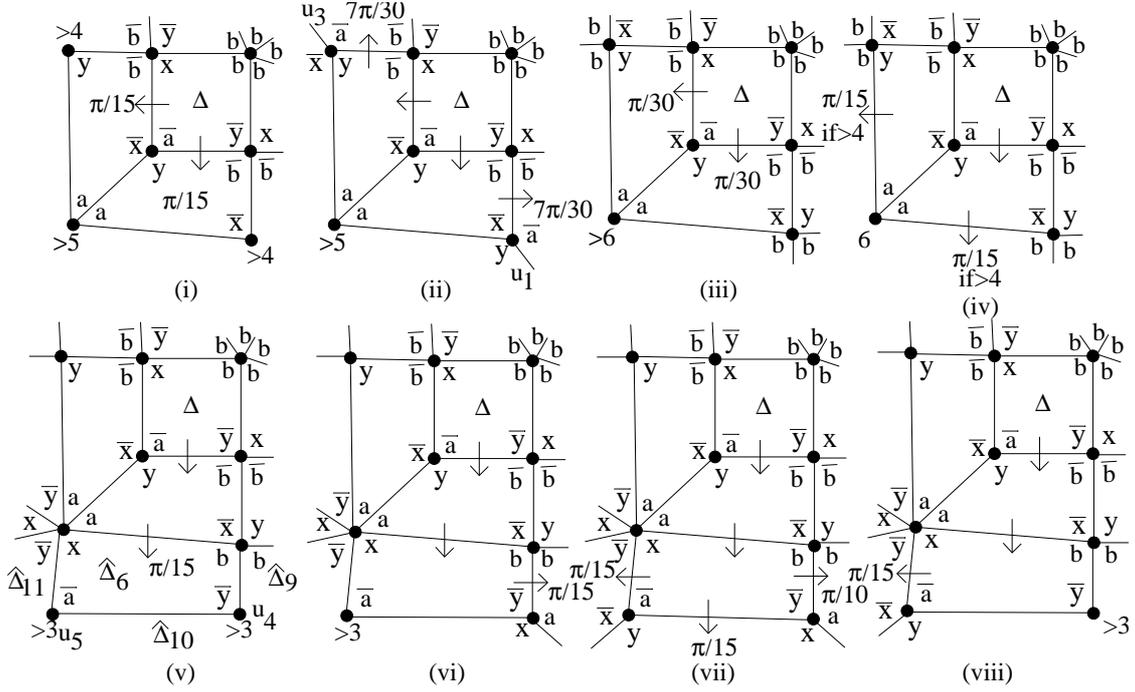}
\end{center}
\caption{$d(v_4)=3 (subcase4)$}
\end{figure}

\begin{figure}
\begin{center}
\psfig{file=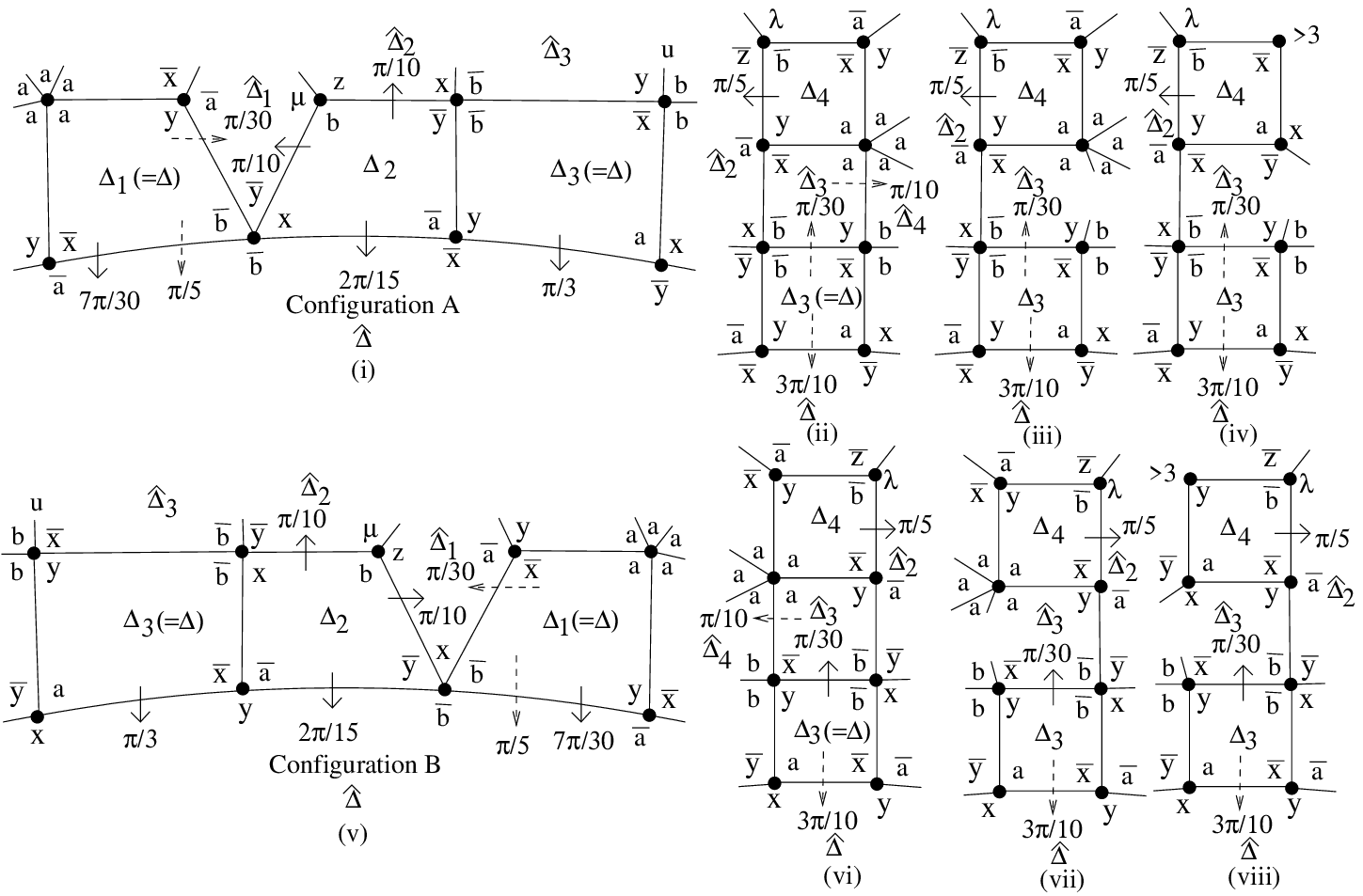}
\end{center}
\caption{ }
\end{figure}

\begin{figure}
\begin{center}
\psfig{file=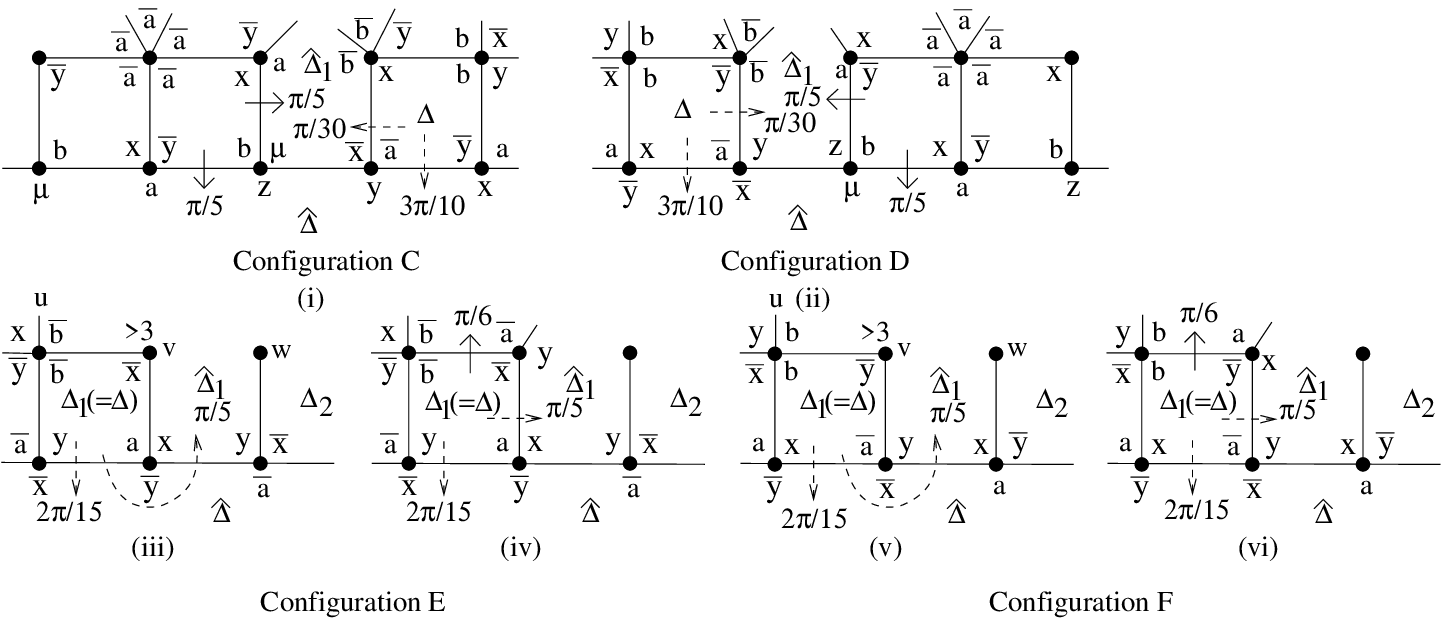}
\end{center}
\caption{}
\end{figure}

If $c(\hat{\Delta}_3) \leq - \frac{\pi}{15}$ then the
$\frac{\pi}{15}$ from $c(\Delta)$ remains with $c(\hat{\Delta}_3)$ as in Figure 17(i); and
if $-\frac{\pi}{15} < c(\hat{\Delta}_3) \leq 0$ then
$\frac{\pi}{15} + c(\hat{\Delta}_3) \leq \frac{\pi}{15}$ is added to $c(\hat{\Delta}_4)$ as in Figure 17(ii).
Assume that $c(\hat{\Delta}_3) > 0$.
We now proceed according to the values of $d(u_1)$ and $d(u_2)$.
If $d(u_1)=4$ and $d(u_2)=5$ then
($c(\hat{\Delta}_3)=c(3,4,4,5)=\frac{\pi}{15}$ and)
$\frac{\pi}{15}+c(\hat{\Delta}_3)=\frac{2 \pi}{15}$ so add
$\frac{\pi}{15}$ to each of $c(\hat{\Delta}_4)$ and $c(\hat{\Delta}_6)$ as in Figure 17(iii);
if $d(u_1)=4=d(u_2)$ then $\frac{\pi}{15}+c(\hat{\Delta}_3)=\frac{7 \pi}{30}$ so add
$\frac{\pi}{15}$ to $c(\hat{\Delta}_4)$ and $\frac{\pi}{6}$ to $c(\hat{\Delta}_6)$ as in (iv);
if $d(u_1)=5$ and $d(u_2)=4$ then add
$\frac{\pi}{15} + c(\hat{\Delta}_3) = \frac{2 \pi}{15}$ to $c(\hat{\Delta}_4)$ as in (v);
if $d(u_1)=3$ and $d(u_2) \geq 6$ then $\frac{\pi}{15} + c(\hat{\Delta}_3) \leq \frac{7 \pi}{30}$ so add
$\frac{\pi}{15}$ to $c(\hat{\Delta}_4)$ and $\frac{\pi}{6}$ to $c(\hat{\Delta}_5)$ as in (vi);
if $d(u_1)=3$ and $d(u_2)=5$ then $\frac{\pi}{15} + c(\hat{\Delta}_3) = \frac{9 \pi}{30}$ so add
$\frac{\pi}{15}$ to $c(\hat{\Delta}_4)$ and $c(\hat{\Delta}_6)$, and add
$\frac{\pi}{6}$ to $c(\hat{\Delta}_5)$ as in (vii); and
if $d(u_1)=3$ and $d(u_2)=4$ then $\frac{\pi}{15} + c(\hat{\Delta}_3)=\frac{6 \pi}{15}$ so add
$\frac{2 \pi}{15}$ to $c(\hat{\Delta}_4)$ and $\frac{4 \pi}{15}$ to $c(\hat{\Delta}_5)$ as in (viii).

Now let $d(\hat{\Delta}_3)=6$ and $d(\hat{\Delta}_4)=4$.  This is shown in Figure 17(ix) where $d(u_3) \geq 3$ and $d(u_2) \geq 4$.
It remains to describe the further transfer (if any) of positive curvature from $c(\hat{\Delta}_4)$.
If $c(\hat{\Delta}_4) \leq -\frac{\pi}{15}$ then the $\frac{\pi}{15}$ from $c(\Delta)$ remains with $c(\hat{\Delta}_4)$ as in Figure 17(ix); and
if $-\frac{\pi}{15} < c(\hat{\Delta}_4) \leq 0$ then $\frac{\pi}{15} + c(\hat{\Delta}_4) \leq \frac{\pi}{15}$ is added to $c(\hat{\Delta}_3)$ as in
Figure 17(x).  Assume that $c(\hat{\Delta}_4) > 0$.  We proceed according to the values of $d(u_2)$ and $d(u_3)$.
If $d(u_3)=4$ and $d(u_2)=5$ then $\frac{\pi}{15}+c(\hat{\Delta}_4)=\frac{2 \pi}{15}$ so add
$\frac{\pi}{15}$ to each of $c(\hat{\Delta}_3)$ and $c(\hat{\Delta}_7)$ as in Figure 17(xi);
if $d(u_3)=4=d(u_2)$ then $\frac{\pi}{15} + c(\hat{\Delta}_4)=\frac{7 \pi}{30}$ so add
$\frac{\pi}{15}$ to $c(\hat{\Delta}_3)$ and $\frac{\pi}{6}$ to $c(\hat{\Delta}_7)$ as in (xii);
if $d(u_3)=5$ and $d(u_2)=4$ then add
$\frac{\pi}{15}+c(\hat{\Delta}_4)=\frac{2\pi}{15}$ to $c(\hat{\Delta}_3)$ as in (xiii);
if $d(u_3)=3$ and $d(u_2) \geq 6$ then add
$\frac{\pi}{15} + c(\hat{\Delta}_4) \leq \frac{7 \pi}{30}$ to $c(\hat{\Delta}_8)$ as in (xiv);
if $d(u_3)=3$ and $d(u_2)=5$ then $\frac{\pi}{15}+c(\hat{\Delta}_4)=\frac{3 \pi}{10}$ so add
$\frac{\pi}{15}$ to $c(\hat{\Delta}_3)$ and $\frac{7 \pi}{30}$ to $c(\hat{\Delta}_8)$ as in (xv); and
if $d(u_3)=3$ and $d(u_2)=4$ then $\frac{\pi}{15}+c(\hat{\Delta}_4)=\frac{6 \pi}{15}$ so add
$\frac{2 \pi}{15}$ to $c(\hat{\Delta}_3)$ and $\frac{4 \pi}{15}$ to $c(\hat{\Delta}_8)$ as in (xvi).

(4)$\bs{d(\hat{\Delta}_3)=d(\hat{\Delta}_4)=4}$. \textbf{Figures 18-19}. This subcase is shown in Figure 18(i) in which $d(u_1) \geq 3$, $d(u_2) \geq 4$ and $d(u_3) \geq 3$ and $\frac{\pi}{10}$ is distributed from $c(\Delta)$ to
each
of $c(\hat{\Delta}_1)$ and $c(\hat{\Delta}_2)$ with $\frac{2 \pi}{15}$ remaining with $c(\Delta)$.  If $d(u_1)=d(u_3)=4$, $d(u_2) \geq 6$ and $d(\hat{\Delta}_6) > 4$ then distribute this remaining $\frac{2 \pi}{15}$ from 
$c(\Delta)$ to $c(\hat{\Delta}_6)$ as shown in Figure 18(ii), noting that $c(\hat{\Delta}_3) \leq 0$; or if
$d(u_1)=d(u_3)=4$, $d(u_2) \geq 6$ and $d(\hat{\Delta}_7) > 4$ then distribute the $\frac{2 \pi}{15}$ from $c(\Delta)$ to $c(\hat{\Delta}_7)$ as shown in Figure 18(ii), noting that $c(\hat{\Delta}_4) \leq 0$.  Assume from now on
that neither of these sets of conditions occur.  Then $\frac{\pi}{15}$ is distributed from $c(\Delta)$ to each of $c(\hat{\Delta}_3)$ and $c(\hat{\Delta}_4)$ as shown ini Figure 18(iii).
If $c(\hat{\Delta}_3) \leq - \frac{\pi}{15}$ then the $\frac{\pi}{15}$ from $c(\Delta)$ remains with $c(\hat{\Delta}_3)$ and similarly for $c(\hat{\Delta}_4)$, so assume from now on that
$c(\hat{\Delta}_3) > - \frac{\pi}{15}$ and $c(\hat{\Delta}_4) > - \frac{\pi}{15}$.  It remains to describe further transfer of positive curvature from $c(\hat{\Delta}_3)$ and $c(\hat{\Delta}_4)
$ (and possibly $c(\hat{\Delta}_6)$ when $d(\hat{\Delta}_6)=4$ and $c(\hat{\Delta}_7)$ when $d(\hat{\Delta}_7)=4$).

Let $d(u_2)=4$.
If $d(u_1) \geq 6$ then add $\frac{\pi}{15} + c(\hat{\Delta}_3) \leq \frac{\pi}{15}$ to $c(\hat{\Delta}_6)$ as in Figure 18(iv);
if $d(u_1)=5$ then add $\frac{\pi}{15} + c(\hat{\Delta}_3) = \frac{2 \pi}{15}$ to $c(\hat{\Delta}_6)$ if $l(u_1)$ is given by (v), or add
$\frac{\pi}{15}$ to each of $c(\hat{\Delta}_5)$ and $c(\hat{\Delta}_6)$ if $l(u_1)$ is given by (vi);
if $d(u_1)=4$ then add $\frac{\pi}{15} + c(\hat{\Delta}_3)=\frac{7 \pi}{30}$ to $c(\hat{\Delta}_6)$ as in (vii);
if $d(u_1)=3$ then $\frac{\pi}{15}+c(\hat{\Delta}_3)=\frac{6 \pi}{15}$ so add
$\frac{4\pi}{15}$ to $c(\hat{\Delta}_5)$ and
$\frac{2 \pi}{15}$ to $c(\hat{\Delta}_6)$ as in (viii);
if $d(u_3) \geq 6$ then add
$\frac{\pi}{15} + c(\hat{\Delta}_4) \leq \frac{\pi}{15}$ to $c(\hat{\Delta}_7)$ as in (iv);
if $d(u_3)=5$ then add
$\frac{\pi}{15} + c(\hat{\Delta}_4)=\frac{2 \pi}{15}$ to $c(\hat{\Delta}_7)$ if $l(u_3)$ is given by (ix), or add
$\frac{\pi}{15}$ to each of $c(\hat{\Delta}_7)$ and $c(\hat{\Delta}_8)$ if $l(u_3)$ is given by (x);
if $d(u_3)=4$ then add
$\frac{\pi}{15} + c(\hat{\Delta}_4)=\frac{7 \pi}{30}$ to $c(\hat{\Delta}_7)$ as in (xi); and
if $d(u_3)=3$ then $\frac{\pi}{15} + c(\hat{\Delta}_4)=\frac{6 \pi}{15}$ so add
$\frac{2 \pi}{15}$ to $c(\hat{\Delta}_7)$ and $\frac{4 \pi}{15}$ to $c(\hat{\Delta}_8)$ as in (xii).

Let $d(u_2)=5$ and so $l(u_2)=a^5$.
If $d(u_1)=5$ then add
$\frac{\pi}{30}$ from $c(\Delta)$ to $c(\hat{\Delta}_3)=c(3,4,5,5)=-\frac{\pi}{30}$ and
$\frac{\pi}{30}$ from $c(\Delta)$ to $c(\hat{\Delta}_6)$ as in Figure 18(xiii) and (xiv);
if $d(u_1)=4$ then add $\frac{\pi}{15} + c(\hat{\Delta}_3)=\frac{2 \pi}{15}$ to $c(\hat{\Delta}_6)$ as in (xv);
if $d(u_1)=3$ then add $\frac{\pi}{15} + c(\hat{\Delta}_3)=\frac{9 \pi}{30}$ to $c(\hat{\Delta}_5)$ as in (xvi);
if $d(u_3)=5$ then add $\frac{\pi}{30}$ from $c(\Delta)$ to $c(\hat{\Delta}_4)$ and $\frac{\pi}{30}$ from $c(\Delta)$ to $c(\hat{\Delta}_7)$ as in
(xvii) and (xviii); if $d(u_3)=4$
then add $\frac{\pi}{15}+c(\hat{\Delta}_4)=\frac{2 \pi}{15}$ to $c(\hat{\Delta}_7)$ as in (xix); and
if $d(u_3)=3$ then add $\frac{\pi}{15} + c( \hat{\Delta}_4)=\frac{3 \pi}{10}$ to $c(\hat{\Delta}_8)$ as in (xx).

Let $d(u_2) \geq 6$ so that by assumption $3 \leq d(u_1)$, $d(u_3) \leq 4$ (since $c(3,4,5,6)=-\frac{\pi}{10}$).
If $d(u_1)=3$ then add $\frac{\pi}{15}+c(\hat{\Delta}_3) \leq \frac{7 \pi}{30}$ to $c(\hat{\Delta}_5)$ as in Figure 18(xxi); and
if $d(u_3)=3$ then add $\frac{\pi}{15}+c(\hat{\Delta}_4) \leq \frac{7 \pi}{30}$ to $c(\hat{\Delta}_8)$ as in (xxii).

This leaves $d(u_1)=d(u_3)=d(\hat{\Delta}_6)=d(\hat{\Delta}_7)=4$.
First consider $\hat{\Delta}_6$ as shown in Figure 19(i).
If $d(u_4) > 3$ and $d(u_5) > 3$ then add
$\frac{\pi}{15} + c(\hat{\Delta}_3) \leq \frac{\pi}{15}$ to $c(\hat{\Delta}_6) \leq c(4,4,4,6)=-\frac{\pi}{6}$ as in 19(i);
if $d(u_4)=3$ and $d(u_5) > 3$ then add $\frac{\pi}{15} + c(\hat{\Delta}_3) + c(\hat{\Delta}_6) \leq \frac{\pi}{15}$ to $c(\hat{\Delta}_9)$ as in
(ii);
if $d(u_4)=3=d(u_5)$ then $\frac{\pi}{15} + c(\hat{\Delta}_3) + c(\hat{\Delta}_6) \leq \frac{7 \pi}{30}$ so add
$\frac{\pi}{10}$ to $c(\hat{\Delta}_9)$ and $\frac{2 \pi}{15}$ to $c(\hat{\Delta}_{10})$ as in (iii);
if $d(u_4)=4$, $d(u_5)=3$ and $d(u_2) > 6$ then $c(\hat{\Delta}_3) \leq - \frac{\pi}{21}$ so add
$\frac{\pi}{21}$ from $c(\Delta)$ to $c(\hat{\Delta}_3)$ and the remaining $\frac{\pi}{15}-\frac{\pi}{21}=\frac{2 \pi}{105}$ to
$c(\hat{\Delta}_6) \leq - \frac{\pi}{21}$ as in (iv);
if $d(u_4)=4$, $d(u_5)=3$ and $d(u_2)=6$ then (checking the star graph $\Gamma_0$ for possible labels shows that) $u_2$ is given by (v) in which
case add $\frac{\pi}{15}+c(\hat{\Delta}_3) + c(\hat{\Delta}_6)=\frac{\pi}{15}$ to $c(\hat{\Delta}_{11})$ as in (v); and
if $d(u_4) > 4$ and $d(u_5)=3$ then add $\frac{\pi}{15}+c(\hat{\Delta}_3) \leq \frac{\pi}{15}$ to $c(\hat{\Delta}_6) \leq - \frac{\pi}{10}$ as in
(vi).

Now consider $\hat{\Delta}_7$ as in Figure 19(vii).
If $d(u_6)>3$ and $d(u_7)>3$ then add
$\frac{\pi}{15}+c(\hat{\Delta}_4) \leq \frac{\pi}{15}$ to $c(\hat{\Delta}_7) \leq -\frac{\pi}{6}$ as in 19(vii);
if $d(u_7)=3$ and $d(u_6) >3$ then add
$\frac{\pi}{15}+c(\hat{\Delta}_6) \leq \frac{\pi}{15}$ to $c(\hat{\Delta}_{14})$ as in (viii);
if $d(u_7)=3=d(u_6)$ then $\frac{\pi}{15}+c(\hat{\Delta}_4)+c(\hat{\Delta}_7)=\frac{7 \pi}{30}$ so add
$\frac{2 \pi}{15}$ to $c(\hat{\Delta}_{13})$ and $\frac{\pi}{10}$ to $c(\hat{\Delta}_{14})$ as in (ix);
if $d(u_7)=4$, $d(u_6)=3$ and $d(u_2) >6$ then $c(\hat{\Delta}_4) \leq -\frac{\pi}{21}$ so add
$\frac{\pi}{21}$ from $c(\Delta)$ to $c(\hat{\Delta}_4)$ and the remaining $\frac{\pi}{15}-\frac{\pi}{21}=\frac{2 \pi}{105}$ to
$c(\hat{\Delta}_7) \leq - \frac{\pi}{21}$ as in (x);
if $d(u_7)=4$, $d(u_6)=3$ and $d(u_2)=6$ then $u_2$ is given by (xi) in which case add
$\frac{\pi}{15}+c(\hat{\Delta}_4)+c(\hat{\Delta}_7)=\frac{\pi}{15}$ to $c(\hat{\Delta}_{12})$ as in (xi); and
if $d(u_7)>4$ and $d(u_6)=3$ then add
$\frac{\pi}{15}+c(\hat{\Delta}_4) \leq \frac{\pi}{15}$ to $c(\hat{\Delta}_7) \leq - \frac{\pi}{10}$ as in (xii).

$\bs{d(v_1)=d(v_3)=3}$.  \textbf{Figures 20-23}.  There are three subcases.

(1)
$\bs{((d(v_2),d(v_4)) \notin \{ (\geq 6,5), (5,\geq 5)\}}$.  \textbf{Figure 20}.  Distribution of $c(\Delta)$ is described in Figure 20 according to possible $d(v_2)$ and $d(v_4)$.  There are two exceptions 
to these rules.  When $\Delta$ is given by Figure 20(v), $\frac{7 \pi}{30}$ is added to $\hat{\Delta}_2$ as indicated, except when  
$\Delta = \Delta_1$ of \textbf{Configuration B} in Figure 31(v,) in which case the dotted lines in Figure 20(v) indicate the new rule, that is, $\frac{\pi}{30},\frac{\pi}{5}$ is added to 
$\hat{\Delta}_1,\hat{\Delta}_2$ respectively
; and when $\Delta$ is given by Figure 20(v), $\frac{7 \pi}{30}$ is added to $\hat{\Delta}_1$ as indicated, except when $\Delta = \Delta_1$ of \textbf{Configuration A} in Figure 31(i), in which case the 
dotted lines in Figure 20(vi) indicate the new rule, that is, 
$\frac{\pi}{5},\frac{\pi}{30}$ is added to $\hat{\Delta}_1,\hat{\Delta}_2$ respectively. 

(2)
$\bs{d(v_2) \geq 6}$ \textbf{and} $\bs{d(v_4)=5}$.  \textbf{Figures 21-22}.  If $\Delta$ is given by Figure 21(i) then $c(\Delta)$ is distributed as shown. Otherwise $l(v_4)=a^5$ and this subcase is now considered using Figures 21(ii)-(xx) and 22.

Let $d(v_2) \geq 6$ and $l(v_4) = a^{5}$.
Then $c(\Delta) \leq \frac{\pi}{15}$, half of which ($\leq \frac{\pi}{30}$) is distributed to $c(\hat{\Delta}_1)$ and $c(\hat{\Delta}_4)$
whilst the other half is distributed to $c(\hat{\Delta}_2)$ and $c(\hat{\Delta}_3)$ (and this will be described in the next paragraph).  The distribution of $\frac{1}{2} c(\Delta)$ to
$c(\hat{\Delta}_1)$ and $c(\hat{\Delta}_4)$ is as follows.
If $d(\hat{\Delta}_1)>4$ then add $\frac{1}{2} c(\Delta) \leq \frac{\pi}{30}$ to $c(\hat{\Delta}_1)$ as in Figure 21(ii), or
if $d(\hat{\Delta}_1)=4$ and $d(\hat{\Delta}_4) >4$ then add $\frac{1}{2} c(\Delta) \leq \frac{\pi}{30}$ to $c(\hat{\Delta}_4)$ again as in (ii).
It can be assumed therefore that $\Delta$, $\hat{\Delta}_1$ and $\hat{\Delta}_j$ ($4 \leq j \leq 8$) are given by Figure 21(iii).
We proceed according to $d(u_4) \geq 3$, $d(u_5) \geq 4$, $d(u_6) \geq 3$ of Figure 21(iii).
If $d(u_6)=3$, $d(u_5)=4$ and $d(u_4) \geq 5$ then add $\frac{1}{2} c(\Delta) \leq \frac{\pi}{30}$ to $c(\hat{\Delta}_4) \leq -\frac{\pi}{30}$ as in
Figure 21(iv);
if $d(u_6)=3$, $d(u_5)=4$ and $d(u_4)=4$ then add $\frac{1}{2} c(\Delta) \leq \frac{\pi}{30}$ to $c(\hat{\Delta}_4)$ and then add
$\frac{\pi}{30}+c(\hat{\Delta}_4) \leq \frac{\pi}{10}$ to $c(\hat{\Delta}_5)$ as in (v);
if $d(u_6)=3$, $d(u_5)=4$ and $d(u_4)=3$ then add $\frac{1}{2} c(\Delta) \leq \frac{\pi}{30}$ to $c(\hat{\Delta}_4)$ and then add
$\frac{\pi}{30}+c(\hat{\Delta}_4) \leq \frac{4 \pi}{15}$ to $c(\hat{\Delta}_6)$ as in (vi);
if $d(u_6)=3$ and $d(u_5)=5$ then add $\frac{1}{2} c(\Delta) \leq \frac{\pi}{30}$ to $c(\hat{\Delta}_1) \leq \frac{\pi}{15}$ and then add
$\frac{\pi}{15}$ to $c(\hat{\Delta}_7)$ and add $\frac{\pi}{30}$ to $c(\hat{\Delta}_8)$ as in (vii);
if $d(u_6)=3$ and $d(u_5) \geq 6$ then add $\frac{1}{2} c(\Delta) \leq \frac{\pi}{30}$ to $c(\hat{\Delta}_1) \leq 0$ and then add
$\frac{\pi}{30}$ to $c(\hat{\Delta}_8)$ as in (viii). This completes $d(u_6)=3$.
If $d(u_6)=4$, $d(u_5) = 4$ and $d(v_2)=7$ (note that $c(3,3,5,8) < 0$) then add $\frac{1}{2} c(\Delta) \leq \frac{\pi}{105}$ to $c(\hat{\Delta}_1) \leq - \frac{\pi}{21}$ as in (ix).
Let $d(u_6)=4$, $d(u_5)=4$ and $d(v_2)=6$ so, in particular, $c(\hat{\Delta}_1)=0$.
If $u_6$ is given by Figure 21(x) then add $\frac{1}{2} c(\Delta)=\frac{\pi}{30}$ to $c(\hat{\Delta}_7)$, so from now on suppose that $u_6$ is given
by
Figure 21(xi).
If $d(\hat{\Delta}_8) > 4$ then add $\frac{1}{2} c(\Delta)=\frac{\pi}{30}$ to $\hat{\Delta}_8$ as shown in Figure 21(xi), so suppose from now on that
$d(\hat{\Delta}_8)=4$.
\textit{Suppose that $\hat{\Delta}_8$ is given by Figure 21(xii).}
If $d(u_4) \geq 5$ then add $\frac{1}{2} c(\Delta) \leq \frac{\pi}{30}$ to $c(\hat{\Delta}_4) \leq - \frac{\pi}{30}$ as in Figure 21(xiii);
if $d(u_4)=4$ then add $\frac{1}{2} c(\Delta) \leq \frac{\pi}{30}$ to $c(\hat{\Delta}_4)$ and then add $\frac{\pi}{30} + c(\hat{\Delta}_4) \leq \frac{\pi}{10}$ to $c(\hat{\Delta}_5)$ as in (xiv); and
if $d(u_4)=3$ then add $\frac{1}{2} c(\Delta) \leq \frac{\pi}{30}$ to $c(\hat{\Delta}_4)$ and then add $\frac{\pi}{30} + c(\hat{\Delta}_4) \leq \frac{4 \pi}{15}$ to $c(\hat{\Delta}_6)$ as in (xv).
\textit{Suppose now that $\hat{\Delta}_8$ is not given by Figure 21(xii).}
Then again add $\frac{1}{2} c(\Delta)=\frac{\pi}{30}$ to $\hat{\Delta}_8$ as in Figure 21(xi).
We proceed according to the degrees of the vertices $w_1$ and $w_2$ of Figure 21(xi).
If $d(w_1)=d(w_2)=3$ then $\frac{1}{2} c(\Delta) + c(\hat{\Delta}_8)=\frac{\pi}{5}$ so add $\frac{\pi}{10}$ to $c(\hat{\Delta}_9)$ and $\frac{\pi}{10}$ to
$c(\hat{\Delta}_{10})$ as shown in Figure 21(xvi);
if $d(w_1)=3$ and $d(w_2) >3$ then $\frac{1}{2}c(\Delta) + c(\hat{\Delta}_8)=\frac{\pi}{30}$ so add $\frac{\pi}{30}$ to $c(\hat{\Delta}_9)$
as shown in (xvii);
if $d(w_1)=4$ and $d(w_2)=3$ then by assumption $\hat{\Delta}_8$ is given by (xviii) and
$c(\hat{\Delta}_8)=0$ so add $\frac{1}{2} c(\Delta) + c(\hat{\Delta}_8) =\frac{\pi}{30}$ to
$c(\hat{\Delta}_{10})$ as shown; and
if either $d(w_1) \geq 5$ and $d(w_2)=3$ or $d(w_1) \geq 4$ and $d(w_2) \geq 4$ then
add $\frac{1}{2} c(\Delta) \leq \frac{\pi}{30}$ to $c(\hat{\Delta}_8) \leq -\frac{\pi}{10}$ as shown in (xix).
This completes $d(u_6)=d(u_5)=4$.
Finally if $d(u_6) \geq 4$ and $d(u_5) \geq 5$ or $d(u_6) \geq 5$ and $d(u_5) = 4$ then add $\frac{1}{2} c(\Delta) = \frac{\pi}{30}$ to
$c(\hat{\Delta}_1) \leq c(3,4,5,6) = -\frac{\pi}{10}$ as shown in Figure 21(xx).

The remaining $\frac{1}{2} c(\Delta) \leq \frac{\pi}{30}$ is distributed to $c(\hat{\Delta}_2)$ and $c(\hat{\Delta}_3)$ as follows.
If $d(\hat{\Delta}_2)>4$ then add $\frac{1}{2} c(\Delta) \leq \frac{\pi}{30}$ to $c(\hat{\Delta}_2)$ as in Figure 21(ii), or
if $d(\hat{\Delta}_2)=4$ and $d(\hat{\Delta}_3) >4$ then add $\frac{1}{2} c(\Delta) \leq \frac{\pi}{30}$ to $c(\hat{\Delta}_3)$ again as in (ii).
It can be assumed therefore that $\Delta$, $\hat{\Delta}_2$, $\hat{\Delta}_3$ and $\hat{\Delta}_j$ ($5 \leq j \leq 8$) are now given by Figure 22(i).
We proceed according to
$d(u_1) \geq 3$, $d(u_2) \geq 4$, $d(u_3) \geq 3$ of Figure 22(i).
If $d(u_1)=3$, $d(u_2)=4$ and $d(u_3) \geq 5$ then add $\frac{1}{2} c(\Delta) \leq \frac{\pi}{30}$ to $c(\hat{\Delta}_3) \leq -\frac{\pi}{30}$ as in
Figure 22(ii);
if $d(u_1)=3$, $d(u_2)=4$ and $d(u_3)=4$ then add $\frac{1}{2} c(\Delta) \leq \frac{\pi}{30}$ to $c(\hat{\Delta}_3)$ and then add
$\frac{\pi}{30}+c(\hat{\Delta}_3) \leq \frac{\pi}{10}$ to $c(\hat{\Delta}_5)$ as in (iii);
if $d(u_1)=3$, $d(u_2)=4$ and $d(u_3)=3$ then add $\frac{1}{2} c(\Delta) \leq \frac{\pi}{30}$ to $c(\hat{\Delta}_3)$ and then add
$\frac{\pi}{30}+c(\hat{\Delta}_3) \leq \frac{4 \pi}{15}$ to $c(\hat{\Delta}_6)$ as in (iv);
if $d(u_1)=3$ and $d(u_2)=5$ then add $\frac{1}{2} c(\Delta) \leq \frac{\pi}{30}$ to $c(\hat{\Delta}_2) \leq \frac{\pi}{15}$ and add
$\frac{\pi}{15}$ to $c(\hat{\Delta}_7)$ and $\frac{\pi}{30}$ to $c(\hat{\Delta}_8)$ as in (v);
if $d(u_1)=3$ and $d(u_2) \geq 6$ then add $\frac{1}{2} c(\Delta) \leq \frac{\pi}{30}$ to $c(\hat{\Delta}_2) \leq 0$ and then add
$\frac{\pi}{30}$ to $c(\hat{\Delta}_8)$ as in (vi). This completes $d(u_1)=3$.
If $d(u_1)=4$, $d(u_2) = 4$ and $d(v_2)=7$ then add $\frac{1}{2} c(\Delta) \leq \frac{\pi}{105}$ to $c(\hat{\Delta}_2) \leq - \frac{\pi}{21}$ as in (vii).
Let $d(u_1)=4$, $d(u_2)=4$ and $d(v_2)=6$ so, in particular, $c(\hat{\Delta}_2)=0$.
If $u_2$ is given by Figure 22(viii) then add $\frac{1}{2} c(\Delta)=\frac{\pi}{30}$ to $c(\hat{\Delta}_7)$, so from now on suppose that $u_2$ is
given by
Figure 22(ix).
If $d(\hat{\Delta}_8) > 4$ then add $\frac{1}{2} c(\Delta)=\frac{\pi}{30}$ to $\hat{\Delta}_8$ as shown in Figure 22(ix), so suppose from now on that
$d(\hat{\Delta}_8)=4$.
\textit{Suppose that $\hat{\Delta}_8$ is given by Figure 22(x).}
If $d(u_3) \geq 5$ then add $\frac{1}{2} c(\Delta) \leq \frac{\pi}{30}$ to $c(\hat{\Delta}_3) \leq - \frac{\pi}{30}$ as in Figure 22(xi);
if $d(u_3)=4$ then add $\frac{1}{2} c(\Delta) \leq \frac{\pi}{30}$ to $c(\hat{\Delta}_3)$ and then add $\frac{\pi}{30}+c(\hat{\Delta}_3) \leq \frac{\pi}{10}$ to $c(\hat{\Delta}_5)$ as in (xii); and
if $d(u_3)=3$ then add $\frac{1}{2} c(\Delta) \leq \frac{\pi}{30}$ to $c(\hat{\Delta}_3)$ and then add $\frac{\pi}{30} + c(\hat{\Delta}_3) \leq \frac{4 \pi}{15}$ to $c(\hat{\Delta}_6)$ as in (xiii).
\textit{Suppose now that $\hat{\Delta}_8$ is not given by Figure 22(x).}
Then again add $\frac{1}{2} c(\Delta) = \frac{\pi}{30}$ to $\hat{\Delta}_8$ as in Figure 22(ix).
We proceed according to the degrees of the vertices $w_3$ and $w_4$ of Figure 22(ix).
If $d(w_3)=d(w_4)=3$ then $\frac{1}{2} c(\Delta) + c(\hat{\Delta}_8)=\frac{\pi}{30} + \frac{\pi}{6} = \frac{\pi}{5}$ so add $\frac{\pi}{10}$ to $c(\hat{\Delta}_9)$ and $\frac{\pi}{10}$ to
$c(\hat{\Delta}_{10})$ as shown in Figure 22(xiv);
if $d(w_3)=3$ and $d(w_4) >3$ then $\frac{1}{2}c(\Delta) + c(\hat{\Delta}_8)=\frac{\pi}{30}$ so add $\frac{\pi}{30}$ to $c(\hat{\Delta}_9)$
as shown in (xv);
if $d(w_3)=4$ and $d(w_4)=3$ then by assumption $\hat{\Delta}_8$ is given by (xvi) and $c(\hat{\Delta}_8)=0$ so add $\frac{1}{2} c(\Delta)=\frac{\pi}{30}$ to
$c(\hat{\Delta}_{10})$ as shown; and
if either $d(w_3) \geq 5$ and $d(w_4)=3$ or $d(w_3) \geq 4$ and $d(w_4) \geq 4$ then
add $\frac{1}{2} c(\Delta) \leq \frac{\pi}{30}$ to $c(\hat{\Delta}_8) \leq -\frac{\pi}{10}$ as shown in (xvii).
This completes $d(u_1)=d(u_2)= 4$.
Finally if $d(u_1) \geq 4$ and $d(u_2) \geq 5$ or $d(u_1) \geq 5$ and $d(u_2) = 4$ then add $\frac{1}{2} c(\Delta) = \frac{\pi}{30}$ to
$c(\hat{\Delta}_2) \leq c(3,4,5,6) = -\frac{\pi}{10}$ as shown in Figure 22(xviii).

(3) $\bs{d(v_2)=5}$ \textbf{and} $\bs{d(v_4) \geq 5}$. \textbf{Figure 23}.
If $\Delta$ is given by Figure 23(i) then add $\frac{1}{2} c(\Delta) \leq \frac{\pi}{15}$ to each of $c(\hat{\Delta}_1)$ and $c(\hat{\Delta}_2)$.
Otherwise $l(v_2)=bx^{-1} \lambda z^{-1} y$ and $\Delta$ is given by Figure 23(ii).  Here add $\frac{1}{2} c (\Delta) \leq \frac{\pi}{15}$ to
$c(\hat{\Delta}_4)$ if $d(\hat{\Delta}_4) > 4$, otherwise add $\frac{1}{2} c(\Delta) \leq \frac{\pi}{15}$ to $c(\hat{\Delta}_1)$; and
add $\frac{1}{2} c(\Delta) \leq \frac{\pi}{15}$ to $c(\hat{\Delta}_3)$ if $d(\hat{\Delta}_3) > 4$, otherwise add $\frac{1}{2} c(\Delta) \leq \frac{\pi}{15}$ to $c(\hat{\Delta}_2)$.
If $\frac{1}{2} c(\Delta) \leq \frac{\pi}{15}$ is added to $c(\hat{\Delta}_1)$ and $d(\hat{\Delta}_1) > 4$ there is no further distribution of curvature from $\hat{\Delta}_1$ and the same statement holds for $\hat{\Delta}_2$.  This leaves the subcases $d(\hat{\Delta}_1)=d(\hat{\Delta}_4)=4$ and
$d(\hat{\Delta}_2) = d(\hat{\Delta}_3)=4$.

Assume first that $d(\hat{\Delta}_1)=d(\hat{\Delta}_4)=4$ in Figure 23(ii).  Then
$\Delta$ is given by Figure 23(iii). We proceed according to  $d(u_5) \geq 4$ and $d(u_6) \geq 3$.
If $d(u_5)=4$ and $d(u_6)=3$ then $\frac{\pi}{15}+c(\hat{\Delta}_1) \leq \frac{3 \pi}{10}$ so add
$\frac{\pi}{10}$ to $c(\hat{\Delta}_7)$ and $\frac{\pi}{5}$ to $c(\hat{\Delta}_8)$ as in Figure 23(iv);
if $d(u_5)=4$ and $d(u_6)=4$ then add $\frac{1}{2} ( \frac{\pi}{15} + c(\hat{\Delta}_1)) \leq \frac{\pi}{15}$ to each of $c(\hat{\Delta}_7)$ and $c(\hat{\Delta}_8)$
if $u_6$ is given by (v), or add $\frac{\pi}{15} + c(\hat{\Delta}_1) \leq \frac{2 \pi}{15}$ to $c(\hat{\Delta}_8)$
if $u_6$ is given by (vi);
if $d(u_5)=4$ and $d(u_6)=5$ then $c(\hat{\Delta}_1)=-\frac{\pi}{30}$ so add
$\frac{\pi}{15}+c(\hat{\Delta}_1) \leq \frac{\pi}{30}$ to $c(\hat{\Delta}_8)$ as in (vii);
if $d(u_5)=4$ and $d(u_6) \geq 6$ then add $\frac{1}{2} c(\Delta) \leq \frac{\pi}{15}$ to $c(\hat{\Delta}_1) \leq -\frac{\pi}{10}$ as in (viii);
if $d(u_5)=5$ and $d(u_6)=3$ then $\frac{\pi}{15}+c(\hat{\Delta}_1) \leq \frac{\pi}{5}$ so add
$\frac{\pi}{15}$ to $c(\hat{\Delta}_7)$ and $\frac{2 \pi}{15}$ to $c(\hat{\Delta}_8)$ as in (ix);
if $d(u_5) = 5$ and $d(u_6)=4$ then $c(\hat{\Delta}_1) \leq - \frac{\pi}{30}$ so add
$\frac{\pi}{15}+c(\hat{\Delta}_1) \leq \frac{\pi}{30}$ to $c(\hat{\Delta}_8)$ as shown in the two possibilities for $u_6$, namely (x) and (xi);
if $d(u_5)=5$ and $d(u_6) \geq 5$ then add
$\frac{1}{2}c(\Delta) \leq \frac{\pi}{15}$ to $c(\hat{\Delta}_1) \leq -\frac{2 \pi}{15}$ as in (xii);
if $d(u_5)>5$ and $d(u_6)=3$ then add
$\frac{\pi}{15} + c(\hat{\Delta}_1) \leq \frac{2 \pi}{15}$ to $c(\hat{\Delta}_8)$ as in (xiii); and
if $d(u_5) >5$ and $d(u_6) >3$ then add
$\frac{1}{2} c(\Delta) \leq \frac{\pi}{15}$ to $c(\hat{\Delta}_1) \leq -\frac{\pi}{10}$ as in (xiv).

Now assume that $d(\hat{\Delta}_2)=d(\hat{\Delta}_3)=4$ in Figure 23(ii).  Then
$\Delta$ is given by Figure 23(xv).  We proceed according to $d(u_1) \geq 3$ and $d(u_2) \geq 4$.
If $d(u_2)=4$ and $d(u_1)=3$ then $\frac{\pi}{15}+c(\hat{\Delta}_2) \leq \frac{3 \pi}{10}$ so add $\frac{\pi}{5}$ to
$c(\hat{\Delta}_9)$ and $\frac{\pi}{10}$ to $c(\hat{\Delta}_{10})$ as in Figure 23(xvi);
if $d(u_2)=4$ and $d(u_1)=4$ then add $\frac{1}{2} ( \frac{\pi}{15} + c(\hat{\Delta})) \leq \frac{\pi}{15}$ to each of $c(\hat{\Delta}_9)$ and $c(\hat{\Delta}_{10})$
if $u_1$ is given by (xvii), or $\frac{\pi}{15}+c(\hat{\Delta}_2) \leq \frac{2 \pi}{15}$ to $c(\hat{\Delta}_9)$
if $u_1$ is given by (xviii);
if $d(u_2)=4$ and $d(u_1)=5$ then
$c(\hat{\Delta}_2)=-\frac{\pi}{30}$ so $\frac{\pi}{15}+c(\hat{\Delta}_2) \leq \frac{\pi}{30}$ is added to $c(\hat{\Delta}_9)$ as shown in (xix);
if $d(u_2)=4$ and $d(u_1) \geq 6$ then add $\frac{1}{2} c(\Delta) \leq \frac{\pi}{15}$ to $c(\hat{\Delta}_2) \leq -\frac{\pi}{10}$ as in (xx);
if $d(u_2) = 5$ and $d(u_1)=3$ then $\frac{\pi}{15} + c(\hat{\Delta}_2) \leq \frac{\pi}{5}$ so add $\frac{2 \pi}{15}$ to $c(\hat{\Delta}_9)$ and
$\frac{\pi}{15}$ to $c(\hat{\Delta}_{10})$ as in (xxi);
if $d(u_2) = 5$ and $d(u_1)=4$ then $c(\hat{\Delta}_2) = -\frac{\pi}{30}$ so add $\frac{\pi}{15} + c(\hat{\Delta}_2) = \frac{\pi}{30}$ to
$c(\hat{\Delta}_9)$ as shown in the two possibilities (xxii) and (xxiii);
if $d(u_2)=5$ and $d(u_1) > 4$ then add $\frac{1}{2} c(\Delta) \leq \frac{\pi}{15}$ to $c(\hat{\Delta}_2) \leq -\frac{2 \pi}{15}$ as in (xxiv);
if $d(u_2)>5$ and $d(u_1)=3$ then add $\frac{\pi}{15} + c(\hat{\Delta}_2) \leq \frac{2 \pi}{15}$
to $c(\hat{\Delta}_9)$ as in (xxv); and
if $d(u_2) > 5$ and $d(u_1) >3$ then add $\frac{1}{2} c(\Delta) \leq \frac{\pi}{15}$ to $c(\hat{\Delta}_2) \leq -\frac{\pi}{10}$ as in (xxvi).

$\bs{d(v_1)=3}$.  \textbf{Figure 24}.  Either $d(v_3)=5$ or $d(v_4)=5$ or $d(v_3)=d(v_4)=4$.

$\bs{d(v_2)=3}$.  \textbf{Figure 25}.  Either $d(v_1)=5$ or $d(v_3)=5$ or $d(v_1)=d(v_3)=4$.

$\bs{d(v_3)=3}$.  \textbf{Figure 26}.  Either $d(v_1)=5$ or $d(v_4)=5$ or $d(v_1)=d(v_4)=4$.

$\bs{d(v_4)=3}$.  \textbf{Figures 27-30}.  There are four subcases.

(1)
$\bs{(d(v_1),d(v_3),l(v_2)) \neq (4,4,b^5)}$.  \textbf{Figure 27}.  Either $d(v_1)=5$ or $d(v_3)=5$ or
$d(v_1)=d(v_3)=4$ but $l(v_2) \neq b^5$ and the distribution of curvature is as shown.

Assume from now on that $d(v_1)=d(v_3)=4$ and $l(v_2)=b^5$.

(2)
$\bs{(d(\hat{\Delta}_3),d(\hat{\Delta}_4)) \neq (4,4)}$.  \textbf{Figure 28}.  $c(\Delta)=\frac{\pi}{15}$ is distributed as shown.

Now assume that $d(\hat{\Delta}_3)=d(\hat{\Delta}_4)=4$ as shown in Figure 29(i).  If
$c(\hat{\Delta}_3) \leq - \frac{\pi}{15}$ then add $c(\Delta)=\frac{\pi}{15}$ to $c(\hat{\Delta}_3)$ 
or if $c(\hat{\Delta}_4) \leq - \frac{\pi}{15}$ then add $c(\Delta)=\frac{\pi}{15}$ to $c(\hat{\Delta}_4)$ as shown in Figure 29(i).  Assume that 
$c(\hat{\Delta}_3) > - \frac{\pi}{15}$ and 
$c(\hat{\Delta}_4) > - \frac{\pi}{15}$.  There are two subcases according to $d(u_2) \geq 4$.

(3)
$\bs{4 \leq d(u_2) \leq 5}$.  \textbf{Figure 29}.

Let $d(u_2)=4$.
If $d(u_1) \geq 6$ then add $c(\Delta)+c(\hat{\Delta}_3) \leq \frac{\pi}{15}$ to $c(\hat{\Delta}_6)$ as in Figure 29(ii);
if $d(u_1)=5$ then add $c(\Delta)+c(\hat{\Delta}_3)=\frac{2 \pi}{15}$ to $c(\hat{\Delta}_6)$
if $\hat{\Delta}_3$ is given by (iii), or add $\frac{1}{2} (c(\Delta)+c(\hat{\Delta}_3))=\frac{\pi}{15}$ to each of $c(\hat{\Delta}_5)$ and $c(\hat{\Delta}_6)$ if $\hat{\Delta}_3$ is given by (iv);
if $d(u_1)=4$ then add $c(\Delta)+c(\hat{\Delta}_3)=\frac{7 \pi}{30}$ to $c(\hat{\Delta}_6)$ as in (v); and
if $d(u_1)=3$ then $c(\Delta)+c(\hat{\Delta}_3)=\frac{6 \pi}{15}$ so add
$\frac{4 \pi}{15}$ to $c(\hat{\Delta}_5)$ and $\frac{2 \pi}{15}$ to $c(\hat{\Delta}_6)$ as in (vi).

Let $d(u_2)=5$ in which case $l(u_2)=a^5$.  In this case add $\frac{1}{2} c(\Delta)=\frac{\pi}{30}$ to each of $c(\hat{\Delta}_3)$ and
$c(\hat{\Delta}_4)$.
If $d(u_1) \geq 5$ then $c(\hat{\Delta}_3) \leq c(3,4,5,5)=-\frac{\pi}{30}$ and the $\frac{\pi}{30}$ from $c(\Delta)$ remains with
$c(\hat{\Delta}_3)$ as shown in Figure 29(vii);
if $d(u_1)=4$ then add $\frac{\pi}{30}+c(\hat{\Delta}_3)=\frac{\pi}{10}$ to $c(\hat{\Delta}_6)$ as in (viii); and
if $d(u_1)=3$ then add $\frac{\pi}{30}+c(\hat{\Delta}_3)=\frac{4 \pi}{15}$ to $c(\hat{\Delta}_5)$ as in (ix).
If $d(u_3) \geq 5$ then $c(\hat{\Delta}_4) \leq -\frac{\pi}{30}$ and the $\frac{\pi}{30}$ from $c(\Delta)$ remains with $c(\hat{\Delta}_4)$
as shown in Figure 29(vii);
if $d(u_3)=4$ then add $\frac{\pi}{30} + c(\hat{\Delta}_4)=\frac{\pi}{10}$ to $c(\hat{\Delta}_7)$ as in (x); and
if $d(u_3)=3$ then add $\frac{\pi}{30}+c(\hat{\Delta}_4)=\frac{4 \pi}{15}$ to $c(\hat{\Delta}_8)$ as in (xi).

(4) $\bs{d(u_2) \geq 6}$. \textbf{Figure 30}.
If $d(u_1) > 4$ then add $c(\Delta)=\frac{\pi}{15}$ to $c(\hat{\Delta}_3) \leq -\frac{\pi}{10}$; or
if $d(u_3) > 4$ then add $c(\Delta)=\frac{\pi}{15}$ to $c(\hat{\Delta}_4) \leq -\frac{\pi}{10}$ as in Figure 30(i).
If $d(u_1)=3$ then add $c(\Delta)+c(\hat{\Delta}_3) \leq \frac{\pi}{15} + \frac{\pi}{6}=\frac{7 \pi}{30}$ to $c(\hat{\Delta}_5)$ as shown in Figure 30(ii); or
if $d(u_3)=3$ then add $c(\Delta)+c(\hat{\Delta}_4) \leq \frac{7 \pi}{30}$ to $c(\hat{\Delta}_8)$ as shown in Figure 30(ii).
This leaves $d(u_1)=d(u_3)=4$.
If $d(u_2) \geq 7$ then add $\frac{1}{2}c(\Delta)=\frac{\pi}{30}$ to each of $c(\hat{\Delta}_3) \leq -\frac{\pi}{21}$ and $c(\hat{\Delta}_4) \leq
-\frac{\pi}{21}$ as in Figure 30(iii),
so assume that $d(u_2)=6$.
If $d(\hat{\Delta}_6) > 4$ then add $c(\Delta)=\frac{\pi}{15}$ to $c(\hat{\Delta}_6)$, or
if $d(\hat{\Delta}_7) > 4$ then add $c(\Delta)=\frac{\pi}{15}$ to $c(\hat{\Delta}_7)$ as in Figure 30(iv).
It can be assumed that $d(\hat{\Delta}_6)=d(\hat{\Delta}_7)=4$ which forces $l(u_2)=aaxy^{-1} xy^{-1}$ as shown in Figure 30(v).
If $d(u_4) > 3$ and $d(u_5) > 3$ in Figure 30(v) then add $c(\Delta)+c(\hat{\Delta}_3)=\frac{\pi}{15}$ to $c(\hat{\Delta}_6) \leq -\frac{\pi}{6}$ as
shown;
if $d(u_4)=3$ and $d(u_5) > 3$ then add $c(\Delta)+c(\hat{\Delta}_3)+c(\hat{\Delta}_6) \leq \frac{\pi}{15}$ to $c(\hat{\Delta}_9)$ as in (vi);
if $d(u_4)=d(u_5)=3$ then $c(\Delta)+c(\hat{\Delta}_3)+c(\hat{\Delta}_6)=\frac{7 \pi}{30}$ and so add $\frac{\pi}{10}$ to $c(\hat{\Delta}_9)$, $\frac{\pi}{15}$ to $c(\hat{\Delta}_{10})$
and $\frac{\pi}{15}$ to $c(\hat{\Delta}_{11})$ as in (vii); and
if $d(u_4)>3$ and $d(u_5)=3$ then add $c(\Delta)+c(\hat{\Delta}_3)+c(\hat{\Delta}_6)=\frac{\pi}{15}$ to $c(\hat{\Delta}_{11})$ as in (viii).

\medskip

This completes the description of distribution of curvature from $\Delta$ when $d(\Delta) =4$
except for six exceptional configurations which we now describe and for which there is an amendment to the rules given above.
(Indeed the amendments, as they relate to Figures 9,13,14 and 20 have already been described. In what follows we detail the amendments as they relate to Figures 31 and 32.)

\medskip

\textbf{Configuration A}.  This is shown in Figure 31(i) in which $c(\Delta_1)=\frac{7 \pi}{30}$ and
$c(\Delta_3)=\frac{\pi}{3}$.  The region $\Delta_1$ in Figure 31(i) corresponds to the region $\Delta$ in Figure 20(vi); and the region $\Delta_3$ in Figure 31(i)-(iv) corresponds to the region $\Delta$ in Figure 13(ii).  The new
rule is: add $\frac{\pi}{5}$ from $c(\Delta_1)$ to $c(\hat{\Delta})$ and add $\frac{\pi}{30}$ from $c(\Delta_1)$ to $c(\hat{\Delta}_1)$ as shown by dotted lines in Figure 31(i) \textit{except} when the neighbouring regions of 
$\Delta_3$ are given by Figure 31(ii)-(iv).  There it is assumed that $\hat{\Delta}_2$ receives $\frac{\pi}{5}$ from $\Delta_4$; and so the region $\Delta_4$ of Figure 31(ii)-(iii) corresponds to the region $\Delta$ of Figure 7(iii), and the region $\Delta_4$ of Figure 31(iv) corresponds to the region $\Delta$ of Figure 10.  When $\Delta_3$ is given by 31(ii)-(iv) add \textit{all} of $c(\Delta_1)=\frac{7 \pi}{30}$ to $c(\hat{\Delta})$ (as shown in Figure 31(i)) as usual and the new rule is: add $\frac{3 \pi}{10}$ from $c(\Delta_3)$ to $c(\hat{\Delta})$ and add $\frac{\pi}{30}$ from $c(\Delta_3)$ to $c(\hat{\Delta}_3)$ as shown by dotted lines.  Moreover, if $d(\hat{\Delta}_3)=4$ then add
$\frac{\pi}{30} + c(\hat{\Delta}_3) = \frac{\pi}{10}$ to $c(\hat{\Delta}_4)$ as shown by dotted line in Figure 31(ii).  Note that it is being assumed that $d(\hat{\Delta}_3) \neq 4$ in Figures 31(iii)-(iv), in which case $\hat{\Delta}_3$ is not given by Figure 4(ii) or (iii) and so $d(\hat{\Delta}_3) \geq 8$.

\textbf{Configuration B}.  This is shown in Figure 31(v) in which $c(\Delta_1) = \frac{7 \pi}{30}$ and
$c(\Delta_3)=\frac{\pi}{3}$.  The region $\Delta_1$ in Figure 31(v) corresponds to the region $\Delta$ in Figure 20(v); and the region $\Delta_3$ in Figure 31(v)-(viii) corresponds to the region $\Delta$ in Figure 14(ii).  The new 
rule is: add $\frac{\pi}{5}$ from $c(\Delta_1)$ to $c(\hat{\Delta})$ and add $\frac{\pi}{30}$ from $c(\Delta_1)$ to $c(\hat{\Delta}_1)$ as shown by dotted lines in Figure 31(v) \textit{except} when the neighbouring regions of $\Delta_3$ are given by Figure 31(vi)-(viii).  There it is assumed that $\hat{\Delta}_2$ receives $\frac{\pi}{5}$ from $\Delta_4$; and so the region $\Delta_4$ in Figure 31(vi)-(vii) corresponds to the region $\Delta$ in Figure 7(iii), and the region $\Delta_4$ in Figure 31(viii) corresponds to the region $\Delta$ in Figure 8(iv).  When $\Delta_3$ is given by Figure 31(vi)-(viii) add \textit{all} of $c(\Delta_1) = \frac{7 \pi}{30}$ to $c(\hat{\Delta})$ (as shown in Figure 31(v)) as usual and the new rule is: add $\frac{3 \pi}{10}$ from $c(\Delta_3)$ to $c(\hat{\Delta})$ and add $\frac{\pi}{30}$ as shown by dotted lines.  Moreover, if $d(\hat{\Delta}_3)=4$ then add
$\frac{\pi}{30} + c(\hat{\Delta}_3) = \frac{\pi}{10}$ to $c(\hat{\Delta}_4)$ as shown by dotted line in Figure 31(vi).  Note that it is being assumed that $d(\hat{\Delta}_3) \neq 4$ in Figure 31(vii)-(viii), in which case $\hat{\Delta}_3$ is not given by Figure 4(ii) or (iii) and so $d(\hat{\Delta}_3) \geq 8$.

\textbf{Configurations C and D}.  These are shown in Figure 32(i), (ii).  The region $\Delta$ of Figure 32(i) corresponds to the region $\Delta$ of Figure 13(iii); and the region $\Delta$ of Figure 32(ii) corresponds to the region $\Delta$ of Figure 14(iii).  In both cases the new rule (given by dotted lines) is: add $\frac{3 \pi}{10}$ from $c(\Delta)$ to $c(\hat{\Delta})$ and add $\frac{\pi}{30}$ from $c(\Delta)$ to $c(\hat{\Delta}_1)$.

\textbf{Configurations E and F}.  This is shown in Figure 32.  There are two cases, namely when $d(v) \geq 4$ and when $d(v)=3$ for the vertex $v$ indicated.  If $d(v) \geq 4$ then the region $\Delta_1$ in Figure 32(iii) 
corresponds to the region $\Delta$ of Figure 13(iv); and the region $\Delta_1$ of Figure 32(v) corresponds to the region $\Delta$ in Figure 14(iv).  If $d(v)=3$ then the region $\Delta_1$ in Figure 32(iv) corresponds to the region $\Delta$ in Figure 9(v); and the region $\Delta_1$ in Figure 32(vi) corresponds to the region $\Delta$ in Figure 9(iv).  For Figure 32(iii) and (v) the new rule is: instead of adding $c(\Delta_1) \leq \frac{\pi}{3}$ to $c(\hat{\Delta})$, as in Figures 13(i) and 14(i), add $\min \{ c(\Delta_1), \frac{\pi}{5} \}$ from $c(\Delta_1)$ to $c(\hat{\Delta}_1)$ via $\hat{\Delta}$ across the edge shown; and add (at most) $\frac{2 \pi}{15}$ from $c(\Delta_1)$ to $c(\hat{\Delta})$.
For Figure 32(iv) and (vi) the new rule is: instead of add $\frac{\pi}{6}$ from $c(\Delta_1)$ to each of $c(\hat{\Delta})$ and $c(\hat{\Delta}_1)$, as in Figure 9(ii), (iii), add $\frac{\pi}{5}$ from 
$c(\Delta_1)$ to 
$c(\hat{\Delta}_1)$ and add $\frac{2 \pi}{15}$ from $c(\Delta_1)$ to $c(\hat{\Delta})$ as shown by the dotted lines.  Note that $d(\hat{\Delta}_1) \geq 8$ in Figure 32(iii)-(vi).


\section{Proof of Proposition 4.1}

Let $\hat{\Delta}$ ($\neq \Delta_0$) be a region that receives positive curvature in Figures 6-32.  Then inspection of these figures shows that $d(\hat{\Delta}) \geq 6$ in Figure 6-12, 13(i), (iii), (iv), 14(i), (iii), (iv), 15-16, 20, 24-28 and 32.

\begin{lemma}
Let $\hat{\Delta}$ be a region of degree 4 that receives positive curvature across at least one edge in Figures 6-32.  Then one of the following holds.
\begin{enumerate}
\item[(i)]
$\hat{\Delta}$ occurs in Figure 17, 18 or 19, in which case we say that $\hat{\Delta}$ is a T24 region.
\item[(ii)]
$\hat{\Delta}$ occurs in Figure 21, 22 or 23, in which case we say that $\hat{\Delta}$ is a T13 region.
\item[(iii)]
$\hat{\Delta}$ occurs in Figure 29 or 30, in which case we say that $\hat{\Delta}$ is a T4 region.
\item[(iv)]
$\hat{\Delta} = \hat{\Delta}_3$ of Figure 31(ii) $= \hat{\Delta}$ of Figure 13(ii).
\item[(v)]
$\hat{\Delta}=\hat{\Delta}_3$ of Figure 31(vi) $= \hat{\Delta}$ of Figure 14(ii).
\end{enumerate}
\end{lemma}

\begin{proof}
The result for Figures 6-30 and 32 follows immediately from the statement preceding the lemma.  To complete the proof observe that all regions in Figure 31 other than $\hat{\Delta}_3$ of (ii) and (vi) that receive positive 
curvature have degree $>4$.
\end{proof}

\newpage   
\begin{figure}
\begin{center}
\psfig{file=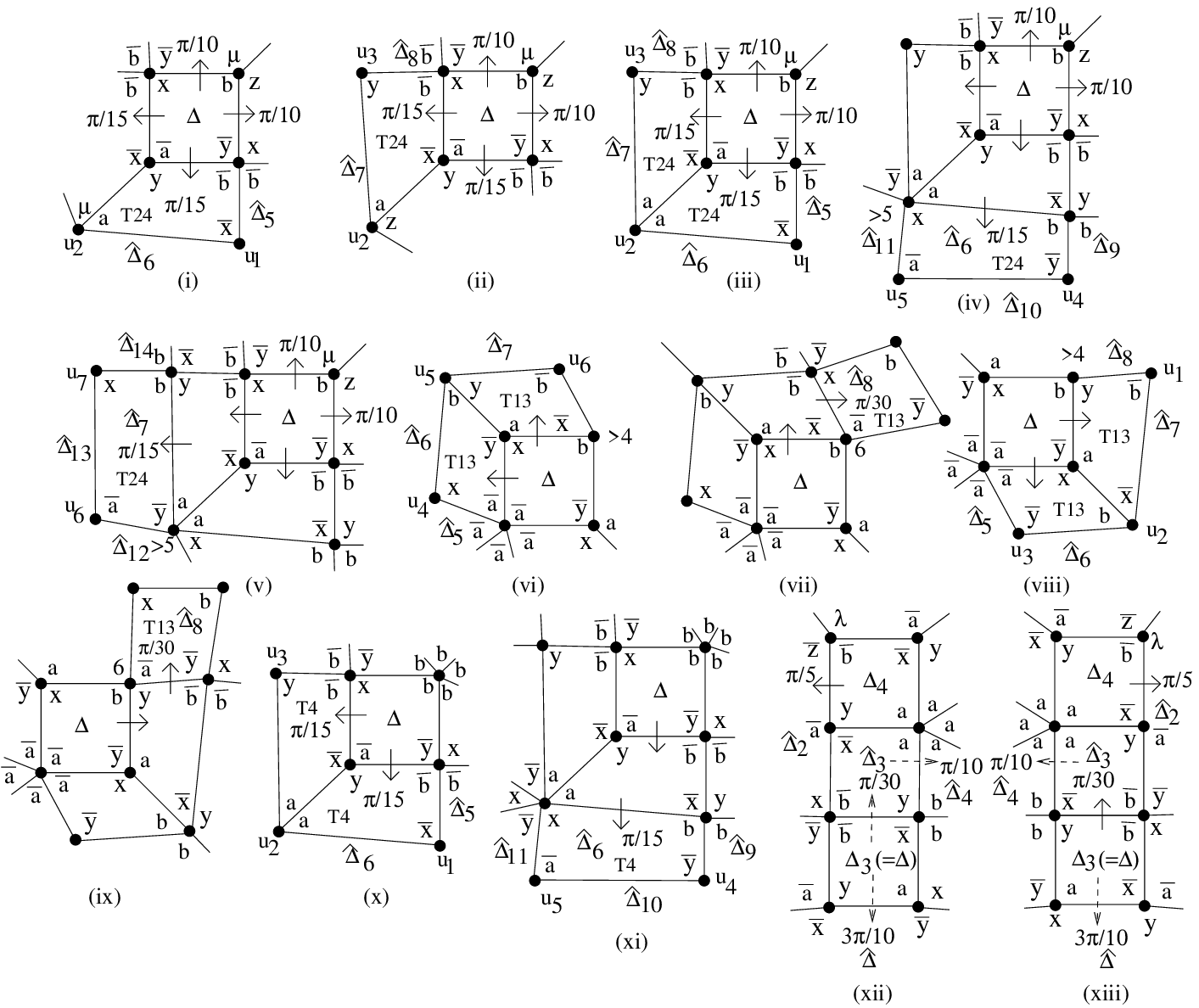}
\end{center}   
\caption{} 
\end{figure}

We remark here that if $\hat{\Delta}$ is a T24 region then an inspection of Figures 17-19 shows that there are essentially six cases for $\hat{\Delta}$, namely $\hat{\Delta}=\hat{\Delta}_3$ of Figure 17(i) and this is again shown 
in Figure 33(i); $\hat{\Delta}=\hat{\Delta}_4$ of Figure 17(ix), see Figure 33(ii); $\hat{\Delta}=\hat{\Delta}_3$ or $\hat{\Delta}_4$ of Figure 18(iii), see Figure 33(iii); $\hat{\Delta}=\hat{\Delta}_6$ of Figure 19(i) for which it is no longer assumed that $d(u_4) > 3$ or that $d(u_5) > 3$, see Figure 33(iv); or $\hat{\Delta}=\hat{\Delta}_7$ of Figure 19(vii) for which it is no longer assumed that $d(u_6) > 3$ or that $d(u_7) > 3$, see Figure 33(v).  If $\hat{\Delta}$ is a T13 region then an inspection of Figures 21-23 shows that there are six cases for $\hat{\Delta}$, namely
$\hat{\Delta}=\hat{\Delta}_1$ of Figure 21(iii) but with $d(v_2) \geq 5$ to take Figure 23(iii)-(xiv) into account, see Figure 33(vi); $\hat{\Delta}=\hat{\Delta}_4$ of Figure 21(iii), see Figure 33(vi); 
$\hat{\Delta}=\hat{\Delta}_8$ of Figure 21(xi) \textit{under the assumption} that $d(\hat{\Delta}_8)=4$ and that $\hat{\Delta}_8$ is \textit{not} given by Figure 21(xii), see Figure 33(vii); $\hat{\Delta}=\hat{\Delta}_2$ of Figure 
22(i) but with $d(v_2) \geq 
5$ to take Figure 23(xv)-(xxvi) into account, see Figure 33(viii); $\hat{\Delta}=\hat{\Delta}_3$ of Figure 22(i), see Figure 33(viii); or $\hat{\Delta}=\hat{\Delta}_8$ of Figure 22(ix) \textit{under the assumption} that $d(\hat{\Delta}_8)=4$ and that
$\hat{\Delta}_8$ is \textit{not} given by Figure 22(x), see Figure 33(ix).  If $\hat{\Delta}$ is a T4 region then inspecting Figures 29-30 shows that there are three cases for $\hat{\Delta}$, 
namely $\hat{\Delta}=\hat{\Delta}_3$ or $\hat{\Delta}_4$ of Figure 29(i), see Figure 33(x); or $\hat{\Delta}=\hat{\Delta}_6$ of Figure 30(v) where it is no longer assumed that $d(u_4) > 3$ or that $d(u_5) > 3$, see 
Figure 33(xi).  The two remaining possibilities for $\hat{\Delta}$, namely $\hat{\Delta}=\hat{\Delta}_3$ of Figure 31(ii), (vi) are given by Figure 33(xii), (xiii).

\begin{lemma}
\textit{Let $\hat{\Delta}$ be a region of degree 4 that receives positive curvature across at least one edge.  Then one of the following occurs.}
\begin{enumerate}
\item[(i)]
$c^{\ast} (\hat{\Delta}) \leq 0$\textit{;}
\item[(ii)]
$c^{\ast} (\hat{\Delta}) > 0$ \textit{is distributed to a region of degree $>4$;}
\item[(iii)]
$c^{\ast} (\hat{\Delta}) > 0$ \textit{is distributed to a region $\Delta '$ of degree 4 and either $c^{\ast} (\Delta ') \leq 0$ or $c^{\ast} (\Delta ') > 0$ is distributed to a region of degree $>4$.}
\end{enumerate}
\end{lemma}

\textit{Proof}.
Let $d(\hat{\Delta})=4$. By Lemma 6.1, $\hat{\Delta}$ is a T24, T13 or T4 region or  $\hat{\Delta} = \hat{\Delta}_3$ of Figure 31(ii),(vi).
We divide the proof of the lemma into two parts.  The first deals with the cases 
when $\hat{\Delta}$ receives positive curvature across exactly one edge and the second part deals with the cases in which $\hat{\Delta}$ receives positive curvature across at least two edges.

If $\hat{\Delta}$ receives positive curvature across exactly one edge then we see by inspection of Figures 17-19, 21-23, 29-30, 31(ii) and
(vi) that in all cases either $c^{\ast} (\hat{\Delta}) \leq 0$
or $c^{\ast} (\hat{\Delta})$ is distributed from $\hat{\Delta}$ to a neighbouring region of degree $> 4$ except when $\hat{\Delta}$ is given by
Figures 19, 21(xi), 
22(ix) or 30(v)-(viii) where $c^{\ast}(\hat{\Delta})$ may initially be distributed further to a region $\Delta '$ of degree 4.
But in each of these cases either $c^{\ast} (\Delta ') \leq 0$ (under the assumption that $\Delta '$ receives positive curvature across exactly one edge -- the case when $\Delta '$ may receive across more 
than one edge is considered below) or $c^{\ast} (\Delta ')$ is again distributed to a region of degree $>4$.

Now suppose that $\hat{\Delta}$ receives positive curvature across at least two edges.
An inspection of the labelling and degrees of the vertices in each of these 17 possibilities for $\hat{\Delta}$ shown in 
Figure 33 immediately rules out the following combinations:
a T24 region with a T24; a T24 with a T4; a T4 with a T4; and either Figure 33(xii) or 33(xiii) with any of the other 16 possibilities.
For example, $\Delta$ of Figure 33(viii) cannot coincide with the inverse of $\Delta_4$ of (xii) as the degrees of the b-corner vertices differ.
This leaves the possibility that at least two T13 regions coincide or a T13 coincides with a T24 or a T13 coincides with a T4.

Suppose that at least two T13 regions coincide.  An inspection of the six T13 regions of Figure 33(vi)-(ix) shows that all
combinations are immediately ruled out by the labelling and 
degree of vertices except for three cases.  
The first case is $\hat{\Delta}$ = $\hat{\Delta}_4$ of Figure 33(vi) = $\hat{\Delta}_3$ of Figure 33(vii). 
This, for example, forces $l(u_5) = ybx^{-1}w$ in Figure 33(vi), in 
particular, $d(u_5) > 4$. But observe that if $\hat{\Delta}_4$ has degree 4 and receives positive curvature from $\Delta$
in Figures 21-23 then $d(u_5) = 4$, a contradiction. 
The second case is  $\hat{\Delta}$ =
$\hat{\Delta}_8$ of Figure 33(vii) = $\hat{\Delta}_8$ of Figure 33(ix).  But this forces $l(v_2)=bx^{-1} a^{-1} ybw$ in Figure 33(vii), and the
fact that
$d(v_2) = 6$ then forces $l(v_2)=bx^{-1} a^{-1} ybb$, a label whose $t$-exponent sum is equal to 6, 

\newpage
\begin{figure}
\begin{center}
\psfig{file=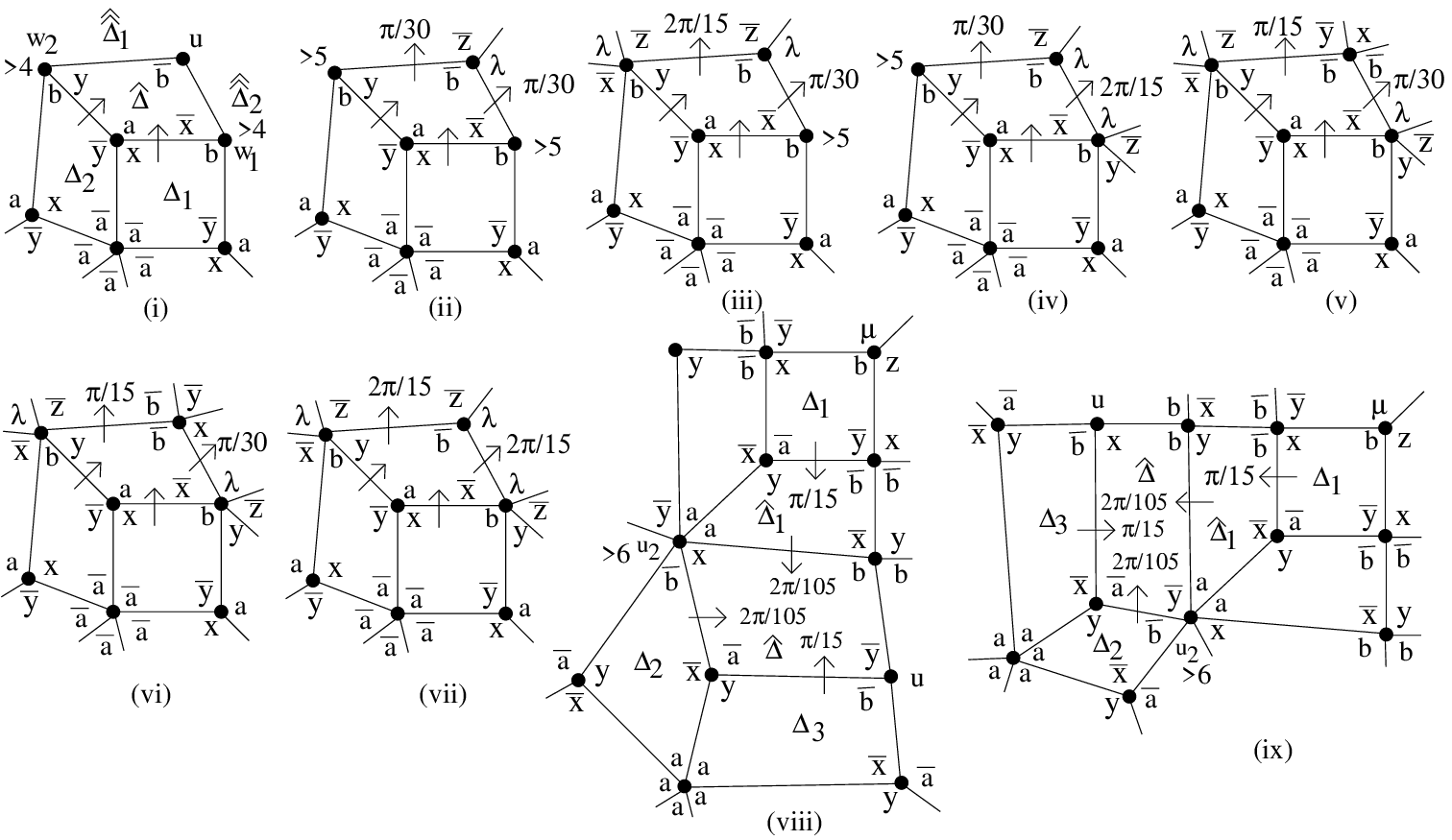}
\end{center}
\caption{}
\end{figure}

\noindent a contradiction.  The third case is 
$\hat{\Delta}$ = $\hat{\Delta}_1$ of Figure 33(vi) = $\hat{\Delta}_2$ of Figure 33(viii).  This case can occur and is shown in Figure 34(i).  It follows that
a combination of more than two T13 regions cannot occur.

Consider Figure 34(i) in which $\hat{\Delta}$ receives positive curvature from the regions $\Delta_1$ and $\Delta_2$ each contributing
at most $\frac{\pi}{15}$ to $c(\hat{\Delta})$.
(Note that we use $\Delta_1,\Delta_2$ and not $\Delta$ as before to denote regions from which positive curvature is distributed.)
Let $d(w_1)>5$ and $d(w_2)>5$.
If $d(u)>3$ then $c^{\ast}(\hat{\Delta}) \leq c(3,4,6,6) + 2 \left( \frac{\pi}{30} \right) < 0$; and
if $d(u)=3$ then $c(\hat{\Delta})+2 \left( \frac{\pi}{30} \right) \leq c(3,3,6,6) + 2 \left( \frac{\pi}{30} \right) = \frac{\pi}{15}$ so add
$\frac{\pi}{30}$ to each of $c(\hat{\hat{\Delta}}_1),c(\Hat{\Hat{\Delta}}_2)$ as shown in Figure 34(ii).
Let $d(w_1) > 5$ and $d(w_2)=5$.
If $d(u) > 3$ then $c^{\ast}(\hat{\Delta}) \leq c(3,4,5,6) + \frac{\pi}{30} + \frac{\pi}{15} = 0$; and
if $d(u)=3$ then $c(\hat{\Delta}) + \frac{\pi}{30} + \frac{\pi}{15} \leq c(3,3,5,6) + \frac{\pi}{30} + \frac{\pi}{15} = \frac{\pi}{6}$ so add
$\frac{2 \pi}{15}$ to $c(\hat{\hat{\Delta}}_1)$ and $\frac{\pi}{30}$ to $c(\hat{\hat{\Delta}}_2)$ as shown in Figure 34(iii).
Let $d(w_1)=5$ and $d(w_2)>5$.
If $d(u)>3$ then $c^{\ast}(\hat{\Delta}) \leq c(3,4,5,6) + \frac{\pi}{15} + \frac{\pi}{30} = 0$; and
if $d(u)=3$ then $c(\hat{\Delta})+\frac{\pi}{15} + \frac{\pi}{30} \leq c(3,3,5,6) + \frac{\pi}{15} + \frac{\pi}{30} = \frac{\pi}{6}$ so add
$\frac{\pi}{30}$ to $c(\hat{\hat{\Delta}}_1)$ and $\frac{2 \pi}{15}$ to $c(\hat{\hat{\Delta}}_2)$ as shown in Figure 34(iv).
This leaves $d(w_1)=d(w_2)=5$.
If $d(u) > 4$ then $c^{\ast}(\hat{\Delta}) \leq c(3,5,5,5) + 2 \left( \frac{\pi}{15} \right) = 0$;
if $d(u)=4$ then $c(\hat{\Delta})+2 \left( \frac{\pi}{15} \right)=c(3,4,5,5) + 2 \left( \frac{\pi}{15} \right) = \frac{\pi}{10}$ so add
$\frac{\pi}{15}$ to $c(\hat{\hat{\Delta}}_1)$ and $\frac{\pi}{30}$ to $c(\hat{\hat{\Delta}}_2)$ as shown in Figures 34(v) and (vi); and
if $d(u)=3$ then $c(\hat{\Delta}) + 2 \left( \frac{\pi}{15} \right) \leq c(3,3,5,5) + \frac{2 \pi}{15} = \frac{4 \pi}{15}$ so add $\frac{2 \pi}{15}$ to
$c(\hat{\hat{\Delta}}_1)$ and $\frac{2 \pi}{15}$ to $c(\hat{\hat{\Delta}}_2)$ as shown in Figure 34(vii).
Observe that $d(\hat{\hat{\Delta}}_1) \geq 6$ and $d(\hat{\hat{\Delta}}_2) \geq 6$ in Figure 34(ii)-(vii).

Now suppose that a T4 region and a T13 region coincide.  Again an inspection of Figure 

\newpage
\begin{figure}
\begin{center}
\psfig{file=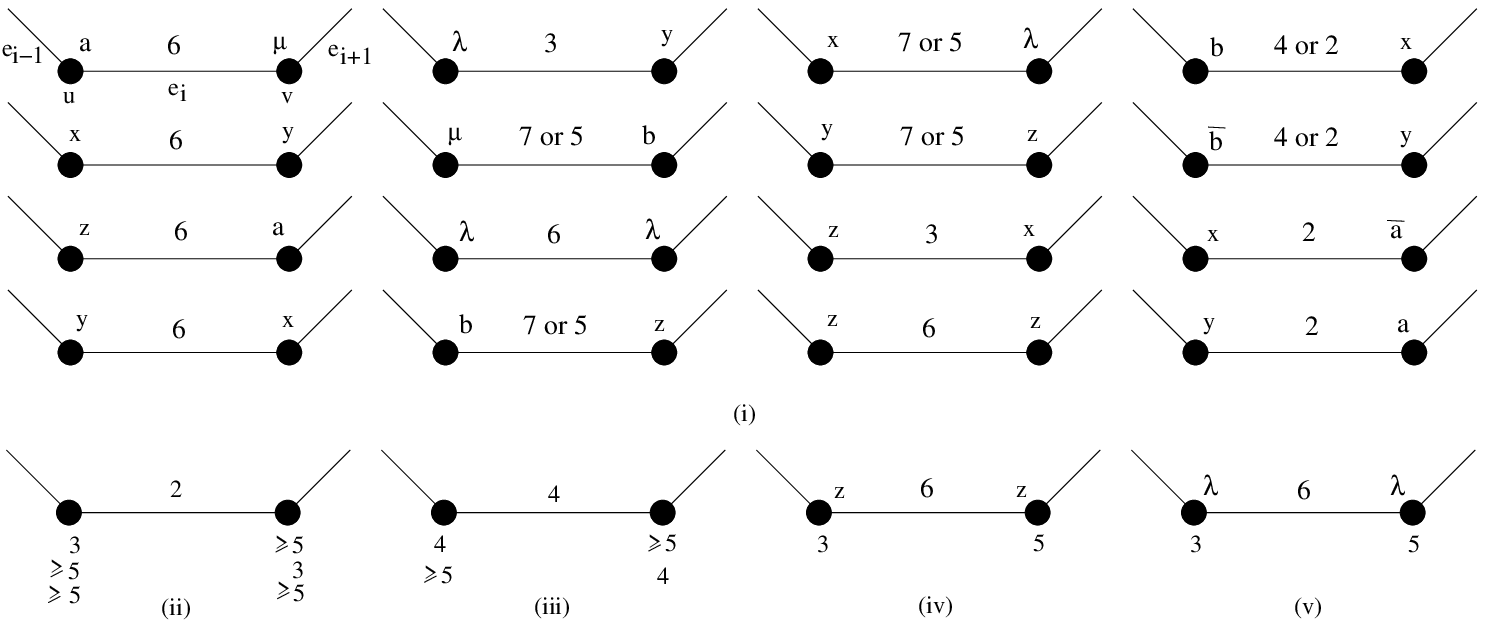} 
\end{center}
\caption{}
\end{figure}

\noindent 33 of the labelling and degrees of the vertices involved immediately rules out all combinations except for three cases.  The first case is $\hat{\Delta}_3$ of Figure 33(x) with
$\hat{\Delta}_8$ of Figure 33(vii), but this forces $\hat{\Delta}_8$ to be given by Figure 21(xii), a contradiction; and the second case is
$\hat{\Delta}_4$ of Figure 33(x) with $\hat{\Delta}_8$ of Figure 33(ix), but this forces $\hat{\Delta}_8$ to be given by Figure 22(x), a
contradiction.  The third case is when
$\hat{\Delta}=\hat{\Delta}_6$ of Figure 33(xi) and $\hat{\Delta}=\hat{\Delta}_8$ of Figure 33(ix).  But then $c^{\ast}(\hat{\Delta}) \leq
c(4,4,6,6)+
\frac{\pi}{15} + \frac{\pi}{30} < 0$. Note that we have also shown that a T4 region coincides with at most one T13 region.

Finally suppose that a T24 region and a T13 region coincide.  An inspection of the 36 possible combinations
immediately rules out all but the following 10 cases.  If $\hat{\Delta}=\hat{\Delta}_3$ of Figure 33(i) or 33(iii) and $\hat{\Delta}=\hat{\Delta}_8$ of
Figure 33(vii) then this forces $\hat{\Delta}_8$ to be given by Figure 21(xii), a contradiction; or if $\hat{\Delta}=\hat{\Delta}_4$ of Figure 33(ii)
or 33(iii) and
$\hat{\Delta}=\hat{\Delta}_8$ of Figure 33(ix) then this forces $\hat{\Delta}_8$ to be given by Figure 22(x), a contradiction.  If
$\hat{\Delta}=\hat{\Delta}_6$ of Figure 33(iv) and $\hat{\Delta}=\hat{\Delta}_8$ of Figure 33(ix), or if $\hat{\Delta}=\hat{\Delta}_7$ of Figure
33(v) and $\hat{\Delta}=\hat{\Delta}_8$ of Figure 33(vii) then $c^{\ast}(\hat{\Delta}) \leq c(4,4,6,6) + \frac{\pi}{15} + \frac{\pi}{30} < 0$.

This leaves $\hat{\Delta}=\hat{\Delta}_6$ of Figure 33(iv) and either $\hat{\Delta}=\hat{\Delta}_1$ of 33(vi) or
or $\hat{\Delta}=\hat{\Delta}_2$ of 33(viii); or $\hat{\Delta}=\hat{\Delta}_7$ of Figure 33(v) and again either
$\hat{\Delta}=\hat{\Delta}_1$ of 33(vi) or $\hat{\Delta}=\hat{\Delta}_2$ of 33(viii).
Observe that if $\hat{\Delta}_6$ of Figure 33(iv) $=\hat{\Delta}_1$ of 33(vi) or if $\hat{\Delta}_6$ of Figure 33(iv)$=\hat{\Delta}_2$ of 33(viii)
then $l(u_2)=y^{-1} a^2 xb^{-1}w$ in Figure 33(iv) and so $d(u_2) \geq 7$, and this is shown in Figure 34(viii). Moreover if $\hat{\Delta}_7$ of Figure 33(v) $=\hat{\Delta}_1$ of 33(vi)
or if $\hat{\Delta}_7$ of Figure 33(v)$=\hat{\Delta}_2$ of 33(viii) then $l(u_2)=b^{-1} y^{-1} a^2 xw $ in Figure 33(v) and again $d(u_2) \geq 7$, and this is shown in Figure 34(ix).
It follows in both Figure 34(viii) and (ix) that $c(\hat{\Delta}_1) \leq
c(3,4,4,7)=-\frac{\pi}{21}$ and $c(\Delta_2) \leq c(3,3,5,7)=\frac{2 \pi}{105}$.  In both configurations $\frac{\pi}{21}$ is added from
$c(\Delta_1)=\frac{\pi}{15}$ to $c(\hat{\Delta}_1)$ and the remaining $\frac{\pi}{15} - \frac{\pi}{21}=\frac{2 \pi}{105}$ to $\hat{\Delta}$ as shown.
If $\hat{\Delta}$ does not receive positive curvature from $\Delta_3$ then $c^{\ast} (\hat{\Delta}) \leq c(3,4,4,7) + 2 \left( \frac{2 \pi}{105} \right) < 0$
so it can be assumed without any loss that $\hat{\Delta}$ receives from $\Delta_1$ (via $\hat{\Delta}_1$), $\Delta_2$ and $\Delta_3$.  But then 
$\hat{\Delta}=\hat{\Delta}_2$ of Figure 33(viii) forces $d(u) \geq 5$ in Figure 34(viii), and $\hat{\Delta}=\hat{\Delta}_1$ of 
Figure 33(vi) forces $d(u) \geq 5$ in Figure 34(ix); therefore in each case
$c^{\ast} (\hat{\Delta} \leq c(3,4,5,7) + \frac{1}{2}c(3,3,5,5) + 2 \left( \frac{2
\pi}{105} \right) < 0$. $\Box$

\medskip

Proposition 4.1 follows immediately from Lemma 6.2 together with the fact that all possibilities for distribution of curvature from a region of degree 4 have been covered by Figures 6-32.

\medskip

We end this section with a summary  that will be helpful in subsequent sections.

\medskip

\textbf{Note}.
In Figure 35(i) the maximum amount of curvature, denoted $c(u,v)$, distributed across an edge $e_i$ with endpoints $u,v$ according to the description
of curvature given in Figures 6-32 and 34  above is shown for each choice of corner labels. The list excludes $(b,a)$-edges and excludes the $(x,y)$-edges
of Figures 13 and 14. In Figure 35(ii) and (iii) $c(u,v)$ is shown when at least one of $d(u),d(v)$ is greater than 4 apart from the two exceptional cases shown in (iv) and (v) (see Figure 23(xvi),(iv)).
The integers shown are multiples of $\frac{\pi}{30}$ with 7 or 5, 4 or 2   meaning that if $c(u,v) < \frac{7 \pi}{30}, \frac{2 \pi}{15}$ then
$c(u,v)=\frac{\pi}{6}, \frac{\pi}{15}$ respectively.
This will be used throughout what follows often without explicit reference.

\section{Distribution of positive curvature from 6-gons}

We turn now to step 2 of the proof of Theorem 1.2 as described in Section 4. Let  $d(\hat{\Delta})=6$ and so $\hat{\Delta}$ is given by Figure 4(xii)-(xiii).
In Figure 36 we fix the labelling of the six neighbours $\hat{\Delta}_i$ ($1 \leq i \leq 6$) of $\hat{\Delta}$ as shown.
We consider regions $\hat{\Delta}$ ($ \neq \Delta_0$) of degree 6 that have received positive curvature in step 1 of Sections 5 and 6. 

\medskip

Again for the benefit of the reader let us indicate briefly that a general rule for distribution is to try whenever possible to add the positive curvature from $\hat{\Delta}$ to neighbouring regions
of degree greater than 6. It turns out that this is not possible for exactly four cases, namely, $\hat{\Delta}_1$ of Figure 36(i) and (x); and $\hat{\Delta}_2$ of Figures 37(iv) and 38(iv).
These four exceptions are dealt with in greater detail in the next section. In all other cases in Figures 36-38 the curvature is added to a region of degree at least 8.

\medskip

First assume that $\hat{\Delta}$ is \textit{not} $\hat{\Delta}_1$ of 
Figure 31(i) (Configuration A) or Figure 31((v) (Configuration B).
Then checking the distribution of curvature described in Figures 6-32 and 34 yields Table 3 in which vertex subscripts are modulo 6;

\begin{table}[h]
\[
\renewcommand{\arraystretch}{1.5}
\begin{array}{llcccccc}
d(u_i)&d(u_{i+1})&c(u_1,u_2)&c(u_2,u_3)&c(u_3,u_4)&c(u_4,u_5)&c(u_5,u_6)&c(u_6,u_1)\\
3&3&0&0&0&6&0&0\\
3&4&0&3&0&0&0&2\\
4&3&0&0&0&0&7&0\\
3&5&0&2&2&0&0&0\\
5&3&0&2&2&2&2&0\\
3&6^+&0&2&2&2&0&0\\
6^+&3&0&2&2&2&2&0\\
4&4&7&0&0&0&0&0\\
4&5&2&0&0&2&0&0\\
5&4&2&2&0&0&4&0\\
4&6^+&4&0&1&0&2&0\\
6^+&4&2&0&0&0&1&4\\
5^+&5^+&1&0&0&1&1&0
\end{array}
\renewcommand{\arraystretch}{1}
\]
\caption{}
\end{table}

\noindent the entries under $c(u_i,u_{i+1})$ are multiples of $\frac{\pi}{30}$
and denote the maximum amount of curvature that $\hat{\Delta}$ can receive across the edge with endpoints $u_i,u_{i+1}$ according to Figure 35; and
$5^+$, $6^+$ means $\geq 5$, $\geq 6$.
Moreover Table 3 applies to $\hat{\Delta}$ both of Figure 4(xii) and of Figure 4(xiii).

\newpage
\begin{figure}  
\begin{center}  
\psfig{file=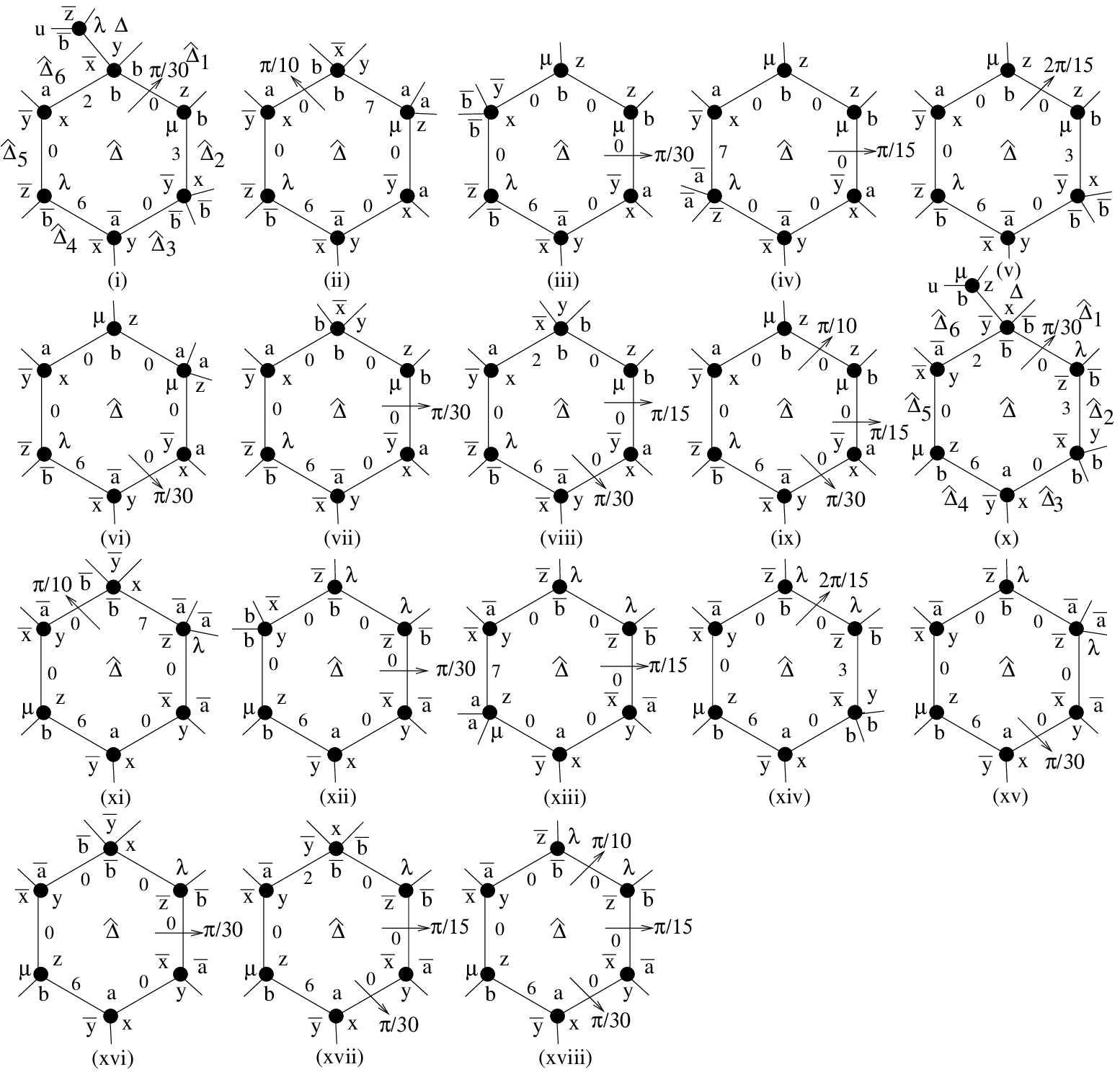}
\end{center}
\caption{}
\end{figure}

\textbf{Notes}.
\begin{enumerate}
\item[1.]
(See Figures 4 and 36.)
$d(u_1)=3 ~ (\Rightarrow d(\hat{\Delta}_1) > 4, \, d(\hat{\Delta}_2) > 4) \Rightarrow c(u_1,u_2)=c(u_6,u_1)=0$;
$d(u_2)=3 \Rightarrow c(u_1,u_2)=0$;
$d(u_2)=4 \Rightarrow c(u_2,u_3)=0$;
$d(u_5)=3 \Rightarrow c(u_5,u_6)=0$;
and
$d(u_5)=4 \Rightarrow c(u_4,u_5)=0$.
\item[2.]
$c(u_1,u_2)>0$ and $c(u_2,u_3)>0 \Rightarrow$ (Table 3) $c(u_1,u_2)+c(u_2,u_3) \leq \frac{2 \pi}{15} + \frac{\pi}{15}$ and since
$c(u_1,u_2) \leq \frac{7 \pi}{30}$, $c(u_2,u_3) \leq \frac{\pi}{10}$ we have
$c(u_1,u_2)+c(u_2,u_3) \leq \frac{7 \pi}{30}$.
\item[3.]
$c(u_4,u_5) > 0$ and
$c(u_5,u_6) > 0 \Rightarrow c(u_4,u_5)+c(u_5,u_6) \leq \frac{\pi}{15} + \frac{2 \pi}{15}$
and since
$c(u_4,u_5) \leq \frac{\pi}{5}$, $c(u_5,u_6) \leq \frac{7 \pi}{30}$
we have
$c(u_4,u_5) + c(u_5,u_6) \leq \frac{7 \pi}{30}$.
\item[4.]
Let $d(u_5)=5$, $d(u_6)=4$.  If $c(u_5,u_6)=\frac{2 \pi}{15}$ then checking $l(u_5),l(u_6)$ shows that $c(u_4,u_5)=\frac{2 \pi}{15}$ (see Figures 23(vi) and
23(xviii));
moreover (see Figure 35(ii),(iii)) if $c(u_5,u_6) \neq \frac{2 \pi}{15}$ then $c(u_5,u_6)=\frac{\pi}{15}$.
\end{enumerate}

\medskip

In what follows much use will be made of Lemma 3.4 when determining the vertex labels and Table 3 when determining $c(u,v)$.

\begin{lemma}
\textit{If $\hat{\Delta}$ is given by Figure 4(xii)-(xiii) (with the assumption that $\hat{\Delta}$ is not $\hat{\Delta}_1$ of Figure 31) and $\hat{\Delta}$ receives positive curvature across at least one edge then
$c^{\ast} (\hat{\Delta}) \leq \frac{2 \pi}{15}$ and if $c^{\ast} (\hat{\Delta}) >0$ then $\hat{\Delta}$ is given by one of the regions of Figure 36.}
\end{lemma}

\textit{Proof}.  It follows from Table 3 and Notes 1-4 above that $c^{\ast} (\hat{\Delta}) \leq c(\hat{\Delta}) + (c(u_1,u_2)+c(u_2,u_3))+
c(u_3,u_4)+(c(u_4,u_5)+c(u_5,u_6))+c(u_6,u_1) \leq c(\hat{\Delta})+ \frac{7 \pi}{30} + \frac{\pi}{15} + \frac{7 \pi}{30} + \frac{2 \pi}{15} = c(\hat{\Delta}) + \frac{2 \pi}{3}$.  Therefore if $\hat{\Delta}$ has at most two vertices of degree 3 then $c^{\ast} (\hat{\Delta}) \leq c(3,3,4,4,4,4) + \frac{2 \pi}{3}=0$.

Let $\hat{\Delta}$ have exactly three vertices of degree 3 so that $c(\hat{\Delta}) \leq -\frac{\pi}{2}$.
If $d(u_1)=3$ then $c(u_1,u_2)=c(u_6,u_1)=0$ and
$c^{\ast} (\hat{\Delta}) \leq -\frac{\pi}{2} + \frac{\pi}{10} + \frac{\pi}{15} + \frac{7 \pi}{30} < 0$, so assume $d(u_1) \geq 4$.
If $d(u_2) = 3$ then $c(u_1,u_2)=0$ so if $d(u_6) \geq 6$ then $c^{\ast} (\hat{\Delta}) \leq c(3,3,3,4,4,6) + \frac{\pi}{10} + \frac{\pi}{15} +
\frac{7 \pi}{30} + \frac{2 \pi}{15} < 0$; otherwise $c(u_6,u_1)=\frac{\pi}{15}$ and
$c^{\ast} (\hat{\Delta}) \leq -\frac{\pi}{2} + \frac{\pi}{10} + \frac{\pi}{15} + \frac{7 \pi}{30} + \frac{\pi}{15} < 0$; so assume $d(u_2) \geq 4$.
This leaves four subcases.
First let $d(u_3)=d(u_4)=d(u_5)=3$.  Then $c(u_3,u_4)=c(u_5,u_6)=0$.  Moreover
if $d(u_6) < 6$ then $c(u_6,u_1)=0$ and $c^{\ast} (\hat{\Delta}) \leq -\frac{\pi}{2} + \frac{7 \pi}{30} + \frac{\pi}{5}=0$; and
if $d(u_6) \geq 6$ then $c^{\ast} (\hat{\Delta}) \leq c(3,3,3,4,4,6) + \frac{7 \pi}{30} + \frac{\pi}{5} + \frac{2 \pi}{15} < 0$.
Let $d(u_3)=d(u_4)=d(u_6)=3$.  Then $c(u_2,u_3)=\frac{\pi}{15}$, $c(u_3,u_4)=0$, $c(u_4,u_5)+c(u_5,u_6) \leq \frac{7 \pi}{30}$ and $c(u_6,u_1)=\frac{\pi}{15}$.  If either $d(u_1) > 4$ or $d(u_2) > 4$ then $c(u_1,u_2) \leq \frac{2 \pi}{15}$ and $c^{\ast}(\hat{\Delta}) \leq c(3,3,3,4,4,5) + \frac{\pi}{2} < 0$;
otherwise $d(u_1)=d(u_2)=4$ which
implies $c(u_2,u_3)=0$ and the labelling (of $u_1$, $u_2$ and $u_6$) either forces $c(u_1,u_2)=0$ and $c^{\ast} (\hat{\Delta}) \leq -\frac{\pi}{2}
+ \frac{3 \pi}{10} < 0$
or forces $c(u_6,u_1)=0$ and $c^{\ast}(\hat{\Delta}) \leq -\frac{\pi}{2} + \frac{7 \pi}{30} + \frac{7 \pi}{30} < 0$.
Let $d(u_3)=d(u_5)=d(u_6)=3$.  Then $c(u_4,u_5)=\frac{\pi}{15}$, $c(u_5,u_6)=0$ and $c(u_6,u_1)=\frac{\pi}{15}$.
Therefore $c^{\ast} (\hat{\Delta}) \leq - \frac{\pi}{2} + \frac{7 \pi}{30} + \frac{\pi}{15} + \frac{\pi}{15} + \frac{\pi}{15} < 0$.
Finally let $d(u_4)=d(u_5)=d(u_6)=3$.  Then $c(u_5,u_6)=0$ and $c(u_6,u_1)=\frac{\pi}{15}$.
If $d(u_1) > 4$ or $d(u_2) > 4$ then $c(u_1,u_2) = \frac{2 \pi}{15}$ and $c^{\ast} (\hat{\Delta}) \leq - \frac{3 \pi}{5} + \frac{2 \pi}{15} + \frac{\pi}{15} + \frac{\pi}{15} + \frac{\pi}{5} + \frac{\pi}{15} < 0$; otherwise
$d(u_1) = d(u_2) = 4$ so $c(u_2,u_3)=0$ and the labelling either forces $c(u_1,u_2)=0$ and $c^{\ast} (\hat{\Delta}) \leq -\frac{\pi}{2} + \frac{\pi}{3} < 0$ or forces $c(u_6,u_1)=0$ and $c^{\ast} (\hat{\Delta}) \leq -\frac{\pi}{2} + \frac{7 \pi}{30} + \frac{\pi}{15} + \frac{\pi}{5} =0$.

Now let $\hat{\Delta}$ have exactly four vertices of degree 3 so that $c(\hat{\Delta}) \leq -\frac{\pi}{3}$.  There are fifteen cases to consider.
In fact if $(d(u_1),d(u_2),d(u_3),d(u_4),d(u_5),d(u_6)) \in \{
(3,3,3,3,\ast,\ast),
(3,3,3,\ast,3,\ast),$
$(3,3,3,\ast,\ast,3),
(3,3,\ast,\ast,3,3),
(3,\ast,3,3,3,\ast),
(3,\ast,3,3,\ast,3),
(3,\ast,3,\ast,3,3),
(3,\ast,\ast,3,3,3),$
$(\ast,3,3,3,3,\ast),
(\ast,3,3,3,\ast,3),
(\ast,3,3,\ast,3,3)\}$
then a straightforward check using Table 3 and Notes 1-4 shows that $c^{\ast} (\hat{\Delta}) \leq -\frac{\pi}{3} + \frac{\pi}{3}=0$.
Let $d(u_1)=d(u_2)=d(u_4)=d(u_5)=3$.Then $c(u_1,u_2)=c(u_5,u_6)=c(u_6,u_1)=0$.
If $d(u_3)>4$ then
$c^{\ast}(\hat{\Delta}) \leq -\frac{13 \pi}{30} + \frac{11 \pi}{30} < 0$;
otherwise $d(u_3)=4$ forcing $c(u_3,u_4)=0$ and
$c^{\ast} (\hat{\Delta}) \leq -\frac{\pi}{3} + \frac{3 \pi}{10} < 0$.
Let $d(u_1)=d(u_2)=d(u_4)=d(u_6)=3$.  Then $c(u_1,u_2)=c(u_6,u_1)=0$.  If $d(u_3) > 4$ then
$c^{\ast} (\hat{\Delta}) \leq -\frac{13 \pi}{30} + \frac{2 \pi}{5} < 0$;
otherwise $d(u_3)=4$ forcing $c(u_3,u_4)=0$ and
$c^{\ast} (\hat{\Delta}) \leq -\frac{\pi}{3} + \frac{\pi}{3} = 0$.
Let $d(u_2)=d(u_4)=d(u_5)=d(u_6)=3$.
Then $c(u_1,u_2)=c(u_5,u_6)=0$ and $c(u_6,u_1)=\frac{\pi}{15}$.
If $d(u_1) > 4$ or $d(u_3) > 4$ then $c^{\ast} (\hat{\Delta}) \leq -\frac{13 \pi}{30} + \frac{13 \pi}{30} = 0$, so assume $d(u_1)=d(u_3)=4$.
Then $c(u_3,u_4)=0$ and $l(u_1)$ either forces $c(u_6,u_1)=0$ and $c^{\ast} (\hat{\Delta}) \leq -\frac{\pi}{3} + \frac{3 \pi}{10} < 0$ or
$\hat{\Delta}$ is given by Figure 36(i) or (x)) in which the numbers assigned to each edge is the value of $c(u_i,u_{i+1})$ in multiples of $\frac{\pi}{30}$ and so  
$c^{\ast} (\hat{\Delta}) \leq -\frac{\pi}{3} + \frac{\pi}{10} + \frac{\pi}{5} + \frac{\pi}{15} = \frac{\pi}{30}$.
(Note that if $c^{\ast}(\hat{\Delta}) > 0$ then $\hat{\Delta}$ must receive $\frac{\pi}{15}$ from $\hat{\Delta}_6$ and 
, since $d(\hat{\Delta}_5) > 4$, this forces
$\hat{\Delta}_6= \Delta$ where $\Delta$ is given by Figure 16(i) which in turn forces $l(u) = b^{-1}z^{-1}{\lambda}$ in Figure 36(i), and $l(u) = b{\mu}z$ in Figure 36(x); and $\hat{\Delta}$ must receive 
$\frac{\pi}{5}$ from $\hat{\Delta}_4$.)
This leaves the case $d(u_j)=3$ ($3 \leq j \leq 6$).
Then $c(u_3,u_4)=c(u_5,u_6)=0$ and $c(u_6,u_1)=\frac{\pi}{15}$.
If $d(u_1) \geq 5$ and $d(u_2) \geq 5$ then $c^{\ast} (\hat{\Delta}) \leq -\frac{8 \pi}{15} + \frac{\pi}{2} < 0$.
If $d(u_1)=4$ and $d(u_2)=5$ or $d(u_1) \geq 5$ and $d(u_2)=4$ then
$c(u_1,u_2)=\frac{\pi}{15}$ and $c^{\ast} (\hat{\Delta}) \leq c(3,3,3,3,4,5) + \frac{\pi}{15} + \frac{\pi}{10} + \frac{\pi}{5} + \frac{\pi}{15}=0$; and
if $d(u_1)=4$ and $d(u_2) \geq 6$ then $c^{\ast} (\hat{\Delta}) \leq c(3,3,3,3,4,6) + \frac{2 \pi}{15} + \frac{\pi}{10} + \frac{\pi}{5} + \frac{\pi}{15}=0$.
Let $d(u_1)= d(u_2)=4$ so $c(u_2,u_3)=0$.
Then $l(u_1)$ either forces $c(u_1,u_2)=0$ and $c^{\ast}(\hat{\Delta}) \leq -\frac{\pi}{3} + \frac{\pi}{5} + \frac{\pi}{15} < 0$ or $\hat{\Delta}$ is
given by Figure 36(ii) or (xi) where $c^{\ast} (\hat{\Delta}) \leq -\frac{\pi}{3} + \frac{7 \pi}{30} + \frac{\pi}{5} = \frac{\pi}{10}$.

Now suppose that $\hat{\Delta}$ has exactly five vertices of degree 3 so that $c(\hat{\Delta}) \leq -\frac{\pi}{6}$.
If $d(u_6) > 3$ then $c(u_1,u_2)=c(u_2,u_3)=c(u_3,u_4)=c(u_5,u_6)=c(u_6,u_1)=0$,
$c^{\ast}(\hat{\Delta}) \leq -\frac{\pi}{6} + \frac{\pi}{5}=\frac{\pi}{30}$ and $\hat{\Delta}$ is given by Figure 36(iii) or (xii).
If $d(u_5) > 3$ then $c(u_i,u_{i+1})=0$ except for $c(u_4,u_5)$ and 
$c(u_5,u_6)$. If $d(u_5) \geq 5$ then $c^{\ast} (\hat{\Delta}) \leq -\frac{4 \pi}{15} + \frac{7 \pi}{30} < 0$; and if  $d(u_5) = 4$
then $c^{\ast} (\hat{\Delta}) \leq -\frac{\pi}{6} + \frac{7 \pi}{30}=\frac{\pi}{15}$ and $\hat{\Delta}$ is given by Figure 36(iv) or (xiii).
If $d(u_4) > 3$ then $c(u_i,u_{i+1})=0$ except for $c(u_3,u_4)=c(u_4,u_5)=\frac{\pi}{15}$ and $c^{\ast}(\hat{\Delta}) \leq -\frac{\pi}{6} + \frac{2 \pi}{15}<0$.
Let $d(u_3) > 3$.  
Then $c(u_1,u_2) = c(u_5,u_6)=c(u_6,u_1)=0$.
If $d(u_3) \geq 6$ then $c^{\ast} (\hat{\Delta}) \leq -\frac{\pi}{3} + 2 \left( \frac{\pi}{15} \right) + \frac{\pi}{5}=0$;
if $d(u_3)=5$ then $l(u_3)$ forces either $c(u_2,u_3)=0$ or $c(u_3,u_4)=0$ so $c^{\ast} (\hat{\Delta}) \leq -\frac{4 \pi}{15} + \frac{\pi}{15} + \frac{\pi}{5}=0$; and
if $d(u_3)=4$ then $c(u_3,u_4)=0$, $c^{\ast}(\hat{\Delta}) \leq -\frac{\pi}{6} + \frac{\pi}{10} + \frac{\pi}{5}=\frac{2 \pi}{15}$ and $\hat{\Delta}$
is given by Figure 36(v) or (xiv).
If $d(u_2) > 3$ then $c(u_1,u_2)=c(u_3,u_4)=c(u_5,u_6)=c(u_6,u_1)=0$.
If $d(u_2) \geq 5$ then $c^{\ast} (\hat{\Delta}) \leq -\frac{4 \pi}{15} + \frac{\pi}{15} + \frac{\pi}{5}=0$; and
if $d(u_2)=4$ then $c(u_2,u_3)=0$, $c^{\ast}(\hat{\Delta}) \leq -\frac{\pi}{6} + \frac{\pi}{5} = \frac{\pi}{30}$ and $\hat{\Delta}$ is given by
Figure 36(vi) or (xv).
Finally
if $d(u_1) >3$ then $c(u_i,u_{i+1})=0$ except for $c(u_4,u_5)=\frac{\pi}{5}$ and $c(u_6,u_1)=\frac{\pi}{15}$.  So
if $d(u_1) \geq 5$ then $c^{\ast} (\hat{\Delta}) \leq -\frac{4 \pi}{15} + \frac{4 \pi}{15}=0$; and
if $d(u_1)=4$ then either $c(u_6,u_1)=0$ or $c(u_6,u_1)=\frac{\pi}{15}$, so either $c^{\ast}(\hat{\Delta}) \leq -\frac{\pi}{6} + \frac{\pi}{5} = \frac{\pi}{30}$ or
$c^{\ast}(\hat{\Delta}) \leq -\frac{\pi}{6} + \frac{\pi}{5} + \frac{\pi}{15} = \frac{\pi}{10}$ and the two cases for $\hat{\Delta}$ are shown in Figure 36(vii),
(viii) or 36(xvi), (xvii).

This leaves the case $d(u_i)=3$ ($1 \leq i \leq 6$).
Then $c(u_i,u_{i+1})=0$ except for $c(u_4,u_5)=\frac{\pi}{5}$,
$c^{\ast} (\hat{\Delta}) \leq 0 + \frac{\pi}{5} = \frac{\pi}{5}$ and $\hat{\Delta}$ is given by Figure 36(ix) or (xviii). $\Box$

We now describe the distribution of curvature from each of the 18 regions $\hat{\Delta}$ of Figure 36.

\bigskip

\textbf{Figure 36(i) and (x)}: $c^{\ast}(\hat{\Delta}) \leq -\frac{\pi}{3} + \frac{11 \pi}{30}$; distribute $\frac{\pi}{30}$ from $\hat{\Delta}$ to
$\hat{\Delta}_1$ in each case.

\textbf{Figure 36(ii) and (xi)}: $c^{\ast} (\hat{\Delta}) \leq -\frac{\pi}{3} + \frac{13 \pi}{30}$; distribute $\frac{\pi}{10}$ from $\hat{\Delta}$ to
$\hat{\Delta}_6$ in each case.

\textbf{Figure 36(iii) and (xii)}: $c^{\ast} (\hat{\Delta}) \leq -\frac{\pi}{6} + \frac{\pi}{5}$; distribute $\frac{\pi}{30}$ from $\hat{\Delta}$ to
$\hat{\Delta}_2$ in each case.

\textbf{Figure 36(iv) and (xiii)}: $c^{\ast}(\hat{\Delta}) \leq -\frac{\pi}{6} + \frac{7 \pi}{30}$; distribute $\frac{\pi}{15}$ from $\hat{\Delta}$ to
$\hat{\Delta}_2$ in each case.

\textbf{Figure 36(v) and (xiv)}: $c^{\ast}(\hat{\Delta}) \leq -\frac{\pi}{6} + \frac{3 \pi}{10}$; distribute $\frac{2 \pi}{15}$ from $\hat{\Delta}$ to
$\hat{\Delta}_1$ in each case. (5.1)

\textbf{Figure 36(vi) and (xv)}: $c^{\ast}(\hat{\Delta}) \leq -\frac{\pi}{6} + \frac{\pi}{5}$; distribute $\frac{\pi}{30}$ from $\hat{\Delta}$ to
$\hat{\Delta}_3$ in each case.

\newpage
\begin{figure}
\begin{center}
\psfig{file=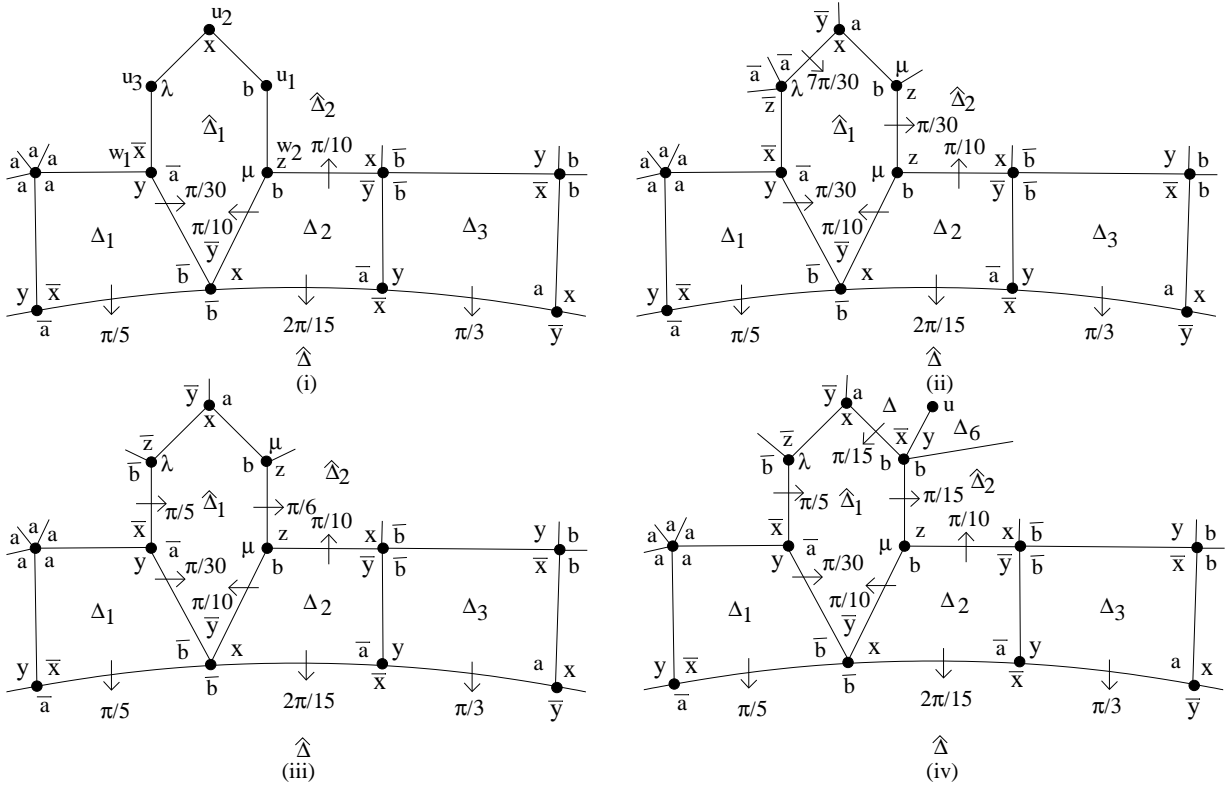}
\end{center}
\caption{Configuration A}
\end{figure}

\textbf{Figure 36(vii) and (xvi)}: $c^{\ast}(\hat{\Delta}) \leq -\frac{\pi}{6} + \frac{\pi}{5}$; distribute $\frac{\pi}{30}$ from $\hat{\Delta}$ to
$\hat{\Delta}_2$ in each case.

\textbf{Figure 36(viii) and (xvii)}: $c^{\ast}(\hat{\Delta}) \leq -\frac{\pi}{6} + \frac{4 \pi}{15}$; distribute $\frac{\pi}{15}$ from $\hat{\Delta}$ to
$\hat{\Delta}_2$ and $\frac{\pi}{30}$ from $\hat{\Delta}$ to $\hat{\Delta}_3$ in each case.

\textbf{Figure 36(ix) and (xviii)}: $c^{\ast}(\hat{\Delta}) \leq 0 + \frac{\pi}{5}$; distribute $\frac{\pi}{10}$ from $\hat{\Delta}$ to $\hat{\Delta}_1$,
$\frac{\pi}{15}$ from $\hat{\Delta}$ to $\hat{\Delta}_2$ and $\frac{\pi}{30}$ from $\hat{\Delta}$ to $\hat{\Delta}_3$ in each case.

\bigskip

\textbf{Note}: in all of the above cases $d(\hat{\Delta}_i) > 6$ for each region $\hat{\Delta}_i$ that receives positive curvature from
$\hat{\Delta}$ except possibly for $\hat{\Delta}_1$ in Figures 36(i) and (x).
\medskip

Now assume that $\hat{\Delta}$ \textit{is}  $\hat{\Delta}_1$ of Figure 31(i) or (v). Then $\hat{\Delta}_1$ is given by Figure 37(i), 38(i). (Recall that for now we are only considering distribution
of curvature from Sections 5 and 6.)  

First assume that $d(u_3) \geq 5$.
Then $c(w_1,u_3) = \frac{\pi}{15}$ and $c(u_3,u_2)=\frac{2 \pi}{15}$ by Figure 35(ii)-(v).
Since $c(u_1,u_2) = \frac{2 \pi}{15}$ it follows that $c^{\ast} (\hat{\Delta}_1) \leq c(\hat{\Delta}) + \frac{7 \pi}{15}$.
If $d(u_1) > 3$ then $c(\hat{\Delta}_1) \leq c(3,3,3,4,4,5)=-\frac{9 \pi}{15}$; on the other hand if $d(u_1)=3$ then
$c(u_1,u_2)=0$ and
$c^{\ast} (\hat{\Delta}_1) \leq c(3,3,3,3,4,5) + \frac{\pi}{3} < 0$.  Now let $d(u_3)=4$.  Then $c(u_1,u_2)=\frac{2 \pi}{15}$,
$c(u_2,u_3)=\frac{7 \pi}{30}$ and 

\newpage
\begin{figure}
\begin{center}
\psfig{file=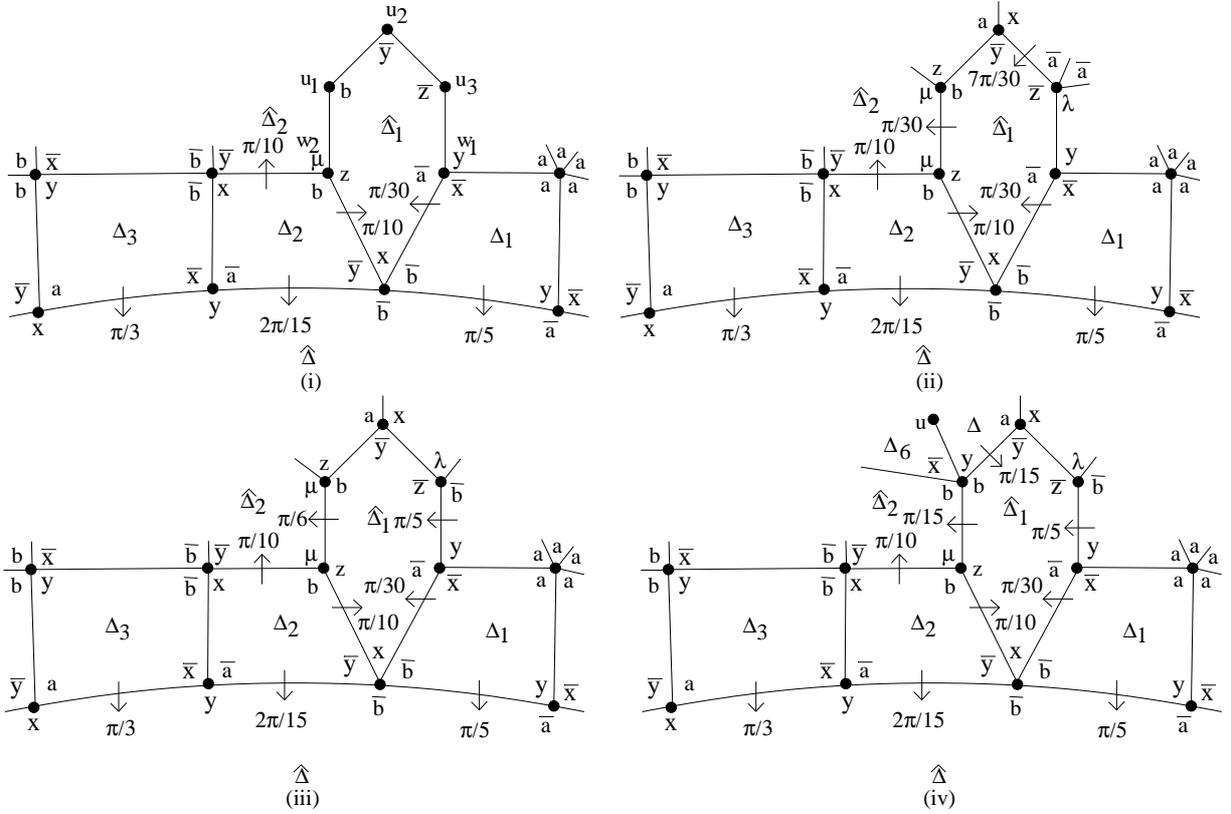}
\end{center}
\caption{Configuration B}
\end{figure}

\noindent $c(u_3,w_1)=0$ so $c^{\ast} (\hat{\Delta}_1) \leq c(\hat{\Delta}_1) + \frac{\pi}{2}$.
If $d(u_1)>3$ or $d(u_2)>3$ then $c(\hat{\Delta}_1) \leq -\frac{\pi}{2}$; on the other hand if $d(u_1)=d(u_2)=3$ then $c^{\ast}(\hat{\Delta}_1) \leq
c(3,3,3,3,4,4) + \frac{11 \pi}{30} = \frac{\pi}{30}$ as shown in Figure 37(ii), 38(ii).
Finally let $d(u_3)=3$.  Then $c(u_3,w_1)=\frac{\pi}{5}$, $c(u_2,u_1) = \frac{2 \pi}{15}$ and $c(u_2,u_3)=0$ so $c^{\ast} (\hat{\Delta}_1) \leq c(\hat{\Delta}_1) + \frac{7 \pi}{15}$.
If $d(u_1)=3$ then $c(u_1,u_2)=0$ and so $d(u_2) \geq 4$ would imply $c^{\ast} (\hat{\Delta}_1) \leq c(3,3,3,3,4,4) + \frac{\pi}{3} = 0$; whereas
if also $d(u_2)=3$ then $c^{\ast} (\hat{\Delta}_1) \leq \frac{\pi}{6}$ as shown in Figure 37(iii), 38(iii).
Let $d(u_1)=4$.
If $d(u_2) \geq 4$ then $c^{\ast} (\hat{\Delta}_1) \leq c(3,3,3,4,4,4) + \frac{7 \pi}{15} < 0$ so assume that $d(u_2)=3$.
Reading clockwise from the $\hat{\Delta}_1$ corner label
if $l(u_1)=bbx^{-1} y$, $bx^{-1} yb$ in Figure 37(i), 38(i) respectively then $c(u_1,u_2)=0$ and $c^{\ast}(\hat{\Delta}_2) \leq -\frac{\pi}{3} +
\frac{\pi}{3}=0$;
otherwise $c(u_1,u_2)=\frac{\pi}{15}$ and $\hat{\Delta}_1$ is given by Figure 37(iv), 38(iv) and $c^{\ast}(\hat{\Delta}_1) \leq -\frac{\pi}{3} +
\frac{6 \pi}{15}=\frac{\pi}{15}$ as shown.
This leaves $d(u_1) \geq 5$ in which case $c(u_1,u_2)=\frac{\pi}{15}$ and $c^{\ast}(\hat{\Delta}_1) \leq c(3,3,3,3,4,5) + \frac{6 \pi}{15} < 0$.

The distribution of curvature in Figures 37 and 38 is as follows.

\bigskip

\textbf{Figure 37(ii) and 38(ii)}: $c^{\ast}(\hat{\Delta}_1) \leq -\frac{\pi}{3} + \frac{11 \pi}{30}$; distribute $\frac{\pi}{30}$ from $\hat{\Delta}_1$ to
$\hat{\Delta}_2$ in each case.

\textbf{Figure 37(iiii) and 38(ii)}: $c^{\ast}(\hat{\Delta}_1) \leq -\frac{\pi}{6} + \frac{\pi}{3}$; distribute $\frac{\pi}{6}$ from $\hat{\Delta}_1$ to
$\hat{\Delta}_2$ in each case.

\textbf{Figure 37(iv) and 38(iv)}: $c^{\ast}(\hat{\Delta}_1) \leq -\frac{\pi}{3} + \frac{6 \pi}{15}$; distribute $\frac{\pi}{15}$ from $\hat{\Delta}_1$ to
$\hat{\Delta}_2$ in each case.

\bigskip

\begin{lemma}
\textit{According to the distribution of curvature so far, that is, in Figures 6-32, 34 and 36-38, $\hat{\Delta}_1$ of Figures 37(i) and 38(i)
does not receive positive curvature from $\hat{\Delta}_2$, that is, $c(w_2,u_1) = 0$.}
\end{lemma}

\textit{Proof}.  Consider $\hat{\Delta}_1$ of Figure 37(i). If $c(w_2,u_1) > 0$ then $\hat{\Delta}_2$ of Figure 37(i) is (the inverse of) $\hat{\Delta}$ of Figure 36(x)
or $\hat{\Delta}_2$ is $\hat{\Delta}_1$ of Figure 38(iv). Suppose $\hat{\Delta}_2$ is $\hat{\Delta}$ of Figure 36(x)
Then since $\hat{\Delta}$ of Figure 36(x) must receive $\frac{\pi}{5}$ across its $(v_4,v_5)$-edge, the region $\hat{\Delta}_4$ of Figure 36(x) is given by $\Delta$
of Figure 7(iii); and this in turn forces $\hat{\Delta}_2$ of
Figure 37(i) to be given by Figure 31(ii)-(iv) and not Configuration A of Figure 31(i), a contradiction.
Moreover the region $\hat{\Delta}_2$ of Figure 37(i) cannot coincide with the region $\hat{\Delta}_1$ of Figure 38(iv) since, for example, the distribution
of curvature from the region $\Delta_2$ of Figure 37(i) is not the same as the distribution of curvature from the corresponding region $\Delta_2$ of Figure 38(iv).

Consider $\hat{\Delta}_1$ of Figure 38(i). If $c(w_2,u_1) > 0$ then $\hat{\Delta}_2$ of Figure 38(i) is $\hat{\Delta}$ of Figure 36(i) or $\hat{\Delta}_2$ is $\hat{\Delta}_1$ of Figure 37(iv).
If $\hat{\Delta}_2$ is $\hat{\Delta}$ of Figure 36(i) then a similar argument to the one above using Figures 36(i), 7(iii) and 31(vi)-(viii) applies to yield a contradiction; and
$\hat{\Delta}_2$ of Figure 38(i) cannot coincide with the region $\hat{\Delta}_1$ of Figure 37(iv) since as above the distribution of curvature from  the corresponding $\Delta_2$ differs. 
$\Box$

\medskip

\textbf{Note }. The upper bounds $c(u,v)$ of Figure 35 remain unchanged as a result of the distribution of curvature described in this section.

\section{Proof of Proposition 4.2}

An inspection of all distribution of curvature described so far yields the following.
If positive curvature is distributed across an $(x,a^{-1})$-edge $e$ into a region of degree $>4$ then $e$ is given by: Figure 21(ii) (two cases);
Figure 23(ii) (two cases); Figure 21(xi); and Figure 31(v).  In particular if the $x$-corner vertex has degree 4 and the $a^{-1}$-corner vertex has
degree 3 then $e$ is given by Figure 31(v) (Configuration B).
If positive curvature is distributed across an $(a^{-1},y^{-1})$-edge $e$ into a region of degree $>4$ then $e$ is given by:
Figure 21(ii) (two cases); Figure 23(ii) (two cases); Figure 22(ix); and Figure 31(i).
In particular if the $a^{-1}$-corner has degree 3 and the $y^{-1}$-corner has degree 4 then $e$ is given by Figure 31(i) (Configuration A).

\begin{lemma}
\textit{Let $\hat{\Delta}$ be a region of degree 6 that receives positive curvature across at least one edge.  Then one of the following occurs.}

\begin{enumerate}
\item[(i)]
$c^{\ast}(\hat{\Delta}) \leq 0$\textit{;}
\item[(ii)]
$c^{\ast}(\hat{\Delta})>0$ \textit{is distributed to a region of degree $>6$;}
\item[(iii)]
$c^{\ast}(\hat{\Delta}) \in \left\{ \frac{\pi}{30} , \frac{\pi}{15} \right\}$ \textit{is distributed to a region $\Delta '$ of degree 6 and
$c^{\ast}(\Delta ') \leq 0$.}
\end{enumerate}
\end{lemma}

\newpage
\begin{figure}
\begin{center}  
\psfig{file=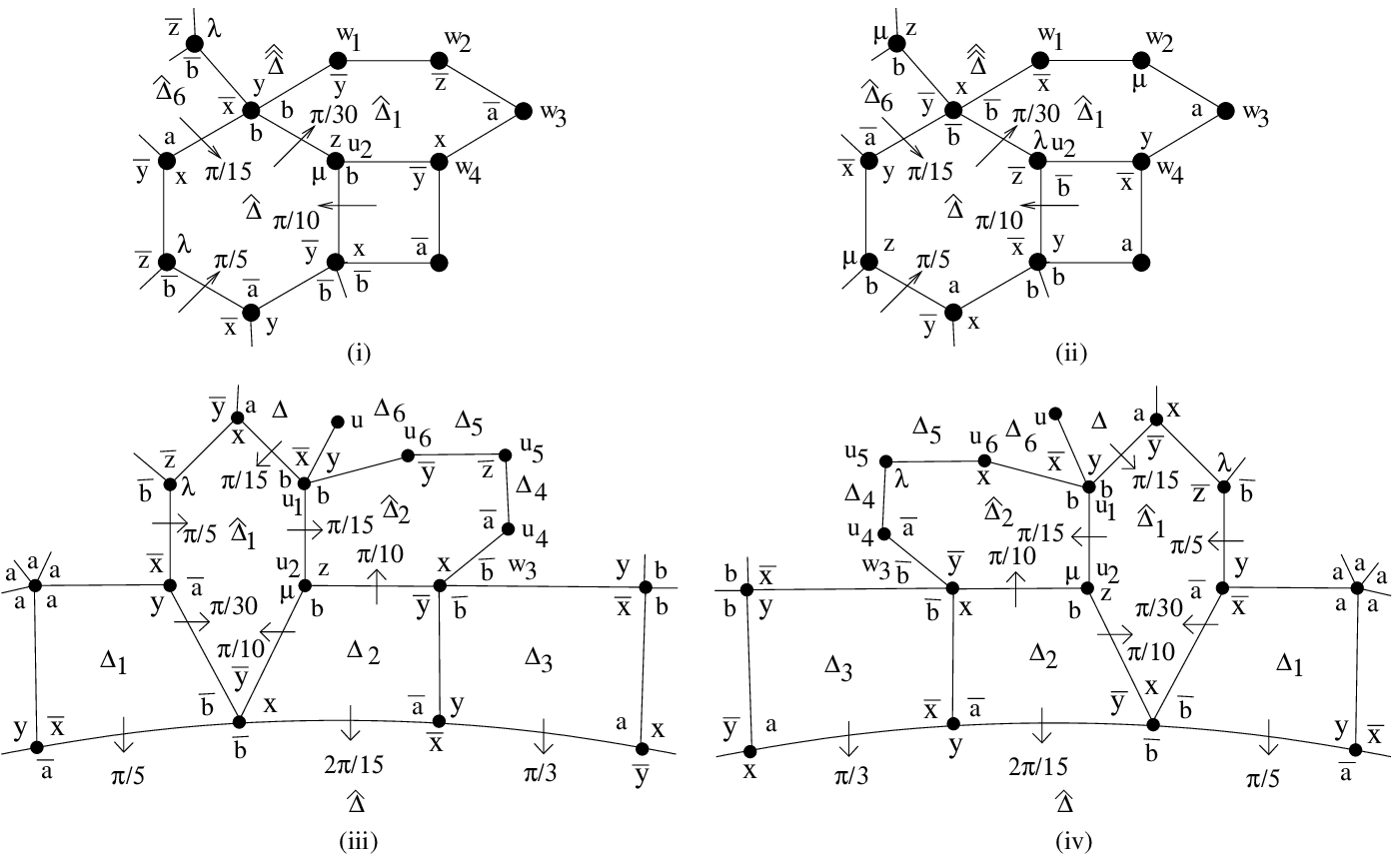}
\end{center}
\caption{}
\end{figure}  

\textit{Proof}.  It is clear from Figures 36-38 that if (i) and (ii) do not hold then
$c^{\ast}(\hat{\Delta}) \in \left\{ \frac{\pi}{30}, \frac{\pi}{15} \right)$ is distributed to $\hat{\Delta}_1$ of Figure 36(i),(x) or
$\hat{\Delta}_2$ of Figure 37(iv),38(iv).
It follows that a region of degree 6 receives positive curvature from at most one region of degree 6.  We treat each pair of cases in turn.

Consider $\hat{\Delta}_1$ of Figure 36(i), (x).  Then $\hat{\Delta}_1$ is given by Figure 39(i), (ii) in which  $d(w_4) > 3$.
Observe that  $d(\hat{\hat{\Delta}}) > 4$  and it follows that  $\hat{\Delta}_1$ does not receive any positive curvature from $\hat{\hat{\Delta}}$ in Figure 39(i), (ii).
Note also from Figure 35 that $c(w_3,w_4)=\frac{2\pi}{15}$ or $\frac{\pi}{15}$;
$c(u_2,w_4)=\frac{\pi}{10}$; and, from Note 3 following Table 3 at the start of Section 7, $c(w_1,w_2)+c(w_2,w_3)=\frac{7 \pi}{30}$.
Therefore $c^{\ast}(\hat{\Delta}_1) \leq c(\hat{\Delta}_1) + \frac{\pi}{2}$.
If however $c(w_3,w_4)=\frac{2\pi}{15}$ then from Figure 35(iii) it follows that $c(\hat{\Delta}_1) \leq c(3,3,3,4,4,5) = -\frac{3 \pi}{5}$
and so $c^{\ast}(\hat{\Delta}_1) \leq 0$; so assume that $c(w_3,w_4)=\frac{\pi}{15}$, $c^{\ast}(\hat{\Delta}_1) \leq c(\hat{\Delta}_1) + \frac{13\pi}{30}$.
If $\hat{\Delta}_1$ has at least three vertices of degree $>3$ then $c(\hat{\Delta}_1) \leq -\frac{\pi}{2}$; and
if $d(w_4) \geq 5$ then $c(\hat{\Delta}_1) \leq c(3,3,3,3,4,5) = -\frac{13 \pi}{30}$;
this leaves $d(w_i)=3$ ($1 \leq i \leq 3$) and $d(w_4)=4$ in which case $c(w_1,w_2)=0$ and $c(w_2,w_3)=\frac{\pi}{5}$.
If $c(w_3,w_4)=0$ then $c^{\ast}(\hat{\Delta}_1) \leq -\frac{\pi}{3} + \frac{\pi}{3}=0$.
On the other hand if $c(w_3,w_4) > 0$ then it follows from the remark preceding the statement of the lemma  that
$\hat{\Delta}_1$ of Figure 39(i),(ii) must coincide with $\hat{\Delta}_1$ of Figure 38(i) (Configuration B), Figure 37(i) (Configuration A). But 
the fact that $\hat{\Delta}_1$ receives $\frac{\pi}{30}$ from $\hat{\Delta}$ in Figure 39(i),(ii) contradicts Lemma 7.2

Now consider $\hat{\Delta}_2$ of Figure 37(iv), 38(iv) and assume that $d(\hat{\Delta}_2)=6$.
Then $\hat{\Delta}_2$ is given by Figure 39(iii), (iv) in which (see Figure 35) the following hold: $c(u_2,w_3)=\frac{\pi}{10}$; $c(u_5,u_6)=\frac{7 \pi}{30}$ if
$d(u_6)<6$; and $c(u_5,u_6)=\frac{2 \pi}{15}$ if $d(u_6) \geq 6$. 
Note that if $\hat{\Delta}_1$ of Figure 39(iii), (iv) does not receive $\frac{\pi}{30}$ from $\Delta_1$ then we are back in the previous case, so assume otherwise.
In particular, according to Configurations A,B of Figure 31, this implies $c(u_4,u_5)\neq \frac{\pi}{5}$ and so $c(u_4,u_5)= \frac{\pi}{6}$; and note that if $d(u_4)=6$ then $c(u_4,u_5)=\frac{2 \pi}{15}$. 
Applying the statement at the beginning of this section, it follows by inspection of Figures 21(ii),(xi), 22(ix), 23(ii) and 31(i),(v) that if $c(w_3,u_4) > 0$ then
$\hat{\Delta}_2$ of Figure 39(iii), (iv) coincides with region $\hat{\Delta}_8$ of Figure 21(xi), 22(xi) in which case $c(w_3,u_4)=\frac{\pi}{30}$ and $d(u_4)=6$.
Finally if $c(u_1,u_2)=\frac{\pi}{15}$ then $\hat{\Delta}_1$ must receive $\frac{\pi}{15}$ from $\Delta$ which implies $d(u)=3$ and $d(\Delta_6) > 4$ and so $c(u_1,u_6) = 0$.
On the other hand if $c(u_1,u_2)=\frac{\pi}{30}$ then (see Figure 35) either $c(u_1,u_6)=\frac{2 \pi}{15}$ in which case $\hat{\Delta}_2$ is given by $\hat{\Delta}_3$ or $\hat{\Delta}_4$
of Figure 18(ii), in particular $d(u_6) \geq 6$; or $d(u_6) < 6$ and $c(u_1,u_6)=\frac{\pi}{15}$

It follows that if $d(u_6) < 6$ then $c^{\ast} (\hat{\Delta}_2) = 
c(\hat{\Delta}_2) + c(u_2,w_3) + c(w_3,u_4) + c(u_4,u_5) + c(u_5,u_6) +(c(u_1,u_2) + c(u_1,u_6)) \leq c(\hat{\Delta}_2) + \frac{\pi}{10} + \frac{\pi}{30} + \frac{\pi}{6} + \frac{7 \pi}{30} + \frac{\pi}{10} =
c(\hat{\Delta}_2) + \frac{19 \pi}{30}$; 
or if $d(u_6) \geq 6$ then $c^{\ast} (\hat{\Delta}_2) \leq c(\hat{\Delta}_2) + \frac{\pi}{10} + \frac{\pi}{30} + \frac{\pi}{6} + \frac{2 \pi}{15} + \frac{\pi}{6} = c(\hat{\Delta}_2) + \frac{18 \pi}{30}$.

Let $d(u_4) \geq 4$.  If $d(u_6) \geq 4$ or $d(u_5) \geq 4$ then $c^{\ast} (\hat{\Delta}_2) \leq - \frac{2 \pi}{3} + \frac{19 \pi}{30} < 0$; 
on the other hand if $d(u_6) = d(u_5) = 3$ then $c(u_5,u_6) = 0$ and $c^{\ast} (\hat{\Delta}_2) \leq c(3,3,3,4,4,4) + \left( \frac{19 \pi}{30} - \frac{7 \pi}{30} \right) < 0$.

Let $d(u_4) = 3$ so, in particular, $c(w_3,u_4)=0$.
If $d(u_6) \geq 4$ and $d(u_5) \geq 4$ or if $d(u_6) = 3$ and $d(u_5) \geq 5$ or if $d(u_6) \geq 5$ and $d(u_5) =3$ then $c(\hat{\Delta}_2) \leq -\frac{3 \pi}{5}$ and it follows that $c^{\ast}(\hat{\Delta}_2) \leq 0$.
If $d(u_6) = 4$ and $d(u_5)=3$ then $d(\Delta_5) > 4$, $c(u_5,u_6)=0$ and $c^{\ast} (\hat{\Delta}_2) \leq -\frac{\pi}{2} + \frac{\pi}{10} + 0 + \frac{\pi}{6} + 0 + \frac{\pi}{10} < 0$; and
if $d(u_6)=3$ and $d(u_5)=4$ then $d(\Delta_4) > 4$, $c(u_4,u_5)=0$ and $c^{\ast} (\hat{\Delta}) \leq - \frac{\pi}{2} + \frac{\pi}{10} + 0 + 0 + \frac{7 \pi}{30} + \frac{\pi}{10} < 0$.
This leaves $d(u_5)=d(u_6)=3$ in which case $c(u_5,u_6)=0$.  Moreover $d(\Delta_5) > 4$ also means that if $c(u_1,u_6)=\frac{\pi}{15}$ then
$\Delta_6$ is given by $\Delta$ of Figure 16(i) forcing the region $\Delta$ of Figure 39(iii), (iv) to have degree $>4$, a contradiction, so
$c(u_1,u_6)=\frac{\pi}{30}$.  Since, as noted above, $c(u_1,u_2)=\frac{\pi}{15}$ implies $c(u_1,u_6)=0$ it follows that $c(u_1,u_2) + c(u_1,u_6) = \frac{\pi}{15}$ and
$c^{\ast} (\hat{\Delta}) \leq c(3,3,3,3,4,4) + \frac{\pi}{10} + 0 + \frac{\pi}{6} + 0 + \frac{\pi}{15} = 0$.
$\Box$

\medskip

Proposition 4.2 follows immediately from Lemma 8.1. 

\section{Two lemmas}

The first two steps of the proof have now been completed. Given this, only step three remains, that is, it remains to consider regions $\hat{\Delta}$ of degree $\geq 8$.  To do this we partition such $\hat{\Delta} \neq 
\Delta_0$ into regions of 
type $\mathcal{A}$ or type $\mathcal{B}$.

\medskip

We say that $\hat{\Delta}$ is a \textit{region of type} $\mathcal{B}$ if $\hat{\Delta}$ receives positive curvature from a region $\Delta$ of degree 4 shown in Figure 5 such that $\Delta$ has not received any positive curvature
from any other region of degree 4 and such that either
$d(v_3) = d(v_4) = 3$ only or $d(v_4) = d(v_1)$ = 3 only. Thus $\hat{\Delta}$ is given by $\hat{\Delta}_3$ of Figure 13(i) or $\hat{\Delta}_4$ of Figure 14(i) or $\hat{\Delta}$ of Figure 31 or $\hat{\Delta}$ of Figure 32(i), (ii), 
(iii) or (v). 
Otherwise we will say that $\hat{\Delta}$ is a \textit{region of type} $\mathcal{A}$. 

\medskip

There will be no further distribution of curvature in what follows and so we collect together in this section results that will be useful in Sections 10 and 11.  
The statements in the following lemma can be verified
by inspecting Figures 6--39.  Further details will appear in the proof of Lemma 10.1.

\begin{lemma}
\textit{Let $e_i$ be an edge with endpoint $u,v$ such that $e_i$ is neither a $(b,a)$-edge nor is the edge of a region $\Delta$ across which positive cuvature is transferred to a type $\mathcal{B}$ region.}
\begin{enumerate}
\item[(i)]
\textit{If $c(e_i) := c(u,v) > \frac{2 \pi}{15}$ then $c(e_i) \in \{ \frac{\pi}{6}, \frac{\pi}{5}, \frac{7 \pi}{30} \}$.}
\item[(ii)]
\textit{If $c(e_i) \in \{ \frac{\pi}{6}, \frac{\pi}{5}, \frac{7 \pi}{30} \}$ then $e_i$ is given by Figure 40 (in which possible $c(e_i)$ is given by multiples of $\frac{\pi}{30}$).}

\begin{figure}
\begin{center}
\psfig{file=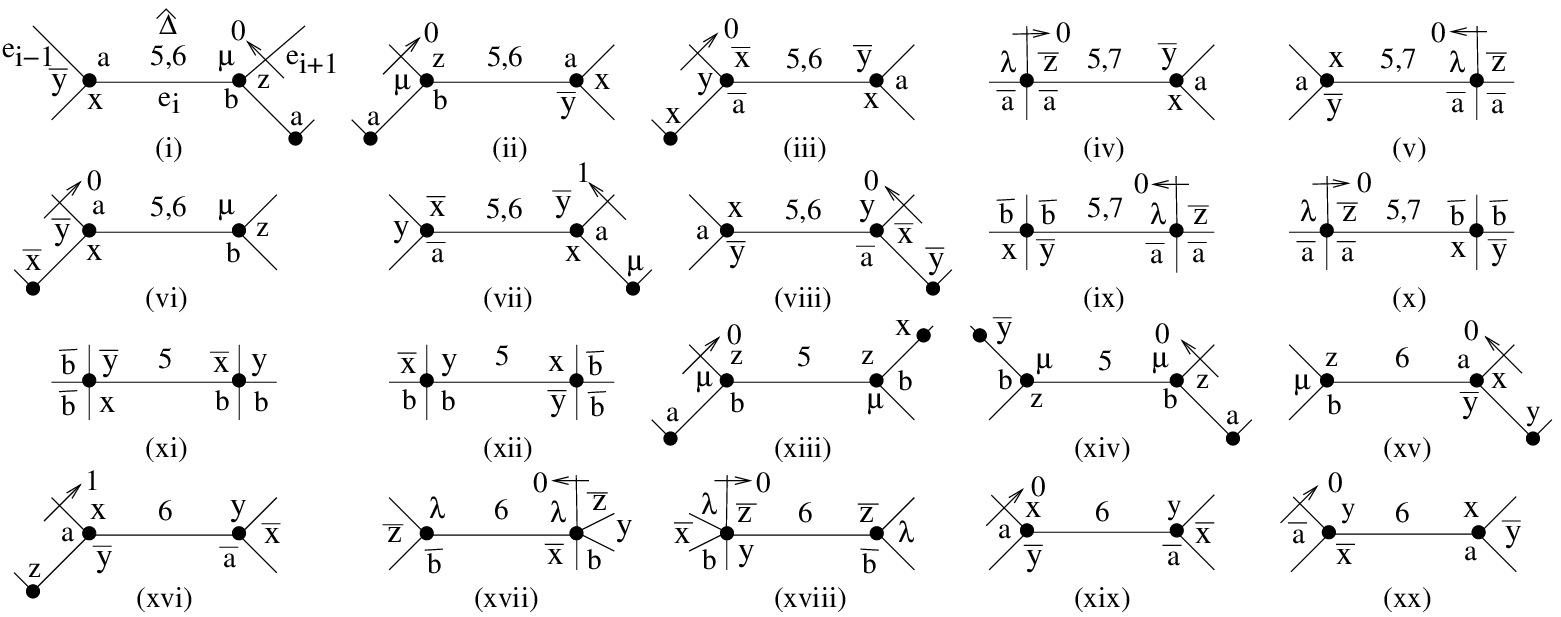}
\end{center}
\caption{}
\end{figure}

\begin{figure}
\begin{center}
\psfig{file=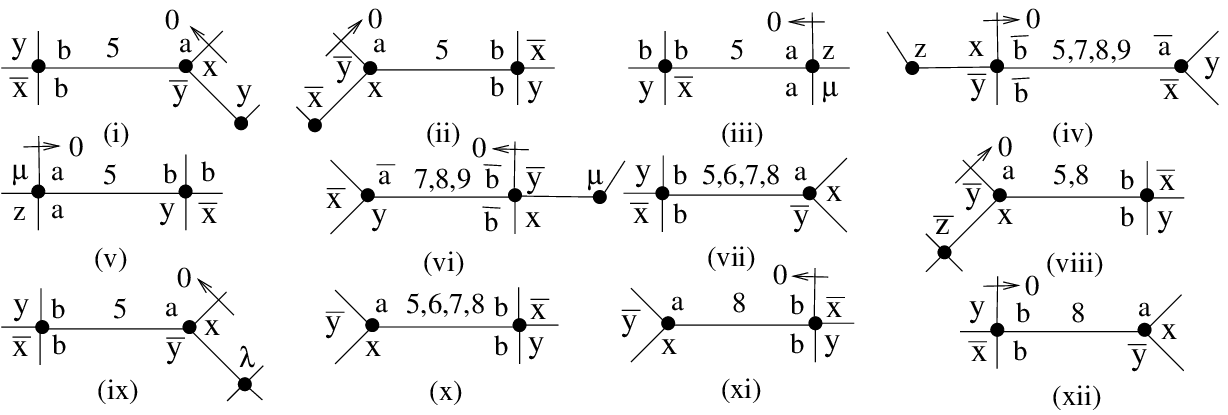}
\end{center}
\caption{}
\end{figure}

\item[(iii)]
\textit{If $c(e_i) > \frac{2 \pi}{15}$ then either $c(e_{i-1})=0$ or $c(e_{i+1})=0$ except for $e_i$ of Figure 40(vii), (xi), (xii) and (xvi).}
\end{enumerate}
\textit{Now assume that $e_i$ be a $(b,a)$-edge} and that transfer of curvature to a type $\mathcal{B}$ region is allowed.
\begin{enumerate}
\item[(iv)]
\textit{If $c(e_i) > \frac{2 \pi}{15}$ then $c(e_i) \in \{ \frac{\pi}{6},\frac{\pi}{5},\frac{7 \pi}{30},\frac{4 \pi}{15}, \frac{3 \pi}{10} \}$.}
\item[(v)]
\textit{If $c(e_i) \in \{ \frac{\pi}{6},\frac{\pi}{5},\frac{7 \pi}{30},\frac{4 \pi}{15},\frac{3 \pi}{10} \}$ then $c_i$ is given by Figure 41.}
\item[(vi)]
\textit{If $c(e_i) > \frac{2 \pi}{15}$ then either $c(e_{i-1})=0$ or $c(e_{i+1})=0$ except for $e_i$ of Figure 41(vii) and (x).}
\end{enumerate}
\end{lemma}

\textbf{Remarks}.
\begin{enumerate}
\item[1.]
In verifying statement (iii) note that $\hat{\Delta}$ of Figure 40(xix), (xx) corresponds to $\hat{\Delta}_1$ of Figure 32(iii),(v) respectively.
\item[2.]
In Figure 40(vii)
if $c(u,v)=\frac{\pi}{6}$ then $\hat{\Delta}=\hat{\Delta}_4$ of Figure 8(i)-(iii);
if $c(u,v)=\frac{\pi}{5}$ then $\hat{\Delta}=\hat{\Delta}_4$ of Figure 8(iv);
moreover the $\frac{\pi}{30}$ distributed across the $e_{i+1}$ edge is given by Figure 36(viii) and (ix).
In Figure 40(xi), $\hat{\Delta}=\hat{\Delta}_1$ of Figure 27(vii).
In Figure 40(xii), $\hat{\Delta}=\hat{\Delta}_2$ of Figure 27(viii).
In Figure 40(xiii), $\hat{\Delta}=\hat{\Delta}_2$ of Figure 37(iii).
In Figure 40(xiv), $\hat{\Delta}=\hat{\Delta}_2$ of Figure 38(iii).
In Figure 40(xvi), $\hat{\Delta}=\hat{\Delta}_3$ of Figure 10(i),(ii);
moreover the $\frac{\pi}{30}$ distributed across the $e_{i-1}$ edge is given by Figure 36(xvii) and (xviii).

\begin{figure}
\begin{center}
\psfig{file=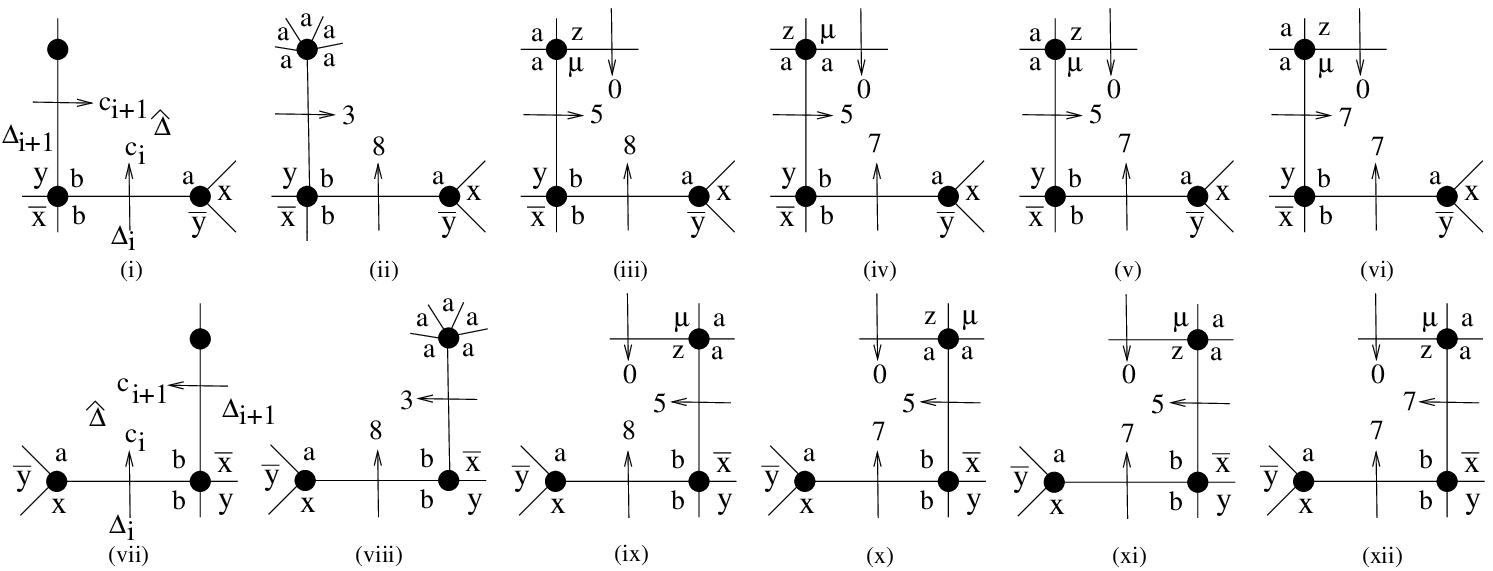}
\end{center}
\caption{}
\end{figure}

\item[3.]
In Figure 41(vii) if $c(u,v)=\frac{\pi}{5}$ then $\hat{\Delta}$ is given by Figure 31(v);
and in Figure 41(x) if $c(u,v)=\frac{\pi}{5}$ then $\hat{\Delta}$ is given by Figure 31(i), in particular, $\hat{\Delta}$ in both cases is a type
$\mathcal{B}$ region.
\end{enumerate}

\begin{lemma}
\textit{Let the regions $\hat{\Delta}$, $\Delta_i$ and $\Delta_{i+1}$ be given by Figure 42(i) or 42(vii).}
\begin{enumerate}
\item[(i)]
\textit{If $c_i = \frac{9 \pi}{30}$ then $c_{i+1}=0$.}
\item[(ii)]
\textit{If $c_i=\frac{8 \pi}{30}$ then $c_{i+1} \leq \frac{5 \pi}{30}$.}
\item[(iii)]
\textit{If $c_i = \frac{8 \pi}{30}$ and $c_{i+1}=\frac{3 \pi}{30}$ then $\hat{\Delta}$ of Figure 42(i) is given by Figure 42(ii) in which
$\Delta_{i+1} = \hat{\Delta}_3$ of Figure 22(xii) or $\Delta_{i+1} = \hat{\Delta}_4$ of Figure 29(x); and
$\hat{\Delta}$ of Figure 42(vii) is given by $\hat{\Delta}$ of Figure 42(viii) in which $\Delta_{i+1}=\hat{\Delta}_4$ of Figure 21(xiv) or $\Delta_{i+1}=\hat{\Delta}_3$ of Figure 29(viii).}
\item[(iv)]
\textit{If $c_i=\frac{8 \pi}{30}$ then $c_{i+1} \neq \frac{4 \pi}{30}$.}
\item[(v)]
\textit{If $c_i = \frac{8 \pi}{30}$ and $c_{i+1}=\frac{5 \pi}{30}$ then $\hat{\Delta}$ of Figure 42(i) is given by Figure 42(iii) in which
$\Delta_{i+1} = \Delta$ of Figure 24(viii); and
$\hat{\Delta}$ of Figure 42(vii) is given by Figure 42(ix) in which $\Delta_{i+1} = \Delta$ of Figure 26(vi).}
\item[(vi)]
\textit{If $c_i = \frac{7 \pi}{30}$ then $c_{i+1} \leq \frac{7 \pi}{30}$.}
\item[(vii)]
\textit{If $c_i = \frac{7 \pi}{30}$ and $c_{i+1}=\frac{5 \pi}{30}$ then $\hat{\Delta}$ of Figure 42(i) is given by Figures 42(iv) and (v) in which
$\Delta_{i+1}=\hat{\Delta}_4$ of Figure 17(xii) and $\Delta_{i+1}=\Delta$ of Figure 24(viii), respectively; and $\hat{\Delta}$ of Figure 42(vii)
is
given by Figures 42(x) and (xi) in which $\Delta_{i+1}=\hat{\Delta}_2$ of Figure 17(iv) and $\Delta_{i+1}=\Delta$ of Figure 26(vi), respectively.}
\item[(viii)]
\textit{If $c_i = \frac{7 \pi}{30}$ then $c_{i+1} \neq \frac{6 \pi}{30}$.}
\item[(ix)]
\textit{If $c_i = c_{i+1} = \frac{7 \pi}{30}$ then $\hat{\Delta}$ of Figure 42(i) is given by Figure 42(vi) in which $\Delta_{i+1}=\hat{\Delta}_4$
of Figure 18(xi); and $\hat{\Delta}$ of Figure 42(vii) is given by Figure 42(xii) in which $\Delta_{i+1}=\hat{\Delta}_2$ of Figure 18(vii) or of
Figure 29(v).}
\end{enumerate}
\end{lemma}

\textit{Proof}.
Statements (i), (vi) and (viii) follow from an inspection of Figures 40 and 41.  

Moreover if $\hat{\Delta}$ is given by Figure 42(i) and
$c_i = \frac{8 \pi}{30}$ then it can be assumed without any loss that either $\Delta_i = \hat{\Delta}_2$ of Figure 22(iv) or (xiii) or $\Delta_i =
\hat{\Delta}_4$ of Figure 29(xi); and if $\hat{\Delta}$ is given by Figure 42(vii) and $c_i=\frac{8 \pi}{30}$ then it can be assumed without any
loss that either $\Delta_i = \hat{\Delta}_4$ of Figure 21(vi) or (xv) or $\Delta_i = \hat{\Delta}_2$ of Figure 29(ix).

(ii) Let $\hat{\Delta}$ be given by Figure 42(i).  If $c_{i+1} > \frac{5 \pi}{30}$ then the only possibility is given by Figure 
40(ix) in which case
$c_{i+1}=\frac{7 \pi}{30}$ and $\Delta_{i+1}=\hat{\Delta}_2$ of Figure 18(xi) where we note that (in $\Delta$) 

\begin{figure} 
\begin{center}
\psfig{file=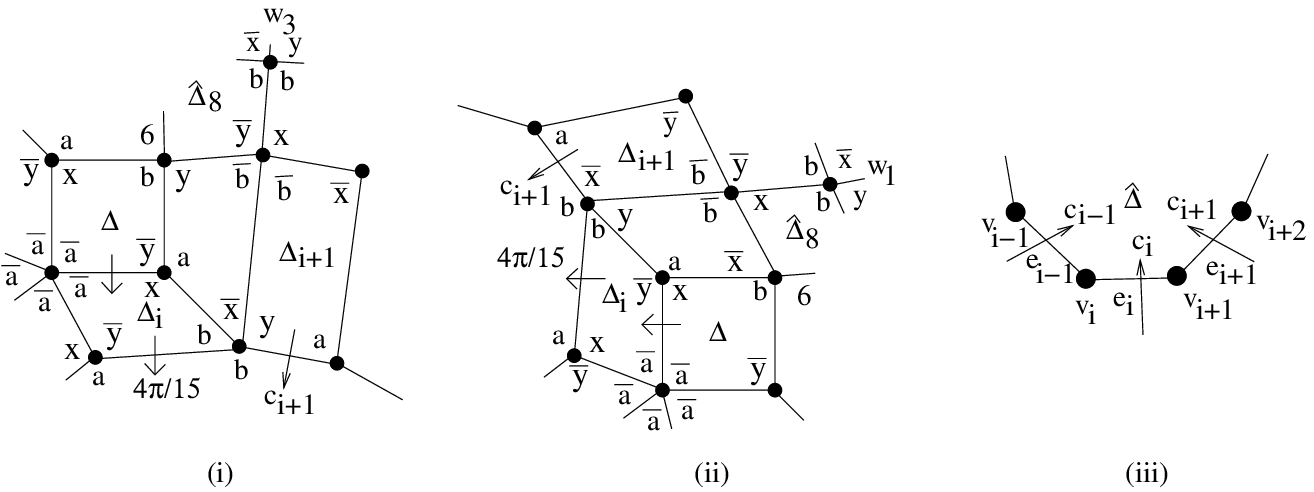}
\end{center}
\caption{}
\end{figure}

\noindent $d(v_1)=4$ and $d(v_2)=3$.  However if
$\Delta_i=\hat{\Delta}_2$ of Figure 22(iv) then the vertex corresponding to $v_1$ is $u_1$ (see Figure 22(i)) which has degree 3; or if $\Delta_i=\hat{\Delta}_4$ of
Figure 29(xi) then the vertex corresponding to $v_1$ is $v_2$ of $\Delta$
which has degree 5, in each case a contradiction.  This leaves $\Delta_i =
\hat{\Delta}_2$ of Figure 22(xiii), where $\Delta_{i+1}=\hat{\Delta}_7$ and this is shown in Figure
43(i) (recall that $\hat{\Delta}_8$ of Figure 22(xiii) is given by $\hat{\Delta}_8$ of Figure 22(x), hence $w_3$ of Figure 43(i)).  But observe that $w_3$ is the vertex of Figure 43(i)
that corresponds to $v_2$ of Figure 18(xi) and since  $w_3$ has degree 4 again there is a contradiction.

Now let $\hat{\Delta}$ be given by Figure 42(vii).
If $c_{i+1} > \frac{5 \pi}{30}$ then the only possibility is given by Figure 40(x) in which case $c_{i+1}=\frac{7 \pi}{30}$ and either
$\Delta_{i+1} = \hat{\Delta}_2$ of Figure 18(vii) where in $\Delta$ $d(v_2)=3$ and $d(v_3)=4$ or $\Delta_{i+1} = \hat{\Delta}_2$ of Figure 29(v) where in $\Delta$
$d(v_2)=5$ and $d(v_3)=4$.  However if $\Delta_i = \hat{\Delta}_4$ of Figure 21(vi) the vertex corresponding to $v_3$ (both cases) is $u_6$ (see Figure 21(iii)) which has
degree 3; or if $\Delta_i = \hat{\Delta}_2$ of Figure 29(ix) the vertex corresponding to $v_3$ (both cases) is $v_2$ which has degree 5, in all
cases
a contradiction. This leaves $\Delta_i = \hat{\Delta}_4$ of Figure 21(xv) where $\Delta_{i+1} = \hat{\Delta}_7$ and this is shown in Figure
43(ii) (and here recall that $\hat{\Delta}_8$ of Figure 21(xv) is given by $\hat{\Delta}_8$ of Figure 21(xii), hence $w_1$ of Figure 43(ii)).
But observe that the vertex of Figure 43(ii) corresponding to $v_2$ of Figures 18(vii), 29(v) is $w_1$ which has degree 4, again a contradiction.

(iii) Checking Figures 6--39 shows that if $c_{i+1} = \frac{3 \pi}{30}$ in Figure 42(i) then $\Delta_{i+1}$ must be one of Figures
11(vii), 11(viii), 22(iii), 22(xii), 29(x) or 31(ii).  Given that $\Delta_i = \hat{\Delta}_2$ of Figure 22(iv) or (xiii) or
$\Delta_i = \hat{\Delta}_4$ of Figure 29(xi) there is a vertex (degree or labelling) contradiction in each possible combination except when $\Delta_{i+1}$ is given
by Figure 22(xii) or Figure 29(x) and these each yield Figure 42(ii).
If $c_{i+1} = \frac{3 \pi}{30}$ in Figure 43(ii) then $\Delta_{i+1}$ must be one of Figure 12(vii), 12(viii), 21(v), 21(xiv), 29(viii) or
31(vi).
Given that $\Delta_i = \hat{\Delta}_4$ of Figure 21(vi) or (xv) or $\Delta_i = \hat{\Delta}_2$ of Figure 29(ix) again there is a vertex 
contradiction in each case except when $\Delta_{i+1}$ is given by Figure 21(xiv) or Figure 29(viii) and these yield Figure 42(viii).

(iv) If $c_{i+1} = \frac{4 \pi}{30}$ in Figure 42(i) then $\Delta_{i+1}$ must be one of the Figures 16(iii), 18(ii), 18(xix) or 31(v), but in
each case there is a vertex contradiction when compared with Figure 22(iv), 22(xiii) or 29(xi). (When comparing Figures 18(xix) and 22(xiii) the
we use Figure 43(i) for 22(xiii) as in case (ii) above.) 
If $c_{i+1} = \frac{4 \pi}{30}$ in Figure 42(vii) then $\Delta_{i+1}$ must be one of Figures 16(ii), 18(ii), 18(xv) or 31(i), and again in
each case there is a vertex contradiction when compared Figure 21(vi), 21(xv) or 29(ix). (When comparing Figures 18(xv) and 21(xv) 
we use Figure 43(ii) for 21(xv) again as in case (ii) above.)

(v) The possibilities for $\Delta_{i+1}$ of Figure 42(i) are (see Figures 40(ix) and 41(v)) $\hat{\Delta}_4$ of Figure 17(xii) which yields a vertex
contradiction when compared with Figure 22(iv), 22(xiii) or 29(xi) and $\Delta$ of Figure 24(viii) which is given by Figure 42(iii); and for $\Delta_{i+1}$ of
Figure 42(vii) are (see Figures 40(x) and 41(iii)) $\hat{\Delta}_2$ of Figure 17(iv) which yields a vertex contradiction when compared Figure 21(vi), 21(xv) or 29(ix) and $\Delta$ of
Figure 26(vi) which is given by Figure 42(ix).

Finally statement (vii) appears in  the proof of (v); and statement (ix) appears in the proof of
(ii). $\Box$


\section{Type $\mathcal{A}$ regions}

Throughout this section many assertions will be based on previous lemmas.  Moreover \textit{checking} means checking Figures 6--34 and Figures
36--39.
The reader is also referred to Figures 35, 40, 41 and 42.

The \textit{surplus} $s_i$ of an edge $e_i$ is defined by $s_i = c_i - \frac{2 \pi}{15}$ ($1 \leq i \leq k$) where $c_i$ is the maximum amount of curvature that is transferred across $e_i$.
If we add $s_i$ to $c_{i+1}, c_{i-1}$ we will say that $e_{i+1}, e_{i-1}$ (respectively) \textit{absorbs} $s_i$ from $c_i$.
Checking Figures 40 and 41 shows, for example, that
if $d(u_i)=d(u_{i+1})=3$ in Figure 43(iii) then $s_i \leq \frac{\pi}{15}$.
The \textit{deficit} $\delta _i$ of a vertex $u_i$ of degree $d_i$ is defined by $\delta_i = 2 \pi ( \frac{1}{d_i} - \frac{1}{3} )$ and so
if $d_i \geq 4$ then $\delta_i \leq - \frac{\pi}{6}$.
If we add $s_{i-1},s_i$ (respectively) to $\delta_i$ we will say that $u_i$ \textit{absorbs} $s_{i-1},s_i$ from $e_{i-1},e_i$ (respectively).

\begin{lemma}
\textit{Let $\hat{\Delta}$ be a type $\mathcal{A}$ region of degree $k$.  Then the following statement holds.
$c^{\ast} (\hat{\Delta}) \leq (2-k) + k. \frac{2 \pi}{3} + k. \frac{2 \pi}{15}$.}
\end{lemma}

\textit{Proof}.
Denote the vertices of $\hat{\Delta}$ by $v_i$ ($1 \leq i \leq k$), the edges by $e_i$ ($1 \leq i \leq k$)
and the degrees of the $v_i$ by $d_i$ ($1 \leq i \leq k$).
Let $c_i$ denote the amount of curvature $\hat{\Delta}$ receives across the edge $e_i$ ($1 \leq i \leq k$).
Consider the edge $e_i$ of $\hat{\Delta}$ as shown in Figure 43(iii).
If $c_i \leq \frac{2 \pi}{15}$ there is nothing to consider, so let $c_i > \frac{2 \pi}{15}$.
Then by Lemma 9.1, $c_i \in \{ \frac{\pi}{6}, \frac{\pi}{5}, \frac{7 \pi}{30}, \frac{4 \pi}{15}, \frac{3 \pi}{10} \}$ and $\hat{\Delta}$ is given by
Figures 40 and 41.
First assume that $e_i$ is not given by Figure 32(iii) or (v).

Let $\hat{\Delta}$ be given by Figure 40.  If $\hat{\Delta}$ is given by Figure 40(i), (vii), (viii), (xiv) or (xv) then the edge $e_{i+1}$ absorbs
$s_i \leq \frac{\pi}{15}$ (from $c_i$). Note that in these cases $(d_{i+1}, c_{i+1}) \in \{ (3,0),(3,\frac{\pi}{30}) \}$.
If $\hat{\Delta}$ is given by Figure 40(ii), (iii), (vi), (xiii), (xvi), (xix) or (xx) then $e_{i-1}$ absorbs $s_i \leq \frac{\pi}{15}$.  Note that
$(d_i, c_{i-1}) \in \{ (3,0), (3,\frac{\pi}{30})\}$.
If $\hat{\Delta}$ is given by Figure 40(iv), (x) or (xviii) then the vertex $v_i$ absorbs $s_i \leq \frac{\pi}{10}$.
Note that $(d_i,c_{i-1}) \in \{ (4,0), (5,0) \}$.  If $\hat{\Delta}$ is given by Figure 40(v), (ix) or (xvii) then $v_{i+1}$ absorbs $s_i \leq
\frac{\pi}{10}$.  Note that $(d_{i+1}, c_{i+1}) \in \{ (4,0),(5,0) \}$.  This leaves Figure 40(xi) and (xii) to be considered.  If $\hat{\Delta}$ is
given by Figure 40(xi) or (xii) then $v_i$ absorbs $s_i = \frac{\pi}{30}$.  Note that $d_i=4$.

Now let $\hat{\Delta}$ be given by Figure 41.  If $\hat{\Delta}$ is given by Figure 41(i) or (ix) then the edge $e_{i+1}$ absorbs
$s_i=\frac{\pi}{30}$.
Note that $(d_{i+1},c_{i+1})=(3,0)$.  If $\hat{\Delta}$ is given by Figure 41(ii) or (viii) restricted to the case $c_i = \frac{5 \pi}{30}$ then
$e_{i-1}$ absorbs $s_i = \frac{\pi}{30}$.
Note that $(d_i,c_{i-1})=(3,0)$.  If $\hat{\Delta}$ is given by Figure 41(iii), (vi) or (xi) then $v_{i+1}$ absorbs $s_i \leq \frac{\pi}{6}$.
Note that $(d_{i+1},c_{i+1})=(4,0)$.  If $\hat{\Delta}$ is given by Figure 41(iv), (v) or (xii) then $v_i$ absorbs $s_i \leq \frac{\pi}{6}$.
Note that $(d_i,c_{i-1})=(4,0)$.  This leaves the cases Figure 41(vii), (viii) with $c_i = \frac{8 \pi}{30}$ and (x).
If $\hat{\Delta}$ is given by Figure 41(vii) then $v_i$ absorbs $s_i \leq \frac{2 \pi}{15}$.
Note that $d_i = 4$.  If $\hat{\Delta}$ is given by Figure 41(viii) or (x) then $v_{i+1}$ absorbs $s_i \leq \frac{2 \pi}{15}$.  Note that
$d_{i+1}=4$.

This completes absorption by edges or vertices when $e_i$ is not given by Figure 32(iii) or (v) (and these correspond to cases of Figure 40(xix),
(xx)).
Observe that if an edge $e_j$ absorbs positive curvature $a_j$, say, then $a_j \leq \frac{\pi}{15}$ and either $c_j=0$ or $c_j=\frac{\pi}{30}$; moreover $e_j$ \textit{always absorbs across a vertex of degree 3}.  If $c_j=0$ then $c_j + a_j \leq \frac{2 \pi}{15}$ so let $c_j = \frac{\pi}{30}$.  We claim that in this case we also have $c_j + a_j \leq \frac{2 \pi}{15}$.  The only possible way this fails is if $s_{j-1}=s_{j+1}=\frac{\pi}{15}$, that is,
$c_{j-1}=c_{j+1}=\frac{\pi}{5}$.  Thus $e_j=e_{i+1}$ of Figure 40(vii) and $\hat{\Delta}=\hat{\Delta}_4$ of Figure 8(iv); and also $e_j=e_{i-1}$ of
Figure 40(xvi) and $\hat{\Delta}=\hat{\Delta}_3$ of Figure 10(i), (ii).  But any attempt at labelling shows that this is impossible and so our
claim follows.
Observe further that any pair of vertices each absorbing more than $\frac{\pi}{30}$ cannot coincide.  This follows immediately from the fact that
either $c_{i-1}=0$ or $c_{i+1}=0$ or the vertex is given by $v_i$ of Figure 41(vii) or $v_{i+1}$ of Figure 41(x) and clearly these cannot coincide.
Also observe that if a vertex $v_i$ say absorbs more than $\frac{2 \pi}{15}$ from $e_i$ or $e_{i-1}$ (respectively) then it absorbs 0 from $e_{i-1}$ or $e_i$ (respectively).
Therefore any given vertex can absorb at most $\frac{\pi}{6}+0=\frac{\pi}{6}$ as in Figure 41(iv) and (vi), or at most
$\frac{2 \pi}{15} + \frac{\pi}{30} = \frac{\pi}{6}$.  But since any vertex that absorbs curvature \textit{has degree at least 4} and so a deficit of at most
$-\frac{\pi}{6}$, the statement of the lemma holds for these cases.

Finally let $e_i$ be given by Figure 32(iii) or (v).  Since $d(v) \geq 4$ in both figures it follows that $e_{i-1}$ does not
absorb
any surplus from $e_{i-2}$. If $s_{i+1} > \frac{\pi}{15}$ then according to the above it must be absorbed by $v_{i+2} = w$ (of Figures 32(iii) and (v)) and in this case $e_{i-1}$ absorbs $s_i  \leq \frac{ \pi}{15}$; or if 
$s_{i+1} \leq \frac{\pi}{15}$ then let $e_{i-1}$ absorb $s_i + s_{i+1} \leq \frac{2 \pi}{15}$. Again the statement follows. $\Box$

\begin{proposition}
\textit{If $\hat{\Delta}$ is a type $\mathcal{A}$ region of degree $k$ and $k \geq 10$ then $c^{\ast} (\hat{\Delta}) \leq 0$.}
\end{proposition}

\textit{Proof}.
This follows from Lemma 10.1 and the fact that $(2-k)+k.\frac{2 \pi}{3} + k.\frac{2 \pi}{15} \leq 0$ if and only if $k \geq 10$. $\Box$

\medskip

It follows from Proposition 10.2 that we need only consider type $\mathcal{A}$ regions of degree at most 9. The following lemma applies to all regions $\hat{\Delta}$.

\begin{lemma}
\textit{If $7 \leq d(\hat{\Delta}) \leq 9$ then
(up to cycle-permutation and corner labelling) either $d(\hat{\Delta})=8$ and $\hat{\Delta}$ is given by Figure 44(i)-(xi) or
$d(\hat{\Delta})=9$ and $\hat{\Delta}$ is given by Figure 44(xii).}
\end{lemma}

\textit{Proof}.
If $7 \leq d(\hat{\Delta}) \leq 9$ then $\hat{\Delta}$ is given by Figure 4(iv)--(xi).
It turns out that there is (up to cyclic permutation and inversion) exactly one way to label $\hat{\Delta}$ of Figure 4(iv), (v), (ix) and (xi);
four ways to label $\hat{\Delta}$ of (vi);
six ways to label $\hat{\Delta}$ of (vii);
and two ways to label $\hat{\Delta}$ 

\newpage
\begin{figure}
\begin{center}
\psfig{file=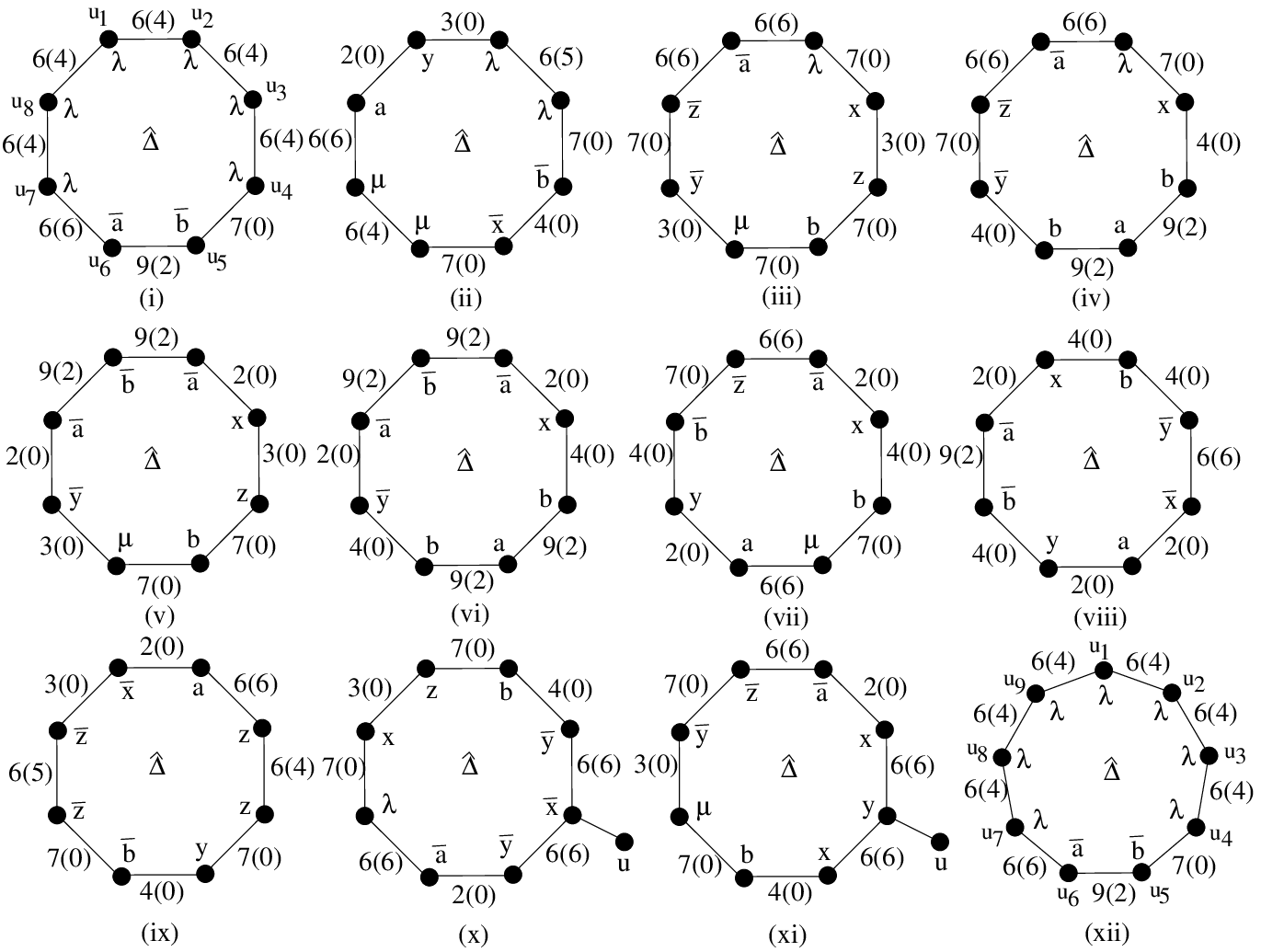}
\end{center}
\caption{}
\end{figure}

\noindent of (viii) and (x).
The resulting set of seventeen labelled regions contains some repeats with respect to corner labelling and deleting these leaves the twelve
$\hat{\Delta}$ of Figure 44(i)-(xii). $\Box$

\medskip

\textbf{Notation}\quad Let $d(\hat{\Delta})=k$ and suppose that the vertices of $\hat{\Delta}$ are $u_i$ ($1 \leq i \leq k$).  We write
$cv(\hat{\Delta}) = (a_1, \ldots, a_k)$, where each $a_i$ is a non-negative integer, to denote the fact that the total amount of curvature $\hat{\Delta}$ receives is bounded above by $(a_1 + \ldots + a_k) \frac{\pi}{30}$ with the understanding that $a_i \frac{\pi}{30}$ is transferred to $\hat{\Delta}$ across the $(u_i,u_{i+1})$-edge (subscripts mod $k$).

\medskip

\textbf{Notation}\quad In the proof of Proposition 10.4 we will use non-negative integers $a_1,a_2$, $b_1,b_2$,
$c_1,c_2$, $d_1,d_2$, $e_1,e_2$, $h_1,h_2$ where: 
\[
a_1+a_2=7; b_1+b_2=8; c_1+c_2=9; d_1+d_2=10; e_1+e_2=11; h_1+h_2=14.
\]

Let $c(\Delta)=c(d_1,\ldots,d_m)$. Suppose that $m=m_1+m_2+m_3=8+k$ where $k \geq 0$ and suppose further that $\Delta$ contains
$m_1,m_2,m_3$ vertices of degree 3, 4, 5 (respectively).  Then we will use the following formula (here and in the next section)
\[
c(\Delta)=c(3,\ldots,3,4,\ldots,4,5,\ldots,5)=-\frac{(20+10k+5m_2+8m_3)\pi}{30}.
\]

\textbf{Remark}\quad Much use will be made here and in Section 11 of the fact that the region $\hat{\Delta}_1$ of Figure 36(i),(x) receives no curvature from the region $\Delta$ shown. If $\hat{\Delta}_2$ of Figure 37(iv), 38(iv) 
receives $\frac{ \pi}{15}$ from $\hat{\Delta}_1$ then $\hat{\Delta}_2$ receives no curvature from the region $\Delta_6$ shown; however if $\hat{\Delta}_2$ receives $\frac{ \pi}{30}$ from $\hat{\Delta}_1$ then $\hat{\Delta}_2$
may receive curvature from $\Delta_6$ although note that $\hat{\Delta}_2$ has at least two vertices of degree $>3$.  

\begin{proposition}
\textit{If $\hat{\Delta}$ is a type $\mathcal{A}$ region and $7 \leq d(\hat{\Delta}) \leq 9$ then $c^{\ast}(\hat{\Delta}) \leq 0$.}
\end{proposition}

\textit{Proof}.  It follows from Lemma 10.3 that we need only consider $\hat{\Delta}$ of Figure 44 in which the label $\alpha ( \beta )$ at the edge
with endpoints $u,v$ indicates $c(u,v)=\frac{\alpha \pi}{30}$ or $c(u,v)=\frac{\beta \pi}{30}$ when $d(u)=d(v)=3$.  We treat each of the twelve
cases of Figure 44 in turn.
(We will make extensive use of \textit{checking} and Figures 35,40,41 and 42 often without explicit reference although for the reader's benefit full details will be given in Cases 1 and 4.)

\textbf{Case 1}\quad
Let $\hat{\Delta}$ be given by Figure 44(i). Note that $\hat{\Delta}$ cannot be $\hat{\Delta}_2$ of Figure 37 or 38.
If $c(u_1,u_2) > \frac{2 \pi}{15}$ then, noting that Figure 40(xiv) does not apply to $\hat{\Delta}$, $d(u_1)>3$ and $c(u_8,u_1)=0$ (see Figure 40(xvii)); and if $\frac{\pi}{15} < c(u_1,u_2) < \frac{\pi}{5}$ then
$c(u_1,u_2) \in \left\{ \frac{2 \pi}{15}, \frac{\pi}{10} \right\}$ and either $c(u_8,u_1)=0$ or $c(u_2,u_3)=0$
(see Figures 15(iii), 18(ix), 23(ix), (xiii), 34(iv), (vii) and 36(xiv),(xviii)).
Similar statements hold for each of $(u_2,u_3)$, $(u_3,u_4)$, $(u_7,u_8)$ and $(u_8,u_1)$.
In particular it follows that $c(u_7,u_8) + c(u_8,u_1) + c(u_1,u_2) + c(u_2,u_3) + c(u_3,u_4) \leq \frac{2 \pi}{3}$. Indeed the maximum is given by $(6 + 4 + 0 + 6 + 4)\frac{\pi}{30}$.
If $c(u_4,u_5) > \frac{2 \pi}{15}$ then (see Figure 40(ix)) $d(u_4)=d(u_5)=4$ and $c(u_3,u_4)=0$; if $c(u_6,u_7) > \frac{2 \pi}{15}$ then (see Figure 40(i), (vi)) $d(u_6)=d(u_7)=3$ and either
$c(u_5,u_6)=0$ or $c(u_7,u_8)=0$; and by Lemma 9.2 (see Figure 42(vi)), $c(u_4,u_5) + c(u_5,u_6) \leq \frac{7 \pi}{15}$.  Therefore
if $c(u_4,u_5) > \frac{2 \pi}{15}$ then $cv(\hat{\Delta}) = (0,6,0,h_1,h_2,6,0,6)$; and
if $c(u_4,u_5) \leq \frac{2 \pi}{15}$ then $c(u_4,u_5) + c(u_5,u_6) \leq \frac{11 \pi}{30}$ (see Figure 42) and $cv(\hat{\Delta}) = (4,0,6,e_1,e_2,6,0,6)$ .  So if $\hat{\Delta}$ has at least three
vertices of degree $> 3$ then $c^{\ast}(\hat{\Delta}) \leq - \frac{35 \pi}{30} + \frac{33 \pi}{30} < 0$.
If $d(u_1)=d(u_2)=3$ and $c(u_1,u_2) > 0$ then  $\hat{\Delta}$ is given by $\hat{\Delta}_1$ of Figure 36(xiv) or (xviii)
therefore $c(u_1,u_2)=\frac{2 \pi}{15}$ and $c(u_2,u_3)=0$.  
Again similar statements hold for $(u_2,u_3)$, $(u_3,u_4)$, $(u_7,u_8)$ and $(u_8,u_1)$.
Suppose that $\hat{\Delta}$ has no vertices of degree $> 3$. In particular $l(u_4)=\lambda b^{-1}z^{-1}$ and $l(u_5)=b^{-1}z^{-1} \lambda$. 
Then $c(u_4,u_5)=0$ and $c(u_5,u_6)=\frac{\pi}{15}$ (see Figure 36(i)-(ix)) so it follows that $cv(\hat{\Delta})=(4,0,4,0,2,4,4,0)$ and
$c^{\ast}(\hat{\Delta}) \leq - \frac{2 \pi}{3} + \frac{3 \pi}{5} < 0$.
Suppose that $\hat{\Delta}$ has exactly one vertex of degree $>3$.
If $d(u_5)=3$ then $c(u_4,u_5)=0$ and $c(u_5,u_6)=\frac{\pi}{15}$ (see Figure 36) and it follows that $cv(\hat{\Delta})=(0,6,4,0,2,6,0,4)$; and
if $d(u_5) > 3$ then $d(u_4) = 3$ and $c(u_4,u_5) = \frac{\pi}{30}$ (see Figure 36(x)-(xviii)) so $cv(\hat{\Delta})=(4,0,4,c_1,c_2,4,4,0)$.
Therefore $c^{\ast}(\hat{\Delta}) \leq - \frac{5 \pi}{6} + \frac{5 \pi}{6} =0$.
Finally suppose that $\hat{\Delta}$ has exactly two vertices $u_i,u_j$ of degree $>3$.
If $d(u_5)=3$ then $c^{\ast}(\hat{\Delta}) < 0$ so it can be assumed without any loss that $i=5$.
If $j=1$ then (since $l(u_8)=\lambda b^{-1}z^{-1}$) $c(u_8,u_1)=0$ and if $c(u_4,u_5) > 0$ then $c(u_4,u_5)=\frac{\pi}{30}$ and $c(u_5,u_6)=0$ (see Figure 36(x)) so $cv(\hat{\Delta}) = (6,0,4,c_1,c_2,4,4,0)$;
if $j=2$ then $c(u_1,u_2)=0$ and $cv(\hat{\Delta})=(0,6,4,c_1,c_2,6,0,4)$;
if $j=3$ then $c(u_2,u_3)=0$ and $cv(\hat{\Delta})=(4,0,6,c_1,c_2,4,4,0)$;
if $j=4$ then $c(u_3,u_4)=0$ and $cv(\hat{\Delta})=(0,4,0,h_1,h_2,6,0,4)$;
if $j=6$ then $c(u_4,u_5)=0$, $c(u_5,u_6)=\frac{\pi}{6}$ and $cv(\hat{\Delta})=(4,0,4,0,5,4,4,0)$;
if $j=7$ then $cv(\hat{\Delta})=(4,0,4,c_1,c_2,4,4,0)$; and
if $j=8$ then $c(u_7,u_8)=0$ and $cv(\hat{\Delta})=(4,0,4,c_1,c_2,6,0,6)$.
It follows that $c^{\ast}(\hat{\Delta}) \leq - \pi + \frac{29 \pi}{30} < 0$.

\textbf{Case 2}\quad
Let $\hat{\Delta}$ be given by Figure 44(ii).
If $c(u_3,u_4) > \frac{2 \pi}{15}$ then, see Figure 40(ix),  $(d(u_3),d(u_4))=(4,4)$ and $c(u_2,u_3)=0$ ; and
if $c(u_5,u_6) > \frac{2 \pi}{15}$ then, see Figure 40 (v),  $(d(u_5),d(u_6))=(3,4)$ and $c(u_6,u_7)=0$.  It follows that if at least three of $u_i$ have degree $\geq 4$ then
$c^{\ast}(\hat{\Delta}) \leq - \frac{7 \pi}{6} + \frac{7 \pi}{6}=0$, so assume otherwise. Note that if  $d(u_2) = d(u_3)=3$ and $c(u_2,u_3) >  \frac{2 \pi}{15}$
then $\hat{\Delta}$ is given by $\hat{\Delta}_2$ of Figure 38(iii); in particular, $c(u_2,u_3)= \frac{ \pi}{6}$ and $d(u_1) = 4$ .
If $\hat{\Delta}$ has no vertices of degree $>3$ then we see (from Figure 44(ii) and Figure 38(iii)) that $cv(\hat{\Delta})=(0,4,0,0,0,4,6,0)$ and
$c^{\ast}(\hat{\Delta}) \leq - \frac{2 \pi}{3} + \frac{7 \pi}{15} < 0$.
Let $\hat{\Delta}$ have exactly one vertex $u_i$ of degree $>3$. Then the following holds.
If $i=1$ then
$cv(\hat{\Delta})=(3,5,0,0,0,4,6,2)$;
if $i=2$ then $cv(\hat{\Delta})=(3,6,0,0,0,4,6,0)$;
if $i=3$ then $cv(\hat{\Delta})=(0,6,4,0,0,4,6,0)$;
if $i=4$ then $cv(\hat{\Delta})=(0,4,4,4,0,4,6,0)$;
if $i=5$ then $cv(\hat{\Delta})=(0,4,0,4,4,4,6,0)$;
if $i=6$ then $cv(\hat{\Delta})=(0,4,0,0,d_1,d_2,6,0)$;
if $i=7$ then $cv(\hat{\Delta})=(0,4,0,0,0,6,6,0)$; and
if $i=8$ then $cv(\hat{\Delta})=(0,4,0,0,0,4,6,2)$.
It\linebreak
follows that $c^{\ast}(\hat{\Delta}) \leq - \frac{5 \pi}{6} + \frac{11 \pi}{15} < 0$.
Let $\hat{\Delta}$ have exactly two vertices $u_i,u_j$ of degree $>3$.
If $d(u_3)=d(u_4)=d(u_5)=3$ or $d(u_4)=d(u_5)=d(u_6)=3$ then $c^{\ast}(\hat{\Delta}) \leq - \pi + \pi = 0$.
This leaves 14 out of 28 cases to be considered. If $(i,j)=(1,4)$ then $cv(\hat{\Delta}) = (3,5,4,4,0,4,6,2)$;
if $(1,5)$ then $(3,5,0,4,4,4,6,2)$;
if $(2,4)$ then $(3,6,4,4,0,4,6,0)$;
if $(2,5)$ then $(3,6,0,4,4,4,6,0)$;
if $(3,4)$ then $(0,d_1,d_2,4,0,4,6,0)$;
if $(3,5)$ then $(0,6,4,4,4,4,6,0)$;
if $(3,6)$ then $(0,6,4,0,d_1,d_2,6,0)$;
if $(4,5)$ then $(0,4,4,4,4,4,6,0)$;
if $(4,6)$ then 
$(0,4,4,4,d_1,d_2,6,0)$;
if $(4,7)$ then $(0,4,4,4,0,6,6,0)$;
if $(4,8)$ then $(0,4,4,4,0,4,6,2)$;
if $(5,6)$ then $(0,4,0,4,4,6,6,0)$;
if $(5,7)$ then $(0,4,0,4,4,6,6,0)$; and
if $(5,8)$ then 
It follows that $c^{\ast}(\hat{\Delta}) \leq - \pi + \frac{14 \pi}{15} < 0$.

\textbf{Case 3}\quad
Let $\hat{\Delta}$ be given by Figure 44(iii).
If $c(u_2,u_3) > \frac{2 \pi}{15}$ then $(d(u_2),d(u_3))=(4,3)$ and $c(u_1,u_2)=0$;
if $c(u_4,u_5) > \frac{2 \pi}{15}$ then, see Figure 40(x), $(d(u_4),d(u_5))=(4,4)$ and $c(u_3,u_4)=0$;
if $c(u_5,u_6) > \frac{2 \pi}{15}$ then $(d(u_5),d(u_6))=(4,4)$ and $c(u_6,u_7)=0$;
if $c(u_7,u_8) > \frac{2 \pi}{15}$ then, see Figure 40(iv), $(d(u_7),d(u_8))=(3,4)$ and $c(u_8,u_1)=0$. Moreover
if $c(u_1,u_2) > \frac{2 \pi}{15}$ then $d(u_2)=3$ and $c(u_2,u_3)=0$; and
if $c(u_8,u_1) > \frac{2 \pi}{15}$ then $d(u_8)=3$ and $c(u_7,u_8)=0$.
It follows that
$c(u_1,u_2)+c(u_2,u_3) \leq \frac{4 \pi}{15}$;
$c(u_3,u_4)+c(u_4,u_5) \leq \frac{7 \pi}{30}$;
$c(u_5,u_6)+c(u_6,u_7) \leq \frac{7 \pi}{30}$; and
$c(u_7,u_8)+c(u_8,u_1) \leq \frac{4 \pi}{15}$.
Therefore if $\hat{\Delta}$ has at least two vertices of degree $>3$ then
$c^{\ast}(\hat{\Delta}) \leq - \pi + \pi =0$.
Suppose that $\hat{\Delta}$ contains no vertices of degree $>3$.
Then we see (from Figure 44(iii)) that $cv(\hat{\Delta})=(6,0,0,0,0,0,0,6)$ and
$c^{\ast} (\hat{\Delta}) \leq - \frac{2 \pi}{3} + \frac{2 \pi}{5} < 0$.
Let $\hat{\Delta}$ have exactly one vertex $u_i$ of degree $>3$.  Then the following holds:
if $i=1$ then $cv(\hat{\Delta}) = (6,0,0,0,0,0,0,6)$;
if $i=2$ then $cv(\hat{\Delta})=(b_1,b_2,0,0,0,0,0,6)$;
if $i=3$ then $cv(\hat{\Delta})=(6,0,3,0,0,0,0,6)$;
if $i=4$ then $cv(\hat{\Delta})=(6,0,3,4,0,0,0,6)$;
if $i=5$ then $cv(\hat{\Delta})=(6,0,0,4,4,0,0,6)$;
if $i=6$ then $cv(\hat{\Delta})=(6,0,0,0,4,3,0,6)$;
if $i=7$ then $cv(\hat{\Delta})=(6,0,0,0,0,3,0,6)$; and
if $i=8$ then $cv(\hat{\Delta})=(6,0,0,0,0,0,b_1,b_2)$.
Therefore $c^{\ast}(\hat{\Delta}) \leq - \frac{5 \pi}{6} + \frac{2 \pi}{3} < 0$.

\textbf{Case 4}\quad
Let $\hat{\Delta}$ be given by Figure 44(iv).
If $c(u_1,u_2) > \frac{2 \pi}{15}$ then $c(u_2,u_3)=0$;
if $c(u_2,u_3) > \frac{2 \pi}{15}$ then $c(u_1,u_2)=0$;
if $c(u_8,u_1) > \frac{2 \pi}{15}$ then $c(u_7,u_8)=0$;
if $c(u_7,u_8) > \frac{2 \pi}{15}$ then $c(u_8,u_1)=0$;
if $c(u_4,u_5) = \frac{3 \pi}{10}$ then, see Figure 41(iv), $c(u_3,u_4)=0$;
if $c(u_4,u_5) = \frac{4 \pi}{15}$ then $c(u_3,u_4)=\frac{\pi}{15}$ (see Figure 42);
if $c(u_5,u_6) = \frac{3 \pi}{10}$ then, see Figure 41(vi), $c(u_6,u_7)=0$; and
if $c(u_5,u_6) = \frac{4 \pi}{15}$ then $c(u_6,u_7)=\frac{\pi}{15}$ (see Figure 42).
It follows that $c(u_3,u_4)+c(u_4,u_5)+c(u_5,u_6)+c(u_6,u_7)=\frac{11 \pi}{15}$.
Therefore $c^{\ast}(\hat{\Delta}) \leq c(\Delta) + \frac{19 \pi}{15}$ so
if $\hat{\Delta}$ has at least four vertices of degree $>3$ then $c^{\ast}(\hat{\Delta}) < 0$.
Let $\hat{\Delta}$ have no vertices of degree $>3$.  Then $cv(\hat{\Delta})=(6,0,0,2,2,0,0,6)$ and
$c^{\ast}(\hat{\Delta}) \leq -\frac{2 \pi}{3} + \frac{8 \pi}{15} < 0$.  Let $\hat{\Delta}$ have exactly one vertex $u_i$ of degree $>3$.
If $d(u_4)=3$ then $c(u_3,u_4)=0$ and $c(u_4,u_5)=\frac{\pi}{15}$; and
if $d(u_6)=3$ then $c(u_5,u_6)=\frac{\pi}{15}$ and $c(u_6,u_7)=0$.  Thus
if $d(u_4)=d(u_6)=3$ then $c^{\ast}(\hat{\Delta}) \leq -\frac{5 \pi}{6} + \frac{2 \pi}{3} < 0$;
if $d(u_4) > 3$ then $cv(\hat{\Delta})=(6,0,e_1,e_2,2,0,0,6)$; and
if $d(u_6) > 3$ then $cv(\hat{\Delta})=(6,0,0,2,e_1,e_2,0,6)$.
Therefore $c^{\ast}(\hat{\Delta}) \leq -\frac{5 \pi}{6} + \frac{5 \pi}{6}=0$.
Let $\hat{\Delta}$ have exactly two vertices of degree $>3$.
If $d(u_4)=3$ or $d(u_6)=3$ then $c^{\ast}(\hat{\Delta}) \leq - \pi + \frac{29 \pi}{30} < 0$ so it can be assumed that $d(u_4) > 3$ and $d(u_6)>3$.
Then $d(u_2)=3$ implies $d(\Delta_2)>4$ and $c(u_2,u_3)=0$; and $d(u_8)=3$ implies $d(\Delta_7) >4$ and $c(u_7,u_8)=0$.
This then prevents $c(u_3,u_4)=\frac{2 \pi}{15}$ or $c(u_6,u_7)=\frac{2 \pi}{15}$ (see Figure 16(ii) and (iii)) so
$c(u_3,u_4)=c(u_7,u_8)=\frac{\pi}{15}$.
Since $c(u_1,u_2)=c(u_8,u_1)=\frac{\pi}{5}$ it follows that
if $c(u_4,u_5) \neq \frac{4 \pi}{15}$ and $c(u_5,u_6) \neq \frac{4 \pi}{15}$ then
$cv(\hat{\Delta})=(6,0,c_1,c_2,c_1,c_2,0,6)$ and
$c^{\ast}(\hat{\Delta}) \leq - \pi + \pi = 0$.
Suppose that $c(u_4,u_5) = \frac{4 \pi}{15}$, say.
If $c(u_3,u_4)>0$ then $\hat{\Delta}= \hat{\Delta}_5$ of Figure 29(vi) or (ix) where  $l(v_2)=b^5$ and so $c(u_3,u_4) > 0$ implies that $\Delta_3 = \Delta$ of Figure 28.
But $d(\Delta_2) > 4$ forces $\Delta_3 = \Delta$ of Figure 28(i) and $c(u_3,u_4)=\frac{\pi}{30}$.
Similarly if $c(u_5,u_6)=\frac{4 \pi}{15}$ and $c(u_6,u_7)>0$ then we see from Figure 29(xi) and Figure 28(i), (iii) that $c(u_6,u_7)=\frac{\pi}{30}$.
It follows that
if $c(u_4,u_5)=\frac{4 \pi}{15}$ or $c(u_5,u_6)=\frac{4 \pi}{15}$ then $c^{\ast}(\hat{\Delta}) \leq - \pi + \pi = 0$.
Finally let $\hat{\Delta}$ have exactly three vertices $u_i,u_j,j_k$ of degree $>3$.  Then $c(\hat{\Delta}) \leq - \frac{7 \pi}{6}$.
If $d(u_2)=d(u_8)=3$ then $c(u_2,u_3)=c(u_7,u_8)=0$ and $cv(\hat{\Delta})=(6,0,e_1,e_2,e_1,e_2,0,6)$;
if $d(u_4)=3$ then $cv(\hat{\Delta})=(b_1,b_2,0,2,e_1,e_2,b_1,b_2)$; and
if $d(u_6)=3$ then $cv(\hat{\Delta})=(b_1,b_2,e_1,e_2,2,0,b_1,b_2)$.  So it can be assumed that
$(i,j,k)=(2,4,6)$ or $(4,6,8)$ and in both cases $c^{\ast}(\hat{\Delta}) \leq c(\hat{\Delta}) + \frac{36 \pi}{30}$.
If $d(u_i)$ or $d(u_j)$ or $d(u_k)$ is greater than 4 then $c^{\ast}(\hat{\Delta}) < 0$ so assume otherwise.
But now $d(u_2)=4$ implies $l(u_2)=\lambda z^{-1}a^{-2}$ and $c(u_1,u_2)=0$; and $d(u_8)=4$ forces $c(u_8,u_1)=0$.
It follows that $cv(\hat{\Delta})=(0,7,e_1,e_2,e_1,e_2,0,6)$ or $(6,0,e_1,e_2,e_1,e_2,7,0)$ so $c^{\ast}(\hat{\Delta}) \leq - \frac{7 \pi}{6}+\frac{7 \pi}{6}=0$.

\textbf{Case 5}\quad
Let $\hat{\Delta}$ be given by Figure 44(v).
If $c(u_1,u_2)=\frac{9 \pi}{30}$ then $c(u_8,u_1)=0$;
if $c(u_1,u_2)=\frac{4\pi}{15}$ then $c(u_8,u_1)=\frac{\pi}{10}$ (see Figure 42);
if $c(u_8,u_1)=\frac{3 \pi}{10}$ then $c(u_1,u_2)=0$;
if $c(u_8,u_1)=\frac{4\pi}{15}$ then $c(u_1,u_2)=\frac{\pi}{10}$;
if $c(u_4,u_5) > \frac{2 \pi}{15}$ then $c(u_3,u_4)=0$; and
if $c(u_5,u_6) > \frac{2 \pi}{15}$ then $c(u_6,u_7)=0$.  It follows that $c(u_8,u_1)+c(u_1,u_2)=\frac{7 \pi}{15}$;
$c(u_3,u_4)+c(u_4,u_5)=\frac{7 \pi}{30}$; and $c(u_5,u_6)+c(u_6,u_7)=\frac{7 \pi}{30}$ so $c^{\ast}(\hat{\Delta}) \leq c(\hat{\Delta}) + \frac{16 \pi}{15}$.
Therefore
if $\hat{\Delta}$ has at least three vertices of degree $>3$ then $c^{\ast}(\hat{\Delta}) < 0$.
If $\hat{\Delta}$ has no vertices of degree $>3$ then we see (from Figure 44(v)) that
$c^{\ast} (\hat{\Delta}) \leq \frac{2 \pi}{3} + \frac{2 \pi}{15} < 0$.  Observe that
if $d(u_1)=3$ then $c(u_8,u_1) + c(u_1,u_2) = \frac{2 \pi}{15}$; and
if $d(u_5)=3$ then $c(u_4,u_5)=c(u_5,u_6)=0$.  It follows that
if $d(u_1)=3$ or $d(u_5)=3$ then $c^{\ast} (\hat{\Delta}) \leq c(\hat{\Delta}) + \frac{11 \pi}{15} < 0$ and so
if $\hat{\Delta}$ has exactly one vertex of degree $>3$ then $c^{\ast}(\hat{\Delta}) < 0$.
If $\hat{\Delta}$ has exactly two vertices $u_i,u_j$ of degree $>3$ it can be assumed that $(i,j)=(1,5)$ in which case
$cv(\hat{\Delta})=(h_1,0,0,7,7,0,0,h_2)$.  Therefore $c^{\ast}(\hat{\Delta}) \leq - \pi + \frac{14 \pi}{15} < 0$.

\textbf{Case 6}\quad
Let $\hat{\Delta}$ be given by Figure 44(vi).
If $c(u_4,u_5)=\frac{3 \pi}{10}$ then $c(u_3,u_4)=0$;
if $c(u_4,u_5)=\frac{4 \pi}{15}$ then $c(u_3,u_4)=\frac{\pi}{15}$ (see Figure 42);
if $c(u_5,u_6)=\frac{3 \pi}{10}$ then $c(u_6,u_7)=0$;
if $c(u_5,u_6)=\frac{4 \pi}{15}$ then $c(u_6,u_7)=\frac{\pi}{15}$;
and as in Case 5, $c(u_8,u_1)+c(u_1,u_2)=\frac{7 \pi}{15}$.
It follows that $c^{\ast}(\hat{\Delta}) \leq c(\hat{\Delta})+\frac{4 \pi}{3}$ so
if $\hat{\Delta}$ has at least four vertices of degree $>3$ then $c^{\ast}(\hat{\Delta}) \leq 0$.
Let $\hat{\Delta}$ have no vertices of degree $>3$.  Then $cv(\hat{\Delta})=(2,0,0,2,2,0,0,2)$ and
$c^{\ast}(\hat{\Delta}) \leq - \frac{2 \pi}{3} + \frac{4 \pi}{15} < 0$.
Let $\hat{\Delta}$ have exactly one vertex $u_i$ of degree $>3$. 
Note that if $d(u_1)=3$ then $c(u_8,u_1)=c(u_1,u_2)=\frac{\pi}{15}$; if $d(u_4)=3$ then $c(u_3,u_4)=0$ and $c(u_4,u_5)=\frac{\pi}{15}$;
and if $d(u_6)=3$ then $c(u_5,u_6)=\frac{\pi}{15}$ and $c(u_6,u_7)=0$.
If $i=1$ then $cv(\hat{\Delta}) = (h_1,0,0,2,2,0,0,h_2)$;
if $i=2$ then $cv(\hat{\Delta})=(2,2,0,2,2,0,0,2)$;
if $i=3$ then $cv(\hat{\Delta})=(2,2,4,2,2,0,0,2)$;
if $i=4$ then $cv(\hat{\Delta})=(2,0,e_1,e_2,2,0,0,2)$;
if $i=5$ then $cv(\hat{\Delta})=(2,0,0,9,9,0,0,2)$;
if $i=6$ then $cv(\hat{\Delta})=(2,0,0,2,e_1,e_2,0,2)$;
if $i=7$ then $cv(\hat{\Delta})=(2,0,0,2,2,4,2,2)$; and
if $i=8$ then $cv(\hat{\Delta})=(2,0,0,2,2,0,2,2)$.
Therefore $c^{\ast}(\hat{\Delta}) \leq - \frac{5 \pi}{6} + \frac{11 \pi}{15} < 0$.
Let $\hat{\Delta}$ have exactly two vertices $u_i,u_j$ of degree $>3$.
Then $c(\hat{\Delta}) \leq - \pi$.
If $d(u_1)=3$ then $cv(\hat{\Delta})=(2,2,e_1,e_2,e_1,e_2,2,2)$ and $c^{\ast}(\hat{\Delta}) \leq 0$
so it can be assumed that $i = 1$.
If $j=2$ then $cv(\hat{\Delta})=(h_1,2,0,2,2,0,0,h_2)$;
if $j=3$ then $cv(\hat{\Delta})=(h_1,2,4,2,2,0,0,h_2)$;
if $j=4$ then $cv(\hat{\Delta})=(h_1,0,e_1,e_2,2,0,0,h_2)$;
if $j=5$ then $cv(\hat{\Delta})=(h_1,0,0,2,2,0,0,h_2)$;
if $j=6$ then $cv(\hat{\Delta})=(h_1,0,0,2,e_1,e_2,0,h_2)$;
if $j=7$ then $cv(\hat{\Delta})=(h_1,0,0,2,2,4,2,h_2)$; and
if $j=8$ then $cv(\hat{\Delta})=(h_1,0,0,2,2,0,2,h_2)$.
Therefore $c^{\ast}(\hat{\Delta}) \leq - \pi + \frac{9 \pi}{10} < 0$.
Let $\hat{\Delta}$ have exactly three vertices of degree $>3$ so that $c(\hat{\Delta}) \leq - \frac{7 \pi}{6}$.
If $d(u_1)=3$ then $c^{\ast}(\hat{\Delta}) \leq - \frac{7 \pi}{6} + \pi$;
if $d(u_4)=3$ then $c^{\ast}(\hat{\Delta}) \leq - \frac{7 \pi}{6} + \frac{31 \pi}{30}$; and
if $d(u_6)=3$ then $c^{\ast}(\hat{\Delta}) \leq \frac{7 \pi}{6} + \frac{31 \pi}{30}$.
So it can be assumed that $d(u_1) > 3$, $d(u_4) > 3$ and
$d(u_6) > 3$ in which case $cv(\hat{\Delta})=(h_1,0,e_1,e_2,e_1,e_2,0,h_2)$.
If $d(u_1) > 4$ then $c^{\ast}(\hat{\Delta}) \leq - \frac{19}{15} \pi + \frac{18}{15} \pi < 0$;
whereas if $d(u_1) = 4$ then the fact that
$d(u_2)=d(u_8)=3$ means that $l(u_1)=bb x^{-1} y$ forces either $c(u_1,u_2)=0$ or $c(u_8,u_1)=0$ and
$c^{\ast}(\hat{\Delta}) \leq - \frac{7 \pi}{6} + \frac{31 \pi}{30} < 0$.

\textbf{Case 7}\quad
Let $\hat{\Delta}$ be given by Figure 44(vii) and note that  $\hat{\Delta}$ cannot be  $\hat{\Delta}_2$ of Figure 37(iv) or 38(iv).
If $c(u_1,u_2) > \frac{2 \pi}{15}$ then $d(u_1)=3$ and $c(u_8,u_1)=\frac{\pi}{30}$;
if $c(u_4,u_5) > \frac{2 \pi}{15}$ then $c(u_5,u_6)=0$;
if $c(u_5,u_6) > \frac{2 \pi}{15}$ then $d(u_5)=3$ and $c(u_4,u_5)=\frac{\pi}{30}$; and
if $c(u_8,u_1) > \frac{2 \pi}{15}$ then $c(u_1,u_2)=0$.  It follows that
$c(u_8,u_1) + c(u_1,u_2) \leq \frac{4 \pi}{15}$ and
$c(u_4,u_5) + c(u_5,u_6) \leq \frac{4 \pi}{15}$.
If $\hat{\Delta}$ has at least two vertices of degree $> 3$ then $c^{\ast} (\hat{\Delta}) \leq - \pi + \frac{14 \pi}{15} < 0$.
If $\hat{\Delta}$ has no vertices of degree $>3$ then we see (from Figure 44(vii)) that
$c^{\ast} (\hat{\Delta}) \leq - \frac{2 \pi}{3} + \frac{8 \pi}{15} < 0$.
Let $\hat{\Delta}$ have exactly one vertex of degree $>3$.
If $d(u_3)=d(u_4)=3$ then $c(u_3,u_4)=0$ and $c^{\ast} (\hat{\Delta}) \leq - \frac{5 \pi}{6} + \frac{4 \pi}{5} < 0$;
if $d(u_3) > 3$ then $cv(\hat{\Delta})=(6,2,4,0,6,0,0,0)$;
if $d(u_4) > 3$ then $cv(\hat{\Delta})=(6,0,4,b_1,b_2,0,0,0)$;
and it follows that $c^{\ast}(\hat{\Delta}) \leq -\frac{5 \pi}{6} + \frac{18 \pi}{30} < 0$.

\textbf{Case 8}\quad
Let $\hat{\Delta}$ be given by Figure 44(viii).  Then $cv(\hat{\Delta})=(4,4,6,2,2,4,9,2)$ so
if $\hat{\Delta}$ has at least three vertices of degree 2 then $c^{\ast}(\hat{\Delta}) \leq -\frac{7 \pi}{6} + \frac{11 \pi}{10} < 0$.
Note that
if $d(u_2)=3$ then $c(u_1,u_2)=c(u_2,u_3)=0$ and
if $d(u_7)=3$ then $c(u_6,u_7)=0$.
If $\hat{\Delta}$ has no vertices of degree $>3$ then $cv(\hat{\Delta})=(0,0,6,0,0,0,2,0)$ and
$c^{\ast}(\hat{\Delta}) \leq -\frac{2 \pi}{3} + \frac{4 \pi}{15} < 0$.
Let $\hat{\Delta}$ have exactly one vertex $u_i$ of degree $>3$.
If $i=1$ then $cv(\hat{\Delta})=(0,0,6,0,0,0,2,2)$;
if $i=2$ then $cv(\hat{\Delta})=(4,4,6,0,0,0,2,0)$;
if $i=3$ then $cv(\hat{\Delta})=(0,0,6,0,0,0,2,0)$;
if $i=4$ then $cv(\hat{\Delta})=(0,0,6,2,0,0,2,0)$;
if $i=5$ then $cv(\hat{\Delta})=(0,0,6,2,2,0,2,0)$;
if $i=6$ then $cv(\hat{\Delta})=(0,0,6,0,2,0,2,0)$;
if $i=7$ then $cv(\hat{\Delta})=(0,0,0,6,0,0,4,9,0)$; and
if $i=8$ then $cv(\hat{\Delta})=(0,0,6,0,0,0,9,2)$.  Therefore $c^{\ast}(\hat{\Delta}) \leq - \frac{5 \pi}{6} + \frac{19 \pi}{30} < 0$.
Let $\hat{\Delta}$ have exactly two vertices of degree $>3$.
If $d(u_2)=3$ or $d(u_7) = 3$ then $c^{\ast}(\hat{\Delta}) \leq - \pi + \frac{29 \pi}{30} < 0$ so assume that $d(u_2) > 3$ and $d(u_7) > 3$.
Then $cv(\hat{\Delta})=(4,4,6,0,0,4,9,0)$ and $c^{\ast}(\hat{\Delta}) \leq - \pi + \frac{9 \pi}{10} < 0$.

\textbf{Case 9}\quad
Let $\hat{\Delta}$ be given by Figure 44(ix).
If $c(u_4,u_5) > \frac{2 \pi}{15}$ then $d(u_4)=4$ and $c(u_3,u_4)=0$; and
if $c(u_6,u_7) > \frac{2 \pi}{15}$ then $(d(u_6),d(u_7))=(4,4)$ and $c(u_7,u_8)=0$.
It follows that if at least three of the $u_i$ have degree $\geq 4$ then $c^{\ast} (\hat{\Delta}) \leq - \frac{7 \pi}{6} + \frac{7 \pi}{6}=0$, so assume otherwise.
If $\hat{\Delta}$ has no vertices of degree $>3$ then we see (from Figure 44(ix) and the fact that $\hat{\Delta}$ cannot then be $\hat{\Delta}_2$ of Figure 37(iii)) that $cv(\hat{\Delta})=(0,6,4,0,0,0,4,0)$ and
$c^{\ast}(\hat{\Delta}) \leq - \frac{2 \pi}{3} + \frac{7 \pi}{30} < 0$.
Let $\hat{\Delta}$ have exactly one vertex $u_i$ of degree $>3$.
Then the following holds.
If $i = 1$ then $cv(\hat{\Delta})=(2,6,4,0,0,0,5,3)$;
if $i=2$ then $cv(\hat{\Delta})=(2,6,4,0,0,0,4,0)$;
if $i=3$ then $cv(\hat{\Delta})=(0,6,6,0,0,0,4,0)$;
if $i=4$ then $cv(\hat{\Delta})=(0,6,d_1,d_2,0,0,4,0)$;
if $i=5$ then $cv(\hat{\Delta})=(0,6,4,4,4,0,4,0)$;
if $i=6$ then $cv(\hat{\Delta})=(0,6,4,0,4,4,4,0)$;
if $i=7$ then $cv(\hat{\Delta})=(0,6,4,0,0,4,6,0)$; and
if $i=8$ then $cv(\hat{\Delta})=(0,6,4,0,0,0,6,3)$.
It follows that $c^{\ast}(\hat{\Delta}) \leq -\frac{5 \pi}{6}+\frac{11 \pi}{15} < 0$.
Let $\hat{\Delta}$ have exactly two vertices $u_i,u_j$ of degree $>3$.
If $d(u_4)=d(u_5)=d(u_6)=3$ or $d(u_5)=d(u_6)=d(u_7)=3$ then $c^{\ast}(\hat{\Delta}) \leq - \pi + \pi = 0$.
This leaves 14 out of 28 cases to be considered.
If $(i,j)=(1,5)$ then $cv(\hat{\Delta}) = (2,6,4,4,4,0,5,3)$;
if $(i,j)=(1,6)$ then $cv(\hat{\Delta})=(2,6,4,0,4,4,5,3)$;
if $(i,j)=(2,5)$ then $cv(\hat{\Delta})=(2,6,4,4,4,0,4,0)$;
if $(i,j)=(2,6)$ then $cv(\hat{\Delta})=(2,6,4,0,4,4,4,0)$;
if $(i,j)=(3,5)$ then $cv(\hat{\Delta})=(0,6,6,4,4,0,4,0)$;
if $(i,j)=(3,6)$ then $cv(\hat{\Delta})=(0,6,6,0,4,4,4,0)$;
if $(i,j)=(4,5)$ then $cv(\hat{\Delta})=(0,6,d_1,d_2,4,0,4,0)$;
if $(i,j)=(4,6)$ then $cv(\hat{\Delta})=(0,6,d_1,d_2,4,4,4,0)$;
if $(i,j)=(4,7)$ then $cv(\hat{\Delta})=(0,6,d_1,d_2,0,4,4,0)$;
if $(i,j)=(5,6)$ then $cv(\hat{\Delta})=(0,6,4,4,4,4,4,0)$;
if $(i,j)=(5,7)$ then $cv(\hat{\Delta})=(0,6,4,4,4,4,4,0)$;
if $(i,j)=(5,8)$ then $cv(\hat{\Delta})=(0,6,4,4,4,0,4,3)$;
if $(i,j)=(6,7)$ then $cv(\hat{\Delta})=(0,6,4,0,4,d_1,d_2,0)$; and
if $(i,j)=(6,8)$ then $cv(\hat{\Delta})=(0,6,4,0,4,4,4,3)$.
It follows that $c^{\ast}(\hat{\Delta}) \leq = - \pi + \frac{14 \pi}{15} < 0$.

\textbf{Case 10}\quad
Let $\hat{\Delta}$ be given by Figure 44(x).
If $c(u_1,u_2) > \frac{2 \pi}{15}$ then $(d(u_1),d(u_2))=(4,4)$ and $c(u_8,u_1)=0$;
if $c(u_7,u_8) > \frac{2 \pi}{15}$ then $(d(u_7),d(u_8))=(4,3)$ and $c(u_6,u_7)=0$; and
if $c(u_6,u_7) > \frac{2 \pi}{15}$ then $d(u_7)=3$ forcing $c(u_7,u_8)=0$.
It follows that $c(u_8,u_1) + c(u_1,u_2) \leq \frac{7 \pi}{30}$; and $c(u_6,u_7)+c(u_7,u_8) \leq \frac{4 \pi}{15}$.
If $\hat{\Delta}$ has at least three vertices of degree $>3$ then $c^{\ast}(\hat{\Delta}) \leq - \frac{7 \pi}{6} + \frac{11 \pi}{10} < 0$.
If $\hat{\Delta}$ has no vertices of degree $>3$ then we see (from Figure 44(x)) 
that $cv(\hat{\Delta})=(0,0,6,6,0,6,0,0)$ and
$c^{\ast}(\hat{\Delta}) \leq -\frac{2 \pi}{3} + \frac{18 \pi}{30} < 0$.
Let $\hat{\Delta}$ have exactly one vertex $u_i$ of degree $>3$. 
Then the following holds.
If $i=1$ then $d(u_2)=3$ and so $cv(\hat{\Delta})=(0,0,6,6,0,6,0,3)$;
if $i=2$ then $d(u_1)=3$ and so $cv(\hat{\Delta})=(x_1,y_1,6,6,0,6,0,0)$ where, by the remark preceeding this lemma, $x_1+y_1=4$;
if $i=3$ then $cv(\hat{\Delta})=(0,4,6,6,0,6,0,0)$;
if $i=4$ then $cv(\hat{\Delta})=(0,0,6,6,0,6,0,0)$;
if $i=5$ then $cv(\hat{\Delta})=(0,0,6,6,2,6,0,0)$;
if $i=6$ then $cv(\hat{\Delta})=(0,0,6,6,2,6,0,0)$;
if $i=7$ then $cv(\hat{\Delta})=(0,0,6,6,0,b_1,b_2,0)$; and
if $i=8$ then $d(u_7)=3$ so $cv(\hat{\Delta})=(0,0,6,6,0,6,0,3)$.  Therefore
$c^{\ast}(\hat{\Delta}) \leq - \frac{5 \pi}{6} + \frac{11 \pi}{15} < 0$.
Let $\hat{\Delta}$ have exactly two vertices of degree $>3$.
If $d(u_1)=3$ or $d(u_2)=3$ then $cv(\hat{\Delta})=(x_2,y_2,6,6,2,b_1,b_2,3)$ where, again by the above remark, $x_2+y_2=5$ and 
$c^{\ast}(\hat{\Delta}) \leq - \pi + \pi = 0$.  On the other hand
if $d(u_1) > 3$ and $d(u_2) > 3$ then $cv(\hat{\Delta})=(a_1,4,6,6,0,6,0,a_2)$
and $c^{\ast}(\hat{\Delta}) \leq - \pi + \frac{29 \pi}{30} < 0$.

\textbf{Case 11}\quad
Let $\hat{\Delta}$ be given by Figure 44(xi).
If $c(u_1,u_2) > \frac{2 \pi}{15}$ then $d(u_1)=3$ and $c(u_8,u_1)=0$;
if $c(u_8,u_1) > \frac{2 \pi}{15}$ then $c(u_1,u_2)=0$; and
if $c(u_6,u_7) > \frac{2 \pi}{15}$ then $c(u_7,u_8)=0$. Therefore $cv(\hat{\Delta})=(b_1,2,6,6,4,a_1,a_2,b_2)$.
It follows that $c^{\ast}(\hat{\Delta}) \leq c(\Delta) + \frac{11 \pi}{10}$ and so
if $\hat{\Delta}$ has at least three vertices of degree $>3$ then $c^{\ast}(\hat{\Delta}) < 0$.
If $\hat{\Delta}$ has no vertices of degree $>3$ then we see (from Figure 44(xi)) that
$c^{\ast}(\hat{\Delta}) \leq - \frac{2 \pi}{3} + \frac{3 \pi}{5} < 0$.  Let $\hat{\Delta}$ have exactly one vertex $u_i$ of degree $>3$.
If $d(u_6)=d(u_7)=d(u_8)=3$ then $c(u_5,u_6)=c(u_6,u_7)=c(u_7,u_8)=0$;
if $i=6$ then $cv(\hat{\Delta})=(6,0,6,6,4,1,0,0)$;
if $i=7$ then $cv(\hat{\Delta})=(6,0,6,6,0,0,3,0)$; and
if $i=8$ then $cv(\hat{\Delta})=(6,0,6,6,0,0,3,0)$.  It follows that $c^{\ast}(\hat{\Delta}) \leq -\frac{5 \pi}{6}+\frac{11 \pi}{15} < 0$.
Let $\hat{\Delta}$ have exactly two vertices of degree $>3$.
If $d(u_6)=3$ then $c(u_5,u_6)=0$ and $cv(\hat{\Delta})=(b_1,2,6,6,0,a_1,a_2,b_2)$;
if $d(u_7)=3$ then $cv(\hat{\Delta})=(b_1,2,6,6,x_2,y_2,3,b_2)$ where again $x_2+y_2=5$;
if $d(u_6)>3$ and $d(u_7)>3$ then $cv(\hat{\Delta})=(6,0,6,6,4,a_1,a_2,0)$.
It follows that $c^{\ast}(\hat{\Delta}) \leq -\pi + \pi = 0$.

\textbf{Case 12}\quad
Finally let $\hat{\Delta}$ be the region of Figure 44(xii). Suppose that $\hat{\Delta}$ has at least one vertex of degree $>4$.
Using a similar analysis as done for Case 1, it follows that
$c^{\ast} (\hat{\Delta}) \leq - \frac{38 \pi}{30} + \frac{36 \pi}{30} < 0$.  Indeed the maximum $\frac{36 \pi}{30}$ can only be obtained when
$cv(\hat{\Delta})=(0,6,0,h_1,h_2,6,0,6,4)$.  
Suppose that $\hat{\Delta}$ has no vertices of degree $>4$ and at least one vertex of degree 4.
Then noting again that $\hat{\Delta}$ is not given by Figure 40(xvii), we see from Figure 40(xiv) that  $c(u_i,u_j) = \frac{2 \pi}{15}$ for $(i,j) \in \{ (7,8),(8,9),(9,1),(1,2),(2,3),(3,4) \}$.
It follows that $c^{\ast} (\hat{\Delta}) \leq - \frac{35 \pi}{30} + \frac{32 \pi}{30} < 0$, the maximum 
$\frac{32 \pi}{30}$ being obtained when
$cv(\hat{\Delta}) = (0,4,0,h_1,h_2,6,0,4,4)$.
But if $\hat{\Delta}$ has no vertices of degree $>3$ then $c(u_4,u_5)=0$, $c(u_5,u_6)=\frac{\pi}{15}$ and,as in Case 1, for example either $c(u_1,u_2)=0$ or $c(u_1,u_2)=\frac{2 \pi}{15}$ and $c(u_2,u_3)=0$.
It follows that $cv(\hat{\Delta})=(4,0,4,0,2,6,0,4,0)$ and $c^{\ast}(\hat{\Delta}) \leq - \pi + \frac{2 \pi}{3} < 0$. This completes the proof. $\Box$


\section{Regions of Type $\boldsymbol{\mathcal{B}}$}

Let $\hat{\Delta}$ be a type $\mathcal{B}$ region as defined at the start of Section 9.  Therefore $\hat{\Delta}$ is given by Figure 13(i),14(i),31 or 32(i),(ii),(iii) or (v); in particular $d(\hat{\Delta})
\geq 8$.
A \textit{$b$-segment} of $\hat{\Delta}$ of length $k$ is a sequence of edges $e_1,\ldots,e_k$ of $\hat{\Delta}$ maximal with
respect to each vertex having degree 3 with vertex label $a(a \lambda )(b^{-1} \mu)=axy^{-1}$ and which (up to inversion) contribute one of
four possible alternating sequences to the corner labelling of $\hat{\Delta}$, namely:
$x^{-1},y^{-1}, \ldots, x^{-1}, y^{-1}$;
$x^{-1},y^{-1},\ldots,y^{-1},x^{-1}$;
$y^{-1},x^{-1}$;
$y^{-1},x^{-1},\ldots,y^{-1},x^{-1}$;
$y^{-1},x^{-1},\ldots,x^{-1},y^{-1}$.
An example showing the first sequence is given in Figure 45(i) and so maximal in this case means that either $d(u_0)>3$ or $d(u_0)=3$ but does not
extend the sequence to $\bar{y},\bar{x},\bar{y},\ldots,\bar{x},\bar{y}$; and that either $d(u_{k+2}) > 3$ or $d(u_{k+2})=3$ but does not extend the sequence to
$\bar{x},\bar{y},\ldots,\bar{x},\bar{y},\bar{x}$.  Since $\hat{\Delta}$ is of type $\mathcal{B}$, it must contain at least one $b$-segment in which
at least one of the regions $\Delta_i$ ($1 \leq i \leq k$) is given by the region $\Delta$ in Figure 13(i),14(i) and we will from now on call
such a region $\Delta_i$ a \textit{$b$-region}. Therefore a $b$-region contributes at most $\frac{\pi}{3}$ to $\hat{\Delta}$.
(If $\Delta_i$ is not a $b$-region then, as shown in Figure 35, it contributes at most $\frac{\pi}{5}$ to $\hat{\Delta}$.)

\newpage
\begin{figure}
\begin{center}
\psfig{file=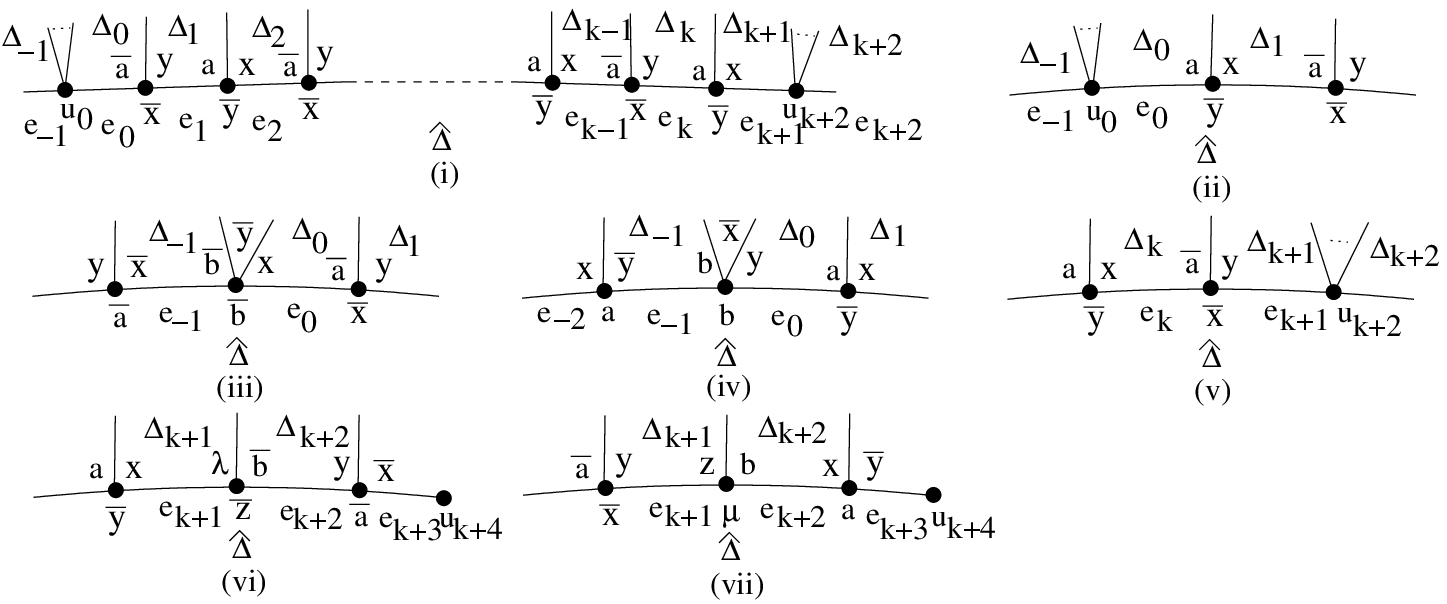}
\end{center}
\caption{}
\end{figure}

\noindent \textit{The absorption rules for edges and vertices described in Section 10 apply also to $\hat{\Delta}$.} In Figure 31 $\hat{\Delta}$ receives $\frac{\pi}{5},\frac{2 \pi}{15}$ from 
$\Delta_1,\Delta_2$
so the
vertex of degree 4 with label $b^{-1} b^{-1} y^{-1} x$ is used to absorb $\frac{\pi}{15}$; and in Figure 32(i)-(ii) $\hat{\Delta}$ receives
$\frac{\pi}{5}$
across an edge, $e$ say, but checking Figures 36-38 shows that $\hat{\Delta}$ receives no curvature from $\hat{\Delta}_1$ across the
neighbouring
edge which is used to absorb $\frac{\pi}{15}$ noting from Figure 32(i)-(ii) that this is all the curvature that this edge will absorb (relative to
curvature transferred to $\hat{\Delta}$).

It follows from the above paragraph and as in the proof of Lemma 10.1 that if the $b$-segments containing at least one $b$-region of $\hat{\Delta}$ contribute a total of $n_1$ edges to $\hat{\Delta}$ then putting $n=n_1+n_2$,
\[
c^{\ast}(\hat{\Delta}) \leq (2-(n_1+n_2)) \pi + 2(n_1 + n_2) \frac{\pi}{3} + n_1 \frac{\pi}{3} + n_2 \frac{2 \pi}{15} = \pi \left( 2 - \frac{n_2}{5} \right).\qquad (\dag)
\]
Therefore if $n_2 \geq 10$ then $c^{\ast}(\hat{\Delta}) \leq 0$.  The next result improves this bound slightly.

\begin{lemma}
\textit{If $n_2 \geq 9$ and $\hat{\Delta}$ is not given by Figure 46 (in which the $b$-segment contains at least one $b$-region) then}
$c^{\ast}(\hat{\Delta}) \leq 0$.
\end{lemma}

\textit{Proof}. We will show that the existence of a $b$-segment in which at least one $\Delta_i$ ($1 \leq i \leq k$) is a $b$-region allows us to decrease the upper bound $(\dag)$ for $c^{\ast}(\hat{\Delta})$ given above.
First consider the region $\Delta_0$ of Figure 45(i) or (ii).
In each case if $\Delta_1$ is not a $b$-region then $\hat{\Delta}$ receives at most $\frac{\pi}{5}$ from $\Delta_1$ and the upper bound for
$c^{\ast}(\hat{\Delta})$ is reduced by at least $\frac{\pi}{3}-\frac{\pi}{5}=\frac{2 \pi}{15}$, so assume the $\Delta_1$ is a $b$-region.
In particular,according to the rules in Section 10 and at the start of this section, $e_0$ absorbs no positive curvature from $\Delta_1$.
Let $d(u_0) \geq 5$ and so $u_0$ can absorb at least $\frac{2 \pi}{3} - \frac{2 \pi}{5} = \frac{4 \pi}{15}$.  Since $\hat{\Delta}$ then receives at
most $\frac{\pi}{15}$ from $\Delta_0$ (see Figure 35(ii)) and since the maximum amount any vertex absorbs is $\frac{\pi}{6}$, in particular $u_0$ from $\Delta_{-1}$,
$u_0$ can absorb the $\frac{\pi}{15}$ crossing $e_0$ and so $n_2$ in $(\dag)$ is reduced by 1, that is, $c^{\ast}(\hat{\Delta})$ is reduced  

\newpage
\begin{figure}
\begin{center}
\psfig{file=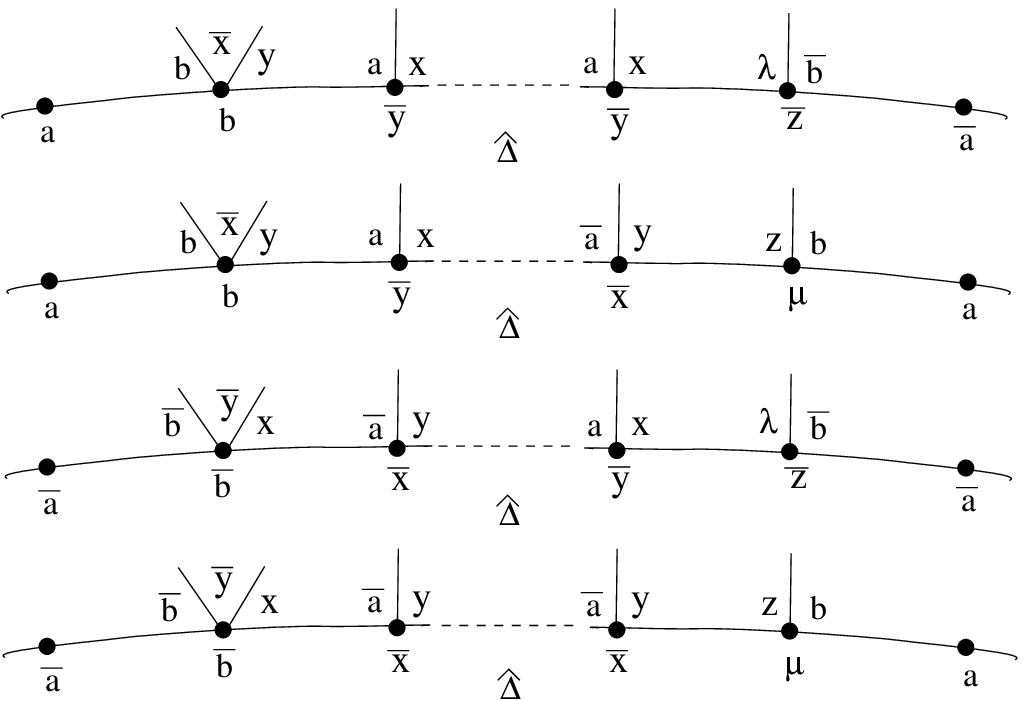}
\end{center}
\caption{}
\end{figure}   

\noindent by at least $\frac{2 \pi}{15}$.  Let $d(u_0)=4$ and so $u_0$ can absorb
$\frac{2 \pi}{3} - \frac{\pi}{2} = \frac{\pi}{6}$.  If the total curvature $\hat{\Delta}$ receives across $e_0$ and $e_{-1}$ is at most
$\frac{3 \pi}{10}$ then $c^{\ast}(\hat{\Delta})$ is reduced by at least $\frac{2 \pi}{15}$, so assume otherwise.
In particular $\hat{\Delta}$ must receive curvature from $\Delta_0$ which forces $l(u_0)$ to be as shown in Figure 45(iii) and (iv) and so (see
Figure 35(i)) $\hat{\Delta}$ receives at most $\frac{2 \pi}{15}$ from $\Delta_0$.
To exceed a total of $\frac{3 \pi}{10}$, therefore, it follows that $\hat{\Delta}$ must receive at least $\frac{\pi}{5}$ across $e_{-1}$ and so (see
Figure 40) $l(u_{-1})$ must be as shown in Figures 45(iii) and (iv) and in these figures the maximum combination $\hat{\Delta}$ can receive across
$e_{-1}, e_0$
is $\frac{7 \pi}{30}$, $\frac{2 \pi}{15}$ (see Figure 42), therefore $c^{\ast} (\hat{\Delta})$ is reduced by at least $\frac{\pi}{15}$.
Let $d(u_0)=3$.  Note that we use the fact that $l(u_0) \neq axy^{-1}$ in Figure 45(i) or (ii) for otherwise the $b$-segment would be extended, a contradiction.
Given this, $l(u_0) = b \mu z$ forces $d(\Delta_0) \geq 6$ and $d(\Delta_{-1}) \geq 
6$ and checking Figures 36-38 shows that
$\hat{\Delta}$ does
not receive curvature across $e_0$ and at most $\frac{2 \pi}{15}$ across $e_{-1}$ so $c^{\ast}(\hat{\Delta})$ is reduced by at least $\frac{2 \pi}{15}$.

Now consider the region $\Delta_{k+1}$ of Figure 45(i) and (v).
Again if $\Delta_k$ is not a $b$-region then $c^{\ast}(\hat{\Delta})$ is reduced by $\frac{2 \pi}{15}$ so assume otherwise.
In particular $e_{k+1}$ absorbs no positive curvature from $\Delta_k$.
Moreover, if $\Delta_k$ is given by $\Delta_1$ of Figure 32(iii) or (v) (Configuration E, F) then $c^{\ast}(\hat{\Delta})$ is again reduced by
$\frac{\pi}{5}$, so assume otherwise, in particular $u_{k+2}$ is not given by the corresponding vertex of $\Delta_2$ of Figure 32(iii) or (v).
Let $d(u_{k+2}) \geq 5$ and so $u_{k+2}$ can absorb $\frac{4 \pi}{15}$.
Since $\hat{\Delta}$ then receives at most $\frac{\pi}{15}$ from $\Delta_{k+1}$ (see Figure 35(ii)) and since the maximum amount $u_{k+1}$ absorbs from $\Delta_{k+2}$ is
$\frac{\pi}{6}$, $u_{k+2}$ can absorb the $\frac{\pi}{15}$ crossing $e_{k+1}$ and so $c^{\ast}(\hat{\Delta})$ is reduced by at least $\frac{2 \pi}{15}$.
Let $d(u_{k+2})=4$ and so $u_{k+2}$ can absorb $\frac{\pi}{6}$.  If $\hat{\Delta}$ does not receive curvature from $\Delta_{k+1}$ then
$c^{\ast}(\hat{\Delta})$ is reduced by $\frac{2 \pi}{15}$; otherwise checking possible vertex labels for $u_{k+2}$ shows that $l(u_{k+2})= a a z \mu$ and $\hat{\Delta}$ receives at most $\frac{7 \pi}{30}$  across $e_{k+1}$ and 0 across $e_{k+2}$, so $c^{\ast}(\hat{\Delta})$ is reduced by $\frac{2 \pi}{15}$. Let $d(u_{k+2})=3$
and so using the maximality of the $b$-segment and the fact that $u_{k+2}$ is not given by Figure 32(iii) or (v) it follows that $l(u_{k+2})$ must
be
as shown in Figures 45(vi) and (vii).  Then $d(\Delta_{k+1}) \geq 6$ and checking Figures 36-38 shows that $\hat{\Delta}$ does not receive
curvature
from $\Delta_{k+1}$.  It follows that $c^{\ast}(\hat{\Delta})$ is reduced by $\frac{2 \pi}{15}$ except possibly when $d(u_{k+3})=3$ and
$\hat{\Delta}$ receives $\frac{\pi}{6}$ or $\frac{\pi}{5}$ from $\Delta_{k+2}$ (see Figure 40).
There are four cases.
Two (see Figures 40(i), (ii), (vi) and (xv))
are given by Figure 45(vi) and (vii) where $\hat{\Delta}$ can receive $\frac{\pi}{5}$ from $\Delta_{k+2}$ and $c^{\ast}(\hat{\Delta})$ is reduced by
$\frac{\pi}{15}$; and two (see Figures 40(xiii) and (xiv)) are given by Figures 37(iii) and 38(iii) in which the region $\hat{\Delta}_2$,
$\hat{\Delta}_1$, $\Delta_2$ (respectively) plays the role of the region $\hat{\Delta}$, $\Delta_{k+2}$, $\Delta_{k+3}$ (respectively) which
implies $l(u_{k+4})=bbx^{-1}y$ so,in particular, $e_{k+3}$ does not absorb curvature from $e_{k+4}$ (relative to $\hat{\Delta}$).  In each of these last two cases $\hat{\Delta}$ receives 
$\frac{\pi}{6}$
from $\Delta_{k+2}$ and $\frac{\pi}{10}$ from $\Delta_{k+3}$, and since $\hat{\Delta}$ does not receive curvature from $\Delta_{k+1}$
it follows that
$c^{\ast}(\hat{\Delta})$ is reduced by $\frac{2 \pi}{15}$.

It follows from the above that if the $b$-segment of Figure 45(i) is not given by Figure 46 then there is a reduction of at least $\frac{\pi}{15} + \frac{2
\pi}{15} = \frac{3 \pi}{15}$ to
$c^{\ast}(\hat{\Delta})$ (if $e_{k+2}=e_0$ the reduction is also $\frac{3 \pi}{35}$)
therefore $c^{\ast}(\hat{\Delta}) \leq \pi \left( 2 - \frac{n_2}{5} \right) - \frac{3 \pi}{15}$ and so $n_2 \geq 9 \Rightarrow c^{\ast}(\hat{\Delta}) \leq 0$. $\Box$

\begin{figure}
\begin{center}
\psfig{file=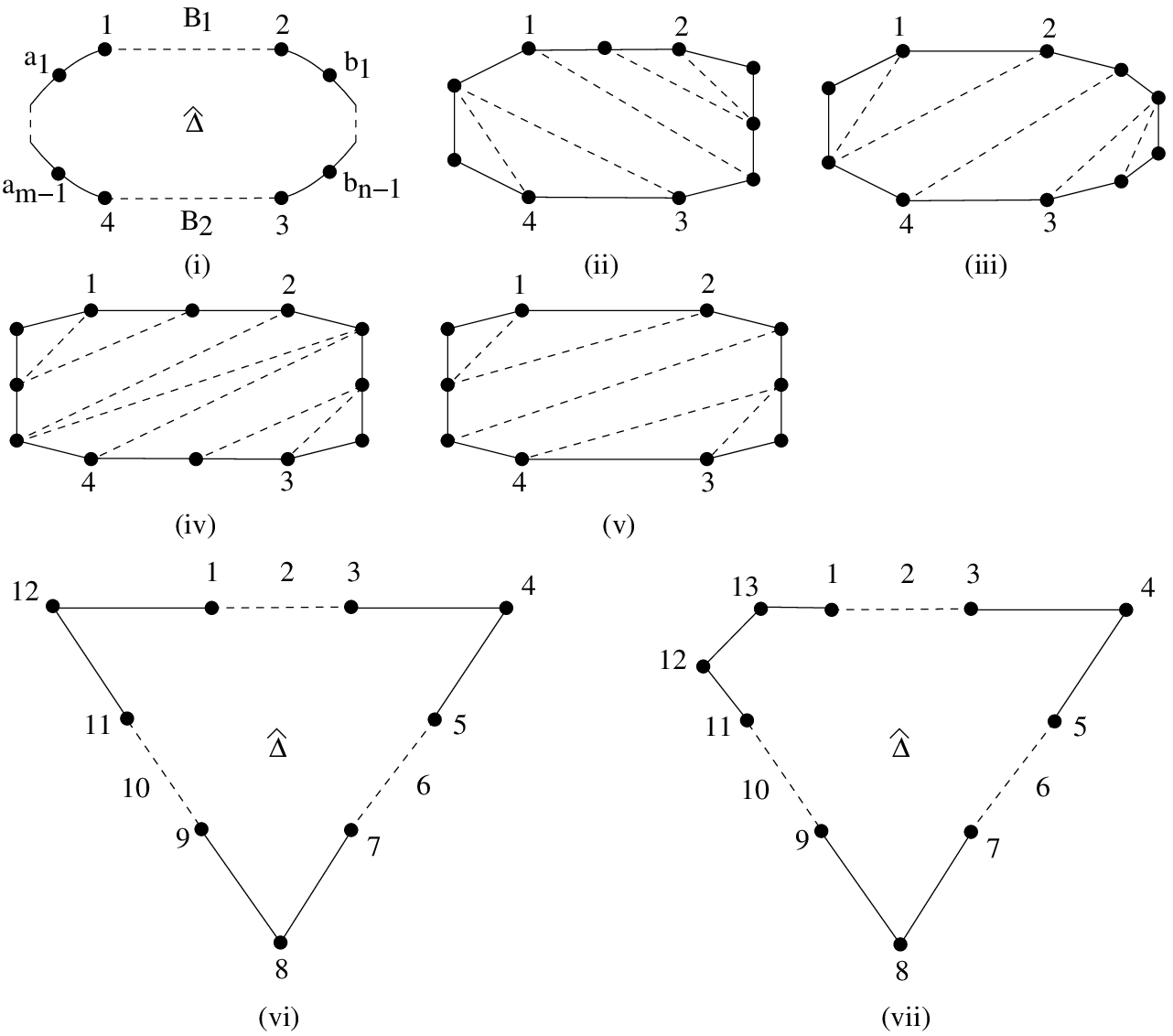}
\end{center}
\caption{}
\end{figure}

\newpage
\begin{figure}
\begin{center}
\psfig{file=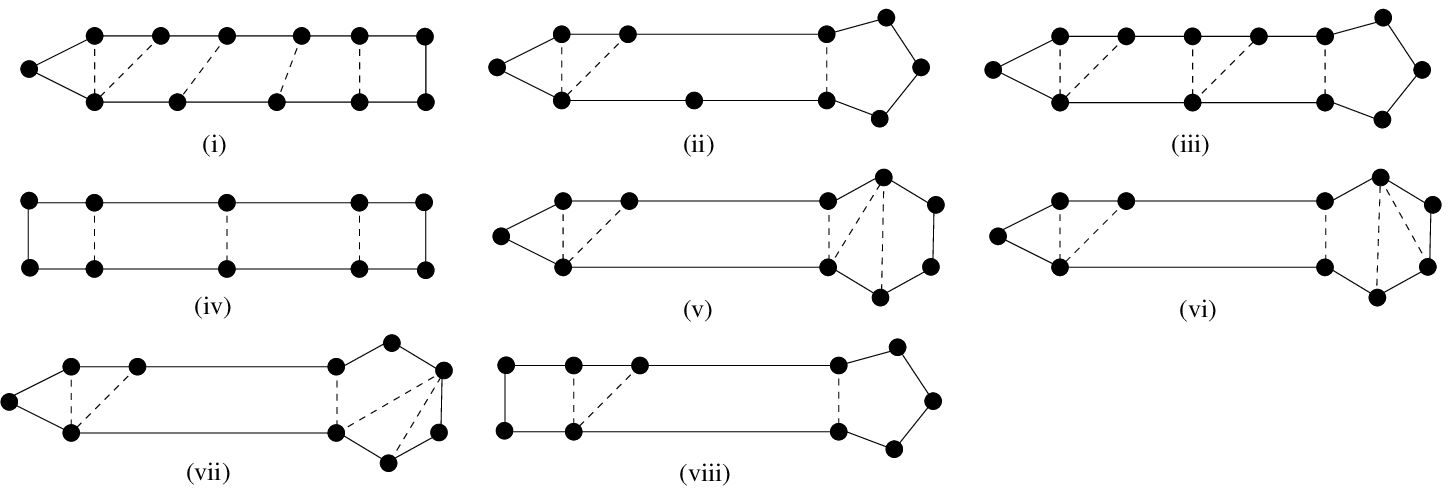}
\end{center}
\caption{}
\end{figure}

\begin{lemma}
\textit{Let $\hat{\Delta}$ be a type $\mathcal{B}$ region such that $d(\hat{\Delta}) \geq 10$.}
\begin{enumerate}
\item[(i)]
\textit{If $\hat{\Delta}$ has exactly three $b$-segments that contain a $b$-region then $n_2 \geq 8$.}
\end{enumerate}
\textit{Assume now that $\hat{\Delta}$ has exactly two $b$-segments $B_1$ and $B_2$ that contain a $b$-region as shown in Figure 47(i) and assume
that
$(m,n) \in \{ (2,j)~(2 \leq j \leq 6), (3,3), (3,4), (3,5), (4,4)\}$ where $m,n$ are given by Figure 47(i).}
\begin{enumerate}
\item[(ii)]
\textit{$\hat{\Delta}$ must contain a shadow edge with an endpoint in $B_1$ and the other endpoint in $B_2$ except when $\hat{\Delta}$ is given by
Figure 47(ii)-(v).}
\item[(iii)]
\textit{If $v \in \hat{\Delta}$ is a vertex of $B_1$ or $B_2$ and $(m,n) \neq (2,6)$ then $i \, \mathrm{deg} (v)=1$ where $i \deg (v)$ denotes the number of shadow edges in $\hat{\Delta}$ incident at $v$.}

\begin{enumerate}
\item[(iv)]
\textit{If $(m,n) \in \{ (3,3), (3,4), (3,5), (4,4) \}$ and $\hat{\Delta}$ is not given by Figure 47(ii)-(v) there must be a shadow edge in
$\hat{\Delta}$ either from 1 to $B_2$ or from 4 to $B_1$; and there must be a shadow edge in $\hat{\Delta}$ either from 2 to $B_2$ or from 3 to $B_1$.}
\end{enumerate}
\textit{Finally assume that $\hat{\Delta}$ has exactly one $b$-segment containing a $b$-region.}
\item[(v)]
\textit{If $n_2 \leq 8$ then $\hat{\Delta}$ is given by Figure 48.}
\item[(vi)]
\textit{If $n_2=9$ and $\hat{\Delta}$ is given by Figure 46 then $\hat{\Delta}$ is one of the regions of Figure 49.}
\end{enumerate}
\end{lemma}

\textit{Proof}.
The proof is elementary but lengthy so we have omitted it.
Full details can be found in the Appendix.
As an illustration we give part of the proof of (i).

Let $\hat{\Delta}$ have exactly three $b$-segments and supppose by way of contradiction that $n_2 \leq 7$. Since there are at least two edges between any two $b$-segments
it follows that $\hat{\Delta}$ is given by Figure 47(vi) ($n_2=6$) or 47(vii) ($n_2=7$) in which 2,6,10 refer to the (possibly empty) set 

\newpage
\begin{figure}  
\begin{center}
\psfig{file=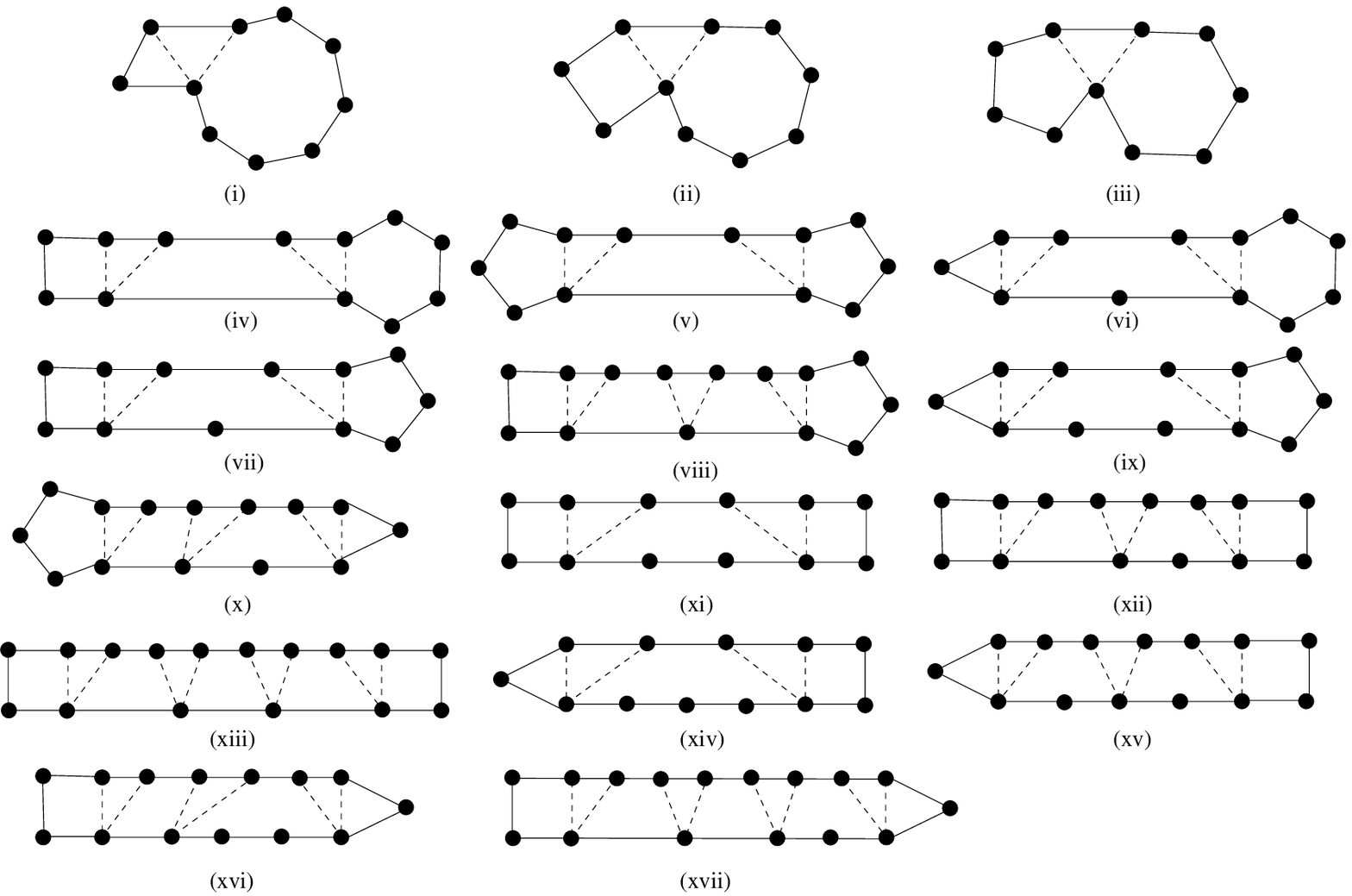}
\end{center}
\caption{}
\end{figure}

\noindent of vertices between vertices 1 and 3, 5 and 7, 9 and 11.
We will consider the case $n_2=6$ only.

We remark here that if the corner label at a vertex $v$ of $\hat{\Delta}$ is $x$ or $y$ then it follows from equations (3.1) in Section 3 that there must be an odd number of shadow edges in $\hat{\Delta}$
incident at $v$ and it is clear that there are no shadow edges in $\hat{\Delta}$ connecting two vertices in the same $b$-segment.   
We write $(ab)$ to indicate there is a shadow edge between vertices $a$ and $b$ with the understanding that if $a=2$, for example, we mean a vertex belonging to $a$.

Consider Figure 47(vi).
By the previous remark the number of $(ab)$ involving each of 1, 3, 5, 7, 9 and 11 must be odd. 
It also follows that if $\{ a,b \} \subseteq \{ 12,1,2,3,4 \}$ or $\{ 4,5,6,7,8 \}$ or $\{ 8,9,10,11,12 \}$ then $(ab)$ does not occur.
Moreover (18) forces (19), ($1\,11$) and this in turn forces a basic labelling contradiction (see Section 3), termed LAC.
It follows that the only pairs involving 4, 8 or 12 are ($4\,10$), (28) and ($6\,12$).
First assume that none of (35), (79) or ($1\,11$) occur. 
Then since (15), (16) and (17) each forces (35), and (19), ($1\,10$) each forces ($1\,11$), we get a contradiction.
Assume exactly one of (35), (79), ($1\,11$) occurs -- without any loss (79).
Then again (15), (16) and (17) each force (35), and (19) and ($1\,10$) each force ($1\,11$), a contradiction.
Assume exactly two of (35), (75), ($1\,11$) occur -- without any loss (35) and (79).
Then (19) and ($1\,10$) each forces ($1\,11$), a contradiction; and (16) and (17) each forces a basic length contradiction at (35) (a shadow edge of length $n-1$) or forces either the pair (52), (52) or (52), (51) or $(36),(36)$ or 
$(36),(37)$ yielding LAC.
This leaves (15).
Since the number of $(ab)$ involving 5 must be odd at least one of (59), ($5\,10$) or ($5\,11$) occurs.
But (59) forces ($11\,5$) and ($5\,10$); ($5\,10$) forces ($11\,5$) and another ($5\,10$); and ($5\,11$) forces either a length contradiction at (79) or forces $(95), (96)$ or $(96),(96)$ or $(7\,10),(7\,10)$ or $(7\,10),(7\,11)$ 
yielding LAC in all cases.
Finally assume that $(1\,11)$, (35) and (79) occur.
Since the length of each is $n-1$ we must have more pairs otherwise there is a length contradiction.
Assume without any loss that 1 is involved in further pairs.
Since (16) and (17) each forces either $(36),(36)$ or $(36)(37)$ or $(52),(52)$ or $(52)(51)$ yielding LAC it follows that at least two of (15), (19) and $(1\,10)$ occur. However
$(19),(1\,10)$ and 
$(1\,10),(1\,10)$ yield LAC and $(15),(19)$ forces (59) and LAC.
This leaves $(15),(1\,10)$ together with at least one of (25), (59), $(5\,10)$.
But (25) yields LAC; (59) forces (19) or (69) and LAC; and finally $(5\,10)$ forces either a length contradiction or one of $(7\,10)$, $(7\,10)$ or (59)(69) or (69)(69) and LAC,
our final contradiction.$\Box$

\medskip

\textbf{Notation} \quad Throughout the following proofs we will use non-negative integers\newline
$a_1,a_2,b_1,b_2,c_1,c_2,d_1,d_2,e_1,e_2,f_1,f_2$ where:
$a_1+a_2=7$; $b_1+b_2=8$; $c_1+c_2=9$; $d_1+d_2=10$; $e_1+e_2=11$; and $f_1 + f_2 = 12$.

\begin{proposition}
\textit{Let $\hat{\Delta}$ be a type $\mathcal{B}$ region.
If $d(\hat{\Delta}) < 10$ then $c^{\ast} (\hat{\Delta}) \leq 0$.}
\end{proposition}

\textit{Proof}.
If $d(\hat{\Delta}) < 10$ then by Lemma 10.3 $\hat{\Delta}$ is given by Figure 44(viii), (x) or (xi).

\textbf{Case 1}\quad
Let $\hat{\Delta}$ be given by Figure 44(viii) in which it is now assumed that $d(u_3)=d(u_4)=3$.  
Observe from Figures 41 and 42 that $c(u_6,u_7) + c(u_7,u_8) \leq \frac{11 \pi}{30}$.
Therefore $cv(\hat{\Delta})=(4,4,10,2,2,e_1,e_2,2)$ so
$c^{\ast}(\hat{\Delta}) \leq c(\hat{\Delta}) + \frac{7 \pi}{6}$ and if $\hat{\Delta}$ has at least three vertices of degree $\geq 4$ then $c^{\ast}(\hat{\Delta}) \leq 0$.
If $\hat{\Delta}$ has no vertices of degree $>3$ then $cv(\hat{\Delta}) = (0,0,10,0,0,0,2,0)$ and $c^{\ast}(\hat{\Delta}) \leq - \frac{2 \pi}{3} + \frac{2 \pi}{5} < 0$.
Let $\hat{\Delta}$ have exactly one vertex $u_i$ of degree $>3$.
If $i=1$ then $cv(\hat{\Delta})=(4,0,10,0,0,0,2,2)$;
if $i=2$ then $cv(\hat{\Delta})=(4,4,10,0,0,0,2,0)$;
if $i=5$ then $cv(\hat{\Delta})=(0,0,10,2,2,0,2,0)$;
if $i=6$ then $cv(\hat{\Delta})=(0,0,10,0,2,4,2,0)$;
if $i=7$ then $cv(\hat{\Delta})=(0,0,10,0,0,e_1,e_2,0)$; and
if $i=8$ then $cv(\hat{\Delta})=(0,0,10,0,0,0,9,2)$.
It follows that $c^{\ast}(\hat{\Delta}) \leq -\frac{5 \pi}{6}+\frac{21 \pi}{30} < 0$.
Let $\hat{\Delta}$ have exactly two vertices $u_i,u_j$ of degree $>3$.
If $d(u_7)=3$ then $cv(\hat{\Delta})=(4,4,10,2,2,4,2,2)$ and $c^{\ast}(\hat{\Delta}) \leq -\pi + \pi = 0$ so assume $i=7$.
If $j=1$ then $cv(\hat{\Delta})=(4,0,10,0,0,e_1,e_2,2)$;
if $j=2$ then $cv(\hat{\Delta})=(4,4,10,0,0,e_1,e_2,0)$;
if $j=5$ then $cv(\hat{\Delta})=(0,0,10,2,2,e_1,e_2,0)$;
if $j=6$ then $cv(\hat{\Delta})=(0,0,10,0,2,e_1,e_2,0)$; and
if $j=8$ then $cv(\hat{\Delta})=(0,0,10,0,0,e_1,e_2,2)$.
It follows that $c^{\ast}(\hat{\Delta}) \leq - \pi + \frac{29 \pi}{30} < 0$.

\medskip

\textbf{Remark 1}\quad If $\hat{\Delta}$ is given by Figure 44(x) or (xi) then it is now assumed that $d(u_4)=3$, at least one of $d(u_3)$, $d(u_5)$ equals 3 and $d(u)>3$
Note that in both figures if $d(u)>4$ and $d(u_3)=3$ then $c(u_3,u_4)= \frac{7 \pi}{30}$;if $d(u)>4$ and $d(u_3)>3$ then $c(u_3,u_4)= \frac{ \pi}{5}$;if $d(u)>4$ and $d(u_5)=3$ then $c(u_4,u_5)= \frac{7 \pi}{30}$;
and if $d(u)>4$ and $d(u_5)>3$ then $c(u_4,u_5)= \frac{ \pi}{5}$.
Note also that if $d(u)=4$,$d(u_5)=d(u_6)=3$ in Figure 44(x) or $d(u_2)=d(u_3)=3$ in Figure 44(xi )and $\hat{\Delta}$ receives more than $\frac{2 \pi}{15}$ across the $(u_4,u_5)$-edge, $(u_3,u_4)$-edge (respectively) 
then according to 
Configuration E in Figure 32(iii), Configuration F in Figure 32(v) (respectively) the surplus
of at most  $\frac{ \pi}{5}$ is distributed out of $\hat{\Delta}$.

\textbf{Remark 2}\quad In Figure 44(x) if $d(u_1)=3$ then, by the remark immediately preceeding Proposition 10.4, $c(u_1,u_2)+c(u_2,u_3) \leq \frac{\pi}{6}$ and this bound can
only be attained when $c(u_1,u_2)=\frac{\pi}{30}$, $c(u_2,u_3)=\frac{2 \pi}{15}$, $\hat{\Delta}=\hat{\Delta}_2$ of Figure 37(iv) and checking shows that $\Delta_6$ of Figure 37(iv)
must then be  $\hat{\Delta}_4$ of Figure 18(ii); in particular, the vertices $u_2$, $u_3$ and $u_8$ of $\hat{\Delta}$ have degree $>3$. Similarly if $d(u_7)=3$ in Figure 44(xi) then 
$c(u_5,u_6)+c(u_6,u_7)=\frac{\pi}{6}$ forces the vertices $u_5$, $u_6$ and $u_8$ of $\hat{\Delta}$ to have degree $>3$ (see Figure 38(iv)).

\medskip

\textbf{Case 2}\quad
Let $\hat{\Delta}$ be given by Figure 44(x) in which case (see Proposition 10.4 Case 10)
$cv(\hat{\Delta})=(a_2,4,10,10,2,b_1,b_2,a_1)$, so
$c^{\ast}(\hat{\Delta}) \leq c(\hat{\Delta}) + \frac{41 \pi}{30}$ and
if $\hat{\Delta}$ has at least five vertices of degree $> 3$ then $c^{\ast}(\hat{\Delta}) \leq 0$.
If $\hat{\Delta}$ has no vertices of degree $>3$ then it follows by Remark 1 that either $d(u)=4$, $cv(\hat{\Delta})=(0,0,10,10,0,6,0,0)$ and
$c^{\ast}(\hat{\Delta}) \leq -\frac{2 \pi}{3} + \frac{13 \pi}{15} - \frac{\pi}{5}=0$ or $d(u) > 4$,
$cv(\hat{\Delta})=(0,0,7,7,0,6,0,0)$ and $c^{\ast}(\hat{\Delta}) \leq - \frac{2 \pi}{3} + \frac{2 \pi}{3} =0$.
Let $\hat{\Delta}$ have exactly one vertex $u_i$ of degree $>3$ and assume that $d(u)=4$.
If $i=1$ then $l(u_2)=b \mu z$ implies $cv(\hat{\Delta})=(0,0,10,10,0,6,0,3)$;
if $i=2$ then $cv(\hat{\Delta})=(x_1,y_1,10,10,0,6,0,0)$ where $x_1+y_1=4$ by Remark 2;
if $i=3$ then $cv(\hat{\Delta})=(0,0,6,10,0,6,0,0)$;
if $i=5$ then $cv(\hat{\Delta})=(0,0,10,6,2,6,0,0)$;
if $i=6$ and $d(u_6)=4$ then $cv(\hat{\Delta})=(0,0,10,10,2,0,0,0)$;
if $i=6$ and $d(u_6) > 4$ then $cv(\hat{\Delta})=(0,0,10,10,2,2,0,0)$;
if $i=7$ then $cv(\hat{\Delta})=(0,0,10,10,0,b_1,b_2,0)$; and
if $i=8$ then $l(u_7)= \lambda b^{-1}z^{-1}$ implies $cv(\hat{\Delta})=(0,0,10,10,0,6,0,3)$.
It follows that
if $d(u_5)=d(u_6)=3$ then $c^{\ast}(\hat{\Delta}) \leq - \frac{5 \pi}{6} + \pi - \frac{\pi}{5} < 0$;
otherwise $c^{\ast}(\hat{\Delta}) \leq -\frac{5 \pi}{6} + \frac{24 \pi}{30} < 0$.
If now $d(u) > 4$ then each $cv(\hat{\Delta})$ is altered by replacing each 10 by 7 and it follows that $c^{\ast}(\hat{\Delta}) \leq - \frac{5 \pi}{6} + \frac{24 \pi}{30} < 0$.
Let $\hat{\Delta}$ have exactly two vertices $u_i,u_j$ of degree $>3$ and assume that $d(u)=4$.
If $(i,j)=(1,2)$ then $cv(\hat{\Delta})=(a_2,4,10,10,0,6,0,a_1)$ and so
if $d(u_1)>4$ or $d(u_2)>4$ then
$c^{\ast}(\hat{\Delta}) \leq -\frac{11 \pi}{10} + \frac{37 \pi}{30} - \frac{\pi}{5} < 0$; and
if $d(u_1)=d(u_2)=4$ then $c(u_8,u_1)=0$ and, moreover, $c(u_1,u_2) > \frac{2 \pi}{15}$ and $l(u_3)=axy^{-1}$ together imply (see Figure 40(x)) that $c(u_2,u_3)=0$ so
$cv(\hat{\Delta})=(b_1,b_2,10,10,0,6,0,0)$ and $c^{\ast}(\hat{\Delta} \leq -\pi + \frac{17 \pi}{15} - \frac{\pi}{5} < 0$.
If $(i,j)=(1,3)$ then $cv(\hat{\Delta})=(0,0,6,10,0,6,0,3)$;
if $(i,j)=(1,5)$ then $cv(\hat{\Delta})=(0,0,10,6,2,6,0,3)$;
if $(i,j)=(1,6)$ and $d(u_6)=4$ then $cv(\hat{\Delta})=(0,0,10,10,2,0,0,3)$;
if $(i,j)=(1,6)$ and $d(u_6) > 4$ then $cv(\hat{\Delta})=(0,0,10,10,2,2,0,3)$;
if $(i,j)=(1,7)$ then (see Proposition 10.4 Case 10) $cv(\hat{\Delta})=(0,0,10,10,0,b_1,b_2,3)$;
if $(i,j)=(1,8)$ then $cv(\hat{\Delta})=(0,0,10,10,0,6,0,3)$;
if $(i,j)=(2,3)$ then $cv(\hat{\Delta})=(x_1,y_1,6,10,0,6,0,0)$;
if $(i,j)=(2,5)$ then $cv(\hat{\Delta})=(x_1,y_1,10,6,2,6,0,0)$;
if $(i,j)=(2,6)$ and $d(u_6)=4$ then $cv(\hat{\Delta})=(x_1,y_1,10,10,2,0,0,0)$;
if $(i,j)=(2,6)$ and $d(u_6)>4$ then $cv(\hat{\Delta})=(x_1,y_1,10,10,2,2,0,0)$;
if $(i,j)=(2,7)$ then $cv(\hat{\Delta})=(x_1,y_1,10,10,0,b_1,b_2,0)$;
if $(i,j)=(2,8)$ then $cv(\hat{\Delta})=(x_1,y_1,10,10,0,6,0,3)$;
if $(i,j)=(3,6)$ then $cv(\hat{\Delta})=(0,0,6,10,2,6,0,0)$;
if $(i,j)=(3,7)$ then $cv(\hat{\Delta})=(0,0,6,10,0,b_1,b_2,0)$;
if $(i,j)=(3,8)$ then $cv(\hat{\Delta})=(0,0,6,10,0,6,0,3)$;
if $(i,j)=(5,6)$ then $cv(\hat{\Delta})=(0,0,10,6,2,6,0,0)$;
if $(i,j)=(5,7)$ then $cv(\hat{\Delta})=(0,0,10,6,2,b_1,b_2,0)$;
if $(i,j)=(5,8)$ then $cv(\hat{\Delta})=(0,0,10,6,2,6,0,3)$;
if $(i,j)=(6,7)$ then $cv(\hat{\Delta})=(0,0,10,10,2,b_1,b_2,0)$;
if $(i,j)=(6,8)$ and $d(u_6)=4$ then $cv(\hat{\Delta})=(0,0,10,10,2,0,0,3)$;
if $(i,j)=(6,8)$ and $d(u_6)>4$ then $cv(\hat{\Delta})=(0,0,10,10,2,2,0,3)$; and
if $(i,j)=(7,8)$ then $cv(\hat{\Delta})=(0,0,10,10,0,b_1,b_2,3)$.
It follows that
if $(i,j) \neq (1,2)$ and
if $d(u_5)=d(u_6)=3$ then $c^{\ast}(\hat{\Delta}) \leq - \pi + \frac{11 \pi}{10} - \frac{\pi}{5} < 0$; or 
if $d(u_5) > 3$ or $d(u_6) > 3$ then $c^{\ast}(\hat{\Delta}) \leq - \pi + \pi = 0$.
If now $d(u) > 4$ then, as before, replacing each 10 by 7 in the above yields $c^{\ast}(\hat{\Delta}) \leq - \pi + \frac{28 \pi}{30} < 0$ except when $(i,j)=(1,2)$ and either $d(u_1) > 4$ or $d(u_2) > 4$ and
$c^{\ast} (\hat{\Delta}) \leq - \frac{11 \pi}{10} + \frac{31 \pi}{30} < 0$.
Let $\hat{\Delta}$ have exactly three vertices $u_i,u_j,u_k$ of degree $>3$.
If $d(u_2)=3$ then $cv(\hat{\Delta})=(0,0,10,10,2,b_1,b_2,3)$ and $c^{\ast}(\hat{\Delta}) \leq -\frac{7 \pi}{6}+\frac{11 \pi}{10} < 0$; or
if $d(u_5)=d(u_6)=3$ then $c^{\ast}(\hat{\Delta}) \leq -\frac{7 \pi}{6} + \frac{41 \pi}{30} - \frac{\pi}{5} = 0$, so assume otherwise.
If $d(u_3)=4$ then ($d(u_4)=3$ implies) $c(u_3,u_4)=0$ and if $d(u_3) \geq 5$ then $c(u_3,u_4)=\frac{\pi}{15}$, and in both cases $c^{\ast}(\hat{\Delta} \leq 0$.
Similarly
if $d(u_5) \neq 3$ then $c^{\ast}(\hat{\Delta}) \leq 0$, so it can be assumed that $d(u_3)=d(u_5)=3$.
If $(i,j,k)=(1,2,6)$ and $d(u_6)=4$ then $cv(\hat{\Delta})=(a_2,4,10,10,2,0,0,a_1)$ and $c^{\ast}(\hat{\Delta}) \leq -\frac{7 \pi}{6} + \frac{11 \pi}{10} < 0$; or
if $d(u_6)>4$ then $cv(\hat{\Delta})=(a_2,4,10,10,2,2,0,a_1)$ and $c^{\ast}(\hat{\Delta}) \leq -\frac{19 \pi}{15} + \frac{7 \pi}{6} < 0$.
If $(i,j,k)=(2,6,7)$ then $cv(\hat{\Delta})=(x_1,y_1,10,10,2,b_1,b_2,0)$ (by Remark 2); and
if $(i,j,k)=(2,6,8)$ then $cv(\hat{\Delta})=(x_1,y_1,10,10,2,6,0,3)$.  In both cases $c^{\ast}(\hat{\Delta})  \leq 0$.
Finally let $\hat{\Delta}$ have exactly four vertices of degree $>3$ and so $c(\hat{\Delta}) \leq -\frac{4 \pi}{3}$.
If any vertex has degree $>4$ or if any of $u_1$, $u_2$, $u_6$ or $u_7$ has degree 3 then clearly
$c^{\ast}(\hat{\Delta}) \leq 0$, so assume otherwise.  But then $d(u_1)=d(u_2)=4$ and $d(u_3)=3$ together imply either $c(u_1,u_2)=0$ or $c(u_2,u_3)=0$ and $c^{\ast}(\hat{\Delta})<0$.

\textbf{Case 3}\quad
Let $\hat{\Delta}$ be given by Figure 44(xi) in which case (see Proposition 10.4 Case 11) $cv(\hat{\Delta})=(b_2,2,10,10,4,a_1,a_2,b_1) = \frac{41 \pi}{30}$
so if $\hat{\Delta}$ has at least five vertices of degree at least 4 then $c^{\ast}(\hat{\Delta}) \leq 0$. 
If $\hat{\Delta}$ has no vertices of degree $>3$ then by Remark 1 preceeding Case 2 either $d(u)=4$, $cv(\hat{\Delta})=(6,0,10,10,0,0,0,0)$ and $c^{\ast}(\hat{\Delta})\leq -\frac{2 \pi}{3} + \frac{13 \pi}{15} - \frac{\pi}{5}=0$
or $d(u) > 4$, $cv(\hat{\Delta})=(6,0,7,7,0,0,0,0)$ and $c^{\ast}(\hat{\Delta}) \leq - \frac{2 \pi}{3} + \frac{2 \pi}{3}-0$.
Let $\hat{\Delta}$ have exactly one vertex $u_i$ of degree $>3$ and assume that $d(u)=4$.
If $i=1$ then $cv(\hat{\Delta})=(b_2,0,10,10,0,0,0,b_1)$;
if $i=2$ and $d(u_2)=4$ then $l(u_1)=z^{-1} \lambda b^{-1}$ forces $c(u_1,u_2)=0$ and $cv(\hat{\Delta})=(0,2,10,10,0,0,0,0)$;
if $i=2$ and $d(u_2)>4$ then $cv(\hat{\Delta})=(2,2,10,10,0,0,0,0)$;
if $i=3$ then $cv(\hat{\Delta})=(6,2,6,10,0,0,0,0)$;
if $i=5$ then ($l(u_6)= b \mu z$ and so) $cv(\hat{\Delta})=(6,0,10,6,0,0,0,0)$;
if $i=6$ then $cv(\hat{\Delta})=(6,0,10,10,x_1,y_1,0,0)$ (where $x_1+y_1=4$ by Remark 2 above);
if $i=7$ then $cv(\hat{\Delta})=(6,0,10,10,0,0,3,0)$;
if $i=8$ and $d(u_8)=4$ then $cv(\hat{\Delta})=(6,0,10,10,0,0,3,0)$; and
if $i=8$ and $d(u_8)>4$ then $cv(\hat{\Delta})=(6,0,10,10,0,0,2,2)$.
Therefore $c^{\ast}(\hat{\Delta}) \leq - \frac{5 \pi}{6} + \frac{24 \pi}{30}$ when $(d(u_2),d(u_3)) \neq (3,3)$; and
if $d(u_2)=d(u_3)=3$ then $c^{\ast}(\hat{\Delta}) \leq -\frac{5 \pi}{6} + \pi - \frac{\pi}{5} < 0$.
If now $d(u) > 4$ then replacing each 10 by 7 in the above yields $c^{\ast}(\hat{\Delta}) \leq - \frac{5 \pi}{6} + \frac{24 \pi}{30} < 0$.
Let $\hat{\Delta}$ have exactly two vertices $u_i,u_j$ of degree $>3$ and assume that $d(u)=4$.
If $(i,j)=(1,2)$ then $cv(\hat{\Delta})=(b_2,2,10,10,0,0,0,b_1)$;
if $(i,j)=(1,3)$ then $cv(\hat{\Delta})=(b_2,2,6,10,0,0,0,b_1)$;
if $(i,j)=(1,5)$ then $cv(\hat{\Delta})=(b_2,0,10,6,0,0,0,b_1)$;
if $(i,j)=(1,6)$ then $cv(\hat{\Delta})=(b_2,0,10,10,x_1,y_1,0,b_1)$;
if $(i,j)=(1,7)$ then $cv(\hat{\Delta})=(b_2,0,10,10,0,0,3,b_1)$;
if $(i,j)=(1,8)$ then $cv(\hat{\Delta})=(b_2,0,10,10,0,0,3,b_1)$;
if $(i,j)=(2,3)$ then $cv(\hat{\Delta})=(6,2,6,10,0,0,0,0)$;
if $(i,j)=(2,5)$ then $cv(\hat{\Delta})=(6,2,10,6,0,0,0,0)$;
if $(i,j)=(2,6)$ and $d(u_2)=4$ then $cv(\hat{\Delta})=(0,2,10,10,x_1,y_1,0,0)$;
if $(i,j)=(2,6)$ and $d(u_2)>4$ then $cv(\hat{\Delta})=(2,2,10,10,x_1,y_1,0,0)$;
if $(i,j)=(2,7)$ and $d(u_2)=4$ then $cv(\hat{\Delta})=c(0,2,10,10,0,0,3,0)$;
if $(i,j)=(2,7)$ and $d(u_2)>4$ then $cv(\hat{\Delta})=c(2,2,10,10,0,0,3,0)$;
if $(i,j)=(2,8)$ and $d(u_2)=4$ then $cv(\hat{\Delta})=(0,2,10,10,0,0,3,0)$;
if $(i,j)=(2,8)$ and $d(u_2)>4$ then $cv(\hat{\Delta})=(2,2,10,10,0,0,3,0)$;
if $(i,j)=(3,6)$ then $cv(\hat{\Delta})=(6,2,6,10,x_1,y_1,0,0)$;
if $(i,j)=(3,7)$ then $cv(\hat{\Delta})=(6,2,6,10,0,0,3,0)$;
if $(i,j)=(3,8)$ then $cv(\hat{\Delta})=(6,2,6,10,0,0,3,0)$;
if $(i,j)=(5,6)$ then $cv(\hat{\Delta})=(6,0,10,6,x_1,y_1,0,0)$;
if $(i,j)=(5,7)$ then $cv(\hat{\Delta})=(6,0,10,6,0,0,3,0)$;
if $(i,j)=(5,8)$ then $cv(\hat{\Delta})=(6,0,10,6,0,0,3,0)$;
if $(i,j)=(6,7)$ then $cv(\hat{\Delta})=(6,0,10,10,4,a_1,a_2,0)$ 
and so if $d(u_6)>4$ or $d(u_7)>4$ then $c^{\ast}(\hat{\Delta}) \leq -\frac{11 \pi}{10} + \frac{37 \pi}{30} - \frac{\pi}{5} < 0$, or
if $d(u_6)=d(u_7)=4$ then $c(u_7,u_8)=0$ and, moreover, $c(u_6,u_7) \geq \frac{2 \pi}{15}$ and $l(u_5)=axy^{-1}$ together imply (see Figure 40(ix)) that $c(u_5,u_6)=0$ so 
$cv(\hat{\Delta})=(6,0,10,10,b_1,b_2,0,$
and $c^{\ast}(\hat{\Delta} \leq -\pi + \frac{17 \pi}{15} - \frac{\pi}{5} < 0$;
if $(i,j)=(6,8)$ then $cv(\hat{\Delta})=(6,0,10,10,x_1,y_1,3,0)$; and
if $(i,j)=(7,8)$ then $cv(\hat{\Delta})=(6,0,10,10,0,0,3,0)$.
It follows that if $(i,j) \neq (6,7)$ and
if $d(u_2)=d(u_3)=3$ then $c^{\ast}(\hat{\Delta}) \leq - \pi + \frac{11 \pi}{10} - \frac{\pi}{5} < 0$;
or if $d(u_2) > 3$ or $d(u_3) > 3$ then $c^{\ast}(\hat{\Delta}) \leq - \pi + \pi = 0$.
If now $d(u) > 4$ then replacing each 10 by 7 in the above yields $c^{\ast}(\hat{\Delta}) \leq - \pi + \frac{14 \pi}{15} < 0$.
Let $\hat{\Delta}$ have exactly three vertices $u_i,u_j,u_k$ of degree $>3$.
If $d(u_6)=3$ then $cv(\hat{\Delta})=(b_2,2,10,10,0,0,3,b_1)$ and $c^{\ast}(\hat{\Delta}) \leq -\frac{7 \pi}{6}+\frac{11 \pi}{10} < 0$; or
if $d(u_2)=d(u_3)=3$ then $c^{\ast}(\hat{\Delta}) \leq -\frac{7 \pi}{6} + \frac{41 \pi}{30} - \frac{\pi}{5} = 0$, so assume otherwise.
If $d(u_3)=4$, $d(u_5)=4$ (respectively) then $d(u_4)=3$ implies $c(u_3,u_4)=0$, $c(u_4,u_5)=0$ (respectively) 
or if $d(u_3) \geq 5$, $d(u_5) \geq 5$ (respectively) then $c(u_3,u_4)=\frac{\pi}{15}$, $c(u_4,u_5)=\frac{\pi}{15}$ (respectively)and in each case $c^{\ast}(\hat{\Delta} \leq 0$.
So it can be assumed that $d(u_3)=d(u_5)=3$.
If $(i,j,k)=(6,2,7)$ and $d(u_2)=4$ then $cv(\hat{\Delta})=(0,2,10,10,4,a_1,a_2,0)$;
if $(i,j,k)=(6,2,7)$ and $d(u_2)>4$ then $cv(\hat{\Delta})=(2,2,10,10,4,a_1,a_2,0)$;
if $(i,j,k)=(6,2,8)$ then  
$cv(\hat{\Delta})=(2,2,10,10,x_1,y_1,3,0)$ (by Remark 2); or if $(i,j,k)=(6,2,1)$ then $cv(\hat{\Delta})=(b_2,2,10,10,x_1,y_1,0,b_1)$.
In each case $c^{\ast}(\hat{\Delta}) \leq 0$ so assume that $\hat{\Delta}$ has exactly four vertices of degree $>3$. Then $c(\hat{\Delta}) \leq - \frac{4 \pi}{3}$. 
If $d(u_1)=3$ then $cv(\hat{\Delta})=(6,2,10,10,4,a_1,a_2,0) = \frac{39 \pi}{30}$; or if 
$d(u_1)=4$ then $cv(\hat{\Delta})=(0,2,10,10,4,a_1,a_2,7) = \frac{40 \pi}{30}$ and in both cases $c^{\ast}(\hat{\Delta}) \leq 0$.
On the other hand if $d(u_1) \geq 5$ then $c^{\ast}(\hat{\Delta}) \leq c(\hat{\Delta}) + cv(\hat{\Delta}) \leq - \frac{43 \pi}{30} + \frac{41 \pi}{30} < 0$.$\Box$

\begin{proposition}
\textit{Let $\hat{\Delta}$ be a type $\mathcal{B}$ region.  If $\hat{\Delta}$ is given by Figure 47(ii)-(v), 48 or 49 then
$c^{\ast} (\hat{\Delta}) \leq 0$.}
\end{proposition}

\textit{Proof}.
Let $\hat{\Delta}$ be given by Figure 47(ii)-(v).  Then (up to cyclic permutation and inversion)
there are two ways to label each of (ii) and (iii); and one way to label each of (iv) and (v) and so $\hat{\Delta}$ is given by Figure 50.
There are six \textit{a-cases}. As usual we rely heavily on Figures 35-38 and 40-42.

\newpage
\begin{figure}
\begin{center}
\psfig{file=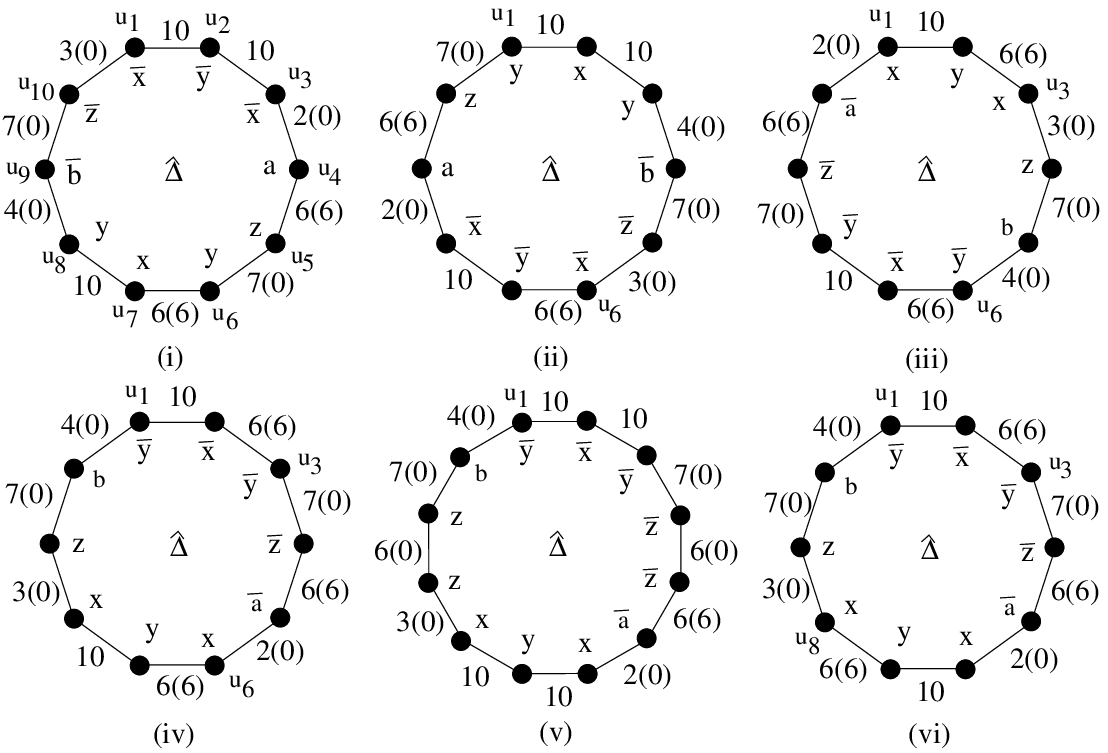}
\end{center}
\caption{}
\end{figure}

\textbf{Case a1}\quad
Let $\hat{\Delta}$ be given by Figure 50(i) in which (it can be seen from Figure 47(ii) that) $d(u_1)=d(u_2)=d(u_3)=d(u_7)=d(u_8)=3$ and $d(u_6) >
3$.  Then (see Figure 40) $cv(\hat{\Delta})=$\
$(10,10,2,d_1,d_2,6,10,4,a_1,a_2)$ so
$c^{\ast}(\hat{\Delta}) \leq c(\hat{\Delta}) + \frac{59 \pi}{30}$.
Let $\hat{\Delta}$ have exactly one vertex $u_6$ of degree $>3$.  Then $cv(\hat{\Delta})=(10,10,0,6,0,6,10,0,0,0)$ and
$c^{\ast}(\hat{\Delta}) \leq -\frac{45 \pi}{30} + \frac{42 \pi}{30} < 0$.
Let $\hat{\Delta}$ have exactly two vertices $u_6,u_i$ of degree $>3$.
If $i=4$ then $cv(\hat{\Delta})=(10,10,2,6,0,6,10,0,0,0)$;
if $i=5$ then $cv(\hat{\Delta})=(10,10,0,d_1,d_2,6,10,0,0,0)$;
if $i=9$ then (using $d(u_5)=d(u_{10})=3$) $cv(\hat{\Delta})=$\
$(10,10,0,6,0,6,10,4,2,0)$; and if $i=10$ then $cv(\hat{\Delta})=(10,10,0,6,0,6,10,0,a_1,a_2)$.
It follows that $c^{\ast}(\hat{\Delta}) \leq -\frac{50 \pi}{30} + \frac{49 \pi}{30} < 0$.
Let $\hat{\Delta}$ have exactly three vertices $u_6,u_i,u_j$ of degree $>3$.
If $d(u_5)=3$ then $cv(\hat{\Delta})=(10,10,2,6,0,6,10,4,a_1,a_2)$;
if $d(u_{10})=3$ then $cv(\hat{\Delta})=(10,10,2,d_1,d_2,6,10,4,2,0)$; and
if $(i,j)=(5,10)$ then $cv(\hat{\Delta})=(10,10,0,d_1,d_2,6,10,0,a_1,a_2)$.
It follows that $c^{\ast}(\hat{\Delta}) \leq -\frac{55 \pi}{30} + \frac{55 \pi}{30}=0$.
If $\hat{\Delta}$ has more than three vertices of degree $>3$ then
$c^{\ast}(\hat{\Delta}) \leq -\frac{60 \pi}{30} + \frac{59 \pi}{30} < 0$.

\textbf{Case a2}\quad 
Let $\hat{\Delta}$ be given by Figure 50(ii) in which
$d(u_1)=d(u_2)=d(u_3)=d(u_7)=d(u_8)=3$ and $d(u_6) > 3$.  Then $cv(\hat{\Delta})=$\
$(10,10,4,a_1,a_2,6,10,2,d_1,d_2)$ so
$c^{\ast}(\hat{\Delta}) \leq c(\hat{\Delta}) + \frac{59 \pi}{30}$.
Let $\hat{\Delta}$ have exactly one vertex $u_6$ of degree $>3$.  Then $cv(\hat{\Delta})=(10,10,0,0,3,6,10,0,6,0)$ and
$c^{\ast}(\hat{\Delta}) \leq -\frac{45 \pi}{30} + \frac{45 \pi}{30}=0$.
Let $\hat{\Delta}$ have exactly two vertices $u_6,u_i$ of degree $>3$.
If $i=4$ and $d(u_6)=4$ then ($l(u_6)$ together with $l(u_7)$ force) $cv(\hat{\Delta})=$\
$(10,10,4,2,3,0,10,0,6,0)$;
if $i=4$ and $d(u_6) > 4$ then (see Figure 35(ii)) $cv(\hat{\Delta}) =$\
$(10,10,4,2,2,2,10,0,6,0)$;
if $i=5$ then $cv(\hat{\Delta})=(10,10,0,a_1,a_2,6,10,0,6,0)$;
if $i=9$ then $cv(\hat{\Delta})=(10,10,0,0,3,6,10,2,6,0)$; and
if $i=10$ then $cv(\hat{\Delta})=$\
$(10,10,0,0,3,6,10,0,d_1,d_2)$.
It follows that either $c^{\ast}(\hat{\Delta}) \leq -\frac{50 \pi}{30} + \frac{49 \pi}{30} < 0$ or
$c^{\ast}(\hat{\Delta}) \leq -\frac{53 \pi}{30} + \frac{46 \pi}{30} < 0$.
Let $\hat{\Delta}$ have exactly three vertices $u_6,u_i,u_j$ of degree $>3$.
If $d(u_{10})=3$ then $cv(\hat{\Delta})=(10,10,4,a_1,a_2,6,10,2,6,0)$ and
$c^{\ast}(\hat{\Delta}) \leq -\frac{55 \pi}{30} + \frac{55 \pi}{30}=0$, so assume $i=10$ and $j \in \{ 4,5,9 \}$.
If $j=4$ then $cv(\hat{\Delta})=(10,10,4,2,3,6,10,0,d_1,d_2)$;
if $j=5$ then $cv(\hat{\Delta})=(10,10,0,a_1,a_2,6,10,0,d_1,d_2)$; and
if $j=9$ then $cv(\hat{\Delta})=(10,10,0,0,3,6,10,2,d_1,d_2)$.
It follows that $c^{\ast}(\hat{\Delta}) \leq -\frac{55 \pi}{30} + \frac{55 \pi}{30}=0$.
If $\hat{\Delta}$ has more than three vertices of degree $>3$ then $c^{\ast}(\hat{\Delta}) \leq -\frac{60 \pi}{30} + \frac{59 \pi}{30}=0$.

\textbf{Case a3}\quad
Let $\hat{\Delta}$ be given by Figure 50(iii) in which (see Figure
47(iii)) $d(u_1)=d(u_2)=d(u_7)=d(u_8)=3$, $d(u_3) > 3$ and $d(u_6) > 3$.
Then $cv(\hat{\Delta})=(10,6,a_1,a_2,4,6,10,d_1,d_2,2)$ and $c^{\ast}(\hat{\Delta}) \leq c(\hat{\Delta}) + \frac{55 \pi}{30}$.
If $\hat{\Delta}$ has at least three vertices of degree $>3$ then $c^{\ast}(\hat{\Delta}) \leq - \frac{55 \pi}{30} + \frac{55 \pi}{30} = 0$.
This leaves the case $d(u_3) > 3$ and $d(u_6) > 3$ only.  Then $cv(\hat{\Delta})=(10,6,3,0,4,6,10,0,6,0)$ and
$c^{\ast}(\hat{\Delta}) \leq - \frac{50 \pi}{30} + \frac{45 \pi}{30} < 0$.

\textbf{Case a4}\quad 
Let $\hat{\Delta}$ be given by Figure 50(iv) in which $d(u_1)=d(u_2)=d(u_7)=d(u_8)=3$, $d(u_3) > 3$ and $d(u_6) > 3$.
Then $cv(\hat{\Delta})=(10,6,d_1,d_2,2,6,10,a_1,a_2,4)$ and $c^{\ast}(\hat{\Delta}) \leq c(\hat{\Delta}) + \frac{55 \pi}{30}$.
If $\hat{\Delta}$ has at least three vertices of degree $>3$ then $c^{\ast}(\hat{\Delta}) \leq -\frac{55 \pi}{30} + \frac{55 \pi}{30}=0$.
This leaves the case $d(u_3) >3$ and $d(u_6) >3$ only.
Then $cv(\hat{\Delta})=(10,6,7,6,2,6,10,0,0,0)$ and $c^{\ast}(\hat{\Delta}) \leq -\frac{50 \pi}{30}+\frac{47 \pi}{30} <0$.

\textbf{Case a5}\quad 
Let $\hat{\Delta}$ be given by Figure 50(v) in which (see Figure 47(iv)) $d(u_1)=d(u_2)=d(u_3)=d(u_7)=d(u_8)=d(u_9)=3$.  Then
$cv(\hat{\Delta})=(10,10,d_1,d_2,6,2,10,10,3,d_1,d_2,4)$ so
$c^{\ast}(\hat{\Delta}) \leq c(\hat{\Delta}) + \frac{75 \pi}{30}$.
If $\hat{\Delta}$ has no vertices of degree $>3$ then $cv(\hat{\Delta})=
(10,10,0,0,6,0,10,10,0,0,0,0)$ and
$c^{\ast}(\hat{\Delta}) \leq -\frac{60 \pi}{30} + \frac{46 \pi}{30} < 0$.
Let $\hat{\Delta}$ have exactly one vertex $u_i$ of degree $>3$.
If $i=4$ then $cv(\hat{\Delta})=(10,10,d_1,d_2,6,0,10,10,0,0,0,0)$;
if $i=5$ then $cv(\hat{\Delta})=(10,10,0,6,6,0,10,10,0,0,0,0)$;
if $i=6$ then $cv(\hat{\Delta})=(10,10,0,0,6,2,10,10,0,0,0,0)$;
if $i=10$ then $cv(\hat{\Delta})=(10,10,0,0,6,0,10,10,3,6,0,0)$;
if $i=11$ then\
$cv(\hat{\Delta})=(10,10,0,0,6,0,10,10,0,d_1,d_2,0)$; and
if $i=12$ then\
$cv(\hat{\Delta})=(10,10,0,0,6,0,10,10,0,0,7,4)$.
It follows that $c^{\ast}(\hat{\Delta}) \leq -\frac{65 \pi}{30} + \frac{57 \pi}{30} < 0$.
Let $\hat{\Delta}$ have exactly two vertices $u_i,u_j$ of degree $>3$.
If $d(u_4)=d(u_5)=3$ then $cv(\hat{\Delta})=(10,10,0,0,6,2,10,10,3,d_1,d_2,4)$;
if $d(u_{10})=d(u_{11})=3$ then\
$cv(\hat{\Delta})=(10,10,d_1,d_2,6,2,10,10,3,0,2,4)$; and
if $d(u_{12})=3$ then $cv(\hat{\Delta})=$\
$(10,10,d_1,d_2,6,2,10,10,3,6,0,0)$.
It follows that $c^{\ast}(\hat{\Delta}) \leq -\frac{70 \pi}{30} + \frac{67 \pi}{30}$.
If $\hat{\Delta}$ has at least three vertices of degree $>3$ then $c^{\ast}(\hat{\Delta}) \leq -\frac{75 \pi}{30} + \frac{75 \pi}{30}=0$.

\textbf{Case a6}\quad 
Let $\hat{\Delta}$ be given by Figure 50(vi) in which (see Figure 47(v)) $d(u_1)=d(u_2)=d(u_6)=d(u_7)=3$, $d(u_3) > 3$ and $d(u_8) > 3$.  Then
$cv(\hat{\Delta})=(10,6,d_1,d_2,2,10,6,a_1,a_2,4)$ and
$c^{\ast}(\Delta) \leq c(\hat{\Delta}) + \frac{55 \pi}{30}$.
If $\hat{\Delta}$ has at least three vertices of degree $>3$ then
$c^{\ast}(\hat{\Delta}) \leq -\frac{55 \pi}{30} + \frac{55 \pi}{30} = 0$.  This leaves the case $d(u_3)>3$ and $d(u_8) > 3$ only.
Then $cv(\hat{\Delta}) = (10,6,0,6,0,10,6,3,0,0)$ and $c^{\ast}(\hat{\Delta}) \leq -\frac{50 \pi}{30} + \frac{41 \pi}{30} < 0$.

Now let $\hat{\Delta}$ be one of the regions of Figure 48.
It turns out that (up to cyclic permutation and inversion) there are two ways to label each of Figure 48(i), (ii), (iii) and (iv);
four ways to label (v); two ways to label each of (vi) and (vii); and four ways to label (viii).  However the labelled regions produced by (vii) already appear in those produced by (vi); 
and two of the labelled regions produced by (viii) already appear in those produced by 

\newpage
\begin{figure}
\begin{center}
\psfig{file=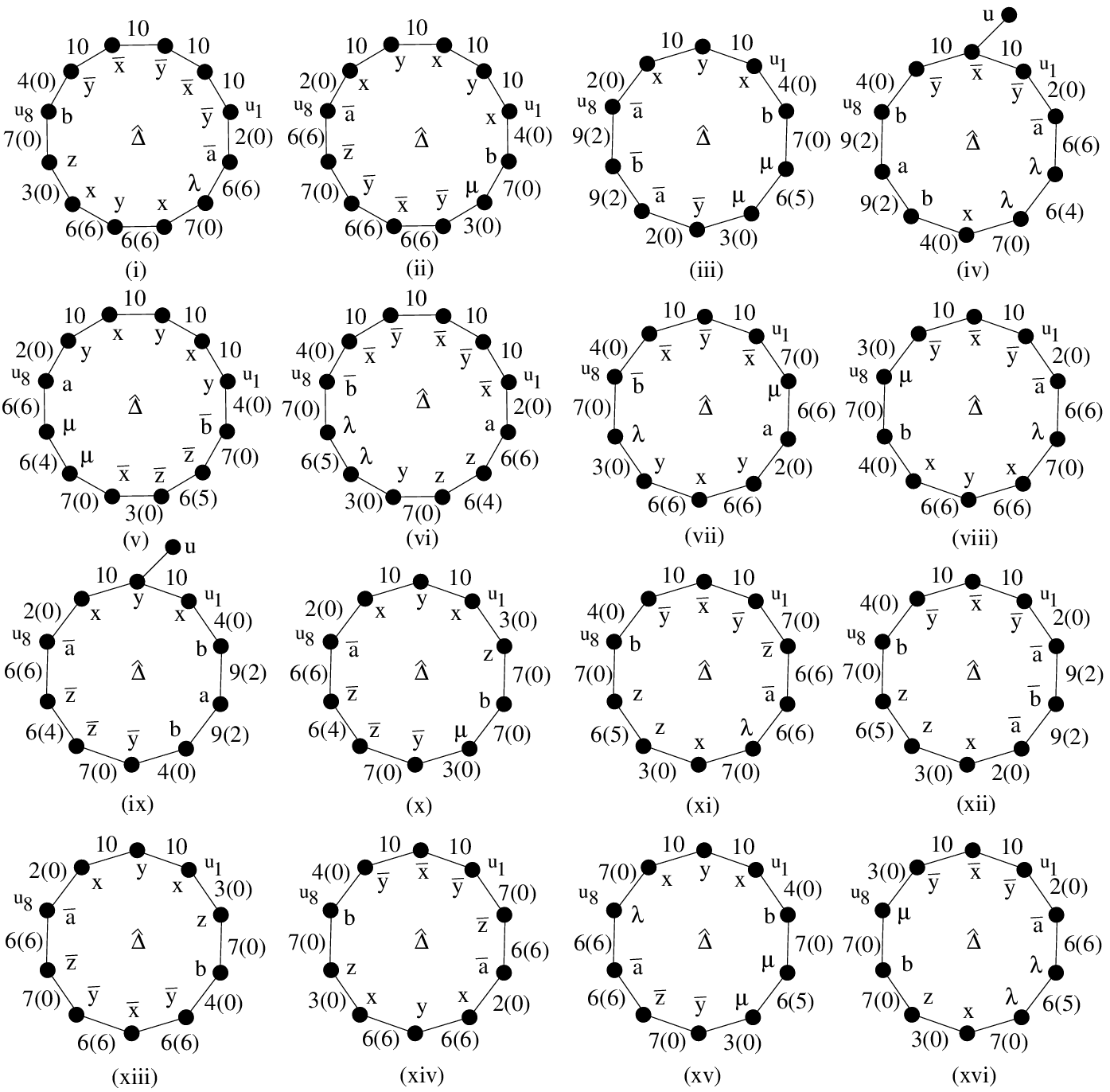}
\end{center}
\caption{}
\end{figure}

\noindent (ii), leaving a total of 16 regions and $\hat{\Delta}$ is given by Figure 51.
Table 4 gives $c(u_i,u_{i+1})$
($1 \leq i \leq 8$) in multiples of $\frac{\pi}{30}$ for each of the sixteen regions of Figure 51 with the total plus the contribution made via the
$b$-segment in the final column. We note here that Lemma 9.2 is used for the bounds $e_1, e_2$ and $f_1, f_2$ in rows (iii), (iv), (ix) and (xii); 
and that Figure 40 (iv), (x), (xiii) annd (xviii) is used to obtain the other bounds $a_1, a_2$, $b_1, b_2$ and $d_1, d_2$ in the table.

\begin{table}[h]
\[
\renewcommand{\arraystretch}{1.5}
\begin{array}{rlllllllll}
\textrm{(i)}&2&d_1&d_2&6&6&a_1&a_2&4&35 + 40 = 75\\
\textrm{(ii)}&4&a_1&a_2&6&6&d_1&d_2&2&35 + 40 = 75\\
\textrm{(iii)}&
4&d_1&d_2&3&2&f_1&f_2&2&33 + 20 = 53\\
\textrm{(iv)}&
2&6&d_1&d_2&e_1&e_2&e_1&e_2&40 + 20 = 60\\
\textrm{(v)}&
4&b_1&b_2&3&d_1&d_2&6&2&33 + 40 = 73\\
\textrm{(vi)}&
2&6&d_1&d_2&3&d_1&d_2&4&35 + 40 = 75\\
\textrm{(vii)}&
d_1&d_2&2&6&6&a_1&a_2&4&35 + 20 = 55\\
\textrm{(viii)}&
2&d_1&d_2&6&6&4&a_1&a_2&35 + 20 = 55\\
\textrm{(ix)}&
e_1&e_2&e_1&e_2&b_1&b_2&6&2&38 + 20 = 58\\
\textrm{(x)}&
a_1&a_2&a_1&a_2&b_1&b_2&6&2&30 + 20 = 50\\
\textrm{(xi)}&
d_1&d_2&d_1&d_2&3&b_1&b_2&4&35 + 20 = 55\\
\textrm{(xii)}&
2&f_1&f_2&2&3&b_1&b_2&4&31 + 20 = 51\\
\textrm{(xiii)}&
a_1&a_2&4&6&6&d_1&d_2&2&35 + 20 = 55\\
\textrm{(xiv)}&
d_1&d_2&2&6&6&a_1&a_2&4&35 + 20 = 55\\
\textrm{(xv)}&
4&d_1&d_2&3&d_1&d_2&d_1&d_2&37 + 20 = 57\\
\textrm{(xvi)}&
2&6&d_1&d_2&a_1&a_2&a_1&a_2&32 + 20 = 52
\end{array}
\renewcommand{\arraystretch}{1}
\]
\caption{}
\end{table}

The regions in Figure 51(i), (ii), (v) and (vi) each have degree 12 and so
$c(\hat{\Delta}) \leq (2-12) \pi + \frac{24 \pi}{3} = -2 \pi$; whereas the rest have degree 10 and in these cases
$c(\hat{\Delta}) \leq - \frac{4 \pi}{3}$.  It follows from Table 4  that
if $\hat{\Delta}$ has at least two vertices of degree $>3$ then $c^{\ast}(\hat{\Delta}) \leq 0$ for (x); if at least three then
$c^{\ast} (\hat{\Delta}) \leq 0$ for (i), (ii), (iii), (v), (vi), (vii), (viii), (xi), (xii), (xiii), (xiv) and (xvi); and if at least four then
$c^{\ast} (\hat{\Delta}) \leq - \frac{40 \pi}{30} + \frac{40 \pi}{30} = 0$.

If $\hat{\Delta}$ has no vertices of degree $>3$ then we see from Figure 51 that
$c^{\ast} (\hat{\Delta}) \leq - \frac{20 \pi}{30} + \frac{18 \pi}{30} < 0$.

We consider each of the sixteen \textit{b-cases} in turn.

\textbf{Case b1}\quad 
Let $\hat{\Delta}$ be given by Figure 51(i).  Suppose that $\hat{\Delta}$ has exactly one vertex $u_i$ of degree $>3$.
If $d(u_3)=d(u_8)=3$ then $cv(\hat{\Delta}) = (10,10,10,10,2,6,0,6,6,3,0,0)$;
if $i=3$ then $cv(\hat{\Delta})=(10,10,10,10,0,d_1,d_2,6,6,0,0,0)$; and
if $i=8$ then 
$cv(\hat{\Delta})=(10,10,10,10,0,6,0,6,6,0,2,4)$.
It follows that $c^{\ast} (\hat{\Delta}) \leq -\frac{65 \pi}{30} + \frac{64 \pi}{30} < 0$.
Let $\hat{\Delta}$ have exactly two vertices $u_i,u_j$ of degree $>3$.
If $d(u_3)=d(u_7)=3$ then $cv(\hat{\Delta})=(10,10,10,10,2,6,0,6,6,3,2,4)$;
if $d(u_8)=3$ then 
$cv(\hat{\Delta})=(10,10,10,10,2,d_1,d_2,6,6,3,0,0)$;
if $(i,j)=(3,8)$ then 
$c^{\ast}(\hat{\Delta})=(10,10,10,10,0,d_1,d_2,6,6,0,2,4)$; and
if $(i,j)=(7,8)$ then 
$cv(\hat{\Delta})=(10,10,10,10,0,6,0,6,6,a_1,a_2,4)$.
It follows that $c^{\ast}(\hat{\Delta}) \leq - \frac{70 \pi}{30} + \frac{69 \pi}{30} < 0$.

\textbf{Case b2}\quad 
Let $\hat{\Delta}$ be given by Figure 51(ii).  Suppose that $\hat{\Delta}$ has exactly one vertex $u_i$ of degree $>3$.
If $d(u_2)=d(u_7)=3$ then $cv(\hat{\Delta})=(10,10,10,10,0,0,3,6,6,0,6,2)$;
if $i=2$ then $cv(\hat{\Delta})=(10,10,10,10,4,2,0,6,6,0,6,0)$; and
if $i=7$ then 
$cv(\hat{\Delta})=(10,10,10,10,0,0,0,6,6,d_1,d_2,2)$.
It follows that $c^{\ast}(\hat{\Delta}) \leq - \frac{65 \pi}{30} + \frac{64 \pi}{30} < 0$.
Let $\hat{\Delta}$ have exactly two vertices $u_i,u_j$ of degree $>3$.
If $d(u_2)=3$ then 
$cv(\hat{\Delta})=(10,10,10,10,0,0,3,6,6,d_1,d_2,2)$;
if $d(u_3)=d(u_7)=3$ then 
$cv(\hat{\Delta})=(10,10,10,10,4,2,3,6,6,0,6,2)$;
if $(i,j)=(2,3)$ then 
$cv(\hat{\Delta})=(10,10,10,10,4,a_1,a_2,6,6,0,6,0)$; and
if $(i,j)=(2,7)$ then 
$cv(\hat{\Delta})=(10,10,10,10,4,2,0,6,6,d_1,d_2,0)$.
It follows that $c^{\ast}(\hat{\Delta}) \leq -\frac{70 \pi}{30} + \frac{69 \pi}{30} < 0$.

\textbf{Case b3}\quad 
Let $\hat{\Delta}$ be given by Figure 51(iii).
Suppose that $\hat{\Delta}$ has exactly one vertex $u_i$ of degree $> 3$.
If $d(u_7)=3$ then $cv(\hat{\Delta})=(10,10,4,d_1,d_2,3,2,2,2,2)$;
if $i=7$ then (see Lemma 9.2) $cv(\hat{\Delta})=(10,10,0,0,5,0,0,f_1,f_2,0)$.
It follows that $c^{\ast}(\hat{\Delta}) \leq -\frac{45 \pi}{30} + \frac{45 \pi}{30}$.
Let $\hat{\Delta}$ have exactly two vertices $u_i,u_j$ of degree $>3$.
If $d(u_7)=3$ then $c^{\ast}(\hat{\Delta}) < 0$;
if $d(u_2)=d(u_8)=3$ then $cv(\hat{\Delta})=(10,10,0,0,6,3,2,f_1,f_2,0)$;
if $(i,j)=(7,2)$ then 
$cv(\hat{\Delta})=(10,10,4,2,5,0,0,f_1,f_2,0)$; and
if $(i,j)=(7,8)$ then 
$cv(\hat{\Delta})=(10,10,0,0,5,0,0,f_1,f_2,2)$.
It follows that $c^{\ast}(\hat{\Delta}) \leq -\frac{50 \pi}{30} + \frac{43 \pi}{30} < 0$.

\textbf{Case b4}\quad 
Let $\hat{\Delta}$ be given by Figure 51(iv).  Suppose that $\hat{\Delta}$ has exactly one vertex $u_i$ of degree $>3$.
If $d(u_4)=d(u_6)=d(u_8)=3$ then $cv(\hat{\Delta})=c(10,10,2,6,6,0,0,2,2,0)$;
if $i=4$ then (see Figure 36) $cv(\hat{\Delta})=(10,10,0,6,0,7,0,2,2,0)$;
if $i=6$ then $cv(\hat{\Delta})=(10,10,0,6,4,0,e_1,e_2,2,0)$; and
if $i=8$ then $cv(\hat{\Delta})=(10,10,0,6,4,0,0,2,e_1,e_2)$.
Therefore $c^{\ast}(\hat{\Delta}) \leq -\frac{45 \pi}{30} + \frac{43 \pi}{30}$.
Let $\hat{\Delta}$ have exactly two vertices $u_i,u_j$ of degree $>3$.
If $d(u_2)=d(u_6)=3$ then $cv(\hat{\Delta})=(10,10,0,6,d_1,d_2,0,2,e_1,e_2)$;
if $d(u_4)=d(u_8)=3$ then $cv(\hat{\Delta})=(10,10,2,d_1,d_2,0,e_1,e_2,2,0)$;
if $(i,j)=(2,4)$ then $cv(\hat{\Delta})=(10,10,2,6,0,7,0,2,2,0)$;
if $(i,j)=(2,8)$ then $cv(\hat{\Delta})=(10,10,2,6,4,0,0,2,e_1,e_2)$;
if $(i,j)=(6,4)$ then $cv(\hat{\Delta})=(10,10,0,6,0,7,e_1,e_2,2,0)$; and
if $(i,j)=(6,8)$ then
$cv(\hat{\Delta})=(10,10,0,6,4,0,e_1,e_2,e_1,e_2)$.
It follows that either $c^{\ast}(\hat{\Delta}) \leq - \frac{50 \pi}{30} + \frac{49 \pi}{30} < 0$ or $(i,j)=(6,8)$, but here either $d(u)=4$ and $\frac{\pi}{5}$ is
distributed from $\hat{\Delta}$ across the $(u_1,u_2)$ edge according to Configuration E of Figure 32(iii) and so $c^{\ast}(\hat{\Delta}) \leq
-\frac{50 \pi}{30} + \frac{53 \pi}{30} - \frac{\pi}{5} <0$ or $d(u) > 4$,
$cv(\hat{\Delta}) = (7,7,0,6,4,0,e_1,e_2,e_1,e_2)$ and so $c^{\ast}(\hat{\Delta}) \leq - \frac{50 \pi}{30} + \frac{46 \pi}{30} < 0$.
Let $\hat{\Delta}$ have exactly three vertices $u_i,u_j,u_k$ of degree $>3$.
If $d(u_4)=3$ then (see Figure 40(xiv),(xvii)) $cv(\hat{\Delta})=(10,10,2,d_1,d_2,0,e_1,e_2,e_1,e_2)$ and $c^{\ast}(\hat{\Delta}) \leq - \frac{55 \pi}{30} + \frac{54 \pi}{30}<0$;
or if $d(u_6)=3$ or $d(u_8)=3$ then $c^{\ast}(\hat{\Delta}) 
\leq - \frac{55 \pi}{30} + \frac{51 \pi}{30} < 0$; and
if $(i,j,k)=(4,6,8)$ then $cv(\hat{\Delta})=(10,10,0,6,0,7,e_1,e_2,e_1,e_2)$ and $c^{\ast}(\hat{\Delta}) \leq - \frac{55 \pi}{30} + \frac{55 \pi}{30}= 0$.

\textbf{Case b5}\quad 
Let $\hat{\Delta}$ be given by Figure 51(v).  Suppose that $\hat{\Delta}$ has exactly one vertex $u_i$ of degree $>3$.
If $d(u_2)=d(u_6)=3$ then $cv(\hat{\Delta})=(10,10,10,10,0,0,6,3,0,d_1,d_2,2)$;
if $i=2$ then $cv(\hat{\Delta})=(10,10,10,10,4,2,5,0,0,4,6,0)$; and
if $i=6$ then $cv(\hat{\Delta})=(10,10,10,10,0,0,5,0,7,0,6,0)$.
It follows that $c^{\ast}(\hat{\Delta}) \leq -\frac{65 \pi}{30} + \frac{61 \pi}{30} < 0$.
Let $\hat{\Delta}$ have exactly two vertices $u_i,u_j$ of degree $>3$.
If $d(u_2)=3$ or $d(u_6)=3$ then in each case $c^{\ast}(\hat{\Delta}) \leq - \frac{70 \pi}{30} + \frac{67 \pi}{30} < 0$; and
if $(i,j)=(2,6)$ then $cv(\hat{\Delta})=(10,10,10,10,4,2,5,0,7,0,6,0)$ and $c^{\ast}(\hat{\Delta}) < 0$.

\textbf{Case b6}\quad 
Let $\hat{\Delta}$ be given by Figure 51(vi).  Suppose that $\hat{\Delta}$ has exactly one vertex $u_i$ of degree $>3$.
If $d(u_4)=d(u_8)=3$ then $cv(\hat{\Delta})=(10,10,10,10,2,6,6,0,3,6,0,0)$;
if $i=4$ then $cv(\hat{\Delta})=(10,10,10,10,0,6,d_1,d_2,0,5,0,0)$; and
if $i=8$ then $cv(\hat{\Delta})=(10,10,10,10,0,6,4,0,0,5,2,4)$.
It follows that $c^{\ast}(\hat{\Delta}) \leq - \frac{65 \pi}{30} + \frac{63 \pi}{30} < 0$.
Let $\hat{\Delta}$ have exactly two vertices $u_i,u_j$ of degree $>3$.
If $d(u_8)=3$ then 
$cv(\hat{\Delta})=(10,10,10,10,2,6,d_1,d_2,3,6,0,0)$;
if $d(u_2)=d(u_4)=3$ then 
$cv(\hat{\Delta})=(10,10,10,10,0,6,6,0,3,d_1,d_2,4)$;
if $(i,j)=(8,2)$ then 
$cv(\hat{\Delta})=(10,10,10,10,2,6,4,0,0,5,2,4)$; and
if $(i,j)=(8,4)$ then 
$cv(\hat{\Delta})=(10,10,10,10,0,6,d_1,d_2,0,5,2,4)$.
It follows that $c^{\ast}(\hat{\Delta}) \leq - \frac{70 \pi}{30} + \frac{69 \pi}{30} < 0$.

\textbf{Case b7}\quad 
Let $\hat{\Delta}$ be given by Figure 51(vii).  Suppose that $\hat{\Delta}$ has exactly one vertex $u_i$ of degree $>3$.
If $d(u_2)=d(u_8)=3$ then $cv(\hat{\Delta})=(10,10,0,6,2,6,6,3,0,0)$;
if $i=2$ then $cv(\hat{\Delta})=(10,10,d_1,d_2,0,6,6,0,0,0)$; and
if $i=8$ then $cv(\hat{\Delta})=(10,10,0,6,0,6,6,0,2,4)$.
It follows that $c^{\ast}(\hat{\Delta}) \leq -\frac{45 \pi}{30} + \frac{44 \pi}{30} < 0$.
Let $\hat{\Delta}$ have exactly two vertices $u_i,u_j$ of degree $>3$.
If $d(u_8)=3$ then $cv(\hat{\Delta})=(10,10,d_1,d_2,2,6,6,3,0,0)$;
if $d(u_2)=d(u_7)=3$ then $cv(\hat{\Delta})=(10,10,0,6,2,6,6,3,2,4)$;
if $(i,j)=(8,2)$ then $cv(\hat{\Delta})=(10,10,d_1,d_2,0,6,6,0,2,4)$; and
if $(i,j)=(8,7)$ then $cv(\hat{\Delta})=(10,10,0,6,0,6,6,a_1,a_2,4)$.
It follows that $c^{\ast}(\hat{\Delta}) \leq -\frac{50 \pi}{30} + \frac{49 \pi}{30} < 0$.

\textbf{Case b8}\quad 
Let $\hat{\Delta}$ be given by Figure 51(viii).
Suppose that $\hat{\Delta}$ has exactly one vertex $u_i$ of degree $>3$.
If $d(u_3)=d(u_7)=3$ then $cv(\hat{\Delta})=(10,10,2,6,0,6,6,0,0,3)$;
if $i=3$ then $cv(\hat{\Delta})=(10,10,0,d_1,d_2,6,6,0,0,0)$; and
if $i=7$ then $cv(\hat{\Delta})=(10,10,0,6,0,6,6,4,2,0)$.
It follows that $c^{\ast}(\hat{\Delta}) \leq - \frac{45 \pi}{30} + \frac{44 \pi}{30} < 0$.
Let $\hat{\Delta}$ have exactly two vertices $u_i,u_j$ of degree $>3$.
If $d(u_7)=3$ then $cv(\hat{\Delta})=(10,10,2,d_1,d_2,6,6,0,0,3)$;
if $d(u_2)=d(u_3)=3$ then $cv(\hat{\Delta})=(10,10,0,6,0,6,6,4,a_1,a_2)$;
if $(i,j)=(7,2)$ then $cv(\hat{\Delta})=(10,10,2,6,0,6,6,4,2,0)$; and
if $(i,j)=(7,3)$ then $cv(\hat{\Delta})=(10,10,0,d_1,d_2,6,6,4,2,0)$.
It follows that $c^{\ast}(\hat{\Delta}) \leq - \frac{50 \pi}{30} + \frac{49 \pi}{30} < 0$.

\textbf{Case b9}\quad 
Let $\hat{\Delta}$ be given by Figure 51(ix).
Suppose that $\hat{\Delta}$ has exactly one vertex $u_i$ of degree $>3$.
If $d(u_2)=d(u_4)=3$ then $cv(\hat{\Delta})=(10,10,0,2,2,0,b_1,b_2,6,2)$;
if $i=2$ then $cv(\hat{\Delta})=(10,10,e_1,e_2,2,0,0,4,6,0)$; and
if $i=4$ then $cv(\hat{\Delta})=(10,10,0,2,e_1,e_2,0,4,6,0)$.
It follows that $c^{\ast}(\hat{\Delta}) \leq - \frac{45 \pi}{30} + \frac{43 \pi}{30} < 0$.

Let $\hat{\Delta}$ have exactly two vertices $u_i,u_j$ of degree $>3$.
If $d(u_2)=d(u_4)=3$ then $c^{\ast}(\hat{\Delta}) < 0$;
if $d(u_2)=d(u_8)=3$ then $cv(\hat{\Delta})=(10,10,0,2,e_1,e_2,b_1,b_2,6,0)$;
if $d(u_4)=d(u_8)=3$ then $cv(\hat{\Delta})=(10,10,e_1,e_2,2,0,b_1,b_2,6,0)$;
if $(i,j)=(2,4)$ then $cv(\hat{\Delta})=(10,10,e_1,e_2,e_1,e_2,0,4,6,0)$;
if $(i,j)=(2,8)$ then $cv(\hat{\Delta})=(10,10,e_1,e_2,2,0,0,4,6,2)$; and
if $(i,j)=(4,8)$ then $cv(\hat{\Delta})=(10,10,0,2,e_1,e_2,0,4,6,2)$.
It follows that $c^{\ast}(\hat{\Delta}) \leq - \frac{50 \pi}{30} + \frac{47 \pi}{30} < 0$ except for $(i,j)=(2,4)$, in which case either $d(u)=4$ and
$\frac{\pi}{5}$ is distributed from $\hat{\Delta}$ across the $(u_8,u_9)$ edge according to Configuration F of Figure 32(v) and
$c^{\ast}(\hat{\Delta}) \leq - \frac{50 \pi}{30} + \frac{52 \pi}{30} - \frac{\pi}{5} < 0$ or $d(u) > 4$,
$cv(\hat{\Delta})=(7,7,e_1,e_2,e_1,e_2,0,4,6,0)$ and $c^{\ast}(\hat{\Delta}) \leq - \frac{50 \pi}{30} + \frac{46 \pi}{30} < 0$.
Let $\hat{\Delta}$ have exactly three vertices $u_i,u_j,u_k$ of degree $>3$.
If $d(u_2)=3$ or $d(u_4)=3$ then $c^{\ast}(\hat{\Delta}) <0$;
if $d(u_6)=d(u_8)=3$ then $cv(\hat{\Delta})=(10,10,e_1,e_2,e_1,e_2,0,6,6,0)$;
if $(i,j,k)=(2,4,6)$ then $cv(\hat{\Delta})=(10,10,e_1,e_2,e_1,e_2,7,0,6,0)$; and
if $(i,j,k)=(2,4,8)$ then $cv(\hat{\Delta})=(10,10,e_1,e_2,e_1,e_2,0,4,6,2)$.
It follows that $c^{\ast}(\hat{\Delta}) \leq - \frac{55 \pi}{30} + \frac{55 \pi}{30} = 0$.

\textbf{Case b10}\quad 
Let $\hat{\Delta}$ be given by Figure 51(x).
Let $\hat{\Delta}$ have exactly one vertex $u_i$ of degree $>3$.
If $d(u_3)=d(u_6)=3$ then $cv(\hat{\Delta})=c(10,10,3,0,0,3,0,6,6,2)$;
if $i=3$ then $cv(\hat{\Delta})=(10,10,0,2,2,0,0,4,6,0)$; and
if $i=6$ then $cv(\hat{\Delta})=(10,10,0,0,0,0,7,0,6,0)$.
It follows that $c^{\ast}(\hat{\Delta}) \leq - \frac{45 \pi}{30} + \frac{40 \pi}{30} < 0$.

\textbf{Case b11}\quad
Let $\hat{\Delta}$ be given by Figure 51(xi).
Let $\hat{\Delta}$ have exactly one vertex $u_i$ of degree $>3$.
If $d(u_2)=d(u_4)=d(u_8)=3$ then $cv(\hat{\Delta})=(10,10,0,6,6,0,3,6,2,0)$;
if $i=2$ then $cv(\hat{\Delta})=(10,10,d_1,d_2,6,0,0,5,0,0)$;
if $i=4$ then $cv(\hat{\Delta})=(10,10,0,6,d_1,d_2,0,5,0,0)$; and
if $i=8$ then $cv(\hat{\Delta})=(10,10,0,6,6,0,0,5,2,4)$.
It follows that $c^{\ast}(\hat{\Delta}) \leq - \frac{45 \pi}{30} + \frac{43 \pi}{30} < 0$.
Let $\hat{\Delta}$ have exactly two vertices $u_i,u_j$ of degree $>3$.
If $d(u_2)=d(u_4)=3$ then $cv(\hat{\Delta})=(10,10,0,6,6,0,3,b_1,b_2,4)$;
if $d(u_2)=d(u_8)=3$ then $cv(\hat{\Delta})=(10,10,0,6,d_1,d_2,3,6,0,0)$;
if $d(u_4)=d(u_8)=3$ then $cv(\hat{\Delta})=(10,10,d_1,d_2,6,0,3,6,0,0)$;
if $(i,j)=(2,4)$ then $cv(\hat{\Delta})=(10,10,d_1,d_2,d_1,d_2,0,5,0,0)$;
if $(i,j)=(2,8)$ then $cv(\hat{\Delta})=(10,10,d_1,d_2,6,0,0,5,2,4)$; and
if $(i,j)=(4,8)$ then $cv(\hat{\Delta})=(10,10,0,6,d_1,d_2,0,5,2,4)$.
It follows that $c^{\ast}(\hat{\Delta}) \leq - \frac{50 \pi}{30} + \frac{47 \pi}{30} < 0$.

\textbf{Case b12}\quad 
Let $\hat{\Delta}$ be given by Figure 51(xii).  Suppose that $\hat{\Delta}$ has exactly one vertex $u_i$ of degree $>3$.
If $d(u_3)=3$ then $cv(\hat{\Delta})=(10,10,2,2,2,2,3,b_1,b_2,4)$; and
if $i=3$ then $cv(\hat{\Delta})=(10,10,0,f_1,f_2,0,0,5,0,0)$.
It follows that $c^{\ast}(\hat{\Delta}) \leq - \frac{45 \pi}{30} + \frac{43 \pi}{30} < 0$.
Let $\hat{\Delta}$ have exactly two vertices $u_3,u_j$ of degree $>3$.
If $d(u_7)=d(u_8)=3$ then $cv(\hat{\Delta})=(10,10,2,f_1,f_2,2,3,6,0,0)$;
if $j=7$ then $cv(\hat{\Delta})=(10,10,0,f_1,f_2,0,0,6,0,0)$; and
if $j=8$ then $cv(\hat{\Delta})=(10,10,0,f_1,f_2,0,0,5,2,4)$.
It follows that $c^{\ast}(\hat{\Delta}) \leq - \frac{50 \pi}{30} + \frac{45 \pi}{30} < 0$.

\textbf{Case b13}\quad 
Let $\hat{\Delta}$ be given by Figure 51(xiii).  Suppose that $\hat{\Delta}$ has exactly one vertex $u_i$ of degree $>3$.
If $d(u_2)=d(u_3)=d(u_8)=3$ then $cv(\hat{\Delta})=(10,10,0,0,0,6,6,d_1,d_2,0)$;
if $i=2$ then $cv(\hat{\Delta})=(10,10,3,0,0,6,6,0,6,0)$;
if $i=3$ then $cv(\hat{\Delta})=(10,10,0,2,4,6,6,0,6,0)$; and
if $i=8$ then $cv(\hat{\Delta})=(10,10,0,0,0,6,6,0,6,2)$.
It follows that $c^{\ast}(\hat{\Delta}) \leq - \frac{45 \pi}{30} + \frac{44 \pi}{30} < 0$.
Let $\hat{\Delta}$ have exactly two vertices $u_i,u_j$ of degree $>3$.
If $d(u_2)=3$ then $cv(\hat{\Delta})=(10,10,0,2,4,6,6,d_1,d_2,2)$;
if $d(u_3)=3$ then $cv(\hat{\Delta})=(10,10,3,0,0,6,6,d_1,d_2,2)$; and
if $(i,j)=(2,3)$ then $cv(\hat{\Delta})=(10,10,a_1,a_2,4,6,6,0,6,0)$.
It follows that $c^{\ast}(\hat{\Delta}) \leq - \frac{50 \pi}{30} + \frac{50 \pi}{30}=0$.

\begin{figure}
\begin{center}
\psfig{file=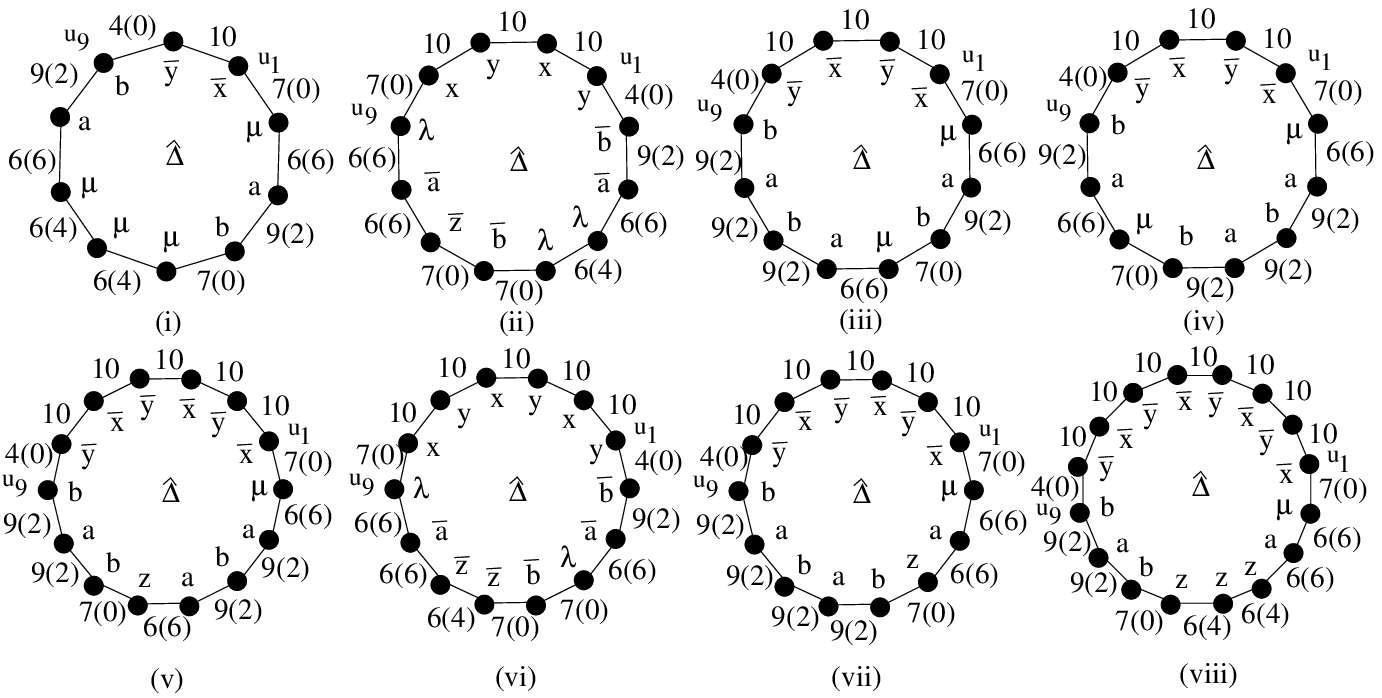}
\end{center}
\caption{}
\end{figure}

\textbf{Case b14}\quad
Let $\hat{\Delta}$ be given by Figure 51(xiv).  Suppose that $\hat{\Delta}$ has exactly one vertex $u_i$ of degree $>3$.
If $d(u_2)=d(u_8)=3$ then $cv(\hat{\Delta})=(10,10,0,6,2,6,6,3,0,0)$;
if $i=2$ then $cv(\hat{\Delta})=(10,10,d_1,d_2,0,6,6,0,0,0)$; and
if $i=8$ then $cv(\hat{\Delta})=(10,10,0,6,0,6,6,0,2,4)$.
It follows that $c^{\ast}(\hat{\Delta}) \leq - \frac{45 \pi}{30} + \frac{44 \pi}{30} < 0$.
Let $\hat{\Delta}$ have exactly two vertices $u_i,u_j$ of degree $>3$.
If $d(u_8)=3$ then $cv(\hat{\Delta})=(10,10,d_1,d_2,2,6,6,3,0,0)$;
if $d(u_6)=d(u_7)=3$ then $cv(\hat{\Delta})=(10,10,d_1,d_2,2,6,6,0,2,4)$;
if $(i,j)=(8,6)$ then $cv(\hat{\Delta})=(10,10,0,6,0,6,6,3,2,4)$; and if 
$(i,j)=(8,7)$ then
$cv(\hat{\Delta})=(10,10,0,6,0,6,6,a_1,a_2,4)$.
It follows that $cv(\hat{\Delta}) \leq - \frac{50 \pi}{30} + \frac{50 \pi}{30} = 0$.

\textbf{Case b15}\quad 
Let $\hat{\Delta}$ be given by Figure 51(xv).  Suppose that $\hat{\Delta}$ has exactly one vertex $u_i$ of degree $>3$.
If $d(u_2)=d(u_6)=3$ then $cv(\hat{\Delta})=(10,10,0,0,6,3,0,6,d_1,d_2)$;
if $i=2$ then $cv(\hat{\Delta})=(10,10,4,2,5,0,0,6,6,0)$; and
if $i=6$ then $cv(\hat{\Delta})=(10,10,0,0,5,0,d_1,d_2,6,0)$.
It follows that $c^{\ast}(\hat{\Delta}) \leq - \frac{45 \pi}{30} + \frac{45 \pi}{30} = 0$.
Let $\hat{\Delta}$ have exactly two vertices $u_i,u_j$ of degree $>3$.
If $d(u_2)=3$ then $cv(\hat{\Delta})=(10,10,0,0,6,3,d_1,d_2,d_1,d_2)$;
if $d(u_6)=d(u_8)=3$ then $cv(\hat{\Delta})=(10,10,4,d_1,d_2,3,0,6,6,0)$,
if $(i,j)=(2,6)$ then $cv(\hat{\Delta})=(10,10,4,2,5,0,d_1,d_2,6,0)$; and
if $(i,j)=(2,8)$ then $cv(\hat{\Delta})=(10,10,4,2,5,0,0,6,d_1,d_2)$.
It follows that $c^{\ast}(\hat{\Delta}) \leq -\frac{50 \pi}{30} + \frac{49 \pi}{30} < 0$.
Let $\hat{\Delta}$ have exactly three vertices of degree $>3$.
If $d(u_2)=3$ or $d(u_6)=3$ or $d(u_8)=3$ then $c^{\ast}(\hat{\Delta}) \leq - \frac{55 \pi}{30} + \frac{53 \pi}{30} < 0$; and
if $(i,j,k)=(2,6,8)$ then $cv(\hat{\Delta})=(10,10,4,2,5,0,d_1,d_2,d_1,d_2)$ and
$c^{\ast}(\hat{\Delta}) \leq - \frac{55 \pi}{30} + \frac{51 \pi}{30} < 0$.

\textbf{Case b16}\quad 
Let $\hat{\Delta}$ be given by Figure 51(xvi).  Suppose $\hat{\Delta}$ has exactly one vertex $u_i$ of degree $>3$.
If $d(u_4)=d(u_8)=3$ then $c^{\ast}(\hat{\Delta})=(10,10,2,6,6,0,a_1,a_2,2,0)$;
if $i=4$ then $cv(\hat{\Delta})=(10,10,0,6,d_1,d_2,0,0,0,0)$; and
if $i=8$ then $cv(\hat{\Delta})=(10,10,0,6,5,0,0,0,2,3)$.
It follows that $c^{\ast}(\hat{\Delta}) \leq - \frac{45 \pi}{30} + \frac{43 \pi}{30} < 0$.
Let $\hat{\Delta}$ have exactly two vertices $u_i,u_j$ of degree $>3$.
If $d(u_4)=3$ or $d(u_8)=3$ then $c^{\ast}(\hat{\Delta}) < 0$; and if 
$(i,j)=(4,8)$ then $cv(\hat{\Delta})=(10,10,0,6,d_1,d_2,a_1,a_2,a_1,a_2)$.
It follows that $c^{\ast}(\hat{\Delta}) \leq -\frac{50 \pi}{30} + \frac{50 \pi}{30}$.

Finally let $\hat{\Delta}$ be given by one of the regions of Figure 49.
It turns out that (up to cyclic permutations and inversion) there is one way to label each of Figure 49(i), (ii), (iii), (v) and (vi);
two ways to label (iv); five ways to label (vii) or (viii); three ways to label (ix) or (x); seven ways to label (xi) or (xii) or (xiii);
and seven ways to label (xiv) or (xv) or (xvi) or (xvii).  This yields a total of twenty-nine regions.  There are however several coincidences
amongst these regions resulting in $\hat{\Delta}$ being one of the eight regions given by Figure 52.
Table 5 gives $c(u_i,u_{i+1})$ ($1 \leq i \leq 9$) in multiples of $\pi / 30$ for each of the eight regions of Figure 52 with the total plus
the contribution via the $b$-segment in the final column.

We claim that $x_1+y_1+z_1=15$ in Table 5.  To see this let $\hat{\Delta}$ be given by Figure 52(i).
If $c(u_5,u_6)=0$ then $x_1+y_1+z_1=14$, so assume otherwise, in which case $c(u_4,u_5)=\frac{2 \pi}{15}$ (Figure 40(ix)).
If now $c(u_5,u_6)=\frac{2 \pi}{15}$ then $x_1+y_1+z_1=15$ by Lemma 9.2.  On the other hand
if $c(u_5,u_6) > \frac{2 \pi}{15}$ then $d(u_5)=3$ (see Figure 40) forcing $c(u_4,u_5)=\frac{\pi}{30}$ and $x_1+y_1+z_1=15$.
Note that we use here and below the fact that labelling prevents $\hat{\Delta}=\hat{\Delta}_2$ of Figure 38.
The arguments for $\hat{\Delta}$ of Figures 52(iii), (iv), (v), (vii) and (viii) are similar although for (v), (vii) and (viii) we use the fact, again
both here and below, that $\hat{\Delta} \neq \hat{\Delta}_2$ of Figure 37.

\begin{table}[h]
\[
\renewcommand{\arraystretch}{1.5}
\begin{array}{rllllllllll}
\textrm{(i)}&d_1&d_2&x_1&y_1&z_1&6&6&e_1&e_2&48 + 10 = 58\\
\textrm{(ii)}&e_1&e_2&6&d_1&d_2&d_1&d_2&d_1&d_2&47 + 30 = 77\\
\textrm{(iii)}&d_1&d_2&x_1&y_1&z_1&f_1&f_2&e_1&e_2&48 + 30 = 78\\
\textrm{(iv)}&d_1&d_2&f_1&f_2&x_1&y_1&z_1&e_1&e_2&48 + 30 = 78\\
\textrm{(v)}&d_1&d_2&f_1&f_2&x_1&y_1&z_1&e_1&e_2&48 + 50 = 98\\
\textrm{(vi)}&e_1&e_2&d_1&d_2&d_1&d_2&6&d_1&d_2&47 + 50 = 97\\
\textrm{(vii)}&d_1&d_2&x_1&y_1&z_1&f_1&f_2&e_1&e_2&48 + 50 = 98\\
\textrm{(viii)}&d_1&d_2&6&6&x_1&y_1&z_1&e_1&e_2&48 + 70 = 118
\end{array}
\renewcommand{\arraystretch}{1}
\]
\caption{}
\end{table}

Observe that $d(\hat{\Delta})=10$ in (i); $d(\hat{\Delta})=12$ in (ii)-(iv); $d(\hat{\Delta})=14$ in (v)-(vii); and $d(\hat{\Delta})=16$ in (viii).
It follows that if $\hat{\Delta}$ has at least four vertices of degree $>3$ then $c^{\ast}(\hat{\Delta}) \leq 0$.
If $\hat{\Delta}$ has no vertices of degree $>3$ then we see from Figure 52 that
$c^{\ast}(\hat{\Delta}) \leq -  7 \pi + \frac{18 \pi}{3} + \frac{24 \pi}{30} < 0$.

We deal with each of the eight \textit{c-cases}  in turn.

\textbf{Case c1}\quad 
Let $\hat{\Delta}$ be given by Figure 52(i).  Suppose $\hat{\Delta}$ has exactly one vertex $u_i$ of degree $>3$.
If $d(u_4)=d(u_9)=3$ then $cv(\hat{\Delta})=(10,d_1,d_2,2,0,6,6,6,2,0)$;
if $i=4$ then $cv(\hat{\Delta})=(10,0,6,c_1,c_2,4,4,6,2,0)$ (the $c_1,c_2$ follows from $\hat{\Delta} \neq \hat{\Delta}_2$ of Figure 38(iv))
; and if $i=9$ then $cv(\hat{\Delta})=(10,0,6,2,0,4,4,6,e_1,e_2)$.
It follows that $c^{\ast}(\hat{\Delta}) \leq -\frac{45 \pi}{30} + \frac{43 \pi}{30} < 0$.
Let $\hat{\Delta}$ have exactly two vertices $u_i,u_j$ of degree $>3$.
If $d(u_2)=d(u_4)=3$ then $cv(\hat{\Delta})=(10,0,6,2,0,6,6,6,e_1,e_2)$;
if $d(u_2)=d(u_9)=3$ then $cv(\hat{\Delta})=(10,0,6,x_1,y_1,z_1,6,6,2,0)$;
if $d(u_4)=d(u_9)=3$ then $cv(\hat{\Delta})=(10,d_1,d_2,2,0,6,6,6,2,0)$;
if $(i,j)=(2,4)$ then $cv(\hat{\Delta})=(10,d_1,d_2,c_1,c_2,4,4,6,2,0)$;
if $(i,j)=(2,9)$ then $cv(\hat{\Delta})=(10,d_1,d_2,2,0,4,4,6,e_1,e_2)$;
if $(i,j)=(4,9)$ then $cv(\hat{\Delta})=(10,0,6,c_1,c_2,4,4,6,e_1,e_2)$.
It follows that $c^{\ast}(\hat{\Delta}) \leq - \frac{50 \pi}{30} + \frac{50 \pi}{30} = 0$.
Let $\hat{\Delta}$ have exactly three vertices $u_i,u_j,u_k$ of degree $>3$.
If $d(u_4)=3$ or $d(u_9)=3$ then $c^{\ast}(\hat{\Delta}) < 0$;
if $d(u_2)=d(u_7)=3$ then (see Figure 40(xiv)) $cv(\hat{\Delta})=(10,0,6,x_1,y_1,z_1,4,6,e_1,e_2)$;
if $(i,j,k)=(4,9,2)$ then $cv(\hat{\Delta})=(10,d_1,d_2,c_1,c_2,4,4,6,e_1,e_2)$; and
if $(i,j,k)=(4,9,7)$ then $cv(\hat{\Delta})=(10,0,6,c_1,c_2,4,6,6,e_1,e_2)$.
It follows that $c^{\ast}(\hat{\Delta}) \leq - \frac{55 \pi}{30} + \frac{54 \pi}{30}<0$.

\textbf{Case c2}\quad 
Let $\hat{\Delta}$ be given by Figure 52(ii).
Suppose that $\hat{\Delta}$ has exactly one vertex $u_i$ of degree $>3$.
If $d(u_2)=d(u_9)=3$ then $cv(\hat{\Delta})=(10,10,10,0,2,6,d_1,d_2,d_1,d_2,6,0)$;
if $i=2$ then $cv(\hat{\Delta})=(10,10,10,e_1,e_2,6,4,0,0,6,6,0)$; and
if $i=9$ then $cv(\hat{\Delta})=(10,10,10,0,2,6,4,0,0,6,d_1,d_2)$.
It follows that $c^{\ast}(\hat{\Delta}) \leq - \frac{65 \pi}{30} + \frac{64 \pi}{30} < 0$.
Let $\hat{\Delta}$ have exactly two vertices $u_i,u_j$ of degree $>3$.
If $d(u_2)=3$ then
$cv(\hat{\Delta})=(10,10,10,0,2,6,d_1,d_2,d_1,d_2,d_1,d_2)$;
if $d(u_6)=3$ then
$cv(\hat{\Delta})=(10,10,10,e_1,e_2,6,6,0,0,6,d_1,d_2)$; and
if $(i,j)=(2,6)$ then
$cv(\hat{\Delta})=(10,10,10,e_1,e_2,6,4,2,2,6,6,0)$.
It follows that $c^{\ast}(\hat{\Delta}) \leq - \frac{70 \pi}{30} + \frac{69 \pi}{30}$.
Let $\hat{\Delta}$ have exactly three vertices $u_i,u_j,u_k$ of degree $>3$.
If $d(u_2)=3$ or $d(u_6)=3$ or $d(u_9)=3$ then $c^{\ast}(\hat{\Delta}) \leq - \frac{75 \pi}{30} + \frac{73 \pi}{30} < 0$; and
if $(i,j,k)=(2,6,9)$ then $cv(\hat{\Delta})=$
$(10,10,10,e_1,e_2,6,4,2,2,6,d_1,d_2)$ and
$c^{\ast}(\hat{\Delta}) \leq - \frac{75 \pi}{30} + \frac{71 \pi}{30} < 0$.

\textbf{Case c3}\quad 
Let $\hat{\Delta}$ be given by Figure 52(iii).  Suppose that $\hat{\Delta}$ has exactly one vertex $u_i$ of degree $>3$.
If $d(u_7)=d(u_9)=3$ then $cv(\hat{\Delta})=(10,10,10,d_1,d_2,x_1,y_1,z_1,2,2,2,0)$;
if $i=7$ then $cv(\hat{\Delta})=(10,10,10,0,6,2,0,6,f_1,f_2,2,0)$; and
if $i=9$ then $cv(\hat{\Delta})=(10,10,10,0,6,2,0,6,2,2,e_1,e_2)$.
It follows that $c^{\ast}(\hat{\Delta}) \leq - \frac{65 \pi}{30} + \frac{61 \pi}{30} < 0$.
Let $\hat{\Delta}$ have exactly two vertices $u_i,u_j$ of degree $>3$.
If $d(u_2)=d(u_7)=3$ then
$cv(\hat{\Delta})=(10,10,10,0,6,x_1,y_1,z_1,2,2,e_1,e_2)$;
if $d(u_9)=3$ then
$cv(\hat{\Delta})=(10,10,10,d_1,d_2,x_1,y_1,z_1,f_1,f_2,2,0)$;
if $(i,j)=(9,2)$ then $cv(\hat{\Delta})=(10,10,10,d_1,d_2,2,0,6,2,2,e_1,e_2)$; and
if $(i,j)=(9,7)$ then $cv(\hat{\Delta})=(10,10,10,0,6,2,0,6,f_1,f_2,e_1,e_2)$.
It follows that $c^{\ast}(\hat{\Delta}) \leq - \frac{70 \pi}{30} + \frac{69 \pi}{30} = 0$.
Let $\hat{\Delta}$ have exactly three vertices $u_i,u_j,u_k$ of degree $>3$.
If $d(u_2)=3$ or $d(u_7)=3$ or $d(u_9)=3$ then 
$c^{\ast}(\hat{\Delta}) \leq - \frac{75 \pi}{30} + \frac{74 \pi}{30} < 0$; and
if $(i,j,k)=(2,7,9)$ then $cv(\hat{\Delta})=(10,10,10,d_1,d_2,2,0,6,f_1,f_2,e_1,e_2)$ and
$c^{\ast}(\hat{\Delta}) \leq - \frac{75 \pi}{30} + \frac{71 \pi}{30} < 0$.

\textbf{Case c4}\quad 
Let $\hat{\Delta}$ be given by Figure 52(iv).  Suppose that $\hat{\Delta}$ has exactly one vertex $u_i$ of degree $>3$.
If $d(u_4)=d(u_9)=3$ then $c^{\ast} (\hat{\Delta})=(10,10,10,d_1,d_2,2,2,x_1,y_1,z_1,2,0)$;
if $i=4$ then $cv(\hat{\Delta})=(10,10,10,0,6,f_1,f_2,2,0,6,2,0)$; and
if $i=9$ then $c^{\ast}(\hat{\Delta})=(10,10,10,0,6,2,2,2,0,6,e_1,e_2)$.
It follows that $c^{\ast}(\hat{\Delta}) \leq - \frac{65 \pi}{30}+\frac{61 \pi}{30} < 0$.
Let $\hat{\Delta}$ have exactly two vertices $u_i,u_j$ of degree $>3$.
If $d(u_9)=3$ then $cv(\hat{\Delta})=(10,10,10,d_1,d_2,f_1,f_2,x_1,y_1,z_1,2,0)$;
if $d(u_2)=d(u_4)=3$ then $cv(\hat{\Delta})=(10,10,10,0,6,2,2,x_1,y_1,z_1,e_1,e_2)$;
if $(i,j)=(9,2)$
then $cv(\hat{\Delta})=(10,10,10,d_1,d_2,2,2,2,0,6,e_1,e_2)$; and if $(i,j)=(9,4)$ then 
$cv(\hat{\Delta})=(10,10,10,0,6,f_1,f_2,2,0,6,e_1,e_2)$.
It follows that $c^{\ast}(\hat{\Delta}) \leq - \frac{70 \pi}{30} + \frac{69 \pi}{30} = 0$.
Let $\hat{\Delta}$ have exactly three vertices $u_i,u_j,u_k$ of degree $>3$.
If $d(u_2)=3$ or $d(u_4)=3$ or $d(u_9)=3$ then $c^{\ast}(\hat{\Delta}) \leq - \frac{75 \pi}{30} + \frac{74 \pi}{30} < 0$; and
if $(i,j,k)=(2,4,9)$ then $cv(\hat{\Delta})=(10,10,10,d_1,d_2,f_1,f_2,2,0,6,e_1,e_2)$ and
$c^{\ast}(\hat{\Delta}) \leq - \frac{75 \pi}{30} + \frac{71 \pi}{30} < 0$.

\textbf{Case c5}\quad 
Let $\hat{\Delta}$ be given by Figure 52(v).  Suppose that $\hat{\Delta}$ has exactly one vertex $u_i$ of degree $>3$.
If $d(u_4)=d(u_9)=3$ then $cv(\hat{\Delta})=(10,10,10,10,10,d_1,d_2,2,2,x_1,y_1,z_1,2,0)$;
if $i=4$ then $cv(\hat{\Delta})=(10,10,10,10,10,0,6,f_1,f_2,6,0,2,2,0)$; and
if $i=9$ then $cv(\hat{\Delta})=(10,10,10,10,10,0,6,2,2,6,0,2,e_1,e_2)$.
It follows that $c^{\ast}(\hat{\Delta}) \leq - \frac{85 \pi}{30} + \frac{81 \pi}{30} < 0$.
Let $\hat{\Delta}$ have exactly two vertices $u_i,u_j$ of degree $>3$.
If $d(u_9)=3$ then 
$cv(\hat{\Delta})=(10,10,10,10,10,d_1,d_2,f_1,f_2,x_1,y_1,z_1,2,0)$;
if $d(u_2)=d(u_4)=3$ then
$cv(\hat{\Delta})=(10,10,10,10,10,0,6,2,2,x_1,y_1,z_1,e_1,e_2)$;
if $(i,j)=(9,2)$ then 
$cv(\hat{\Delta})=(10,10,10,10,10,d_1,d_2,2,2,6,0,2,e_1,e_2)$; and
if $(i,j)=(9,4)$ then 
$cv(\hat{\Delta})=(10,10,10,10,10,0,6,f_1,f_2,6,0,2,e_1,e_2)$.
It follows that $c^{\ast}(\hat{\Delta}) \leq - \frac{90 \pi}{30} + \frac{89 \pi}{30}=0$.
Let $\hat{\Delta}$ have exactly three vertices $u_i,u_j,u_k$ of degree $>3$.
If $d(u_2)=3$ or $d(u_4)=3$ or $d(u_9)=3$ then $c^{\ast}(\hat{\Delta}) \leq - \frac{95 \pi}{30} + \frac{94 \pi}{30} < 0$; and
if $(i,j,k)=(2,4,9)$ then $cv(\hat{\Delta}) = (10,10,10,10,10,d_1,d_2,f_1,f_2,6,0,2,e_1,e_2)$ and
$c^{\ast}(\hat{\Delta}) \leq - \frac{95 \pi}{30} + \frac{91 \pi}{30} < 0$.

\textbf{Case c6}\quad 
Let $\hat{\Delta}$ be given by Figure 52(vi).  Suppose that $\hat{\Delta}$ has exactly one vertex $u_i$ of degree $>3$.
If $d(u_2)=d(u_9)=3$ then $cv(\hat{\Delta})=(10,10,10,10,10,0,2,d_1,d_2,d_1,d_2,6,6,0)$;
if $i=2$ then $cv(\hat{\Delta})=(10,10,10,10,10,e_1,e_2,6,0,0,4,6,6,0)$; and
if $i=9$ then $cv(\hat{\Delta})=(10,10,10,10,10,0,2,6,0,0,4,6,d_1,d_2)$.
It follows that $c^{\ast}(\hat{\Delta}) \leq - \frac{85 \pi}{30} + \frac{84 \pi}{30} < 0$.
Let $\hat{\Delta}$ have exactly two vertices $u_i,u_j$ of degree $>3$.
If $d(u_2)=3$ then 
$cv(\hat{\Delta})=(10,10,10,10,10,0,2,d_1,d_2,d_1,d_2,6,d_1,d_2)$;
if $d(u_5)=3$ then 
$cv(\hat{\Delta})=(10,10,10,10,10,e_1,e_2,6,0,0,6,6,d_1,d_2)$;
and if $(i,j)=(2,5)$ then 
$cv(\hat{\Delta})=(10,10,10,10,10,e_1,e_2,6,2,2,4,6,6,0)$.
It follows that $c^{\ast}(\hat{\Delta}) \leq - \frac{90 \pi}{30} + \frac{89 \pi}{30}$.
Let $\hat{\Delta}$ have exactly three vertices $u_i,u_j,u_k$ of degree $>3$.
If $d(u_2)=3$ or $d(u_5)=3$ or $d(u_9)=3$ then $c^{\ast}(\hat{\Delta}) \leq - \frac{95 \pi}{30} + \frac{93 \pi}{30} < 0$; and
if $(i,j,k)=(2,5,9)$ then $cv(\hat{\Delta})=$
$(10,10,10,10,10,e_1,e_2,6,2,2,4,6,d_1,d_2)$ and
$c^{\ast}(\hat{\Delta}) \leq - \frac{95 \pi}{30} + \frac{91 \pi}{30} < 0$.

\textbf{Case c7}\quad 
Let $\hat{\Delta}$ be given by Figure 52(vii).  Suppose that $\hat{\Delta}$ has exactly one vertex $u_i$ of degree $>3$.
If $d(u_7)=d(u_9)=3$ then $cv(\hat{\Delta})=(10,10,10,10,10,d_1,d_2,x_1,y_1,z_1,2,2,2,0)$;
if $i=7$ then $cv(\hat{\Delta})=(10,10,10,10,10,0,6,6,0,2,f_1,f_2,2,0)$; and
if $i=9$ then $cv(\hat{\Delta})= (10,10,10,10,10,0,6,6,0,2,2,2,e_1,e_2)$.
It follows that $c^{\ast}(\hat{\Delta}) \leq - \frac{85 \pi}{30} + \frac{81 \pi}{30}<0$.
Let $\hat{\Delta}$ have exactly two vertices $u_i,u_j$ of degree $>3$.
If $d(u_9)=3$ then 
$cv(\hat{\Delta})=(10,10,10,10,10,d_1,d_2,x_1,y_1,z_1,f_1,f_2,2,0)$;
if $d(u_2)=d(u_7)=3$ then $cv(\hat{\Delta})=(10,10,10,10,10,0,6,x_1,y_1,z_1,2,2,e_1,e_2)$;
if $(i,j)=(9,2)$ then 
$cv(\hat{\Delta})=(10,10,10,10,10,d_1,d_2,6,0,2,2,2,e_1,e_2)$; and
if $(i,j)=(9,7)$ then 
$cv(\hat{\Delta})=(10,10,10,10,10,0,6,6,0,2,f_1,f_2,e_1,e_2)$.
It follows that $c^{\ast}(\hat{\Delta}) \leq - \frac{90 \pi}{30}+\frac{89 \pi}{30}=0$.
Let $\hat{\Delta}$ have exactly three vertices $u_i,u_j,u_k$ of degree $>3$.
If $d(u_2)=3$ or $d(u_7)=3$ or $d(u_9)=3$ then $c^{\ast}(\hat{\Delta}) \leq - \frac{95 \pi}{30} + \frac{94 \pi}{30} < 0$;
if $(i,j,k)=(2,7,9)$ then $cv(\hat{\Delta})=(10,10,10,10,10,d_1,d_2,6,0,2,f_1,f_2,e_1,e_2)$ and
$c^{\ast}(\hat{\Delta}) \leq - \frac{95 \pi}{30} + \frac{91 \pi}{30} < 0$.

\textbf{Case c8}\quad 
Let $\hat{\Delta}$ be given by Figure 52(viii).  Suppose that $\hat{\Delta}$ has exactly one vertex $u_i$ of degree $>3$.
If $d(u_7)=d(u_9)=3$ then $cv(\hat{\Delta})=(10,10,10,10,10,10,10,d_1,d_2,6,6,6,0,2,2,0)$; if $i=7$ then
$cv(\hat{\Delta})=(10,10,10,10,10,10,10,0,6,6,4,4,c_1,c_2,2,0)$ (the $c_1,c_2$ follows from $\hat{\Delta} \neq \hat{\Delta}_2$ of Figure 37(iv)); and
if $i=9$ then $cv(\hat{\Delta})=(10,10,10,10,10,10,10,0,6,6,4,4,0,2,e_1,e_2)$.
It follows that $c^{\ast}(\hat{\Delta}) \leq - \frac{105 \pi}{30} + \frac{103 \pi}{30} = 0$.
Let $\hat{\Delta}$ have exactly two vertices, $u_i,u_j$ of degree $>3$.
If $d(u_2)=d(u_7)=3$ then $cv(\hat{\Delta})=(10,10,10,10,10,10,10,0,6,6,6,6,0,2,e_1,e_2)$;
if $d(u_2)=d(u_9)=3$ then $cv(\hat{\Delta})=(10,10,10,10,10,10,10,0,6,6,6,x_1,y_1,z_1,2,0)$;
if $d(u_7)=d(u_9)=3$ then $c^{\ast}(\hat{\Delta}) \leq - \frac{110 \pi}{30} + \frac{102 \pi}{30}$;
if $(i,j)=(2,7)$ then 
$cv(\hat{\Delta}) = (10,10,10,10,10,10,10,d_1,d_2,6,4,4,c_1,c_2,2,0)$;
if $(i,j)=(2,9)$ then 
$cv(\hat{\Delta})=(10,10,10,10,10,10,10,d_1,d_2,6,4,4,0,2,e_1,e_2)$; and
if $(i,j)=(7,9)$ then
$cv(\hat{\Delta})=(10,10,10,10,10,10,10,0,6,6,4,4,c_1,c_2,e_1,e_2)$.
It follows that $c^{\ast}(\hat{\Delta}) \leq - \frac{110 \pi}{30} + \frac{110 \pi}{30} = 0$.
Let $\hat{\Delta}$ have exactly three vertices $u_i,u_j,u_k$ of degree $>3$.
If $d(u_7)=3$ or $d(u_9)=3$ then 
$c^{\ast}(\hat{\Delta}) \leq - \frac{115 \pi}{30} + \frac{111 \pi}{30} < 0$;
if $d(u_2)=d(u_4)=3$ then 
$cv(\hat{\Delta})=(10,10,10,10,10,10,10,0,6,6,4,x_1,y_1,z_1,e_1,e_2)$;
if $(i,j,k)=(7,9,2)$ then $cv(\hat{\Delta})=(10,10,10,10,10,10,10,d_1,d_2,6,4,4,c_1,c_2,e_1,e_2)$; and
if $(i,j,k)=(7,9,4)$ then $cv(\hat{\Delta})=(10,10,10,10,10,10,10,0,6,6,6,4,c_1,c_2,e_1,e_2)$.
It follows that $cv(\hat{\Delta}) \leq - \frac{115 \pi}{30} + \frac{114 \pi}{30} < 0$. $\Box$

\medskip

\newpage
\begin{figure}
\begin{center}
\psfig{file=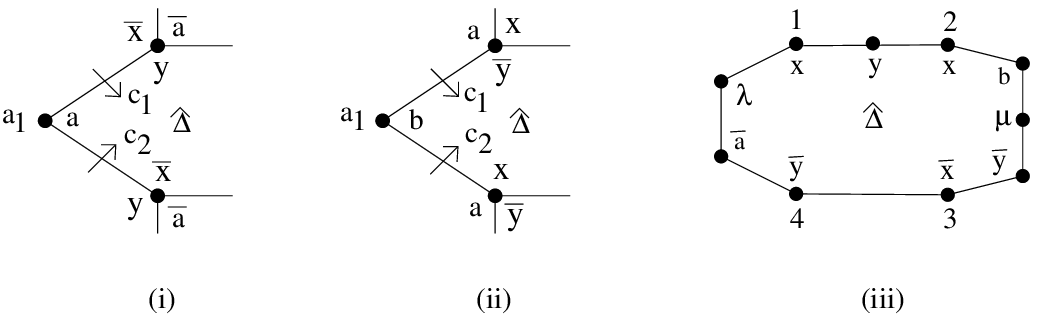}
\end{center}
\caption{}
\end{figure}

\begin{figure}
\begin{center}   
\psfig{file=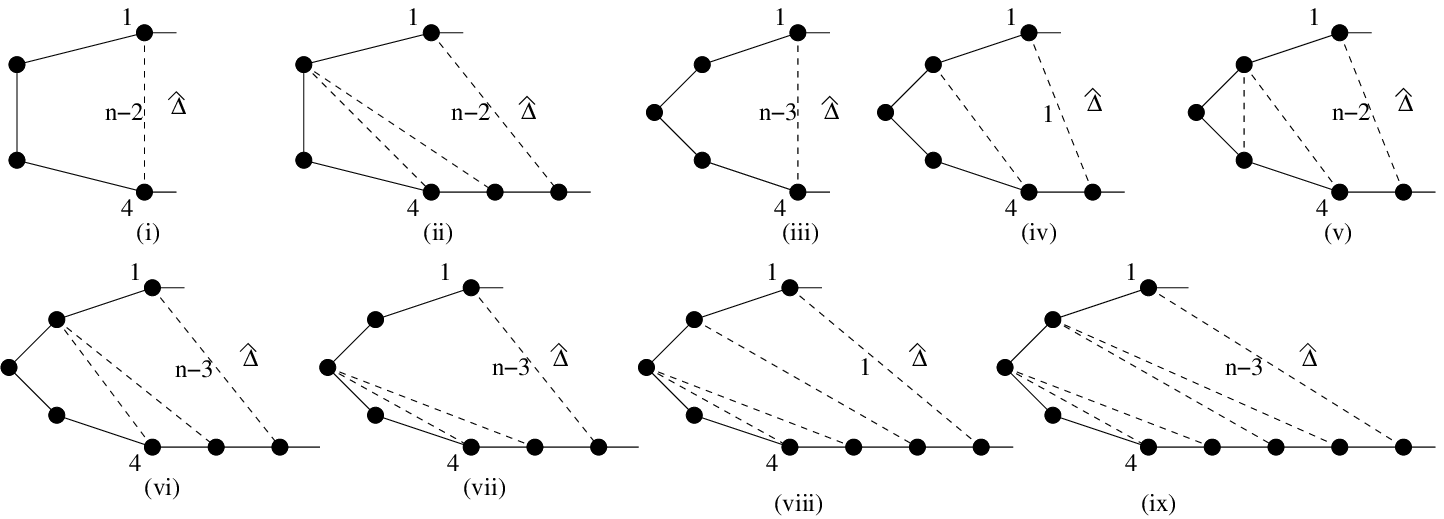}
\end{center}
\caption{}
\end{figure}

The next result completes the proof of Proposition 4.3.

\begin{proposition}
\textit{If $\hat{\Delta}$ is a type $\mathcal{B}$ region and $d(\hat{\Delta}) \geq 10$ then $c^{\ast}(\hat{\Delta}) \leq 0$.}
\end{proposition}

\textit{Proof}.
It can be assumed that $d(\hat{\Delta}) \geq 10$ and that $\hat{\Delta}$ is not one of the regions of Figures 47(ii)-(v), 48 or 49, otherwise
Proposition 11.4 applies.
Moreover if $n_2 \geq 10$ then $c^{\ast} (\hat{\Delta}) \leq 0$ so assume that $n_2 \leq 9$.  It follows from the proof of Lemma 11.1 that the upper bound ($\dag$) immediately preceeding Lemma 11.1 is reduced by at least $\frac{2 
\pi}{15}$ for each gap between two 
$b$-segments that contain $b$-regions so if there are at least three such $b$-segments then $c^{\ast}(\hat{\Delta}) \leq \pi (2 - \frac{n_2}{5}) - 3 \left( \frac{2 \pi}{15}\right)$ implying $c^{\ast}(\hat{\Delta}) \leq 0$ for
$n_2 \geq 8$.  Since there are at least two edges between $b$-segments it follows that if $\hat{\Delta}$ contains more than three such $b$-segments
then $c^{\ast}(\hat{\Delta}) \leq 0$ or if exactly three then $n_2 \geq 8$ by Lemma 11.2(i) and again $c^{\ast}(\hat{\Delta}) \leq 0$.  If
$\hat{\Delta}$ has exactly one $b$-segment that contains 
a $b$-region then $c^{\ast}(\hat{\Delta}) \leq 0$ by Proposition 11.4 together with Lemma
11.2(v), (vi) so suppose from
now on that $\hat{\Delta}$ contains exactly two such segments.  Then $c^{\ast}(\hat{\Delta}) \leq \pi 
\left( 2 -
\frac{n_2}{5} \right)-2 \left( \frac{2 \pi}{15} \right)$ which implies $c^{\ast} (\hat{\Delta}) \leq 0$ for $n_2 \geq 9$, so assume $n_2 \leq 8$ in
which case $\hat{\Delta}$ is given by Figure 47(i) where
$(m,n) \in \{ (2,2),(2,3),(2,4),(2,5),(2,6),(3,3),(3,4),(3,5),(4,4)\}$. Applying Proposition 11.4 and Lemma 11.2(ii) shows it can be assumed
that there is at least one shadow edge in $\hat{\Delta}$ between the two $b$-segments.

\newpage
\begin{figure}
\begin{center}   
\psfig{file=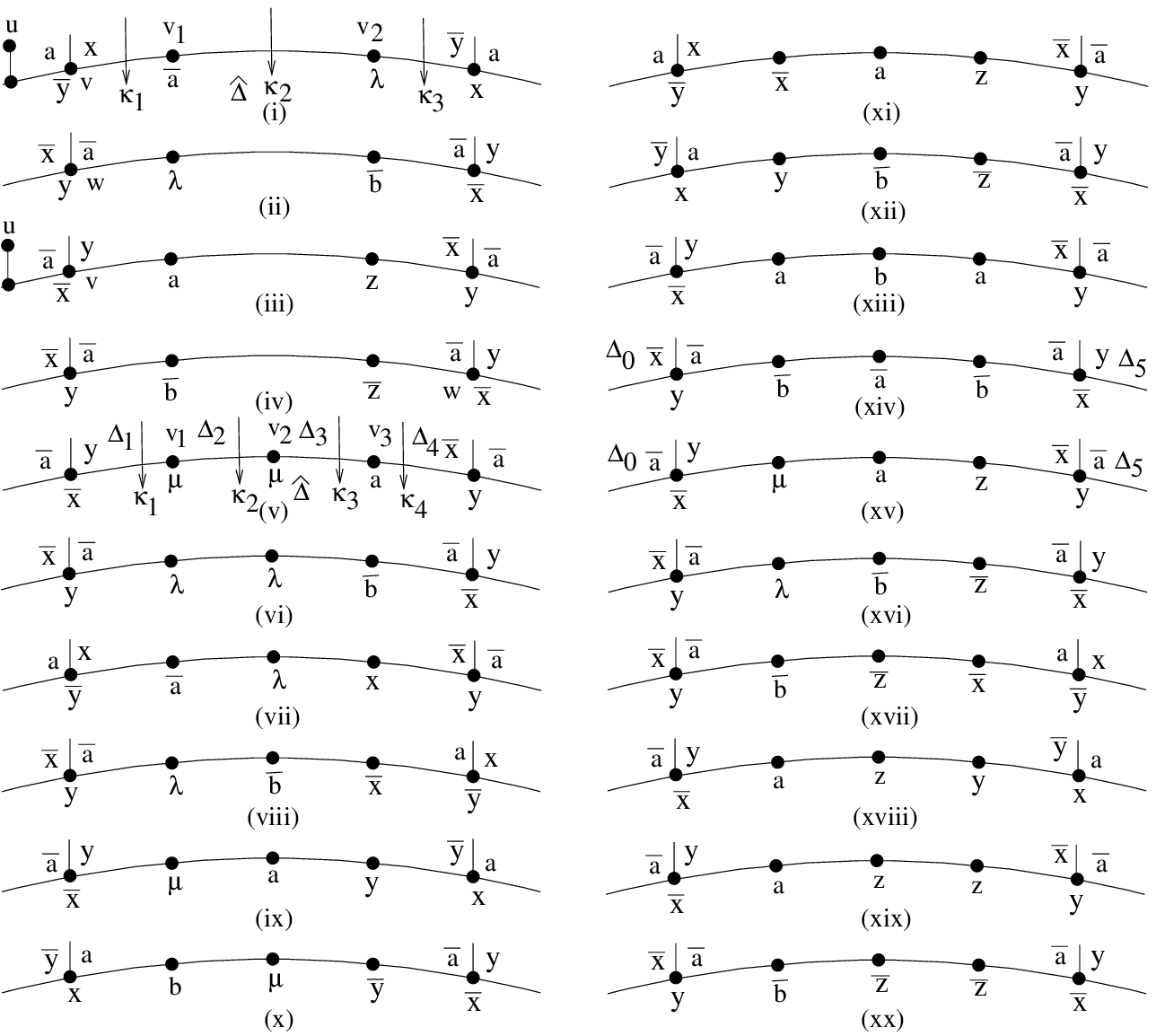}
\end{center}
\caption{}
\end{figure}

Let $m=2$.  It follows from the statement at the end of the above paragraph that $\hat{\Delta}$ contains the shadow edge (14) (of length $n-1$) and $\hat{\Delta}$ is
given by Figure 53(i)-(ii).  If $(m,n) \neq (2,6)$ 
then $i \, \textrm{deg} (1) = i \, \textrm{deg}(4)=1$ by Lemma 11.2(iii) and this leads to a length
contradiction so let $(m,n)=(2,6)$.  We claim that there is a reduction to ($\dag$) of $\frac{4 \pi}{15}$ between vertices 1 and 4.
Given this and the fact that there is a reduction of $\frac{2 \pi}{15}$ between 2 and 3 we obtain $c^{\ast}(\hat{\Delta}) \leq \pi \left( 2 -
\frac{n_2}{5} \right) - \frac{6 \pi}{15}$ and $c^{\ast} (\hat{\Delta}) \leq 0$ for $n_2 \geq 8$, in particular when $(m,n)=(2,6)$.  To prove the
claim observe that if $d(a_1)=3$ in Figure 53(i) or (ii) then $c_1=c_2=0$; and if $d(a_1) \geq 4$ then $c_1 + c_2 \leq \frac{2 \pi}{15}$ (see Figure
35).  In the first instance there is a deficit of at least $\left( \frac{2 \pi}{3} + 2 \left( \frac{2 \pi}{15} \right)\right) - \frac{2 \pi}{3} =
\frac{4 \pi}{15}$; and in the second case the deficit is at least $\left( \frac{2 \pi}{3} + 2 \left( \frac{2 \pi}{15} \right) \right) - \left( \frac{2 \pi}{4} + \frac{2 \pi}{15} \right) = \frac{3 \pi}{10}$.

Let $m=3$ or 4.  Applying Lemma 11.2(ii)-(iv) and Proposition 11.4 it can be assumed that $\hat{\Delta}$ is given by Figure 54 with the understanding
that the segment of $\hat{\Delta}$ between vertices 2 and 3 is also one of these nine possibilities.
(Note that in Figure 54 the length of the shadow edge incident at vertex 1 is shown.)
We claim that if $m=3$ then the edges between 1 and 4 produce a deficit of at least $\frac{2 \pi}{5}$; and if $m=4$ then the reduction is at least $\frac{\pi}{5}$.  Given this, if $(m,n)=(3,3)$ then
$c^{\ast} (\hat{\Delta}) \leq \pi \left( 2 - \frac{6}{5} \right) - \frac{4 \pi}{5} = 0$;
if $(m,n)=(3,4)$ then $c^{\ast} (\hat{\Delta}) \leq \pi \left( 2 - \frac{7}{5} \right) - \frac{3 \pi}{5} = 0$; if $(m,n)=(3,5)$ then
$c^{\ast}(\hat{\Delta}) \leq \pi \left( 2 - \frac{8}{5} \right) - \left( \frac{2 \pi}{5} + \frac{2 \pi}{15} \right) < 0$; and if
$(m,n)=(4,4)$ then $c^{\ast} (\hat{\Delta}) \leq \pi \left( 2 - \frac{8}{5} \right) - 2 \left( \frac{\pi}{5} \right) = 0$, so it remains to prove the
claim for the possible labellings
of the regions of Figure 54 and these are shown in Figure 55(i)-(xx).
Indeed there are four ways to label each of Figure 54(iv) and (vi); and two ways to label each of the others.
However the labelling obtained from Figure 54(vii) already appears in Figure 54(vi).

Let $m=3$.  Then Tables 6-9 give maximum values for  $\kappa_1,\kappa_2$ and $\kappa_3$ of Figure 55(i)-(iii) as multiples of $\frac{\pi}{30}$. Also indicated
in each case as a multiple of $\frac{\pi}{30}$ is the deficit 
$= \pi \left( \frac{2}{3} - \frac{2}{d(v_1)} \right) + \pi \left( \frac{2}{3} - \frac{2}{d(v_2)} \right) + \left( \frac{\pi}{30}\right)(12 - (\kappa_1 + \kappa_2 + \kappa_3))$.
The entries in each final column show that the deficit in each case is $\frac{2 \pi}{5}$, as
required, except for $d(v_1)=d(v_2)=3$ and $d(u) > 4$ in Figure 55(i) and we consider this below.
Note that in Table 6,8  when $d(u)=4$ in Figure 55(i),(iii) and either $d(v_1)=d(v_2)=3$ or $d(v_1)=3$, $d(v_2)=4$, $\frac{\pi}{5}$ is distributed from $\hat{\Delta}$
according to Configuration E, F of Figure 32(iii), (v) resulting in deficits of 12 and 16 as shown.
Note that in Table 7 when $d(v_1)=3$ and $d(v_2)=4$ the region $\hat{\Delta}$ cannot be $\hat{\Delta}_2$ of Figure 38(iv) because $d(w)=3$ in Figure 55(ii)
but the corresponding vertex in Figure 38(ii) has degree 4;and so $\kappa_2 = 1$ by Figure 36(x).  Similarly in Table 9 when $d(v_1)=4$ and $d(v_2)=(3)$ the region $\hat{\Delta}$ cannot be
$\hat{\Delta}_2$ of Figure 37(iv);and so $\kappa_2 = 1$ by Figure 36(i).

\begin{table}[h]
\[
\renewcommand{\arraystretch}{1.0}
\begin{array}{rllllllllll}
&d(v_1)&d(v_2)&\kappa_1&\kappa_2&\kappa_3&deficit&\\
\textrm{(i)}&3&3&0&6&0&12 (9)&(Note)\\
\textrm{(i)}&4&3&2&0&0&15&\\
\textrm{(i)}&3&4&0&0&7&16 (13)&(Note)\\
\textrm{(i)}&5+&3&2&2&0&16&\\
\textrm{(i)}&3&5+&0&2&2&16&\\
\textrm{(i)}&4&4&2&0&7&13&\\
\textrm{(i)}&4&5+&2&4&2&17&\\
\textrm{(i)}&5+&4&2&4&7&12&\\
\textrm{(i)}&5+&5+&2&2&2&22&
\end{array}
\renewcommand{\arraystretch}{1}
\]
\caption{}
\end{table}   

\begin{table}[h]
\[
\renewcommand{\arraystretch}{1.0}
\begin{array}{rllllllllll}
&d(v_1)&d(v_2)&\kappa_1&\kappa_2&\kappa_3&deficit&\\
\textrm{(ii)}&3&3&0&0&0&12&\\
\textrm{(ii)}&4&3&0&0&0&17&\\
\textrm{(ii)}&3&4&0&1&4&12&(Note)\\
\textrm{(ii)}&5+&3&2&2&0&16&\\
\textrm{(ii)}&3&5+&0&2&2&16&\\
\textrm{(ii)}&4&4&0&7&0&15&\\
\textrm{(ii)}&4&4&0&0&4&18&\\
\textrm{(ii)}&4&5+&0&4&2&19&\\
\textrm{(ii)}&5+&4&2&4&4&15&\\ 
\textrm{(ii)}&5+&5+&2&2&2&22&
\end{array}
\renewcommand{\arraystretch}{1}
\]
\caption{}
\end{table}

\begin{table}[h]
\[
\renewcommand{\arraystretch}{1.0}
\begin{array}{rllllllllll}
&d(v_1)&d(v_2)&\kappa_1&\kappa_2&\kappa_3&deficit&\\
\textrm{(iii)}&3&3&0&6&0&12&(Note)\\
\textrm{(iii)}&4&3&2&0&0&15&\\
\textrm{(iii)}&3&4&0&0&7&16&(Note)\\
\textrm{(iii)}&5+&3&2&2&0&16&\\
\textrm{(iii)}&3&5+&0&2&2&16&\\
\textrm{(iii)}&4&4&2&0&7&13&\\ 
\textrm{(iii)}&4&5+&2&4&2&17&\\
\textrm{(iii)}&5+&4&2&4&7&12&\\
\textrm{(iii)}&5+&5+&2&2&2&22& 
\end{array}
\renewcommand{\arraystretch}{1}
\]
\caption{}
\end{table}

\begin{table}[h]
\[
\renewcommand{\arraystretch}{1.0}
\begin{array}{rllllllllll}
&d(v_1)&d(v_2)&\kappa_1&\kappa_2&\kappa_3&deficit&\\
\textrm{(iv)}&3&3&0&0&0&12&\\
\textrm{(iv)}&4&3&4&1&0&12&(Note)\\
\textrm{(iv)}&3&4&0&0&0&17&\\
\textrm{(iv)}&5+&3&2&2&0&16&\\
\textrm{(iv)}&3&5+&0&2&2&16&\\
\textrm{(iv)}&4&4&4&0&0&18&\\ 
\textrm{(iv)}&4&4&0&7&0&15&\\ 
\textrm{(iv)}&4&5+&4&4&2&15&\\
\textrm{(iv)}&5+&4&2&4&0&19&\\
\textrm{(iv)}&5+&5+&2&4&2&20& 
\end{array}
\renewcommand{\arraystretch}{1}
\]
\caption{}
\end{table}

\begin{table}[h]
\[
\renewcommand{\arraystretch}{1.0}
\begin{array}{rlllllrlllll}
&\kappa_1&\kappa_2&\kappa_3&\kappa_4&&&\kappa_1&\kappa_2&\kappa_3&\kappa_4\\
\textrm{(v)}&b_1&b_2&6&2&16&\textrm{(xiii)}&x_1&x_2&x_3&x_4&18\\
\textrm{(vi)}&3&d_1&d_2&4&17&\textrm{(xiv)}&e_1&e_2&e_1&e_2&22\\
\textrm{(vii)}&2&6&7&2&17&\textrm{(xv)}&y_1&y_2&y_3&y_4&17\\
\textrm{(viii)}&3&7&4&2&16&\textrm{(xvi)}&a_1&a_2&a_1&a_2&14\\
\textrm{(ix)}&7&6&2&2&17&\textrm{(xvii)}&4&7&3&2&16\\
\textrm{(x)}&4&7&3&2&16&\textrm{(xviii)}&2&6&7&2&17\\
\textrm{(xi)}&2&2&6&7&17&\textrm{(xix)}&2&6&d_1&d_2&18\\
\textrm{(xii)}&2&4&7&3&16&\textrm{(xx)}&4&d_1&d_2&3&17
\end{array}
\renewcommand{\arraystretch}{1} 
\]
\caption{}
\end{table}

Suppose that $d(u) > 4$ in Figure 55(i) in which the vertex $v$ corresponds to the vertex 4 of Figure 54(i).
If there are at least two regions in the $b$-segment between vertices 4 and 3 then
$2 \left( \frac{10 \pi}{30} - \frac{7 \pi}{30} \right) = \frac{\pi}{5}$ is contributed to the deficit and so we obtain the totals $\frac{12 \pi}{30}$ when $d(v_1)=d(v_2)=3$
and $\frac{16 \pi}{30}$ when $d(v_1)=3$, $d(v_2)=4$ as shown in Table 6.
If however there is exactly one region in the $b$-segment then only $\frac{10 \pi}{30} - \frac{7 \pi}{30} = \frac{\pi}{10}$ is contributed to the deficit and so the total is 
$\frac{9 \pi}{30}$ when $d(v_1)=d(v_2)=3$ and
$\frac{13 \pi}{30}$ when $d(v_1)=3$, $d(v_2)=4$ as shown in brackets in Table 6.
If $(m,n)=(3,5)$ then $c^{\ast} (\hat{\Delta}) \leq \pi (2 - \frac{8}{5}) - \left( \frac{9 \pi}{30} + \frac{2 \pi}{15} \right) < 0$ so it can be
assumed that $n \in \{ 3,4 \}$.  But given that there are no vertices between 4 and 3, it follows immediately from length considerations that (i) of
Figure 54 can only be combined with (iv) or (viii), and so, in particular, $n=4$.  Any attempt at labelling shows that (i) with (viii) is impossible
and the unique region $\hat{\Delta}$ obtained from (i) with (iv) is given by Figure 53(iii) in which the segment of vertices from 2 to 3 corresponds to
Figure 55(x).  We show below that for Figure 55(x), the deficit is at least $\frac{\pi}{3}$ and so $c^{\ast} (\hat{\Delta}) \leq \pi \left( 2 -
\frac{7}{5} \right) -  \frac{19 \pi}{30} < 0$.
If $d(u) > 4$ in Figure 55(iii) in which the vertex $v$ corresponds to the vertex 4 of Figure 54(ii), then since there are at least two regions in
the $b$-segment between 4 and 3 it follows that, as in the above case, the total deficit is 12 and 16 as shown in Table 8.

Now let $m=4$ and consider Figure 55. (Recall that it remains to show that there is a deficit of at least $\frac{\pi}{5}$ in all cases except for Figure 55(x)
where we must show that there is a deficit of at least $\frac{\pi}{3}$.) 
Checking Figures 35-38, 40 and 41 and Lemma 9.1 shows that $\kappa_1+ \kappa_2+ \kappa_3+ \kappa_4 \leq \frac{11 \pi}{15}$ for (xiv);
and $\kappa_1 + \kappa_2 + \kappa_3 + \kappa_4 \leq \frac{9 \pi}{15}$ in all other cases.  Indeed the upper bounds are shown in Table 10.
Note that $\kappa_4 \leq 2$ in (vii)--(x), (xvii) and (xviii) follows from the fact that $d(v_3) \geq 4$ and if $d(v_3)= 4$ then $\kappa_4 = 0$;
$\kappa_1 \leq 2$ in (xi) and (xii) follows from the fact that $d(v_1) \geq 4$ and if $d(v_1)= 4$ then $\kappa_1 = 0$.
Note further that in (v) $\kappa_1 > 4$ implies $\kappa_2 = 0$ and $\kappa_2 > 4$ implies $\kappa_1 = 0$; 
in (vi) that $\kappa_3 > 4$ implies $\kappa_2 = 0$; 
that $x_1 + x_2 + x_3 + x_4 \leq 18$ in (xiii) follows from the fact that $d(v_1)=3$ implies $\kappa_1=0$, $d(v_1) > 3$ implies $\kappa_2=5$, $d(v_3)=3$ implies $\kappa_4=0$ and $d(v_3) > 3$ implies $\kappa_3=5$;
that in (xiv) $\kappa_2 = 9$ or $8$ forces $\kappa_1 = 0$ or $2$, that $\kappa_2 \leq 4$ and that similar statements hold for $\kappa_3$ and $\kappa_4$;  
in (xv) the fact that $\kappa_1 > 4$ implies $\kappa_2=0$, $\kappa_2 > 4$ implies $\kappa_1=0$ or $\kappa_3=0$, $\kappa_3 >4$ implies $\kappa_2=0$ or $\kappa_4=0$ and $\kappa_4 >4$ implies $\kappa_3=0$ forces 
$y_1 + y_2 + y_3 + y_4 \leq 17$; in (xvi) that $\kappa_2 > 4$ implies $\kappa_1 = 0$ and $\kappa_3 > 4$ implies $\kappa_4 = 0$;
in (xix) $\kappa_4 >4$ implies $\kappa_3=0$; and in (xx) $\kappa_2 >4$ implies $\kappa_3=0$. All other numerical entries for the upper bounds in Table 10 can be read directly from Figure 35.

\newpage
\begin{table}[h]
\[
\renewcommand{\arraystretch}{1.0}
\begin{array}{rllllllllll}
&d(v_1)&d(v_2)&d(v_3)&\kappa_1&\kappa_2&\kappa_3&\kappa_4&deficit\\
\textrm{(v)}&3&3&3&0&4&6&0&6\\   
\textrm{(v)}&4&3&3&7&0&6&0&8\\   
\textrm{(v)}&3&4&3&0&0&0&0&21\\ 
\textrm{(v)}&3&3&4&0&4&0&2&15\\ 
\textrm{(vi)}&3&3&3&0&5&0&0&11\\
\textrm{(vi)}&4&3&3&0&0&0&0&21\\
\textrm{(vi)}&3&4&3&0&0&0&0&21\\
\textrm{(vi)}&3&3&4&0&5&2&4&10\\
\textrm{(xiii)}&3&3&3&0&2&2&0&12\\
\textrm{(xiii)}&4&3&3&2&0&2&0&17\\
\textrm{(xiii)}&3&4&3&0&9&0&0&12\\
\textrm{(xiii)}&3&4&3&0&0&9&0&12\\
\textrm{(xiii)}&3&3&4&0&2&0&2&17\\
\textrm{(xv)}&3&3&3&0&6&6&0&6&(Note)\\
\textrm{(xv)}&4&3&3&7&0&6&0&8&\\
\textrm{(xv)}&3&4&3&0&0&0&0&21&\\
\textrm{(xv)}&3&3&4&0&6&0&7&8&\\
\textrm{(xvi)}&3&3&3&0&0&0&0&16\\
\textrm{(xvi)}&4&3&3&0&0&0&0&21\\
\textrm{(xvi)}&3&4&3&0&2&2&0&17\\
\textrm{(xvi)}&3&3&4&0&0&0&0&21\\
\textrm{(xix)}&3&3&3&0&6&4&0&6\\
\textrm{(xix)}&4&3&3&2&0&4&0&15\\
\textrm{(xix)}&3&4&3&0&0&0&0&21\\
\textrm{(xix)}&3&3&4&0&6&0&7&8\\
\textrm{(xx)}&3&3&3&0&0&5&0&11\\
\textrm{(xx)}&4&3&3&4&2&5&0&10\\
\textrm{(xx)}&3&4&3&0&0&0&0&21\\
\textrm{(xx)}&3&3&4&0&0&0&0&21
\end{array}
\renewcommand{\arraystretch}{1}  
\]
\caption{}
\end{table}

In the cases  where $\kappa_1 + \kappa_2 + \kappa_3 + \kappa_4 \leq \frac{9 \pi}{15}$ if at least one of $v_1,v_2,v_3$ has degree $\geq 5$ then there is a deficit of at least
$\left( \frac{2 \pi}{3} + \frac{8 \pi}{15} \right) - \left( \frac{2 \pi}{5} + \frac{9 \pi}{15} \right) = \frac{\pi}{5}$; and if at least two have degree $\geq 4$ then the deficit is at 
least $\left( \frac{4 \pi}{3} + \frac{8 \pi}{15} \right) - \left( \pi + \frac{9 \pi}{15} \right) = \frac{4 \pi}{15}$, so it can be assumed that $(d(v_1),d(v_2),d(v_3)) \in \{ (3,3,3),(4,3,3),(3,4,3),(3,3,4) \}$.  
For cases (vii)--(x) and (xvii)--(xviii), $d(v_3)=4$ which forces $\kappa_4 = 0$ and so  $\kappa_1+\kappa_2+\kappa_3+\kappa_4 \leq \frac{\pi}{2}$ which gives a deficit of at least $\frac{\pi}{5}$.
For (xi) and (xii), $d(v_1)=4$ which forces $\kappa_1 = 0$ and so $\kappa_1+\kappa_2+\kappa_3+\kappa_4 \leq \frac{\pi}{2}$ which gives a deficit of at least $\frac{\pi}{5}$.
In fact for case (x), $d(v_3) \geq 4$ and so if at least one of $v_1$ or $v_2$ has degree $\geq 4$ then the deficit is at least $\frac{\pi}{3}$;
or if $d(v_1)=d(v_2)=3$ then $\kappa_1 = \kappa_2 = 0$ and we see from Table 10 that the deficit is at least $\frac{8 \pi}{15}$, as required.

Table 11 shows the deficit for (v), (vi), (xiii), (xv), (xvi), (xix) and (xx).

As can be seen from Table 11 for these cases it remains to explain the first row for (xv).
Consider (xv) with $d(v_1)=d(v_2)=d(v_3)=3$.
Then $\kappa_1=\kappa_4=0$, $\kappa_2 \leq \frac{\pi}{5}$ and $\kappa_3 \leq \frac{\pi}{5}$.
If $\kappa_1+\kappa_2 \leq \frac{\pi}{3}$ then the deficit is at least $\frac{\pi}{5}$ so assume otherwise.
If $\hat{\Delta}$ receives less than $\frac{\pi}{5}$ from each of $\Delta_2$ and $\Delta_3$ then
deficit $\geq \frac{\pi}{5}$, so assume otherwise.  If $\hat{\Delta}$ receives $\frac{\pi}{5}$ from $\Delta_2$ then $\Delta_2$ is given by $\Delta$
of Figure 7(iii) or Figure 8(iv).  But if $\Delta_2$ is $\Delta$ of Figure 8(iv) then $\hat{\Delta}$ does not receive any curvature from $\Delta_3$
and we are done; and if $\Delta_2$ is $\Delta$ of Figure 7(iii) then according to Configuration D of Figure 32(ii), $\hat{\Delta}$ receives $\frac{3 \pi}{10}$ from
$\Delta_0$ and the deficit is increased by $\frac{\pi}{30}$.
If $\hat{\Delta}$ receives $\frac{\pi}{5}$ from $\Delta_3$ then $\Delta_3$ is given by Figure 7(iii) or Figure 10(i), (ii).  But if $\Delta_3$ is
$\Delta$ of Figure 10(i), (ii) then $\hat{\Delta}$ does not receive any curvature from $\Delta_2$ and we are done; and if $\Delta_3$ is $\Delta$
of Figure 7(iii) then according to Configuration C of Figure 32(i), $\hat{\Delta}$ receives $\frac{3 \pi}{10}$ from $\Delta_5$ and the deficit is increased by $\frac{\pi}{30}$.  It follows that 
the deficit is at least $\frac{2\pi}{15} + \frac{2\pi}{30} = \frac{\pi}{5}$.

\begin{table}[h]
\[
\renewcommand{\arraystretch}{1.0}
\begin{array}{rllllllllll}
&d(v_1)&d(v_2)&d(v_3)&\kappa_1&\kappa_2&\kappa_3&\kappa_4&deficit\\
\textrm{(xiv)}&3&3&3&0&2&2&0&12&\\
\textrm{(xiv)}&4&3&3&e_1&e_2&2&0&8&\\  
\textrm{(xiv)}&3&4&3&0&0&0&0&21&\\
\textrm{(xiv)}&3&3&4&0&2&e_1&e_2&8&\\  
\textrm{(xiv)}&5+&3&3&2&2&2&0&18&\\
\textrm{(xiv)}&3&5+&3&0&2&2&0&20&\\
\textrm{(xiv)}&3&3&5+&0&2&2&2&18&\\
\textrm{(xiv)}&4&4&3&x_1&x_2&0&0&17&\\ 
\textrm{(xiv)}&4&3&4&e_1&e_2&e_1&e_2&4(6)&(Note)\\
\textrm{(xiv)}&3&4&4&0&0&x_1&x_2&17
\end{array}
\renewcommand{\arraystretch}{1}
\]  
\caption{}
\end{table}

Finally the case Figure 55(xiv) is given by Table 12. Note that $d(v_1)=d(v_2)=4$ implies $\kappa_1 = 0$ or $\kappa_2 = 0$ and that $d(v_3)=d(v_4)=4$ implies $\kappa_3 = 0$ or $\kappa_4 = 0$ which implies
$x_1+x_2=5$ in Table 12;and the $e_1$,$e_2$ entries are explained by Figure 42 and Lemma 9.2(iv). 
Since (see Table 10) $\kappa_1 + \kappa_2 + \kappa_3 + \kappa_4 \leq \frac{11 \pi}{15}$, if there is a vertex of degree $\geq 4$ and one of degree $\geq 5$ it follows that the deficit is at least $\frac{7 
\pi}{30}$; 
and if there are at least three vertices of degree $\geq 4$ then the deficit $\geq \frac{3 \pi}{10}$ and so we see from Table 12 that to complete the proof the penultimate row for (xiv), that is, case 
(xiv) 
with $d(v_1)=d(v_3)=4$ and $d(v_2)=3$ must be considered. 
Note that for this subcase the deficit is at least $\frac{2 \pi}{15}$ and so it remains to show that the deficit is in fact at least $\frac{\pi}{5}$.
If $\kappa_1+\kappa_2 \leq \frac{\pi}{3}$ and $\kappa_3+\kappa_4 \leq \frac{\pi}{3}$ then the deficit $\geq \frac{\pi}{5}$, so assume otherwise.
If $\kappa_1 + \kappa_2 > \frac{\pi}{3}$ then the only way this can occur (see Figure 42) is if
$\kappa_1 = \frac{2 \pi}{15}$ and $\kappa_2 = \frac{7 \pi}{30}$ forcing $\Delta_1$ to be given by $\Delta$ of Figure 20(v) and $\Delta_2$ to be given
by $\Delta$ of Figure 16(iii).  But either this gives Configuration B of Figure 31(v), a contradiction since then $\kappa_2 = \frac{ \pi}{5}$ only, or $d(u) > 4$ in Figure 31(v) and the deficit
is at least $\frac{2 \pi}{15} + \left(\frac{\pi}{3}-\frac{7 \pi}{30}\right)=\frac{7 \pi}{30}$.
If $\kappa_3 + \kappa_4 > \frac{\pi}{3}$ then the only way this can occur is if
$\kappa_3 = \frac{7 \pi}{30}$ and $\kappa_4 = \frac{2 \pi}{15}$ forcing $\Delta_3$ to be given by $\Delta$ of Figure 20(vi) and $\Delta_4$ to be given
by $\Delta$ of Figure 16(ii).  But either this gives Configuration A of Figure 31(i), a contradiction since then $\kappa_3 = \frac{ \pi}{5}$ only, or $d(u) > 4$ in Figure 31(i) and again the deficit is at least $\frac{7 
\pi}{30}$.$\Box$

\medskip

The proof of Proposition 4.3 follows from Propositions 10.2, 10.4, 11.3 and 11.5.




{\Large\textbf{Appendix}}

In what follows we use LEC to denote a length contradiction as defined in Section 3; and we use LAC to denote labelling contradiction, which will often be a basic labelling contradiction
corresponding to Figure 2(v),(vi).

\textit{Proof of Lemma 3.5} \quad The proof is immediate for $n=3,4$ and has been given for $n=6$ in Section 3 so let $n=5$ and $\hat{\Delta}$ be given by Figure A.1(i).
If $\hat{\Delta}$ contains no shadow edges then LEC.
If $\hat{\Delta}$ contains exactly one shadow edge up to symmetry it is (13) and this yields LEC.
Up to symmetry this leaves the case $(13),(14)$ forcing LEC.

Let $n=7$ and $\hat{\Delta}$ be given by Figure A.1(ii).
If only (13) or only (14) or only $(14),(15)$ occurs this forces LEC so it can be assumed without any loss that $(13),(14)$ occurs.
If there are no more shadow edges then LEC;
if exactly one more then it is one of $(15),(16),(46),(47),(57)$ and each forces LEC; and
if two more then they are one of the pairs $(15),(16)$; $(15),(57)$; $(16),(46)$; $(46),(47)$; or $(47),(57)$.
But (15) yields LAC; $(16),(46)$ yields LAC; $(46),(47)$ yields LAC; and $(47),(57)$ yields LEC.

Let $n=8$ and $\hat{\Delta}$ be given by Figure A.1(iii).
If there are no shadow edges then $\hat{\Delta}$ is given by Figure 4(iv).
If (15) occurs only then $\hat{\Delta}$ is given by Figure 4(v).
If there are now no shadow edges $(i~i+2)$ (subscripts modulo 8) then up to symmetry one of (14); $(14),(15)$; $(14),(16)$; $(14),(58)$; $(14),(15),(16)$;
or $(14),(15),(58)$ occurs.  But $(14),(15),(16)$ yields LAC; $(14),(15),(58)$ is given by Figure 4(vi); and the other cases each yield LEC.
It can be assumed without loss that $(13),(14)$ occurs.  The remaining possible shadow edges are: (15), (16), (17), (46), (47), (48), (57), (58) and (68).
If (15) then LAC so assume otherwise.
No more shadow edges yields LEC.  Exactly one more shadow edge yields LEC except for (47) given by Figure 4(vii) or (16) given by Figure 4(viii) or (58) given by Figure 4(x).
If there are exactly two more shadow edges then it is one of the pairs
$(16),(17)$; $(16),(46)$; $(16),(68)$; $(17),(46)$; $(17),(47)$; $(17),(57)$;
$(46),(47)$; $(46),(48)$; $(46),(68)$; $(47),(48)$; $(47),(57)$; $(48),(58)$;
$(48),(68)$; $(57),(58)$; or $(58),(68)$.  But each case forces LEC.

If there are exactly three more shadow edges then it is one of the triples $(16),(17),(46)$; $(16),(46),(68)$; $(17),(46),(47)$; $(17),(47),(57)$;
$(46),(47),(48)$; $(46),(48),(68)$; $(47),(48),(57)$; $(48),(57),(58)$; or $(48),(58),(68)$.
Each of these forces LAC except for (13), (14), (47), (48), (57) which is given by Figure 4(ix).

Finally let $n=9$ and $\hat{\Delta}$ be given by Figure A.1(iv).  If there are no shadow edges then $\hat{\Delta}$ is given by Figure 4(xi).
If there is exactly one shadow edge then (up to symmetry) it is one of (13), (14) or (15) and each forces LEC.
If the girth of $\hat{\Delta}$ together with shadow edges is five then either (15) only or $(15),(16)$ only occurs and each forces LEC.
Suppose (14) 

\newpage
\begin{figure}[t]
\begin{center}
\psfig{file=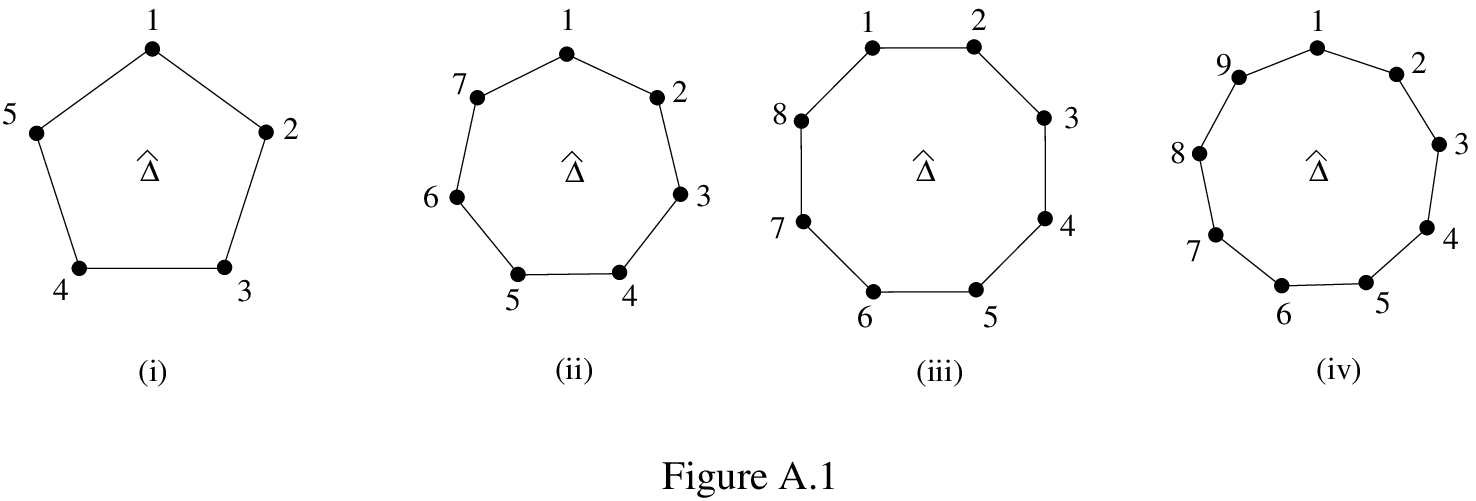}
\end{center}
\end{figure}

\noindent occurs (the girth four case).
If at most one more shadow edge occurs then it is one of (15), (16), (17), (58), (59), (69) and each case forces LEC.
If there are exactly two more shadow edges then LAC or one of $(15),(17)$; $(15),(58)$; $(15),(59)$, $(15),(69)$; $(16),(17)$; $(16),(69)$;
$(17),(47)$; $(47),(48)$; $(47),(49)$; $(48),(49)$; $(48),(58)$; or $(59),(69)$ occur forcing LEC.
If there are three more then this forces LAC or $(14),(15),(59),(69)$ which yields LEC.
It can now be assumed without any loss that $(13),(14)$ occurs.
If there are no more shadow edges then LEC.
If there is exactly one more then it is (15) and LAC or one of (16), (17), (18), (46), (47), (48), (49), (57), (58), (59), (68), (69) or (79) and LEC.
If exactly two more then LAC or one of $(16),(17)$; $(16),(18)$; $(16),(46)$; $(16),(68)$; $(16),(69)$; $(16,(79)$; $(17),(18)$; $(17),(46)$; $(17),(47)$;
$(17),(57)$; $(17),(79)$; $(18),(46)$; $(18),(47)$; $(18),(48)$; $(18),(57)$; $(18),(58)$; $(18),(68)$; $(46),(47)$; $(46),(48)$; $(46),(49)$; $(46),(68)$;
$(46),(69)$; $(46),(79)$; $(47),(48)$; $(47),(49)$; $(47),(57)$; $(47),(79)$; $(48),(57)$; $(48),(58)$; $(48),(68)$; $(49),(57)$; $(49),(58)$; $(49),(59)$;
$(49),(68)$; $(49),(69)$; $(49),(79)$; $(57),(58)$; $(57),(59)$; $(57),(79)$; $(58),(59)$; $(58),(68)$; $(59),(68)$; $(59),(69)$; $(59),(79)$; $(68),(69)$;
or $(69),(79)$ occur and forcing LEC.

If there are three or four more shadow edges this forces either LAC or one of
$(16),(17),(79)$; $(16),(69),(79)$; $(17),(18),(46)$; $(17),(18),(57)$; $(17),(46),(79)$; $(17),(57),(79)$; $(18),(46),(47)$; $(18),(46),(68)$;
$(18),(47),(57)$; $(18),(57),(58)$; $(18),(58),(68)$; $(46),(47),(49)$; $(46),(47),(79)$; $(46),(49),(68)$; $(46),(49),(79)$; $(46),(68),(69)$;
$(46),(69),(79)$; $(47),(48),(57)$; $(47),(49),(57)$; $(47),(57),(79)$; $(48),(57),(58)$; $(49),(78),(58)$; $(49),(78),(59)$; $(49),(57),(79)$;
$(49),(58),(59)$; $(49),(58),(68)$; $(49),(45),(68)$; $(49),(59),(79)$; $(49),(68),(69)$; $(49),(69),(79)$; $(57),(58),(59)$; $(58),(59),(68)$;
$(59),(68),(69)$; or $(49),(58),(59),(68)$ occurs each forcing LEC. $\Box$

\textbf{Remark 1}\quad If the corner label at the vertex $v$ of the region $\hat{\Delta}$ is $x$ or $y$ then it follows from equations (3.1) in Section 3 that there must be an
odd number of shadow edges in $\hat{\Delta}$ incident at $v$.

\textbf{Remark 2}\quad Let $v_1,v_2$ be vertices of the same $b$-segment of the region $\hat{\Delta}$.
It follows from Remark 1 that there are no shadow edges in $\hat{\Delta}$ from $v_1$ to $v_2$.

\newpage
\begin{figure}[t]
\begin{center}
\psfig{file=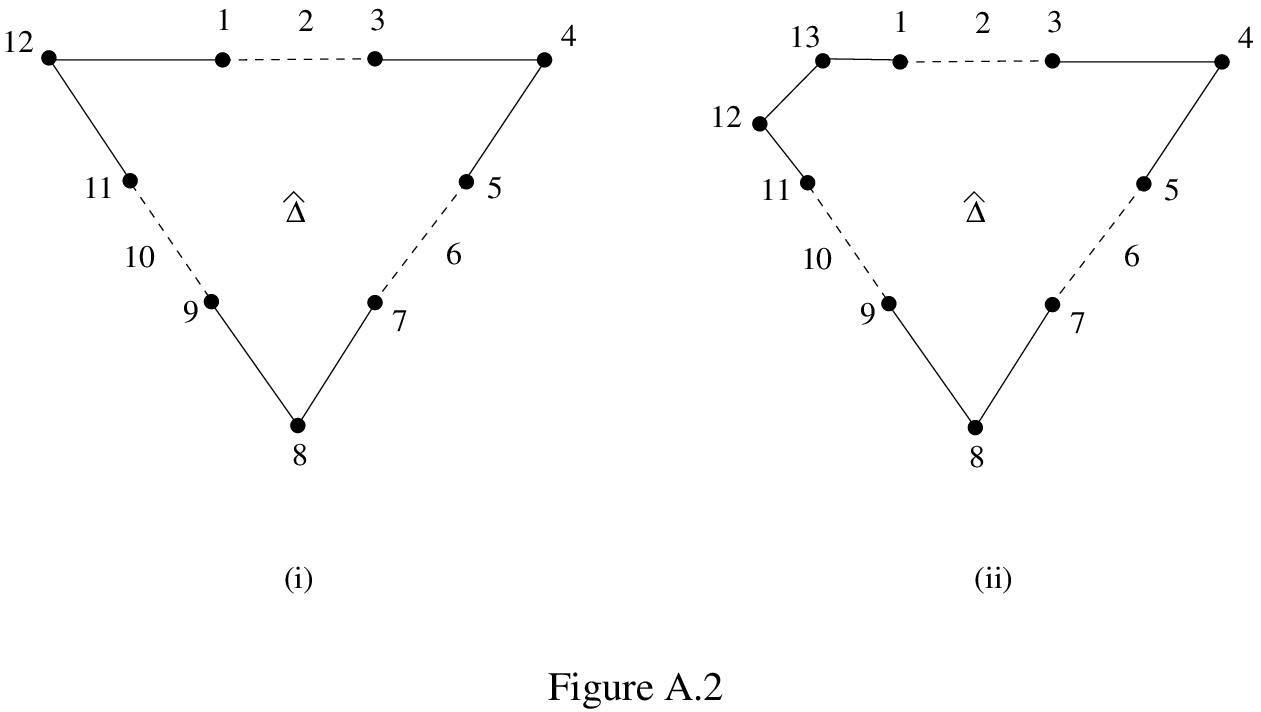}
\end{center}
\end{figure}

\textit{Proof of Lemma 11.2} \quad
\begin{enumerate}
\item[(i)]
Consider the regions $\hat{\Delta}$ of Figure A.2(i)-(ii) in which 2, 6, 10
refer to the (possibly empty) set of vertices between vertices 1 and 3, 5 and 7, 9 and 11.
We write $(ab)$ to indicate that there was a 2-segment between vertices $a$ and $b$ in $D$ with the understanding that if $a=2$, for example, we mean a vertex belonging to $a$.
By remark 1 above the number of $(ab)$ involving each of 1, 3, 5, 7, 9 and 11 must be odd and at least one.
First consider Figure A.2(i).
It follows from remark 2 above that if $\{ a,b \} \subseteq \{ 12,1,2,3,4 \}$ or $\{ 4,5,6,7,8 \}$ or $\{ 8,9,10,11,12 \}$ then $(ab)$ does not occur.
Moreover (18) forces (19), ($1\,11$) and this yields LAC.
It follows that the only pairs involving 4, 8 or 12 are ($4\,10$), (28) and ($6\,12$).
First assume that none of (35), (79) or ($1\,11$) occur.
Then since (15), (16) and (17) each forces (35), and (19), ($1\,10$) each forces ($1\,11$), we get a contradiction.
Assume exactly one of (35), (79), ($1\,11$) occurs -- without any loss (79).
Then again (15), (16) and (17) each force (35), and (19) and ($1\,10$) each force ($1\,11$), a contradiction.
Assume exactly two of (35), (75), ($1\,11$) occur -- without any loss (35) and (79).
Then (19) and ($1\,10$) each forces ($1\,11$), a contradiction; and (16) and (17) each forces LEC at (35) or forces either (52), (52) or (52), (51) or $(36),(36)$ or $(36),(37)$ yielding LAC.
This leaves (15).
Since the number of $(ab)$ involving 5 must be odd at least one of (59), ($5\,10$) or ($5\,11$) occurs.
But (59) forces ($11\,5$) and ($5\,10$); ($5\,10$) forces ($11\,5$) and another ($5\,10$); and ($5\,11$)
forces either a length contradiction at (79) or forces $(95), (96)$ or $(96),(96)$ or $(7\,10),(7\,10)$ or $(7\,10),(7\,11)$ yielding LAC in all cases.
Finally assume that $(1\,11)$, (35) and (79) occur.
Since the length of each is $n-1$ we must have more pairs otherwise there is a length contradiction.
Assume without any loss that 1 is involved in further pairs.
Since (16) and (19) each forces either $(36),(36)$ or $(36)(37)$ or $(52),(52)$ or $(52)(51)$ yielding LAC it follows that at least two of (15), (19) and $(1\,10)$ occur.
However $(19),(1\,10)$ and $(1\,10),(1\,10)$ yield LAC and $(15),(19)$ forces (59) and LAC.
This leaves $(15),(1\,10)$ together with at least one of (25), (59), $(5\,10)$.
But (25) yields LAC; (59) forces (19) or (69) and LAC; and finally $(5\,10)$ forces either a length contradiction or one of $(7\,10)$, $(7\,10)$ or (59)(69) or (69)(69) and LAC.

Now consider Figure A.2(ii).
Here if $\{ a,b \} \subseteq \{ 3,4,5,6,7 \}$ or $\{ 7,8,9,10,11 \}$ or $\{ 11,12,13,1,2 \}$ then $(ab)$ does not occur.
First assume that (68) and $(10\,12)$ do not occur.
Then (69), $(6\,10)$ and $(6\,11)$ force (68); $(6\,12)$ forces (68) or $(10\,12)$; $(6\,13)$ and (61) each force (68) or $(10\,12)$; and (62) forces one of
$(12\,6)$, $(12\,7)$, $(12\,8)$ or $(12\,9)$ and each forces (68) or $(10\,12)$ -- in all cases a contradiction.

Now assume that exactly one of (68) and $(10\,12)$ occurs -- without any loss (68).  Since this segment has length $n-1$ this forces at least one of (69), $(6\,10)$, (58), (48) to occur.
If (68) and (69) occur then at least one (69), $(6\,10)$, $(6\,11)$, $(6\,12)$, $(6\,13)$, (61) or (62) occurs.
But (69) and $(6\,10)$ yields LAC; and each of $(6\,11)$, $(6\,12)$, $(6\,13)$, (61) and (62) forces $(6\,10)$ or $(10\,12)$ a contradiction.
If (68) and $(6\,10)$ occurs then at least one of $(6\,11)$, $(6\,12)$, $(6\,13)$, (61) and (62) occurs.  But $(6\,11)$ yields LAC and the rest force $(10\,12)$ or $(6\,12)$ and LAC.
If (68) and (58) occur then at least one of (85), (84), (83), (82), (81), $(8\,13)$ and $(8\,12)$ occurs.
But (85) and (84) yield LAC; (83) forces (84); and each of the rest forces $(10\,12)$.
If (68) and (48) occur then at least one of (83), (82), (81), $(8\,13)$ or $(8\,12)$ occurs.
But (83) yields LAC and the rest force $(10\,12)$.

Finally assume that (68) and $(10\,12)$ occur.
Then length implies that at least one of (69), $(6\,10)$, (85), (84) occurs and at least one of
$(10\,13)$, $(10\,1)$, $(12\,9)$, $(12\,8)$ occurs.  Let (69) and $(10\,13)$ occur.
Then at least one of $(10\,13)$, $(10\,1)$, $(10\,2)$, $(10\,3)$, $(10\,4)$, $(10\,5)$ and $(10\,6)$ also occurs.
But $(10\,13)$ and $(10\,1)$ yield LAC; $(10\,2)$ forces $(10\,1)$; and each of $(10\,3)$, $(10\,4)$ and $(10\,5)$ forces another (69) or $(6\,10)$ and LAC. Let (69) and $(10\,1)$ occur.
Then at least one of $(10\,2)$, $(10\,3)$, $(10\,4)$ and $(10\,5)$ occurs.
But $(10\,2)$ yields LAC; and each of $(10\,3)$, $(10\,4)$ and $(10\,5)$ forces (69) and LAC.
Let (69) and $(12\,9)$ occur.
Then at least one of $(6\,12)$, $(6\,13)$, (61) and (62) occurs.
But $(6\,12)$ yields LAC; and $(6\,13)$, (61) and (62) each forces either $(6\,12)$ or $(12\,9)$ and LAC.
The segments (69) and $(12\,8)$ cannot both occur.
Let $(6\,10)$ and $(10\,13)$ occur.
Then at least one of $(10\,13)$, $(10\,1)$, $(10\,2)$, $(10\,3)$, $(10\,4)$ and $(10\,5)$ must also occur.
But $(10\,13)$ and $(10\,1)$ each force LAC; $(10\,2)$ forces $(10\,1)$; and each of $(10\,3)$, $(10\,4)$ and $(10\,5)$ forces (69), and LAC.
Let $(6\,10)$ and $(10\,1)$ occur.
Then at least one of $(10\,2)$, $(10\,3)$, $(10\,4)$ and $(10\,5)$ occurs and there is a contradiction as in the subcase above.
The segment $(6\,10)$ cannot occur with $(12\,9)$ or $(12\,8)$.
The subcases (84), $(12\,9)$; (84), $(12\,8)$; (85), $(12\,9)$; and (85), $(12\,8)$ follow by symmetry.
Let (84) and $(10\,13)$ occur.
Then at least one of (85), (83), (82), (81) and $(8\,13)$ occurs.
But (85) and (83) yield LAC; and the others force either $(10\,13)$ or $(10\,1)$ and LAC.
Let (84) and $(10\,1)$ occur.
Then at least one of (85), (83), (82) and (81) occurs and similarly there is a labelling contradiction.
Let (85) and $(10\,13)$ occur.
Then at least one of (85), (84), (83), (82), (81) and $(8\,13)$ also occurs.

\newpage
\begin{figure}[t]
\begin{center}
\psfig{file=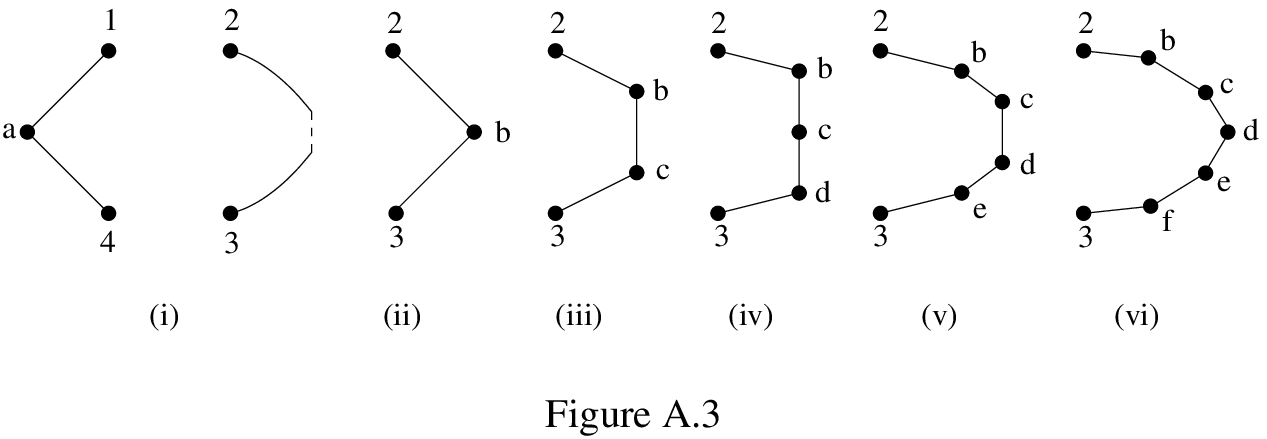}
\end{center}
\end{figure}

\noindent But (85), (84) and (83) yield LAC; and the rest forces either $(10\,13)$ or $(10\,1)$ and LAC. 
Finally if (85) and $(10\,1)$ occur then at least one of (85), (84), (83), (82) and (81) occurs and similarly there is a labelling contradiction.

\item[(ii)]
Let $m=2$.  Then $\hat{\Delta}$ is given by Figure A.3(i).
If $n=2$ as in Figure A.3(ii) then $(1b)$ is forced yielding $(2?)$; and
if $n=3$ as in Figure A.3(iii) then $(1b)$ forces $(2?)$ and $(1c)$ forces $(3a)$ and $(4?)$.  Let $n=4$ as in Figure A.3(iv).
Then $(1b)$ forces $(2?)$ and $(1d)$ forces $(3a), (4?)$ so $(1c)$ must occur.  This forces $(2c)$ and no vertices between 1 and 2 otherwise LAC.
If $(3a)$ then $(4?)$ so must have $(3c)$ and $(4c)$.  Thus there are no vertices between 3 and 4, otherwise LAC, and so $d(\hat{\Delta})=8$, a contradiction.
Let $n=5$ as in Figure A.3(v).
If $(1b)$ then $(2?)$ and if $(1e)$ then $(3a)$ and $(4?)$ so must have $(1c)$ or $(1d)$.  Suppose $(1c)$ occurs.
This forces $(2c)$ and no vertices between 1 and 2.
If $(4d)$ then $(3d)$ must occur and no vertices between 3 and 4.  But then $d(\hat{\Delta})=9$, a contradiction, so must have $(4c)$ since $(4e)$ forces $(3?)$.
If now $(3c)$ then $d(\hat{\Delta})=9$ as before so $(3d)$ must occur.
If $(ac)$ occurs or there are at least three vertices between 3 and 4 then LAC, so assume otherwise.
If 5 is the only vertex between 3 and 4 then $(5c)$ yields LEC and $(5d)$ yields LAC; and
if 5 and 6 are between 3 and 4 then must have $(5d)$ and $(6c)$ otherwise LAC.  There are no more shadow edges and the resulting region yields LEC.
Now suppose $(1d)$ occurs.  This forces $(3d),(4d)$ and no vertices between 3 and 4.
If $(2d)$ then $d(\hat{\Delta})=9$ so $(2c)$ occurs.  By symmetry we can now argue as in the above to obtain a contradiction.

Finally let $n=6$ as in Figure A.3(vi). Note that if $v \notin \{ a,b,c,d,e,f \}$ then $i \deg (v) = 1$ otherwise LAC.
If $(1b)$ then $(2?)$;
if $(1f)$ then $(3a)$ and $(4?)$;
if $(4f)$ then $(3?)$; and
if $(4b)$ then $(2a)$ and $(1?)$.  Thus there are six cases: $(1c),(4c)$; $(1c),(4d)$; $(1c),(4e)$; $(1d),(4d)$; $(1d),(4e)$; and $(1e),(4e)$.   

Consider $(1c),(4c)$.  This forces $(2c)$ and no vertices between 1 and 2.
If $(ac)$ then LAC.
If $(3c)$ then there are no vertices between 3 and 4 and LAC so must have $(3d)$ or $(3e)$.  Suppose $(3d)$.
If there are at least three vertices between 3 and 4 then LAC; and
if there are less than three this forces LEC in each case.  Suppose $(3e)$.  Then $(5e)$ must occur for some vertex 5 between 3 and 4 otherwise LEC.
If there are no other vertices between 3 and 4 then LEC; and
if there is at least one more vertex between 3 and 4 then LAC.

Consider $(1c),(4d)$.  This forces $(2c)$ and no vertices between 1 and 2.  Observe that $(ac)$ forces LAC.  There must be $(3d)$ or $(3e)$.
Suppose $(3d)$ occurs.  Then there are no vertices between 3 and 4.  If $(da)$ or $(df)$ then LAC so $\hat{\Delta}$ has no more vertices and this yields LEC.
Suppose $(3e)$ occurs.  Then must have $(5e)$ where 5 is a vertex between 3 and 4, otherwise LAC.
If there are no other vertices between 3 and 4 then either $(da)$ occurs or does not occur, but in both cases there is LEC;
if there is exactly one more, 6 say, then $(6e)$ yields LAC and if $(6d)$ then $(da)$ yields LAC and if not $(da)$ then LEC.

Consider $(1c),(4e)$.  This forces $(2c),(3e)$, no vertices between 1 and 2 and no vertices between 3 and 4.
If there are no other shadow edges then LEC so at least one of $(ac)$, $(ad)$, $(ae)$ or $(ce)$ occurs.
But $(ac)$, $(ad)$ or $(ae)$ forces LAC; and $(ce)$ yields LEC.

Consider $(1d),(4d)$.  If $(ad)$ then LAC.  By symmetry there are three cases: $(2c),(3e)$; $(2d),(3d)$; and $(2c),(3d)$.
If $(2c),(3e)$ then length forces at least one vertex between $1,2$ and $3,4$ and labelling implies at most two.
If there is exactly one vertex 5, say, between 1 and 2 and one vertex 6, say, between 3 and 4 then must have $(5c),(6e)$ yielding LEC; and
if exactly two vertices, between 1 and 2 and either one or two vertices between 3 and 4 then LAC.
If $(2d),(3d)$ then there are no vertices between 1 and 2 or 3 and 4 yielding LAC.
If $(2c),(3d)$ then there are no vertices between 3 and 4; $(df)$ yields LAC; and there is at least one vertex 5 say between 1 and 2 with $(5c)$, otherwise LEC.
If only 5 occurs between 1 and 2 then this forces LEC; and
if there are any more vertices between 1 and 2 this forces LAC.

Case $(1d),(4e)$ is the same as $(1c),(4d)$ by symmetry; and case $(1e),(4e)$ is the same as $(1c),(4c)$ by symmetry.

\begin{figure}[t]
\begin{center}
\psfig{file=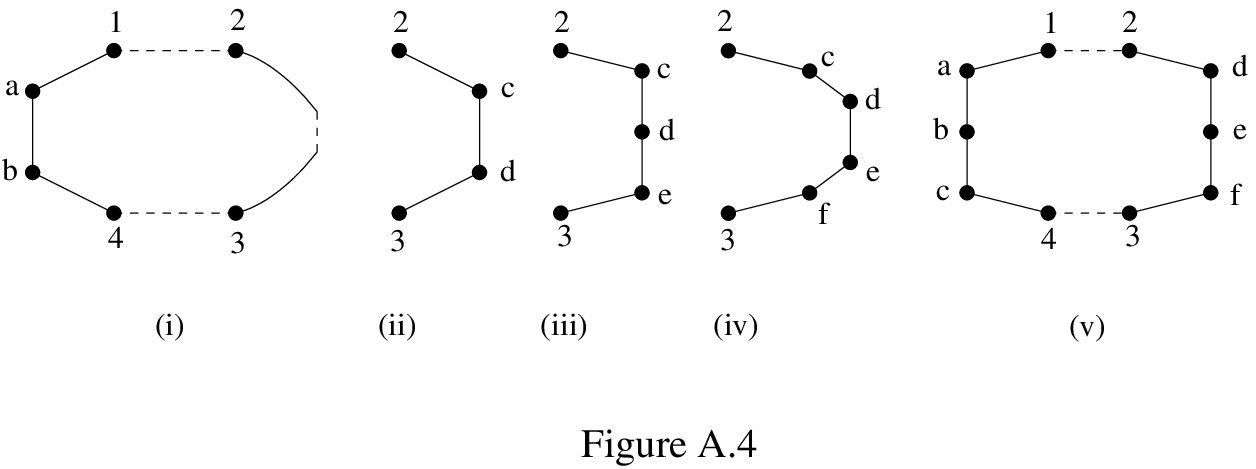}
\end{center}
\end{figure}


Let $m=3$.  Then $\hat{\Delta}$ is given by Figure A.4(i).  
If $n=3$ as in Figure A.4(ii) then $(1c)$ forces $(2?)$ and $(3b)$ forces $(4?)$.  So must have either $(1d)$ or $(1b)$.
Suppose $(1d)$ occurs.  This forces $(2d)$, $(3a)$, $(4a)$ and no vertices between 1 and 2 or between 3 and 4 otherwise LAC.
But then $d(\hat{\Delta})=8$, a contradiction.
Suppose $(1b)$ occurs.  This forces $(3c), (4c), (2b)$, no vertices between 1 and 2 or between 3 and 4 and again $d(\hat{\Delta})=8$.

Let $n=4$ as in Figure A.4(iii).  If $(1c)$ then $(2?)$ and if $(4e)$ then $(3?)$.
The possible cases up to symmetry are $(1b),(4c)$; $(1b),(4d)$; and $(1d),(4d)$.

Consider $(1b),(4c)$.  This forces $(2b),(3d)$ and no vertices between 1 and 2.
If $(bc)$ then LAC.
If there are no vertices between 3 and 4 then $d(\hat{\Delta})=9$, a contradiction;
if there is exactly one vertex, 5 say, between 3 and 4 then must have $(5d)$ otherwise LEC, there are no more shadow edges and $\hat{\Delta}$ is given (up tp symmetry) by Figure 47(ii); and
if there are at least two vertices between 3 and 4 then LAC.

Consider $(1b),(4d)$.  This forces $(3d)$ and no vertices between 3 and 4.
If now $(2b)$ then there are no vertices between 1 and 2 so $d(\hat{\Delta})=9$, a contradiction.
This leaves $(2d)$ and either LEC or there are two vertices, 5 and 6 say, between 1 and 2 with $(5b)$ and $(6d)$.
Any other shadow edges yields LAC and no more shadow edges yields LAC.

Consider $(1d),(4d)$.  This forces $(2d),(3d)$, no vertices between 1 and 2 or between 3 and 4 and so $d(\hat{\Delta})=9$, a contradiction.

Let $n=5$ as in Figure A.4(iv).  As before $(1c)$, $(2a)$, $(3b)$ and $(4f)$ yield contradictions.
Up to symmetry the cases are $(2b),(3c)$; $(2b),(3d)$; $(2b),(3e)$; $(2d),(3d)$; and $(2d),(3e)$.

Consider $(2b),(3c)$.  This forces $(1b),(4c)$ and no vertices between 1 and 2 or between 3 and 4.
If $(bc)$ occurs then LAC otherwise there is still a labelling contradiction.

Consider $(2b),(3d)$. This forces $(1b)$ and no vertices between 1 and 2.
The subcases are $(4c)$ and $(4d)$.  Consider first $(4d)$.  Then there are no vertices between 3 and 4, $(bc)$ yields LAC, $(bd)$ yields LAC and $(df)$ yields LAC.
Therefore there are no more shadow edges and this yields LEC.  Now suppose $(4c)$ occurs.
If $(bc)$ then LAC and any other shadow edge incident at vertex $c$ forces LAC. 
Suppose that there are no vertices between 3 and 4.
If $(df)$ does not occur then LEC and if $(df)$ occurs then $\hat{\Delta}$ is given by Figure 47(iii).
If there is more than one vertex between 3 and 4 then LAC and if there is exactly one, 5 say, then must have $(5d)$ yielding LEC with or without $(df)$.

Consider $(2b),(3e)$.  This forces $(1b)$ and no vertices between 1 and 2.
The subcases are $(4c)$, $(4d)$ and $(4e)$.  Consider first $(4c)$.
If $(bc)$ then LAC; or if there any further shadow edges at $c$ then LAC.
If $i \deg (3) > 1$ then LAC so there must be a vertex, 5 say, between 3 and 4 and $(5e)$ otherwise LEC.
If there are no more shadow edges then LEC; if $(ce)$ occurs then LEC; and
if there is either one more vertex, 6 say between 4 and 5 with $(6d)$ or two more, 6 and 7 say, between 4 and 5 with $(6d),(7d)$ then this forces LEC
in both cases.  Now suppose $(4d)$ occurs.  If $(bc)$ then LAC or if $(4c)$ and $(4e)$ then LAC.
If there is exactly one vertex, 5 say, between 3 and 4 then $(5e)$ occurs, otherwise LEC.
If there are no more shadow edges then LEC and the only other possibility is $(db)$ and again LEC; and
if there are exactly two vertices, 5 and 6, say between 3 and 4 then must have $(5e)$ and $(6d)$.
Any further shadow edges yields LAC and no more yield LEC.  Now suppose $(4e)$ occurs.
This forces $(3e)$ and no vertices between 3 and 4.
Now $(bc)$ and $(be)$ each yield LAC and if $(ce)$ then LEC.
If there are no more shadow edges then LEC or if $(bd)$ then LAC.

Consider $(2d),(3d)$.  The subcases are $(1d),(4d)$; $(1d),(4a)$; and $(1b),(4d)$.
Suppose $(1d),(4d)$ occurs.  Then there are no vertices between 1 and 2 or between 3 and 4, and each of $(df)$, $(da)$ and $(db)$ forces LAC.
Thus there are no more shadow edges and this forces LEC.  Suppose $(1d),(4a)$ occurs.
Then there are no vertices between 1 and 2.
If there are no vertices between 3 and 4 then LEC or if there are at least three vertices between 3 and 4 then LAC.
If there is exactly one vertex, 5 say, between 3 and 4 then must have $(5a)$ otherwise LEC, and if $(da)$ then LAC or if not $(da)$ then LAC.
If there are two vertices, 5 and 6 say, between then must have $(5a)$ and $(6d)$.
Again if $(da)$ then LAC and if not $(da)$ then LAC.
Suppose $(1b),(4d)$ occurs.  There are no vertices between 3 and 4 and either LEC or there are vertices 5 and 6 between 1 and 2 with $(5b)$ and $(6d)$.
If now $(bd)$ then LAC or if not $(bd)$ then LAC.

Consider $(2d),(3e)$.  Up to symmetry the subcases are $(1b),(4d)$; $(1b),(4e)$; $(1d),(4d)$; and $(1d),(4e)$.
Suppose $(1b),(4d)$ occurs.  Then either LEC or there are vertices 5 and 6 between 1 and 2 with $(5b),(6d)$.
If now $(bd)$ then LAC.  There must be a vertex 7 between 3 and 4 with $(7e)$ otherwise LEC.
If there are no more shadow edges then LEC; otherwise $(d8)$ occurs where 8 is between 7 and 4, and this yields LEC.
Suppose $(1b),(4e)$ occurs.  Then there are no vertices between 3 and 4, and either LEC or there are variances 5 and 6 between 1 and 2 with $(5b)$ and $(6d)$.
If now $(bd)$ then LAC.
If there are no more shadow edges then LEC;
if there is one more shadow edge $(e7)$ where 7 lies between 5 and 6 then LAC; and
if there is a further shadow edge $(e8)$ where 8 lies between 7 and 6 then again LAC. 
Suppose $(1d),(4d)$ occurs.  Then there are no vertices between 1 and 2.
If $(ad)$ then LAC.  There must be $(5e)$ where vertex 5 is between 3 and 4 otherwise LEC.
If $(bd)$ then LAC; if no other shadow edges then LEC; and
if $(d6)$ where 6 lies between 5 and 4 then LEC.

Finally let $m=n=4$ as shown in Figure A.4(v).  Up to symmetry the subcases are
$(1b),(4b),(2e),(3e)$; $(1b),(4b),(2b),(3d)$; $(1b),(4b),(2b),(3c)$; $(1b),(4b),(2b),(3b)$;\newline
$(1b),(3e),(4e),(2b)$; $(1b),(3e),(4d),(2c)$; $(1b),(3e),(4e),(2c)$; $(1c),(3d),(2c),(4d)$;\newline
$(1b),(3d),(4d),(2b)$; $(1b),(3d),(4d),(2c)$.
Note that $(1d)$, $(2a)$, $(3c)$ and $(4f)$ each yield a contradiction.

Suppose that $i \deg (1) > 1$.  Then $i \deg (1) = 3$ and the three cases are $(1b),(1c),(1e)$; $(1b),(1e),(1f)$; and $(1c),(1e),(1f)$.
Let $(1b),(1c),(1e)$ occur.  This forces $(2e),(3e),(4e)$ and no vertices between 1 and 2 or between 3 and 4.
If $(ce)$ then LAC and if not $(ce)$ then LEC.  Let $(1b),(1e),(1f)$ occur.  This forces $(3b),(4b)$ and LEC.
If $(1c),(1e),(1f)$ occurs this forces $(3c)$ and $(4?)$.
Now suppose that $i \deg (5) > 1$ where 5 is a vertex between 1 and 2.
Up to symmetry this forces $(5b),(5e),(5f)$ and then $(2e),(1b),(3b),(4b)$ with no more vertices between 1 and 2 and no vertices between 3 and 4.
If $(bf)$ then LAC and not $(bf)$ yields LEC.

By symmetry it can be assumed from now on that $i \deg (P) = 1$ where $P \in \{ 1,2,3,4 \}$ or $P$ is a vertex between 1 and 2 or between 3 and 4.

Consider $(1b),(4b),(2e),(3e)$.  This forces $(5e),(6b),(7e),(8b)$ where $5,6$ lie between 1 and 2 and $7,8$ lie between 3 and 4.
If now $(be)$ then LAC and if not $(be)$ then again LAC.

Consider $(1b),(4b),(2b),(3d)$.  This forces $(5b)$ where 5 lies between 3 and 4 and there are no vertices between 1 and 2.
If $(bd)$ then LAC so assume not $(bd)$.
If there are no more shadow edges then LEC.  The remaining possibilities are: $(d6)$ only where 6 lies between 5 and 3, yielding LAC;
$(df)$ only, yielding LAC; and $(d6),(df)$ yielding LAC.

Consider $(1b),(4b),(2b),(3e)$.  This forces $(5e),(6b)$ where $5,6$ lie between 3 and 4; and there are no vertices between 1 and 2.
If $(bd)$ then LAC; if $(be)$ then LAC; if $(7d)$ then LAC or if $(7d),(8d)$ then LAC, where $7,8$ lie between 5 and 6.

Consider $(1b),(4b),(2b),(3b)$.
If there are no more shadow edges then LEC; if $(bd)$ or $(bf)$ then LAC; if $(df)$ then LEC; and if $(be)$ then LAC.

Consider $(1b),(3e),(4e),(2b)$.
If there are no more shadow edges then LEC;
if $(bd)$ or $(ce)$ then LAC;
if $(be)$ then LAC; and
if $(cd)$ only then this is $\hat{\Delta}$ of Figure 47(v).

Consider $(1b),(3e),(4d),(2c)$.
If there are no more shadow edges then LAC;
if $(c7)$ or $(d8)$ only then LEC;
if $(c7),(d8)$ only then LAC;
if $(c7),(cd)$ or $(d8),(cd)$ only then LAC;
if $(c7),(d8),(cd)$ then LAC; and
if $(cd)$ only then we obtain $\hat{\Delta}$ of Figure 47(iv).

Consider $(1b),(3e),(4e),(2c)$.  This forces $(5b)$ where 5 lies between 1 and 2 and there are no vertices between 3 and 4.
If $(ec)$ then LAC;
if $(c6)$ only then LEC, where 6 lies between 5 and 2;  
if $(cd)$ only then LEC; and
if $(c6),(cd)$ then LAC.

Consider $(1c),(3d),(2c),(4d)$.  Then there are no vertices between 1 and 2 or between 3 and 4.
If $(df)$ or $(dc)$ or $(ca)$ then LAC.
If there are no more shadow edges then LAC.

Consider $(1b),(3d),(4d),(2b)$.  Then there are no vertices between 1 and 2 or between 3 and 4.
If $(df)$ or $(dc)$ or $(bd)$ then LAC.
If there are no more shadow edges then LEC.

Finally consider $(1b),(3d),(4d),(2c)$.  Then there are no vertices between 3 and 4.
There must be $(5b)$ where 5 lies between 1 and 2.
If $(df)$ or $(cd)$ then LAC; and
if there are no more shadow edges then LEC.

\item[(iii)]
Let $m=2$ and so $\hat{\Delta}$ is given by Figure A.3(i).  It follows from (ii) that (14) must occur.
It can be assumed therefore that $i \deg (1) > 1$ otherwise LEC.
If $n=2$ or 3 then this forces LAC.  Let $n=4$ and so $\hat{\Delta}$ is given by Figure A.3(i), (iv).
If $(1c)$ does not occur then this forces $(1d)$, (13) and LAC.  Let $(1c)$ occur.
If $(1d)$ occurs then $(13)$ and LAC is forced so assume $(1d)$ does not occur.
Since $(2c)$ is forced there are no vertices between 1 and 2 and this forces LAC.
Let $n=5$ and so $\hat{\Delta}$ is given by Figure A.3(i), (v).  Suppose that (13) does not occur.
Then must have $(1c), (1d), (3d), (2c)$, no vertices between 1 and 2 and either LEC or $(5d)$ where 5 lies between 3 and 4 which yields LAC.
Let (13) occur so that there are no vertices between 3 and 4.
If $(1c)$ then $d(\hat{\Delta})=9$, a contradiction, if $(1e)$ then LAC, so assume $(1d)$ occurs.
Then $(2d)$ forces $d(\hat{\Delta})=9$ so assume $(2c)$ occurs.  This forces LEC or $(5c)$ where 5 lies between 1 and 2.
If there are no more shadow edges then LAC; and
if $(6d)$ occurs where 6 lies between 1 and 5 then LEC.

Let $m=n=3$ and so $\hat{\Delta}$ is given by Figure A.4(i), (ii).  Up to symmetry there are two cases, namely $i \deg (3) > 1$ and $i \deg (5) > 1$ where 5 lies between 3 and 4.
But any triple from $(5a), (51), (56), (57), (52), (5c)$ or from $(3a), (31), (36), (37), (32), (3c)$ yields LAC.

Let $m=3$ and $n=4$ and so $\hat{\Delta}$ is given by Figure A.4(i), (iii).  It can be assumed without any loss that $i \deg (5) > 1$ or
$i \deg (4) > 1$ or $i \deg (3) > 1$ where 5 lies between 3 and 4.  Suppose $i \deg (5) > 1$.  Let 6 lie between 1 and 2.  Then each of the pairs
$(5a),(56)$; $(5a),(52)$; $(5a),(5c)$; $(51),(5c)$; $(56),(5c)$; and $(52),(5c)$ forces LAC.
If $i \deg (5) > 3$ then LAC.  This leaves the cases $(5a),(51),(5d)$; $(51),(56),(52)$; $(51),(56),(5d)$; $(51),(52),(5d)$; $(56),(52),(5d)$; and
$(56),(57)$ where 7 lies between 6 and 2.  
If $(5a),(51),(5d)$ occurs then LEC or $(2d)$ and (26), and LAC.
If $(51),(56),(52)$ occurs then LAC.
If $(51),(56),(5d)$ occurs this forces $(2d),(27)$ and LAC.

If $(51),(52),(5d)$ occurs this forces $(3d)$.
If now $i \deg (1) > 3$ then LAC so $(4a),(a8)$ must occur where 8 lies between 4 and 5.  Any further shadow edges yields LAC, and no more yields LEC.
If $(56),(52),(5d)$ occurs this forces $(3d)$.  But now $i \deg (6) > 1$ yields LAC and $i \deg (6) = 1$ yields LAC.
If $(56),(57)$ occurs then this immediately forces LAC except for $(5d)$.
But this forces $(2d),(8d)$ and LAC where 8 lies between 7 and 2.

Now suppose $i \deg (4) > 1$.  Similarly to the above the cases are $(4a),(41),(4d)$; $(41),(46),(42)$ yielding LAC; $(41),(46),(4d)$;
$(41),(42),(4d)$; $(46),(42),(4d)$; and $(46),(47)$ where 6 lies between 1 and 2 and 7 lies between 2 and 6.
If $(4a),(41),(4d)$ occurs this forces $(2d),(6d)$ and LAC.
If $(41),(46),(4d)$ occurs this forces $(2d),(7d)$ and LAC.
If $(41),(42),(4d)$ occurs this forces $(3d)$ and no vertices between 1 and 2 or between 3 and 4.
Any other shadow edges yields LAC and no more yields LEC.
If $(46),(44),(4d)$ occurs this forces $(3d)$, no vertices between 3 and 4 and LAC.
If $(46),(47)$ occurs this forces LAC except possibly for $(4d)$.
But $(2d)$ is then forced yielding LEC or $(8d)$, where 8 lies between 7 and 2, yielding LAC.

Suppose $i \deg (3) > 1$.  Similarly to the above the cases are $(3a),(31),(3d)$; $(31),(36),(3c)$ yielding LAC;
$(31),(36),(3d)$; $(31),(32),(3d)$; $(36),(32),(3d)$; and $(36),(37)$.
If $(3a),(31),(3d)$ occurs this forces $(4a)$ and LAC.
If $(31),(36),(3d)$ occurs this forces $(2d)$ and LAC.
If any of the remaining cases occur they immediately force LEC or LAC.

Let $m=3$, $n=5$ and so $\hat{\Delta}$ is given by Figure A.4(i), (iv).  Again up to symmetry the cases are $i \deg (5) > 1$ or $i \deg (4) > 1$ or
$i \deg (3) > 1$ where 5 has between 3 and 4.

Suppose $i \deg (5) > 1$.  The pairs $(5a),(56)$; $(5a),(52)$; $(5a),(5c)$; $(51),(5c)$; $(56),(5c)$ where 6 has between 1 and 2 each force LAC.
If $(51),(52)$ or $(51),(56)$ occurs this forces $(4a)$ then LEC or $(4a),(a8)$, where 8 has between 4 and 5, then LAC.
If $(5a),(5e)$ occurs this forces $(3d),(4a)$ and LAC.  This leaves the triples $(56),(57),(5d)$; $(56),(57),(5e)$; $(5a),(51),(5d)$;
$(51),(5d),(5e)$; $(56),(52),(5d)$; $(56),(52),(5e)$; $(56),(5d),(5e)$; $(52),(5c),(5d)$ and LAC; $(52),(5c),(5e)$; $(52),(5d),(53)$;
$(5c),(5d),(5e)$ and LAC, where 6 lies between 1 and 2 and 7 lies between 1 and 6.
If $(56),(57),(5d)$ or $(5a),(51),(5d)$ occurs this forces $(2d)$ and LEC or $(2d),(d8)$, where 8 lies between 6 and 2, and LAC.
If $(56),(57),(5e)$ or $(56),(52),(5e)$ or $(52),(5c),(5e)$ occurs this forces $(3e)$ and LAC.
If $(51),(5d),(5e)$ occurs this forces $(4a)$ and LEC or $(4a),(a8)$, where 8 lies between 4 and 5, and LAC.
If $(56),(5d),(5e)$ occurs this forces $(2d),(d8)$, where 8 lies between 6 and 2, and LAC.
If $(52),(5d),(5e)$ occurs this forces $(3e)$ and LEC.
Finally suppose that $(56),(52),(5d)$ occurs.  This forces $(3d)$ or $(3e)$.
If $(3e)$ then LEC or $(e8)$, where 8 lies between 5 and 3.
If there are no further shadow edges at $d$ then LAC;
if $(2d)$ then LAC; or
if $(d8)$ where 8 lies between 5 and 3, then LAC.
If $(3d)$ occurs then there are no vertices between 3 and 5 and LAC.

Now suppose $i \deg (4) > 1$.  The pairs $(4a),(46)$; $(4a),(42)$; $(4a),(4c)$; $(41),(4c)$ and $(46),(4c)$ where 6 lies between 1 and 2 each yield LAC, and $(4a),(4e)$
forces $(3e)$ and either LEC or (41), yielding LAC.   
This leaves the triples $(4a),(41),(4d)$; $(41),(42),(4d)$; $(41),(42),(4e)$; $(41),(4d),(4e)$; $(46),(47),(4d)$
where 6 lies between 1 and 2, and 7 between 6 and 2; $(46),(47),(43)$; $(46),(4d),(4e)$; $(42),(4c),(4d)$ and LAC; $(42),(4c),(4e)$; $(42),(4d),(4e)$; and $(4c),(4d),(4e)$ yielding LAC.
If $(4a),(41),(4d)$ or $(41),(4d),(4e)$ occurs this forces $(2d)$ and either LEC or $(6d)$, where 6 lies between 1 and 2, and LAC.
If $(41),(42),(4d)$ occurs this forces $(3d)$ or $(3e)$.
If $(3d)$ then there are no vertices between 1 and 2 or between 3 and 4, any further shadow edges yields LAC and no more yields LEC.
If $(3e)$ then either LEC or $(5e)$ where 5 lies between 3 and 4 and if there are no more shadow edges at $d$ then LEC otherwise
$(6d)$ occurs where 6 lies between 5 and 4, and again LEC or there are further shadow edges and LAC.
If $(41),(42),(4e)$ occurs this forces $(3e)$ and there are no vertices between $1,2$ or $3,4$.
If now $(2e)$ then LAC or if either $(2d)$ or $(ec)$ then either LEC or $(2e)$.
So there are no more shadow edges and LEC.
If $(46),(47),(4d)$ occurs this forces $(2d)$ and either LEC or $(8d)$ where 8 lies between 7 and 2 yielding LAC.
If $(46),(47),(4e)$ or $(46),(4d),(4e)$ or $(42),(4c),(4e)$ or $(42),(4d),(4e)$ occurs this forces no more vertices between 3 and 4, $(1b)$ and either LEC or $(68)$ where 8 lies between 1 and 6, yielding 
LAC.

Suppose that $i \deg (3) > 1$.  The pairs $(3a),(36)$; $(3a),(32)$; $(3a),(3c)$; $(31),(3c)$; $(36),(3c)$ where 6 lies between 1 and 2 each yield LAC;
and the pair $(3a),(3e)$ forces $(4a)$, no vertices between 3 and 4, and either LEC or $(3d)$ which yields LAC.  This leaves the triples
$(3a),(31),(3d)$; $(31),(36),(32)$ and LAC; $(31),(36),(3d)$; $(31),(36),(3e)$; $(31),(32),(3d)$; $(31),(32),(3e)$ and LEC; $(31),(3d),(3e)$;
$(36),(37),(3d)$; $(36),(37),(3e)$ and LEC; $(36),(32),(3d)$; $(36),(32),(3e)$ and LEC; $(36),(3d),(3e)$; $(32),(3c),(3d)$ and LAC;
$(32),(3c),(3e)$ and LEC; $(32),(3d),(3e)$ and LAC; $(3c),(3d),(3e)$ and LAC, where 6 lies between 1 and 2 and 7 between 6 and 2.
If $(3a),(31),(3d)$ or $(31),(36),(3d)$ or $(31),(3d),(3e)$ or $(36),(37),(3d)$ or $(36),(3d),(3e)$ occurs this forces $(2d)$ and either LEC or $(d8)$ where 8 lies between 6 and 2 and LAC.
If $(31),(36),(3e)$ occurs this forces LEC or $(3d)$ which forces LEC or LAC as in the previous cases.
If $(31),(32),(3d)$ occurs this yields LEC, if $(df)$ also occurs then again LEC or $(2d)$ which yields LAC.
If $(36),(32),(3d)$ occurs then either $(df)$ occurs and LAC, otherwise LEC.

Finally let $m=n=4$ and so $\hat{\Delta}$ is given by Figure A.4(v).  Up to symmetry there are the two cases $i \deg (5) > 1$ and $i \deg (4) > 1$.
Let $i \deg (5) > 1$.  The pairs $(5a),(56)$; $(5a),(52)$; $(5a),(5d)$; $(51),(5d)$; $(56),(5d)$, where 6 lies between 1 and 2 each yield LAC.
Up to symmetry this leaves the triples $(5b),(5a),(51)$ and LAC; $(5b),(5a),(5e)$; $(5b),(51),(56)$; $(5b),(51),(52)$; $(5b),(51),(5e)$; $(5b),(56),(52)$;
$(5b),(56),(5e)$; $(5b),(52),(5d)$; $(56),(57),(5b)$; $(51),(56),(52)$ and LAC; where 6 lies between 1 and 2 and 7 between 6 and 2.
If $(5b),(5a),(5e)$ occurs this forces $(2e),(1e)$ and LAC.
If $(5b),(51),(56)$ occurs this forces $(4b)$ and LAC when $i \deg (6)=1$ or $i \deg (6)=3$.
If $(5b),(51),(52)$ occurs this forces $(32)$ or $(3d)$ or $(3e)$:
if (32) then LAC;
if $(3d)$ then LAC; and
if $(3e)$ then LEC.
If $(5b),(51),(5e)$ occurs this forces $(2e)$ and LEC or $(2e),(6e)$ and LAC, where 6 lies between 1 and 2.
If $(5b),(56),(52)$ or $(5b),(56),(5e)$ or $(5b),(52),(5d)$ occurs this forces $(4b),(1b)$ and LEC or $(4b),(1b),(7b)$ and LAC, where 7 lies between 1 and 6.

Let $i \deg (4) > 1$.  The pairs $(4c),(46)$; $(4a),(42),(4a),(4d)$; $(41),(4d)$; $(46),(4d)$ where 6 lies between 1 and 2 each forces LAC.
The pair $(4b),(41)$ force LEC or LAC and the pairs $(4b),(46)$; $(4b),(42)$; $(4b),(4d)$; and $(4b),(4e)$ each force LEC.
This leaves the triples $(4b),(4a),(41)$ and LAC; $(4b),(4a),(4e)$; $(4a),(41),(4e)$; $(41),(46),(42)$ and LAC; $(41),(46),(4e)$; $(41),(42),(4e)$;
$(46),(42),(4e)$; and $(42),(4d),(4e)$ which yields LAC.
If $(4b),(4a),(4e)$ or $(4a),(41),(4e)$ or $(41),(46),(4e)$ occurs this forces $(1e)$ and LAC.
If $(41),(42),(4e)$ occurs with no more shadows then LEC; and any further shadow edge forces LAC or LEC.
If $(46),(42),(4e)$ occurs this forces $(3e)$ and no vertices between 3 and 4; and either $(1b)$ or $(1c)$.
If $(1b)$ occurs then LEC or $(b7)$ where 7 lies between 1 and 6 again forcing LEC.
If $(1c)$ or $(1c),(c7)$ occurs then LAC.

\medskip

\item[(iv)]
By (iii) it can be assumed that $i \deg (v)=1$ for $v \in B_1 \cup B_2$; and by (iii) the statement clearly holds when $(m,n)=(3,3)$.
Let $(m,n)=(3,4)$ and so $\hat{\Delta}$ is given by Figure A.4(i), (iii).
Clearly the statement holds for 1 and 4.
Up to symmetry there are two cases: $(14),(2d),(3d)$; and $(1b),(2d),(3d)$.  Suppose $(14),(2d),(3d)$ occurs.
This forces LEC or $(5d),(6d)$ where 5 lies between 3 and 4 and 6 lies between 1 and 2 and this forces LAC.
Suppose $(1b),(2d),(3d)$ occurs.  This forces LEC or $(5d),(6d)$ as before and $(47),(b8)$ where $7,8$ lie between 1 and 6, and then LAC.

Let $(m,n)=(3,5)$ and so $\hat{\Delta}$ is given by Figure A.4(i), (iv).  Up to symmetry the cases are $(2d),(3d),(1b)$; $(2d),(3d),(14)$;
$(2d),(3d),(4a)$; $(2d),(3e),(1b)$; and $(2d),(3e),(14)$.  Suppose $(2d),(3d),(1b)$ occurs.
This forces $(b8),(47),(6d)$ where $6,7,8$ lie between 1 and 2, otherwise LEC.
Other possible shadow edges are $(df)$ and those between additional vertices between $6,7$ and $3,4$; but in all cases this forces LEC.
A similar argument forces LEC when $(2d),(3d),(14)$ or $(2d),(3d),(4a)$ occurs.
Suppose $(2d),(3e),(1b)$ occurs.  This forces $(b8),(47),(6d),(5e)$ where $6,7,8$ lie between 1 and 2 and 5 lies between 3 and 4, otherwise LEC.

\newpage
\begin{figure}[t]
\begin{center}   
\psfig{file=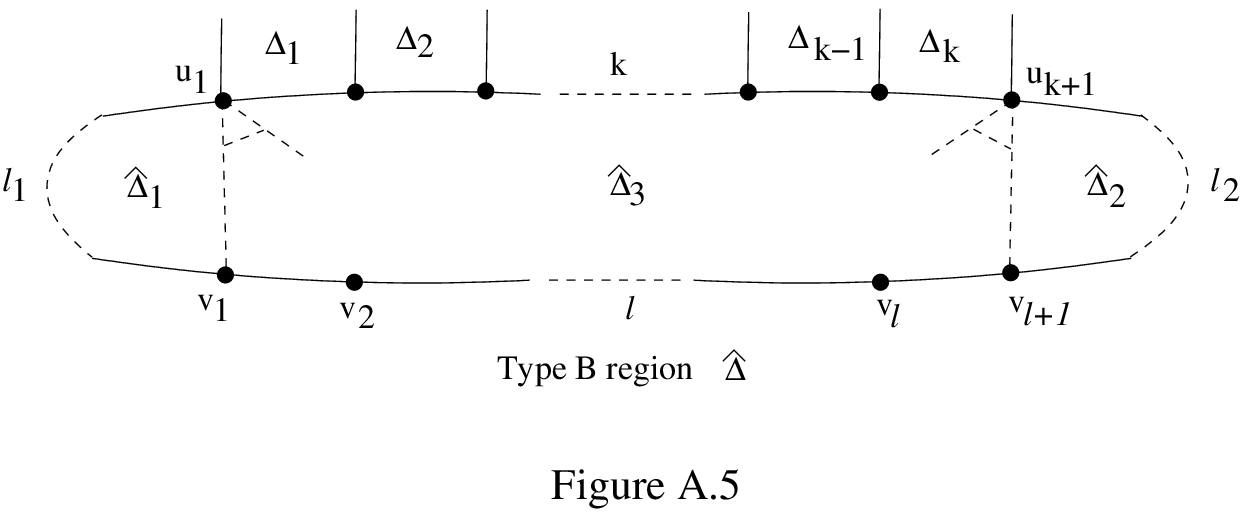}
\end{center}
\end{figure}   

If there are no more shadow edges from $d$ this forces LEC;
if $(d9)$ occurs only, where 9 lies between 4 and 5, this forces LAC;
if $(d9),(d10)$ occurs where 10 lies between 4 and 9 this forces LEC; and
if there are more than two further shadow edges at $d$ this forces LAC.
The argument is the same when $(2d),(3e),(14)$ occurs.

Let $(m,n)=(4,4)$ and so $\hat{\Delta}$ is given by Figure A.4(v).  Up to symmetry it can be assumed that $(1b),(4b)$ occurs.
This forces $(6b),(5b)$ where 6 lies between 1 and 2, and 5 lies between 3 and 4, otherwise LEC.
If any of (23), (27) or (38) now occurs where 7 lies between 5 and 3, and 8 lies between 6 and 2 this forces LAC.
This leaves $(2e),(3e)$ forcing LEC or $(7e),(8e)$ which again forces LAC.

\medskip

\item[(v)]
The region $\hat{\Delta}$ is given by Figure A.5 in which $\hat{\Delta}$ has been partitioned into three regions $\hat{\Delta}_i$ ($1 \leq i \leq 3$);
$l_1 \geq 2$ denotes the number of edges between $u_1$ and $v_1$; $l_2 \geq 2$ denotes the number between $u_{k+1}$ and $v_{l+1}$; $l$ the number between $v_1$ and $v_{l+1}$;
and $k$ the number between $u_1$ and $u_{k+1}$.  Moreover, as shown, there is a shadow edge in $\hat{\Delta}$ between $u_1$ and $v_1$ and between $u_{k+1}$ and $v_{l+1}$;
and, further to this, all shadow edges from $u_i$ ($1 \leq i \leq k+1$) within $\hat{\Delta}$ are connected to some $v_j$ where $1 \leq j \leq l+1$.
Thus $n_2 = l_1 + l + l_2 \geq l + 4$ and so $4 \leq n_2 \leq 9$.

Using LAC and LEC we list in Figure A.6 all the possibilities for $\hat{\Delta}_j$ ($1 \leq j \leq 2$) for $2 \leq l_j \leq 5$;
listed in Figure A.7 are the possibilities for $\hat{\Delta}_3$ when $0 \leq l \leq 3$; and when $l=4$ $\hat{\Delta}_3$ is given by Figure A.8
under the assumption that at least one $\hat{\Delta}_1, \hat{\Delta}_2$ has degree 2.
In what follows we use the labelling of Figures A.6-8 and work up to the symmetry of $\hat{\Delta}_3$.

Let $n_2 \leq 7$.  If $l=0$ then, using the fact that $i \deg (u_i) \geq 1$, $k > 1$ forces LAC and $k=1$ forces $d(\hat{\Delta})=n_2+1 \leq 8$;
if $l=1$ then $k > 3$ forces LAC and if $k \leq 3$ then $n_2 \leq 6$ implies $d(\hat{\Delta}) \leq 9$;
if $l=2$ then $k \leq 5$ and $n_2 \geq 6$.  Thus we are left with the cases $l=1$, $n_2=7$; $l=2$, $n_2=6$; $l=2$, $n_2=7$; and $l=3$, $n_2=7$.

Let $l=1$ and $n_2=7$.  Then $k \leq 3$ and since $k \leq 2$ implies $d(\hat{\Delta}) < 10$ it can be assumed that $\hat{\Delta}_3=C13$.
Since each of $R3B,R4B,R4C$ yields LAC it follows that $(\hat{\Delta}_1,\hat{\Delta}_2) \in \{(R3A,R3A),(R2,R4A)\}$ and LEC.

\newpage
\begin{figure}[t]
\begin{center}
\psfig{file=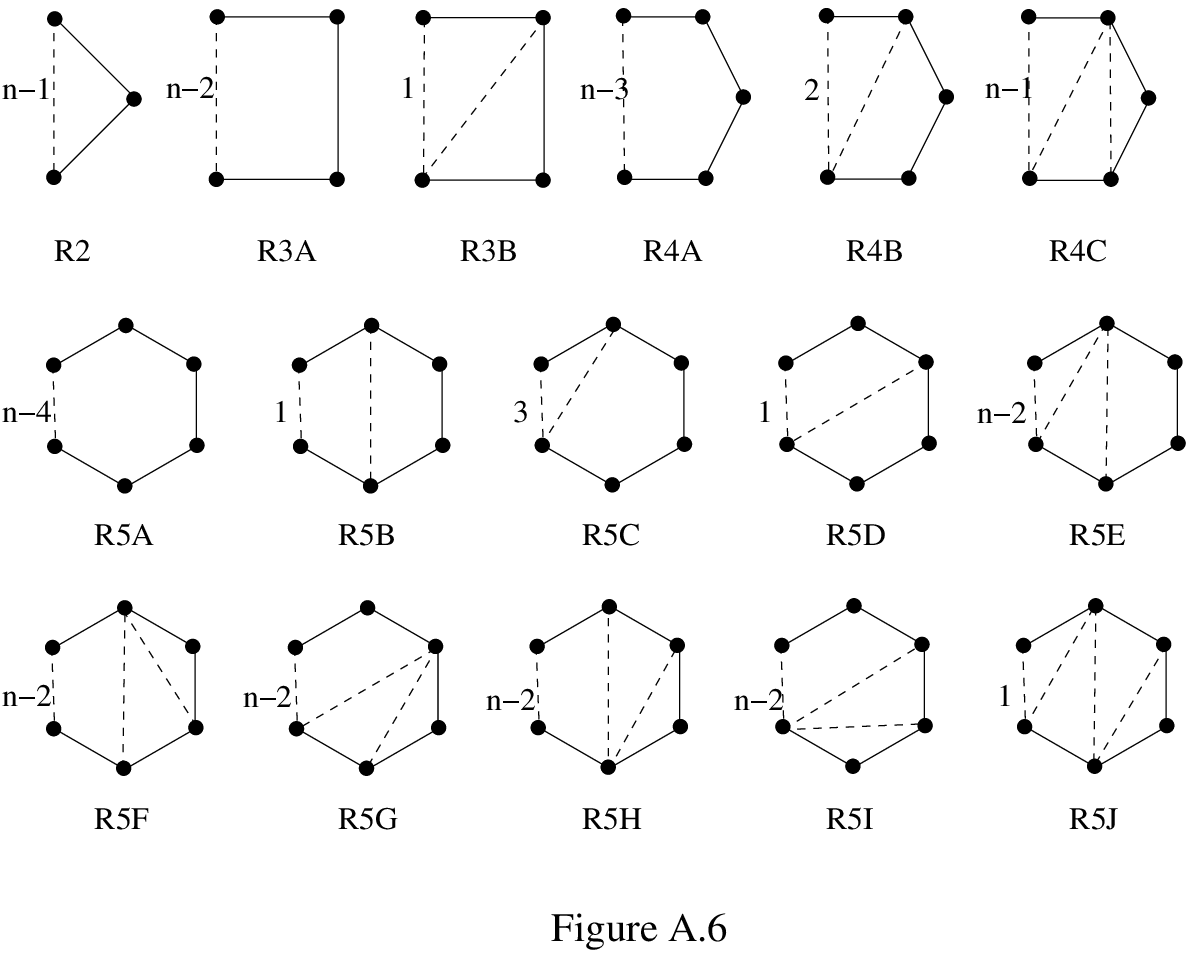}
\end{center}
\end{figure}

Let $l=2$ and $n_2=6$.  Then $k \leq 5$ and since $k \leq 3$ implies $d(\hat{\Delta}) < 10$ it can be assumed that
$\hat{\Delta}_3 \in \{ C24,C25 \}$.  But $l_1=l_2=2$ implies $\hat{\Delta}_1=\hat{\Delta}_2=R2$ and this forces LEC.

Let $l=2$ and $n_2=7$.  This forces $k \leq 5$, and since $k \leq 2$ implies $d(\hat{\Delta}) < 10$ it can be assumed that $3 \leq k \leq 5$.
Therefore $\hat{\Delta}_3 \in \{ C23A,C23B,C23C,C24,C25 \}$; and $(\hat{\Delta}_1,\hat{\Delta}_2) \in \{ (R2,R3A),(R2,R3B),(R3A,R2),(R3B,R2)\}$.
Each case either yields LAC or yields LEC.

Let $l=3$ and $n_2=7$.  Then $k \leq 7$ and since $k \leq 2$ implies $d(\hat{\Delta}) < 10$ it can be assumed that $3 \leq k \leq 7$.
Therefore $\hat{\Delta}_1=\hat{\Delta}_2=R2$ and this yields LAC except when $\hat{\Delta}_3 \in \{ C33C,C35A,C35B,C37 \}$ and each of these four cases yields LEC.

Now let $n_2=8$.  Then $d(\hat{\Delta}) \geq 10$ forces $k \geq 2$.  Since $l=0$ implies $d(\hat{\Delta}) < 10$ it follows that
$1 \leq l \leq 4$.

Let $l=4$.  Then $\hat{\Delta}_1=\hat{\Delta}_2=R2$ and so $\hat{\Delta}_3$ is one of $C43B,C,D,E$ or $C45A,B,C,D$ or $C47B,C,D$ or $C49$ and each case forces LEC.

Let $l=3$.  Then $(l_1,l_2) \in \{(2,3),(3,2)\}$.  Checking each $C3$ for $\hat{\Delta}_3$ yields LEC except for $\hat{\Delta}_3=C34D$ and the region $\hat{\Delta}$ is given by Figure 48(i).

Let $l=2$.  Then $(l_1,l_2) \in \{ (2,4),(3,3),(4,2) \}$.  Checking each C2 case for $\hat{\Delta}_3$ yields LEC except when 

\newpage
\begin{figure}[t]
\begin{center}
\psfig{file=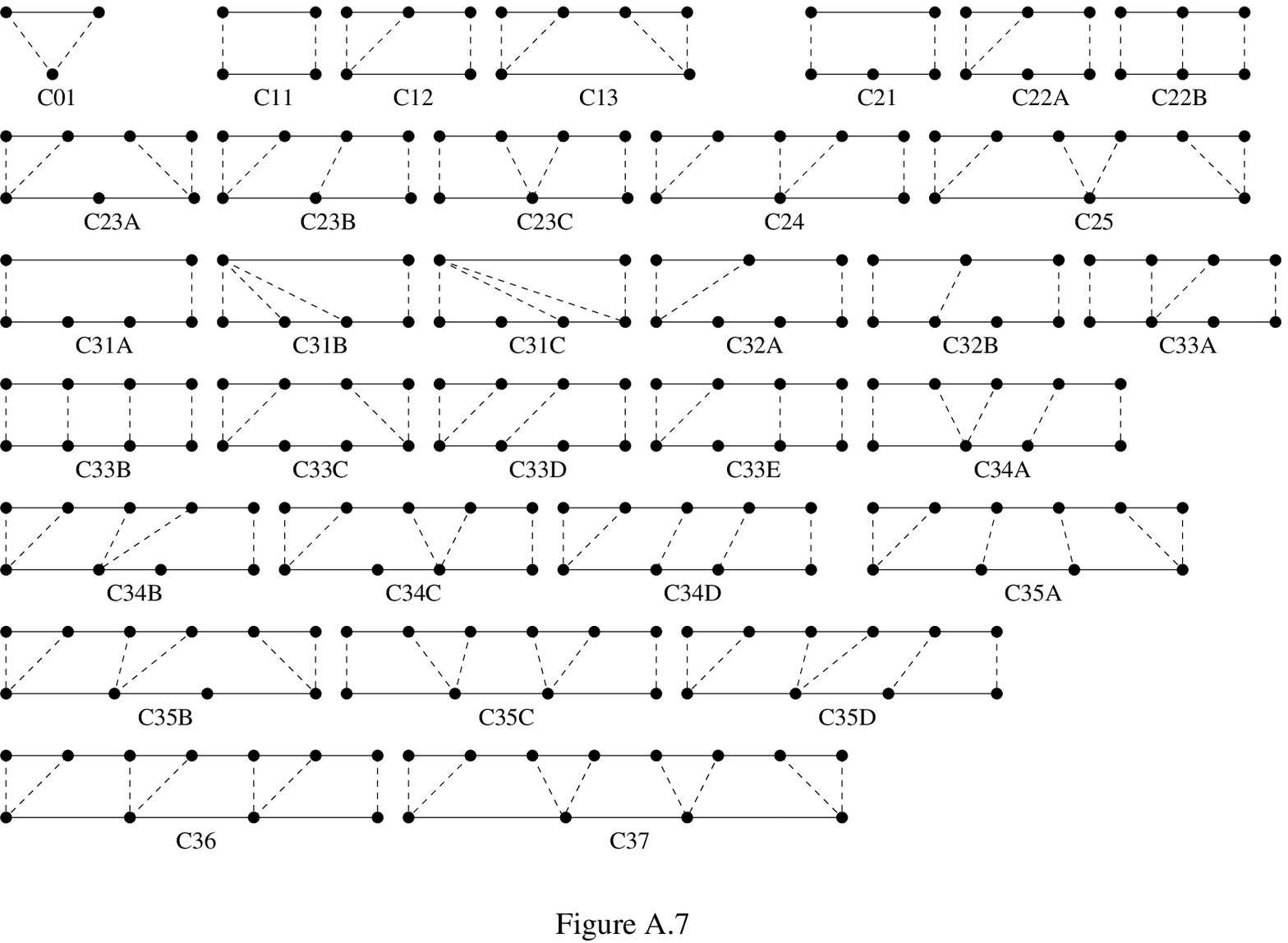}
\end{center}
\end{figure}

\noindent $(\hat{\Delta}_1,\hat{\Delta}_3,\hat{\Delta}_2) \in \{(R2,C22A,R4A),(R2,C24,R4A),(R3A,C22B,R3A),(R3B,C22B,R3B)\}$.
Now $(R2,C22A,R4A),(R2,C24,R4A)$ yields the region $\hat{\Delta}$ of Figure 48(ii)--(iii); $(R3A,C22B,R3A)$ yields the region $\hat{\Delta}$ of Figure 48(iv); and $(R3B,C22B,R3B)$ forces LAC.

Let $l=1$.  Then $(l_1,l_2) \in \{(2,5),(3,4),(4,3),(5,2)\}$ and $\hat{\Delta}_3 \in \{ C12,C13 \}$.  Each case forces LAC or LEC except for\newline
$(\hat{\Delta}_1,\hat{\Delta}_3,\hat{\Delta}_2) \in \{ (R2,C12,R5E),(R2,C12,R5F),(R2,C12,R59),(R3A,C12,R4)\}$ and the four resulting regions $\hat{\Delta}$ are given by Figure 48(v)-(viii).

\medskip

\item[(vi)]
As in (v) the region $\hat{\Delta}$ is given by Figure A.5.
In this case however Lemma 11.1 and Figure 46 can be applied.
Therefore if $l_j \leq 5$ ($j=1,2$) then\newline
$\hat{\Delta}_j \in \{ R2,R3A,R4A,R4C,R5A,R5E,R5F,R5G \}$.
A further simple check now shows that if $l_j=6$ then $\hat{\Delta}_j$ is given by Figure A.9; and if $l_1=7$, say, then $l_2=2$, $\hat{\Delta}_3=CO1$ and this forces $\hat{\Delta}_1=R7$ of Figure A.9.

Let $l=0$ so that $\hat{\Delta}_3=C01$.  Then $(l_1,l_2) \in \{ (2,7),(3,6),(4,5)\}$.
If $(l_1,l_2)=(2,7)$ 

\newpage
\begin{figure}[ht]
\begin{center}
\psfig{file=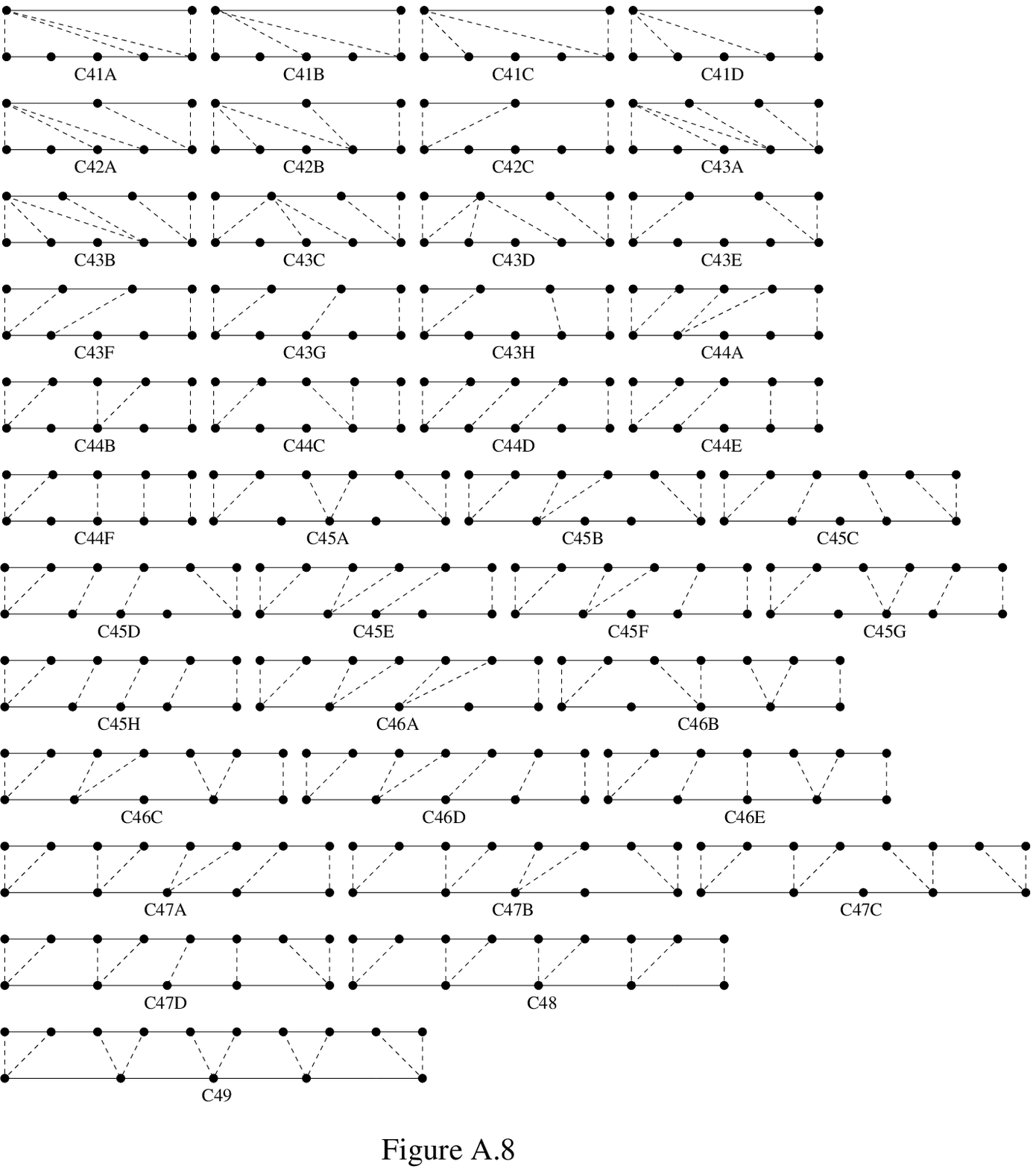}
\end{center}
\end{figure}

\noindent we see from the previous paragraph that $\hat{\Delta}$ is given by Figure 49(i).
If $(l_1,l_2)=(3,6)$ then $\hat{\Delta}_1=R3A$ and this forces $\hat{\Delta}_3=R6J$ and $\hat{\Delta}$ to be given by Figure 49(ii).
If $(l_1,l_2)=(4,5)$ then there is no choice for $\hat{\Delta}_3$ when $\hat{\Delta}_1=R4C$ but if $\hat{\Delta}_1=R4A$ then
$\hat{\Delta}_3=R5A$ and $\hat{\Delta}$ is given by Figure 49(iii).

Let $l=1$ and so $(l_1,l_2) \in \{ (2,6),(3,5),(5,3),(6,2),(4,4)\}$.
If $(l_1,l_2) \in \{ (2,6),(6,2) \}$ then $\hat{\Delta}_3=C12$ or $C13$ and up to symmetry $\hat{\Delta}_1=R2$.
But now checking the $C6$ case yields LEC for each choice of $\hat{\Delta}_2$.
If $(l_1,l_2)=(3,5)$ then $\hat{\Delta}_1=R3A$ and now any pairing of $\hat{\Delta}_3,\hat{\Delta}_2$ forces LEC except for
$(\hat{\Delta}_3,\hat{\Delta}_2)=(C13,R5A)$ and $\hat{\Delta}$ is given by Figure 49(iv).
If $(l_1,l_2)=(5,3)$ then it can be assumed that $\hat{\Delta}_3=C12$ and 

\newpage
\begin{figure}[t]
\begin{center}
\psfig{file=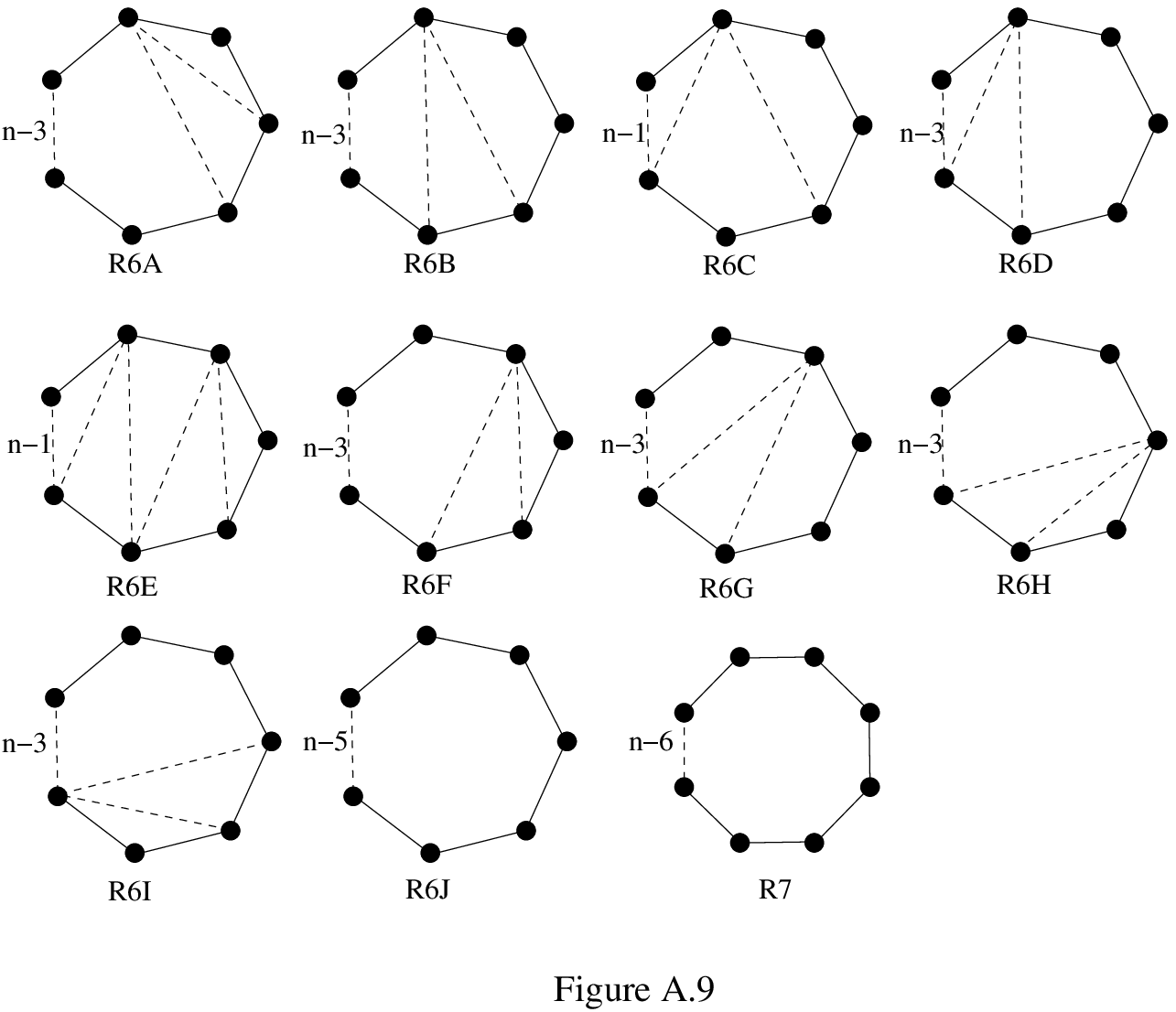}
\end{center}
\end{figure}

\noindent each choice of $\hat{\Delta}_1$ forces LEC.
If $(l_1,l_2)=(4,4)$ this forces LEC except for $(\hat{\Delta}_1,\hat{\Delta}_3,\hat{\Delta}_2)=(R4A,C13,R4A)$ and $\hat{\Delta}$ is given by Figure 49(v).

Let $l=2$ and so $(l_1,l_2) \in \{(2,5),(3,4),(4,3),(5,2)\}$.
If $(l_1,l_2) \in \{ (2,5),(5,2) \}$ then it can be assumed that $\hat{\Delta}_1=R2$ and now each choice of $\hat{\Delta}_2$ forces LEC except for when
$\hat{\Delta}_3=C23A$, $\hat{\Delta}_2=R5A$ and $\hat{\Delta}$ is given by Figure 49(vi).
If $(l_1,l_2) \in \{ (3,4),(4,3)\}$ then each case forces LEC except when $\hat{\Delta}_3=C23A$ or $C25$ and up to symmetry $\hat{\Delta}$ is given by Figure 49(vii), (viii).

Let $l=3$ and so $(l_1,l_2) \in \{(2,4),(3,3),(4,2)\}$.
If $(l_1,l_2) \in \{ (2,4),(4,2) \}$ it can be assumed that $\hat{\Delta}_1=R2$ except when $\hat{\Delta}_3=C35B$.
When $\hat{\Delta}_1=R2$ each case yields LEC except for $(\hat{\Delta}_3,\hat{\Delta}_2)=C33C,R4A)$ and $\hat{\Delta}$ is given by Figure 49(ix);
and $(\hat{\Delta}_3,\hat{\Delta}_2)=(C35B,R4A)$ or $(C37,R4A)$ each forcing LAC.
When $\hat{\Delta}_2=R2$ then LEC except when $\hat{\Delta}_1=R4A$ and $\hat{\Delta}$ is given by Figure 49(x).
If $(l_1,l_2)=(3,3)$ then $\hat{\Delta}_1=\hat{\Delta}_2=R3A$ forcing $\hat{\Delta}_3 \in \{ C33C,C35B,C37 \}$ and $\hat{\Delta}$ is given by Figure 49(xi)-(xiii).

Let $l=4$ so that $(l_1,l_2)=(2,3)$ or $(3,2)$.  Then length forces\newline
$\hat{\Delta}_3 \in \{C43E,C45A,C45B,C47B,C47C,C49\}$.
But $C47C$ and $C49$ each forces LAC; $C45B$ forces LAC when $(l_1,l_2)=(3,3)$; and $C47B$ forces LAC when $(l_1,l_2)=(2,3)$.
It follows that up to symmetry $\hat{\Delta}$ is given by Figure 49(xiv)-(xvii).

Finally let $l=5$.  This forces $(l_1,l_2)=(2,2)$ and Figure 46 immediately yields LAC. $\Box$
\end{enumerate}

\end{document}